\documentclass[a4paper, reqno, 14pt]{amsart}

\usepackage[usenames,dvipsnames]{color}

\usepackage{amsthm,amsfonts,amssymb,amsmath,amsxtra}
\usepackage[all]{xy}
\SelectTips{cm}{}
\usepackage{xr-hyper}
\usepackage[colorlinks=
   citecolor=Black,
   linkcolor=Red,
   urlcolor=Blue]{hyperref}
\usepackage{verbatim}

\usepackage[margin=1.25in]{geometry}

\usepackage{mathrsfs}

\RequirePackage{xspace}
\RequirePackage{etoolbox}
\RequirePackage{varwidth}
\RequirePackage{enumitem}
\RequirePackage{tensor}
\RequirePackage{mathtools}
\RequirePackage{longtable}
\RequirePackage{multirow}
\RequirePackage{bigstrut}
\RequirePackage{bigints}

\newcommand{\fkb}{\ensuremath{\mathfrak{b}}\xspace}
\newcommand{\fkc}{\ensuremath{\mathfrak{c}}\xspace}

\newcommand{\fkg}{\ensuremath{\mathfrak{g}}\xspace}

\newcommand{\fkk}{\ensuremath{\mathfrak{k}}\xspace}
\newcommand{\fkl}{\ensuremath{\mathfrak{l}}\xspace}

\newcommand{\fks}{\ensuremath{\mathfrak{s}}\xspace}

\newcommand{\fku}{\ensuremath{\mathfrak{u}}\xspace}

\newcommand{\club}{{\clubsuit}}
\newcommand{\nat}{{\natural}}

\newcommand{\spade}{{\spadesuit}}

\newcommand{\BA}{\ensuremath{\mathbb {A}}\xspace}

\newcommand{\BC}{\ensuremath{\mathbb {C}}\xspace}

\newcommand{\BE}{\ensuremath{\mathbb {E}}\xspace}

\newcommand{\BG}{\ensuremath{\mathbb {G}}\xspace}

\newcommand{\BL}{\ensuremath{\mathbb {L}}\xspace}
\newcommand{\BM}{\ensuremath{\mathbb {M}}\xspace}

\newcommand{\BP}{\ensuremath{\mathbb {P}}\xspace}
\newcommand{\BQ}{\ensuremath{\mathbb {Q}}\xspace}

\newcommand{\BV}{\ensuremath{\mathbb {V}}\xspace}

\newcommand{\BX}{\ensuremath{\mathbb {X}}\xspace}

\newcommand{\BZ}{\ensuremath{\mathbb {Z}}\xspace}

\newcommand{\CC}{\ensuremath{\mathcal {C}}\xspace}
\newcommand{\CD}{\ensuremath{\mathcal {D}}\xspace}
\newcommand{\CE}{\ensuremath{\mathcal {E}}\xspace}
\newcommand{\CF}{\ensuremath{\mathcal {F}}\xspace}

\newcommand{\CM}{\ensuremath{\mathcal {M}}\xspace}
\newcommand{\CN}{\ensuremath{\mathcal {N}}\xspace}
\newcommand{\CO}{\ensuremath{\mathcal {O}}\xspace}

\newcommand{\CS}{\ensuremath{\mathcal {S}}\xspace}
\newcommand{\CT}{\ensuremath{\mathcal {T}}\xspace}
\newcommand{\CU}{\ensuremath{\mathcal {U}}\xspace}

\newcommand{\CW}{\ensuremath{\mathcal {W}}\xspace}

\newcommand{\CY}{\ensuremath{\mathcal {Y}}\xspace}
\newcommand{\CZ}{\ensuremath{\mathcal {Z}}\xspace}

\newcommand{\RM}{\ensuremath{\mathrm {M}}\xspace}
\newcommand{\RN}{\ensuremath{\mathrm {N}}\xspace}

\newcommand{\RU}{\ensuremath{\mathrm {U}}\xspace}

\DeclareMathOperator{\charac}{char}

\newcommand{\corr}{\mathrm{corr}}

\newcommand{\del}{\operatorname{\partial Orb}}
\DeclareMathOperator{\diag}{diag}

\DeclareMathOperator{\dist}{dist}

\DeclareMathOperator{\End}{End}

\newcommand{\Fil}{\ensuremath{\mathrm{Fil}}\xspace}

\newcommand{\GL}{\mathrm{GL}}
\newcommand{\gl}{\frak{gl}}

\newcommand{\GU}{\mathrm{GU}}

\DeclareMathOperator{\Hom}{Hom}

\newcommand{\id}{\ensuremath{\mathrm{id}}\xspace}
\let\Im\relax
\DeclareMathOperator{\Im}{Im}

\newcommand{\inj}{\hookrightarrow}
\newcommand{\Int}{\ensuremath{\mathrm{Int}}\xspace}
\newcommand{\lInt}{\ensuremath{\text{$\ell$-}\mathrm{Int}}\xspace}

\DeclareMathOperator{\Ker}{Ker}

\DeclareMathOperator{\length}{length}
\DeclareMathOperator{\Lie}{Lie}
\newcommand{\loc}{\ensuremath{\mathrm{loc}}\xspace}

\newcommand{\M}{\mathrm{M}}

\newcommand{\naive}{\ensuremath{\mathrm{naive}}\xspace}

\DeclareMathOperator{\Nm}{Nm}

\newcommand{\OGr}{\mathrm{OGr}}
\DeclareMathOperator{\Orb}{Orb}
\DeclareMathOperator{\ord}{ord}

\DeclareMathOperator{\rank}{rank}

\newcommand{\PGL}{{\mathrm{PGL}}}

\newcommand{\rc}{\ensuremath{\mathrm{rc}}\xspace}
\renewcommand{\Re}{{\mathrm{Re}}}
\newcommand{\red}{\ensuremath{\mathrm{red}}\xspace}

\DeclareMathOperator{\Res}{Res}
\newcommand{\rs}{\ensuremath{\mathrm{rs}}\xspace}

\newcommand{\SL}{{\mathrm{SL}}}
\DeclareMathOperator{\Spec}{Spec}
\DeclareMathOperator{\Spf}{Spf}

\newcommand{\Sp}{{\mathrm{Sp}}}
\newcommand{\SU}{{\mathrm{SU}}}
\DeclareMathOperator{\supp}{supp}

\DeclareMathOperator{\tr}{tr}

\newcommand{\U}{\mathrm{U}}

\DeclareMathOperator{\vol}{vol}

\newcommand{\wt}{\widetilde}

\newcommand{\ov}{\overline}
\newcommand{\incl}{\hookrightarrow}
\newcommand{\lra}{\longrightarrow}

\newcommand{\bs}{\backslash}

\newcommand{\uF}{\underline{F}}
\newcommand{\uV}{\underline{V}}

\newtheorem{theorem}{Theorem}
\newtheorem{proposition}[theorem]{Proposition}
\newtheorem{lemma}[theorem]{Lemma}
\newtheorem {conjecture}[theorem]{Conjecture}
\newtheorem {`conjecture'}[theorem]{``Conjecture''}
\newtheorem{corollary}[theorem]{Corollary}

\theoremstyle{definition}
\newtheorem{definition}[theorem]{Definition}
\newtheorem{example}[theorem]{Example}

\newtheorem{remark}[theorem]{Remark}
\newtheorem{remarks}[theorem]{Remarks}

\newenvironment{altenumerate}
   {\begin{list}
      {\textup{(\theenumi)} }
      {\usecounter{enumi}
       \setlength{\labelwidth}{0pt}
       \setlength{\labelsep}{0pt}
       \setlength{\leftmargin}{0pt}
       \setlength{\itemsep}{\the\smallskipamount}
       \renewcommand{\theenumi}{\roman{enumi}}
      }}
   {\end{list}}
\newenvironment{altenumerate2}
   {\begin{list}
      {\textup{(\theenumii)} }
      {\usecounter{enumii}
       \setlength{\labelwidth}{0pt}
       \setlength{\labelsep}{0pt}
       \setlength{\leftmargin}{2em}
       \setlength{\itemsep}{\the\smallskipamount}
       \renewcommand{\theenumii}{\alph{enumii}}
      }}
   {\end{list}}

\newenvironment{altitemize}
   {\begin{list}
      {$\bullet$}
      {\setlength{\labelwidth}{0pt}
	   \setlength{\itemindent}{5pt}
       \setlength{\labelsep}{5pt}
       \setlength{\leftmargin}{0pt}
       \setlength{\itemsep}{\the\smallskipamount}
      }}
   {\end{list}}

\numberwithin{equation}{section}
\numberwithin{theorem}{section}

\newcommand{\aform}{\ensuremath{\langle\text{~,~}\rangle}\xspace}

\newcounter{filler}

\setitemize[0]{leftmargin=*,itemsep=\the\smallskipamount}
\setenumerate[0]{leftmargin=*,itemsep=\the\smallskipamount}

\renewcommand{\to}{%
   \ifbool{@display}{\longrightarrow}{\rightarrow}%
   }
\let\shortmapsto\mapsto
\renewcommand{\mapsto}{%
   \ifbool{@display}{\longmapsto}{\shortmapsto}%
   }
\newlength{\olen}
\newlength{\ulen}
\newlength{\xlen}
\newcommand{\xra}[2][]{%
   \ifbool{@display}%
      {\settowidth{\olen}{$\overset{#2}{\longrightarrow}$}%
       \settowidth{\ulen}{$\underset{#1}{\longrightarrow}$}%
       \settowidth{\xlen}{$\xrightarrow[#1]{#2}$}%
       \ifdimgreater{\olen}{\xlen}%
          {\underset{#1}{\overset{#2}{\longrightarrow}}}%
          {\ifdimgreater{\ulen}{\xlen}%
             {\underset{#1}{\overset{#2}{\longrightarrow}}}
             {\xrightarrow[#1]{#2}}}}%
      {\xrightarrow[#1]{#2}}
   }
\makeatother
\newcommand{\xyra}[2][]{%
   \settowidth{\xlen}{$\xrightarrow[#1]{#2}$}%
   \ifbool{@display}%
      {\settowidth{\olen}{$\overset{#2}{\longrightarrow}$}%
       \settowidth{\ulen}{$\underset{#1}{\longrightarrow}$}%
       \ifdimgreater{\olen}{\xlen}%
          {\mathrel{\xymatrix@M=.12ex@C=3.2ex{\ar[r]^-{#2}_-{#1} &}}}%
          {\ifdimgreater{\ulen}{\xlen}%
             {\mathrel{\xymatrix@M=.12ex@C=3.2ex{\ar[r]^-{#2}_-{#1} &}}}
             {\mathrel{\xymatrix@M=.12ex@C=\the\xlen{\ar[r]^-{#2}_-{#1} &}}}}}%
      {\mathrel{\xymatrix@M=.12ex@C=\the\xlen{\ar[r]^-{#2}_-{#1} &}}}%
   }
\makeatletter
\newcommand{\xla}[2][]{%
   \ifbool{@display}%
      {\settowidth{\olen}{$\overset{#2}{\longleftarrow}$}%
       \settowidth{\ulen}{$\underset{#1}{\longleftarrow}$}%
       \settowidth{\xlen}{$\xleftarrow[#1]{#2}$}%
       \ifdimgreater{\olen}{\xlen}%
          {\underset{#1}{\overset{#2}{\longleftarrow}}}%
          {\ifdimgreater{\ulen}{\xlen}%
             {\underset{#1}{\overset{#2}{\longleftarrow}}}
             {\xleftarrow[#1]{#2}}}}%
      {\xleftarrow[#1]{#2}}
   }
\newcommand{\isoarrow}{%
   \ifbool{@display}{\overset{\sim}{\longrightarrow}}{\xrightarrow\sim}%
   }
\renewcommand{\lra}{%
   \ifbool{@display}{\longleftrightarrow}{\leftrightarrow}%
   }   
\newcommand{\undertilde}{\raisebox{0.4ex}{\smash[t]{$\scriptstyle\sim$}}}

\begin{document}


\title[On the arithmetic transfer conjecture for exotic smooth moduli spaces]{On the arithmetic transfer conjecture for exotic smooth formal moduli spaces}
\author{M. Rapoport}
\address{Mathematisches Institut der Universit\"at Bonn, Endenicher Allee 60, 53115 Bonn, Germany}
\email{rapoport@math.uni-bonn.de}
\author{B. Smithling}
\address{Johns Hopkins University, Department of Mathematics, 3400 N.\ Charles St.,\ Baltimore, MD  21218, USA}
\email{bds@math.jhu.edu}
\author{W. Zhang}
\address{Columbia University, Department of Mathematics, 2990 Broadway, New York, NY 10027, USA}
\email{wzhang@math.columbia.edu}


\begin{abstract}

In the relative trace formula approach to the arithmetic Gan--Gross--Prasad conjecture, we formulate a local conjecture (arithmetic transfer) in the case of an exotic smooth formal moduli space of $p$-divisible groups, associated to a unitary group relative to a ramified quadratic extension of a $p$-adic field. We prove our conjecture in the case of a unitary group in three variables. 
\end{abstract}

\maketitle

\tableofcontents
\section{Introduction}\label{Intro section}

The theorem of Gross and Zagier \cite{GZ} relates the Neron--Tate heights of Heegner points on modular curves to special values of derivatives of certain $L$-functions. This has been generalized in various ways to higher-dimensional Shimura varieties. One such generalization, which is still conjectural, has been proposed by Gan--Gross--Prasad \cite{GGP} and the third-named author \cite{Z09,Z12}. This arithmetic Gan--Gross--Prasad conjecture is inspired by the (usual) Gan--Gross--Prasad conjecture relating period integrals on classical groups to special values of certain $L$-functions.  In \cite{JR} Jacquet and Rallis proposed a relative trace formula approach to this last conjecture in the case of unitary groups, which led them to formulate two local conjectures in this context: a fundamental lemma (FL) conjecture, and a smooth transfer (ST) conjecture. Both of their local conjectures are now proved to a large extent, the first for $p\gg0$ thanks to the work of Yun \cite{Y} (and Gordon \cite{Go}), and the second for arbitrary $p$-adic non-archimedean fields by  the third-named author \cite{Z14}.

In \cite{Z12} the third-named author proposed a relative trace formula approach to the arithmetic Gan--Gross--Prasad conjecture. In this context, he formulated the arithmetic fundamental lemma (AFL) conjecture, cf.~\cite{Z12,RTZ}. The AFL conjecturally relates the special value of the derivative of an orbital integral to an arithmetic intersection number on a Rapoport--Zink formal moduli space of $p$-divisible groups attached to a unitary group. The AFL is proved for low ranks of the unitary group ($n=2$ and $3$) in \cite{Z12}, and for arbitrary rank $n$ and \emph{minuscule} group elements in \cite{RTZ}. A simplified proof for $n=3$ appears in \cite{M-AFL}. At present, the general case of the AFL seems out of reach, even though Yun has obtained interesting results concerning the function field analog \cite{Y12}. 

In the present paper, we address an arithmetic transfer (AT) analog of the ST conjecture in the arithmetic context, in a very specific case; we refer to \cite{RSZ} for a more general context in which we expect such arithmetic analogs of ST. The special feature of the case at hand is that, despite the fact that we take the unitary group to be ramified, the corresponding RZ space is smooth, cf.~\cite{PR}. For this reason we speak of \emph{exotic smoothness}. 

Now that we have explained the title of the paper, let us describe its contents in more detail. 

\smallskip

Let $p$ be an odd prime number, and let $F_0$ be a finite extension of $\BQ_p$. Let $F/F_0$ be a quadratic field extension. We denote by $a\mapsto \ov a$ the non-trivial automorphism of $F/F_0$, and by $\eta=\eta_{F/F_0}$ the corresponding quadratic character on $F_0^\times$. Let $e:=(0,\dotsc,0,1)\in F_0^n$, and let $\GL_{n-1}\incl \GL_n$ be the natural embedding that identifies $\GL_{n-1}$ with the subgroup fixing $e$ under left multiplication, and fixing the transposed vector $\tensor[^t]{e}{}$ under right multiplication. Let 
\[
   S_n := \{\, s\in \Res_{F/F_0}\GL_n \mid s\ov s=1 \,\},
\]
with its action by conjugation of $\GL_{n-1}$. On the other hand, let $W_0$ and $W_1$ be the respective split and non-split non-degenerate $F/F_0$-hermitian spaces of dimension $n$. For $i=0$ and $1$, fix a vector $u_i\in W_i$ of length $1$, and denote by $W_i^\flat$ the orthogonal complement of the line spanned by $u_i$. The unitary group $\U(W_i^\flat)$ acts by conjugation on $\U(W_i)$.

We now explain the matching relation between regular semi-simple elements of $S_n(F_0)$ and of $\U(W_0)(F_0)$ and $\U(W_1)(F_0)$. Here an element of  $S_n(F_0)$, resp.~of $\U(W_i)(F_0)$, is called \emph{regular semi-simple} (rs) if its orbit under $\GL_{n-1}$, resp.~$\U(W_i^\flat)$, is Zariski-closed of maximal dimension.  For each $i$, choose a basis of $W_i$ by first choosing a basis of $W_i^\flat$ and then appending $u_i$ to it. This identifies $\U(W^\flat_i)(F_0)$ with a subgroup of $\GL_{n-1}(F)$ and  $\U(W_i)(F_0)$ with a subgroup of $\GL_{n}(F)$. An element $\gamma\in S_n(F_0)_\rs$ is said to \emph{match} an element $g\in \U(W_i)(F_0)_\rs$ if both elements are conjugate under $\GL_{n-1}(F)$ when considered as elements in $\GL_n(F)$. This matching relation induces a bijection 
\[
   \big[ \U(W_0)(F_0)_\rs\big]\amalg \big[\U(W_1)(F_0)_\rs\big]\simeq \big[S_n(F_0)_\rs\big],
\]
cf.~\cite[\S2]{Z12}, where the brackets indicate the sets of orbits under $\U(W_i^\flat)(F_0)$, resp.~$\GL_{n-1}(F_0)$.

Dual to the matching of elements is the transfer of functions, which is defined through weighted, resp.~ordinary, orbital integrals. For a function $f'\in C_c^\infty(S_n(F_0))$, an element $\gamma\in S_n(F_0)_\rs$, and a complex parameter $s\in\BC$, we define the weighted orbital integral
\[
   \Orb(\gamma,f', s) := \int_{\GL_{n-1}(F_0)}f'(h^{-1}\gamma h)\lvert\det h \rvert^s\eta(\det h) \, dh,
\]
as well as its special value 
\[
   \Orb(\gamma,f') := \Orb(\gamma,f', 0) .
\]
Here the Haar measure on $\GL_{n-1}(F_0)$ is normalized so that $\vol(\GL_{n-1}(O_{F_0}))=1$. 
For a function $f_i\in C_c^\infty(\U(W_i)(F_0))$ and an element $g\in  \U(W_i)(F_0)_\rs$, we define the orbital integral
\[
   \Orb(g,f_i) := \int_{\U(W^\flat_i)(F_0)}f_i(h^{-1} g h)\, dh .
\]
Then the function $f'\in C_c^\infty(S_n(F_0))$ is said to \emph{transfer} to the pair of functions $(f_0,f_1)$ in $ C_c^\infty(\U(W_0)(F_0))\times  C_c^\infty(\U(W_1)(F_0))$
if 
\[
   \omega(\gamma)\Orb(\gamma,f')=\Orb(g,f_i)
\]
whenever $\gamma\in S_n(F_0)_\rs$ matches the element $g\in \U(W_i)(F_0)_\rs$. Here 
\[
   \omega\colon S_n(F_0)_\rs\to \BC^\times
\]
is a fixed transfer factor \cite[p.~988]{Z14}, and the Haar measures on $\U(W_i^\flat)(F_0)$ are fixed. The ST conjecture asserts that for any $f'$, a transfer $(f_0,f_1)$ exists (non-uniquely), and that any pair $(f_0,f_1)$ arises as a transfer from some (non-unique) $f'$.  The FL conjecture asserts a specific transfer relation in a completely unramified situation.

When $F/F_0$ is unramified, if one takes for $\omega\colon S_n(F_0)_\rs \to \{\pm 1\}$ the \emph{natural} transfer factor (see \cite[(1.5)]{RTZ}), and normalizes the Haar measure on  $\U(W_0^\flat)(F_0)$ by giving a hyperspecial maximal compact subgroup volume one, then the FL conjecture asserts that $\mathbf{1}_{S_n(O_{F_0})}$  transfers to $(\mathbf{1}_{K_0},0)$, where $K_0 \subset \U(W_0)(F_0)$ denotes a hyperspecial maximal open subgroup.

By contrast, when $F/F_0$ is ramified, there is no natural choice of a conjugacy class of open compact subgroups $K_0$, no natural choice of a transfer factor, and no natural candidate for $f'$ transferring to $(\mathbf{1}_{K_0},0)$.

\smallskip

We next pass to the AFL conjecture, which requires, just as in the FL conjecture, that $F/F_0$ is unramified. We take the same transfer factor as in the FL conjecture, and the same fixed Haar measure on $\U(W^\flat_0)(F_0)$. For $f'\in C_c^\infty(S_n(F_0))$ and $\gamma \in S_n(F_0)_\rs$, set
\[
   \del(\gamma,f') := \frac d{ds} \Big|_{s=0} \Orb(\gamma,f', s).
\]
Then the AFL conjecture asserts that 
\begin{equation}\label{IntroAFL}
\omega(\gamma)\del \bigl(\gamma, \mathbf{1}_{S_n(O_{F_0})}\bigr) = -\Int(g)\cdot\log q,
\end{equation}
whenever $\gamma\in S_n(F_0)_\rs$ matches $g\in \U(W_1)(F_0)_\rs$ (note that the FL conjecture asserts that $\Orb(\gamma, \mathbf{1}_{S_n(O_{F_0})}) = 0$ for such $\gamma$). Here $q$ denotes the number of elements in the residue field of $F_0$. 

The term $\Int(g)$ requires explanation. Let $\CN_n=\CN_{F/F_0,n}$ denote the formal scheme over $\Spf O_{\breve{F}}$ which represents the following functor on the category of $O_{\breve F}$-schemes $S$ such that $p\cdot \CO_S$ is a locally nilpotent ideal sheaf. The functor associates to $S$ the set of isomorphism classes of tuples $(X,\iota,\lambda,\rho)$ where $X$ is a formal $p$-divisible $O_{F_0}$-module of relative height $2n$ and dimension $n$, where $\iota \colon O_F\to \End(X)$ is an action of $O_F$ satisfying the Kottwitz condition of signature $(1,n-1)$ on $\Lie(X)$ (cf.~\cite[\S 2]{KR-U1}), where $\lambda$ is a principal polarization whose Rosati involution induces the automorphism $a\mapsto \ov a$ on $\iota(O_F)$, and where $\rho\colon X\times_S \ov S\to \BX_n\times_{\Spec\ov k} \ov S$ is a framing of the restriction of $X$ to the special fiber $\ov S$ of $S$, compatible with $\iota$ and $\lambda$ in a certain sense, cf.~\cite[\S2]{RTZ}. Then $\CN_n$ is  formally smooth of relative formal dimension $n-1$ over $\Spf O_{\breve F}$. The automorphism group (in a certain sense) of the framing object $\BX_n$ can be identified with $\U(W_1)(F_0)$; it acts on $\CN_n$ by changing the framing. Let $\CE$ be the canonical lifting of the formal $O_F$-module of relative height $1$ and dimension $1$ over $\Spf O_{\breve F}$, with its canonical $O_F$-action $\iota_\CE$ and its natural polarization $\lambda_\CE$. There is a natural closed embedding of $\CN_{n-1}$ into $\CN_n$,
\[
   \delta_\CN\colon
	\xymatrix@R=0ex{
	   \CN_{n-1} \ar[r]  &  \CN_n\\
		Y \ar@{|->}[r]  &  Y \times \ov\CE.
	}
\]
Here all auxiliary structure ($O_F$-action, polarization, framing) has been suppressed from the notation, and $\ov\CE$ denotes $\CE$, with $\iota_\CE$ replaced by its conjugate. 
Let 
$$
   \Delta\subset \CN_{n-1}\times_{\Spf O_{\breve F}} \CN_n
$$
denote the graph of $\delta_\CN$. Then $\Int(g)$ is defined as the intersection number of $\Delta$ with its translate under the automorphism $1\times g$ of $\CN_{n-1}\times_{\Spf O_{\breve F}} \CN_n$,
\[
   \Int(g) = \chi\bigl( \CO_\Delta\otimes^\BL \CO_{(1 \times g)\Delta} \bigr) .
\]
This concludes the statement of the AFL conjecture. 

It should be true in the situation of the AFL that for \emph{any} $f'\in C_c^\infty(S_n(F_0))$ with transfer  $(\mathbf{1}_{K_0},0)$, there exists a function $f'_{\corr}\in C_c^\infty(S_n(F_0))$ such that
\begin{equation*}
   \omega(\gamma)\del (\gamma, f') = -\Int(g)\cdot\log q+\omega(\gamma)\Orb (\gamma,f'_{\corr})
\end{equation*}
whenever $\gamma\in S_n(F_0)_\rs$ matches $g\in \U(W_1)(F_0)_\rs$.
This would follow along the lines of Lemma \ref{lem a to b} below from a conjectural \emph{density principle} on weighted orbital integrals.
 See Conjecture \ref{conj density} for the statement of the density principle in the setting of this paper.

\smallskip

Now we come to the formulation of our AT conjecture. We assume for this that $F/F_0$ is \emph{ramified}. We then modify the definition of the formal moduli space $\CN_{F/F_0,n}=\CN_n$ by slightly changing the conditions on the tuples  $(X,\iota,\lambda,\rho)$. Namely, in addition to the Kottwitz condition of signature $(1,n-1)$, we impose on $\iota$ the Pappas wedge condition of signature $(1,n-1)$, and the spin condition. The latter condition states that for a uniformizer $\pi$ of $F$, the endomorphism $\iota(\pi) \mid \Lie(X)$ is nowhere zero on $S$. Furthermore, we change the condition that $\lambda$ is principal to the condition 
\[
   \Ker(\lambda) \subset X[\iota(\pi)]
	\quad\text{with}\quad
	\lvert\Ker(\lambda)\rvert = q^{2\lfloor n/2 \rfloor}.
\]
It turns out that $\CN_n$ is again formally smooth of relative formal dimension $n-1$ over $\Spf O_{\breve F}$, and is essentially proper when $n$ is even. We stress that this result is quite surprising in the presence of ramification.

The morphism $\delta_\CN\colon \CN_{n-1}\to\CN_n$ can be defined exactly as before when discussing the AFL setup, provided that $n$ is odd, since then $2\lfloor\frac{n-1}{2}\rfloor = 2\lfloor \frac{n}{2}\rfloor$. We then define $\Int(g)$ as before. Our AT conjecture is as follows. 

\begin{conjecture}\label{Intro-Conj}
Let $F/F_0$ be ramified, and let $n\geq 3$ be odd. 
\begin{altenumerate}
\renewcommand{\theenumi}{\alph{enumi}}
\item\label{intro conj a}
There exists a function $f'\in C^\infty_c(S(F_0))$ with transfer $(\mathbf{1}_{K_0}, 0)$ such that 
\begin{equation}\label{introAT}
   2 \omega(\gamma)\del(\gamma,f')=-\Int(g)\cdot\log q 
\end{equation}
for any $\gamma\in S(F_0)_\rs$ matching an element $g\in \U(W_1)(F_0)_\rs$.
\item\label{intro conj b}
For any $f'\in C^\infty_c(S(F_0))$  with transfer $(\mathbf{1}_{K_0}, 0)$, there exists a function  $f'_\corr\in C^\infty_c(S(F_0))$ such that 
$$
2 \omega(\gamma)\del(\gamma,f') = -\Int(g)\cdot\log q + \omega(\gamma)\Orb(\gamma,f'_\corr)
$$ for any $\gamma\in S(F_0)_\rs$ matching an element $g\in \U(W_1)(F_0)_\rs$.
\end{altenumerate}
\end{conjecture}

Here $K_0$ denotes the maximal compact subgroup stabilizing a \emph{nearly $\pi$-modular lattice} $\Lambda_0$ in $W_0$ (see \eqref{Lambda_0} below), and $\omega$ is the transfer factor defined in \eqref{sign Sn} below. The Haar measure on $\U(W^\flat_0)(F_0)$ is defined by $\vol K^\flat_0=1$ for a special maximal compact subgroup $K^\flat_0$ of $\U(W^\flat_0)(F_0)$. 
We also formulate a ``homogeneous'' variant of the AT conjecture (Conjecture \ref{conj homog}), which we show is equivalent to the above conjecture in \S\ref{inhom setting section}. Note that between the statements of the AFL conjecture \eqref{IntroAFL} and the AT conjecture \eqref{introAT}, there is a discrepancy of a factor of $2$. This is a genuine difference between the unramified and ramified cases, and we refer to \cite{RSZ3} for an explanation by way of a global comparison between the height pairing and the derivative of a relative trace formula.

Our main result concerns the first non-trivial case $n=3$ of the AT conjecture. More precisely, we prove the following.

\begin{theorem}\label{Intr-Mainthm}
 Let $F_0=\BQ_p$, and let $n=3$. Then Conjecture \ref{Intro-Conj} holds true. In addition, for any $g\in \U(W_1)(F_0)_\rs$, the intersection of $\Delta$ and $(1\times g)\Delta$, if non-empty,  is an artinian scheme with two points, and $\Int(g)={\rm length} (\Delta\cap (1\times g)\Delta)$. 
 \end{theorem}

We also prove a Lie algebra version of the above theorem; see Theorem \ref{thm lie}.

\smallskip

Let us comment on the proof of Theorem \ref{Intr-Mainthm}. In those cases in which the AFL conjecture has been established, the proof proceeds by calculating explicitly both sides of the conjectured identity and comparing the results. This approach fails for the AT conjecture because the left-hand side of the identity is not well-determined by the pair of transfer functions $(\mathbf{1}_{K_0}, 0)$. Unlike the AFL situation, there is no canonical choice for the function $f'$; in fact, $f'$ cannot come from the Iwahori Hecke algebra, cf.~Remark \ref{rem pro-u}\eqref{it pro-u}.  In particular, we note that the characteristic function $\mathbf{1}_{S(O_{F_0})}$ has vanishing orbital integrals at \emph{all} regular semi-simple elements in the ramified setting, i.e.~it transfers to $(0,0)$.  Instead we prove part \eqref{intro conj b} of Conjecture \ref{Intro-Conj} (in the case $F_0=\BQ_p$ and $n=3$) by showing that, for any $f'$ as in the statement, the sum 
\begin{equation}\label{introsum}
2 \omega(\gamma)\del (\gamma, f')+\Int(g)\cdot\log q
\end{equation}
is an \emph{orbital integral function}, i.e.~of the form $\omega(\gamma)\Orb(\gamma, f'_\corr)$ for a suitable function $f'_\corr$. Then part \eqref{intro conj a} of Conjecture \ref{Intro-Conj} follows easily.

To prove that \eqref{introsum} is an orbital integral function, we first remark that \eqref{introsum} may be viewed as a function on an open subset of the categorical quotient of $S_n$ by $\GL_{n-1}$, and then use the fact \cite{Z14} that the desired property may be checked locally on the base. To achieve this goal, we proceed in two steps. First, we determine explicitly $\Int(g)$. Second, we develop a \emph{germ expansion} around each point of the categorical quotient,  which is sufficiently explicit that it determines $\omega(\gamma)\del (\gamma, f')$ up to a \emph{local orbital integral function}. Putting these two steps together, we check that \eqref{introsum} is an orbital integral function. The description just given is inaccurate, insofar as we first perform a reduction to a Lie algebra analog. Here the \emph{Cayley transform} from \cite{Z14} plays a key role. For the first step we use, similarly to \cite{KR-U1}, the results of Gross and Keating \cite{ARGOS} on \emph{quasi-canonical liftings} (the reduction to \cite{ARGOS} in \cite[\S8]{KR-U1} is transposed here to the ramified case; in fact, we found a drastic simplification of the proof (due to Zink) in loc.~cit., which applies to both the unramified and the ramified case). For the second step, we base ourselves on the results on local harmonic analysis in \cite{Z12b}, which we complete and make more explicit in various ways.

Additionally, let us point out two group-theoretic features in the case $n=3$ which seem to be important. The first is the exceptional isomorphism $\SL_2\simeq \SU_2$. One geometric manifestation of this is that there is a natural isomorphism between \CM and a connected component of $\CN_2$, where $\CM$ is the Lubin--Tate deformation space over $\Spf O_{\breve F}$ of the formal $O_{F_0}$-module of dimension $1$ and height $2$; see Proposition \ref{Serre isom}. 
To state the second feature, we note that the aforementioned maximal compact subgroups $K_0$ and $K_0^\flat$ have symplectic reductions; see Remark \ref{rem syp red}. When $n=3$, associated to the reduction of $K_0^\flat$ is a second exceptional isomorphism $\Sp_2\simeq \SL_2$, which plays a role in reducing the conjecture to a Lie algebra version; see the proof of Lemma \ref{lem cayley U}. 

Let us also remark on the restriction to $F_0 = \BQ_p$ in Theorem \ref{Intr-Mainthm}.  While we certainly expect that the overall framework of this paper should be valid for any $p$-adic field $F_0$, there are a few instances, all of which occur when working with $\CN_n$ or related formal schemes, where we need to appeal to results in the literature which are only established at the level of generality of $q = p$ or $F_0 = \BQ_p$.  Indeed, strictly speaking, this is already the case for the representability result in \cite{RZ} 
which is needed to know
that $\CN_n$ is a formal scheme in the first place! In accordance with our expectations, we will use the general notation $q$ and $F_0$ throughout the paper, but when working in a context where formal schemes are present, we will always tacitly take $q = p$ and $F_0 = \BQ_p$. By contrast, the 
parts of this paper lying in the realm of harmonic analysis
are valid without any restriction on $F_0$.

\smallskip

Now let us comment on the possibility of extending our main result to odd integers $n>3$. The difficulties seem formidable. First, one would have to deal with degenerate intersections. Related to this is the fact that the reduction procedure to a Lie algebra analog breaks down. In fact, we are unable to even formulate a conjectural Lie algebra version of Conjecture \ref{Intro-Conj}, since we are lacking a reasonable definition of an intersection multiplicity in this context; see Conjecture \ref{lieconj}  below, in which we have to assume that the intersection is artinian.  The second difficulty is that our knowledge of local harmonic analysis when $n>3$ is not advanced enough; even a germ expansion principle is missing beyond the case of $n=3$ \cite{Z12b}.  
One possibility for making further progress would be to consider Conjecture \ref{Intro-Conj} only for elements $\gamma$ and $g$ that satisfy certain simplifying restrictions, in the spirit of \cite{RTZ} 
(which considers only minuscule elements).

On the positive side, there are other instances of AT conjectures. Indeed, in \cite{RSZ}, we formulate AT conjectures for $F/F_0$ ramified (as in the present paper) and $n$ even, and also 
for $F/F_0$ unramified and any $n$.  The methods developed in the present paper can be applied to some low-dimensional cases of them, 
cf.~\cite{RSZ}.

Let us finally comment on the global motivation behind all of these conjectures.  
The AFL was introduced by the third-named author in \cite{Z12} in the context of a relative trace formula approach to the hermitian case of the arithmetic Gan--Gross--Prasad conjecture, to facilitate a comparison between geometric and analytic sides at \emph{inert} primes of the unitary group.  In forthcoming work \cite{RSZ3} we will explain how our AT conjectures facilitate analogous comparisons in this framework, which in particular encompass the \emph{ramified} primes of the unitary group.  We refer to loc.\ cit.\ for details.

\smallskip

We now give an overview of the contents of this paper. The paper consists of four parts. 

In Part 1, we give the group-theoretic setup (in its homogeneous, its inhomogeneous, and its Lie algebra versions); we define the formal moduli spaces of $p$-divisible groups, and establish some structural properties for them; and we define the arithmetic intersection numbers that enter into the formulation of our conjectures and results. In \S\ref{conj and results}, we formulate our main results. 

In Part 2, we explicitly calculate the arithmetic intersection numbers in the case $n = 3$, by reduction to the Lie algebra and by relating this case to the Gross--Keating formulas. 

In Part 3, we tackle the left-hand side of the identities to be proved in Theorem \ref{Intr-Mainthm}. This is done by reducing the problem to one on the \emph{reduced subset} of the Lie algebra. The rest of part 3 is devoted to explicitly evaluating the germ expansion of the orbital integral of a function $f'$ with transfer $({\bf 1}_{K_0}, 0)$, and then making the comparison with the result of part 2. At the end of part 3, the proof of Theorem \ref{Intr-Mainthm} is complete.

In Part 4, we prove the germ expansion of the orbital integral of a general function $f'$. 
This part of the paper can be read independently of the rest and lies squarely in the domain of local harmonic analysis for the Jacquet--Rallis relative trace formula approach to the Gan--Gross--Prasad conjecture.

\subsection*{Acknowledgments} We are grateful to A.~Mihatsch for his help; in particular, he alerted us to the difficulties arising when considering degenerate intersections. We also thank B.~ Gross, S.~Kudla, P.~Scholze, and Z.~Yun for helpful discussions.  We finally thank the referees for their remarks on the text.

M.R. was supported by a grant from the Simons foundation and by the Deutsche Forschungsgemeinschaft through the grant SFB/TR 45. B.S. is supported by an NSA grant H98230-16-1-0024.  W.Z. is supported by an NSF grant DMS \#1301848 and a Sloan research fellowship.

\subsection*{Notation} We list here some notation that we use throughout the whole paper. We denote by $p$ an odd prime number. 
 
$F/F_0$ is a ramified quadratic extension of finite extensions of $\BQ_p$. We denote by $a\mapsto \ov a$ the nontrivial automorphism of $F/F_0$, and by  $\eta = \eta_{F/F_0}$ the corresponding  quadratic character on $F_0^\times$.  When $h$ is a square matrix with entries in $F_0$, we sometimes abbreviate $\eta(\det h)$ to $\eta(h)$.  Since $p \neq 2$, we may and do choose uniformizers $\pi$ of $F$ and $\varpi$ of $F_0$ such that $\pi^2=\varpi$. We denote by $\ov k$ an algebraic closure of the common residue field $k$ of $F$ and $F_0$, and we set $q := \#k$.  As stated above, when working in an algebro-geometric context where formal schemes are present, we always understand that $q = p$ and $F_0 = \BQ_p$.  We denote the group of squares in $F_0^\times$ by $F_0^{\times,2}$.  We denote the group of norm $1$ elements in $F^\times$ by
\[
   F^1 := \{\, a \in F \mid a \ov a = 1 \,\}.
\]
We denote by $\breve F_0$ the completion of a maximal unramified extension of $F_0$, and by $\breve F := \breve F_0 \otimes_{F_0} F$ the analogous object for $F$.

A polarization on a $p$-divisible group $X$ is an anti-symmetric isogeny $X \to X^\vee$, where $X^\vee$ denotes the dual.  We use a superscript $\circ$ to denote the operation $-\otimes_\BZ \BQ$ on groups of homomorphisms, so that for example
\[
   \Hom^\circ(X,Y) := \Hom(X,Y) \otimes_\BZ \BQ,
\]
where $Y$ is another $p$-divisible group.  For any quasi-isogeny $\rho \colon X \to Y$ and polarization $\lambda$ on $Y$, we define the pullback polarization
\[
   \rho^* (\lambda) := \rho^\vee \circ \lambda \circ \rho.
\]

We denote by $\BE$ the unique (up to isomorphism) formal $O_{F_0}$-module of relative height $2$ and dimension $1$ over $\Spec \ov k$.  We set
\[
   O_D := \End_{O_{F_0}}(\BE)
	\quad\text{and}\quad
	D := O_D \otimes_{O_{F_0}} F_0.
\]
Thus $D$ is ``the'' quaternion division algebra over $F_0$, and $O_D$ is its maximal order.  Since $F/F_0$ is ramified, any $F_0$-embedding of $F$ into $D$ makes $\BE$ into a formal $O_F$-module of relative height $1$.  We fix such an embedding
\begin{equation*}
   \iota_\BE\colon F \to D
\end{equation*}
once and for all, and we always understand $\BE$ to be a formal $O_F$-module via $\iota_\BE$.  We denote by $\CE$ the corresponding canonical lift of $\BE$ over $\Spf O_{\breve F}$, equipped with its $O_F$-action $\iota_\CE$ and $O_F$-linear framing isomorphism $\rho_\CE\colon \CE_{\ov k} \isoarrow \BE$.  We denote by $\ov\BE$ the same object as $\BE$, except 
where the $O_F$-action $\iota_{\ov\BE}$ is
equal to $\iota_\BE$ precomposed by the nontrivial automorphism of $F/F_0$; and ditto for $\ov\CE$ in relation to $\CE$, which is furthermore equipped with the same framing $\rho_{\ov\CE} := \rho_\CE$ on the level of $O_{F_0}$-modules over $\Spec \ov k$.  Of course $\ov\BE$ is a formal $O_F$-module of relative height $1$ in its own right, but note that $\ov\CE$ is \emph{not} its canonical lift.  In \S\ref{framing obs} we will 
specify a principal polarization $\lambda_\BE$ on $\BE$,
and we will denote by $\lambda_\CE$ the principal polarization on $\CE$ lifting $\lambda_\BE$.  We write $\lambda_{\ov\BE} := \lambda_{\BE}$ and $\lambda_{\ov\CE} := \lambda_{\CE}$ for the same polarizations when we regard them as defined on $\ov \BE$ and $\ov \CE$, respectively.

We denote the main involution on $D$ by $c \mapsto \ov c$, and the reduced norm by $\RN$.  We also write $\RN$ for the norm map $F^\times \to F_0^\times$; of course, all of this is compatible with any embedding of $F$ into $D$.  We write $v_D$ for the normalized valuation on $D$, and we use $\pi$ as a uniformizer for $D$, via $\iota_\BE$.  We write $v$ for the normalized (i.e.~$\varpi$-adic) valuation on $F_0$.  For $c \in D^\times$, we define the conjugate embedding
\begin{equation}\label{conjugate emb}
\begin{gathered}
	\tensor[^c]{\iota}{_\BE} \colon
	\xymatrix@R=0ex{
	   F \ar[r]  &  D\\
		a \ar@{|->}[r]  &  c\iota_\BE(a)c^{-1}.
	}
\end{gathered}
\end{equation}
We denote by $\tensor[^c]{F}{}$ the image of $\tensor[^c]{\iota}{_\BE}$.

Given a variety $V$ over $\Spec F_0$, we denote by $C_c^\infty(V(F_0))$ the set of locally constant, compactly supported functions on the space $V(F_0)$, endowed with its $\varpi$-adic topology; and
we typically abbreviate this set to $C_c^\infty(V)$. We normalize the respective Haar measures on $F_0$, $F$, $F_0^\times$, and $F^\times$ such that $\vol(O_{F_0}) = \vol(O_F) = \vol(O^\times_{F_0}) = \vol(O^\times_F) = 1$.

We write $\BA$ for the affine line, $\M_n$ for the scheme of $n \times n$ matrices, and $1_n$ for the $n \times n$ identity matrix.  We use a subscript $S$ to denote base change to $S$, and when $S = \Spec A$, we often use a subscript $A$ instead.

\part{The conjectures}
In this first part of the paper we introduce the  objects involved in the statements of our AT conjectures. In \S\ref{conj and results} we state the conjectures and our main results.

\section{Group-theoretic setup and orbit matching}\label{setup}

We begin by explaining the general group-theoretic setup in this paper and the attendant matching relation for regular semi-simple elements.  We consider three cases: the homogeneous group setting, the inhomogeneous group setting, and the Lie algebra setting.  In this section $n \geq 2$ is an integer.

\subsection{Homogenous setting}\label{homog case}
We begin with the algebraic group
\[
   G' := \Res_{F/F_0}(\GL_{n-1}\times\GL_n)
\]
over $F_0$.  We consider the following two subgroups of $G'$. The first subgroup is
\[
   H_1' := \Res_{F/F_0} \GL_{n-1},
\]
which is embedded diagonally, via the inclusion of $\GL_{n-1}$ in $\GL_n$ sending $A \mapsto \diag(A,1)$.  The second subgroup is
\[
   H_2' := \GL_{n-1}\times\GL_n,
\]
with its obvious embedding into $G'$.  Let
\[
   H_{1,2}' := H_1'\times H_2'.
\]
Then $H_{1, 2}'$ acts on $G'$ by $(h_1, h_2)\colon \gamma\mapsto h_1^{-1}\gamma h_2$. We call an element $\gamma\in G'(F_0)$ \emph{regular semi-simple} if it is regular semi-simple for the action of $H_{1, 2}'$, i.e.~its orbit under $H_{1, 2}'$ is closed, and its stabilizer is of minimal dimension.  In the case at hand, it is equivalent that $\gamma$ have closed orbit and trivial stabilizer, which follows from \cite[Th.~6.1]{RS}. We denote by $G'(F_0)_\rs$ the set of regular semi-simple elements in $G'(F_0)$.

We next consider non-degenerate $F/F_0$-hermitian spaces of dimension $n$.  Up to isomorphism there are two of them, a split one $W_0$ and a non-split one $W_1$.  They are distinguished by the rule
\begin{equation}\label{det def}
   \eta\bigl((-1)^{n(n-1)/2}\det W_i\bigr) = (-1)^i ,
\end{equation}
where $\det W_i := \det J_i$ for any hermitian matrix $J_i$ (relative to the choice of a basis) representing the hermitian form.  Let $i \in \{0,1\}$.  Write
\begin{equation}\label{U_i def}
   U_i := \U(W_i),
\end{equation}
and let
\[
   u_i \in W_i
\]
be a non-isotropic vector, which we call the \emph{special vector}.  We will choose $u_0$ and $u_1$ such that their norms are congruent mod $\RN F^\times$, which can always be done since $n \geq 2$.  Let $W_i^\flat$ denote the orthogonal complement in $W_i$ of the line spanned by $u_i$, and let
\[
   H_i := \U(W_i^\flat).
\]
Then $H_i$ naturally embeds into $U_i$ as the stabilizer of $u_i$.  Since we assume that $u_0$ and $u_1$ have congruent norms mod $\RN F^\times$, $W^\flat_0$ and $W^\flat_1$ are non-isomorphic as hermitian spaces. 

To lighten notation, now set $W := W_i$, and define $W^\flat$, $U$, $H$, and $u$ analogously. Let
\[
   G_W := H \times U,
\]
and  consider $H$ as a subgroup of $G_W$, embedded diagonally. Then $H\times H$ acts on $G_W$ via the rule
\begin{equation*}
   (h_1, h_2)\colon g\mapsto h_1^{-1} g h_2.
\end{equation*}
An element $g\in G_W(F_0)$ is called \emph{regular semi-simple} if it is regular semi-simple for the action of $H\times H$. We denote by $G_W(F_0)_\rs$ the set of regular semi-simple elements in $G_W(F_0)$. 

We now recall the matching relation between regular semi-simple elements, as in \cite{Z12}.  Choose an $F$-basis for $W^\flat$ and complete it to a basis for $W$ by adjoining $u$.  This identifies $W^\flat$ with $F^{n-1}$ and $W$ with $F^n$ in such a way that $u$ corresponds to the column vector
\[
   e := (0,\dotsc,0,1)
\]
in $F^n$, and hence determines embeddings of groups $U \inj \Res_{F/F_0} \GL_n$ and $G_W \inj G'$.  We call the embeddings obtained in this way \emph{special embeddings}.  An element $\gamma\in G'(F_0)_\rs$ and an element $g\in G_{W}(F_0)_\rs$ are said to \emph{match}
if these two elements, when considered as elements in $G'(F_0)$, are conjugate under $H_{1, 2}'(F_0)$.
The matching relation is independent of the choice of special embedding and induces a bijection \cite[\S2]{Z12}\footnote{In loc.~cit.~only the inhomogeneous version, which will be taken up  in the next subsection, is considered. However, the inhomogeneous version easily implies the homogeneous version.}
\[
   \bigl[ G_{W_0}(F_0)_\rs\bigr]\amalg \bigl[ G_{W_1}(F_0)_\rs\bigr]\simeq \bigl[G'(F_0)_\rs\bigr] .
\]
Here on the left-hand side, the square brackets denote the sets of orbits under the respective actions of $H_{0}(F_0)\times H_0(F_0)$ and $H_{1}(F_0)\times H_1(F_0)$, whereas the brackets on the right-hand side denote the set of orbits under $H_{1, 2}'(F_0)$.

\subsection{Inhomogeneous setting} 
Now we pass to the inhomogeneous version of the previous subsection.
Consider the following identifications of algebraic varieties over $F_0$. First, 
\begin{equation*}
   H'_1\backslash G' \isoarrow \Res_{F/F_0}\GL_n, \quad \gamma=(\gamma_1, \gamma_2)\mapsto \gamma_1^{-1}\gamma_2. 
\end{equation*}
Second, let
\[
    S := S_n := \{\,g\in \Res_{F/F_0}\GL_n\mid g\ov g=1_n\,\}
\]
and
\[
   H' :=\GL_{n-1} .
\]
Then 
\[
   \Res_{F/F_0}(\GL_n)\big/\GL_n \isoarrow S, \quad \gamma\mapsto \gamma\ov\gamma^{-1} ,
\]
and the above two identifications induce an identification on $F_0$-rational points
\begin{equation}\label{idGL}
   G'(F_0)\big/H_{1, 2}'(F_0)\simeq S(F_0)\big/H'(F_0) . 
\end{equation}
Here the action of $H'$ on $S$ is through conjugation. In other words, the map
\[
G'(F_0)\to S(F_0) , \quad \gamma=(\gamma_1, \gamma_2)\mapsto s(\gamma):=(\gamma_1^{-1} \gamma_2) \bigl(\ov{\gamma_1^{-1} \gamma_2}\bigr)^{-1}  ,
\]
induces the bijection \eqref{idGL}.

On the unitary group side, let $i \in \{0,1\}$ and use the notation $W$, $U$, $H$, etc.~as in the previous subsection.  Then we similarly have an identification of $F_0$-points of algebraic varieties over $F_0$,
\[
   G_W(F_0)\big/\bigl(H(F_0)\times H(F_0)\bigr) \simeq U(F_0) \big/ H(F_0), \quad (g_1, g_2)\mapsto g_1^{-1} g_2 . 
\]
Here the action of $H(F_0)$ on $U(F_0)$ is by conjugation. 

The definitions from the homogeneous setting readily transfer to the inhomogeneous setting. Thus an element $\gamma\in S(F_0)$ is called \emph{regular semi-simple} if it is regular semi-simple for the action of $\GL_{n-1}$ on $\GL_n$.  As in the homogeneous setting, it is again equivalent that $\gamma$ have closed orbit and trivial stabilizer, cf.~\cite[Th.~6.1]{RS}.  We denote by $S(F_0)_\rs$ or $S_n(F_0)_\rs$ the set of regular semi-simple elements in $S(F_0)$. An element $g\in U(F_0)$ is \emph{regular semi-simple} if it is regular semi-simple for the action of $H$ on $U$; equivalently, when $g$ is considered as an element in $\GL_n(F)$ upon choosing a special embedding for $U$ as in the previous subsection, it is regular semi-simple for the action of $\GL_{n-1}$ on $\GL_n$.  We denote by  $U(F_0)_\rs$ the set of regular semi-simple elements in $U(F_0)$.

An element $\gamma\in S(F_0)_\rs$ \emph{matches} an element $g\in U(F_0)_\rs$
if these two elements are conjugate under $\GL_{n-1}(F)$ when considered as elements of $\GL_n(F)$, upon choosing a special embedding for $U$.  The matching relation is again independent of the choice of special embedding, and induces a bijection \cite[\S 2]{Z12}
\[
   \bigl[ U_0(F_0)_\rs\bigr] \amalg \bigl[  U_1(F_0)_\rs\bigr]\simeq \bigl[S(F_0)_\rs\bigr] ,
\]
where on the left-hand side are the respective sets of orbits under $H_0(F_0)$ and $H_1(F_0)$, and on the right-hand side the set of orbits under $\GL_{n-1}(F_0)$.  In particular, there is a disjoint union decomposition
\begin{equation}\label{Srs decomp}
   S(F_0)_\rs = S_{\rs,0} \amalg S_{\rs,1},
\end{equation}
where $S_{\rs,i}$ denotes the set of elements in $S(F_0)_\rs$ that match with elements in $U_i(F_0)_\rs$.

\subsection{Lie algebra setting} 
\label{Lie algebra setting}
We also consider a Lie algebra version of the inhomogeneous setup. We introduce the Lie algebra version of $S_n$, 
\begin{equation}\label{fks def}
   \fks := \fks_n := \bigl\{\, y \in \Res_{F/F_0}\M_n \bigm| y + \ov y = 0 \,\bigr\} .
\end{equation}
Then $H'=\GL_{n-1}$ acts on $\fks$, and we call an element of $\fks(F_0)$ \emph{regular semi-simple} if its $H'$-orbit is closed and of maximal dimension.  It is again equivalent that the element have closed orbit and trivial stabilizer.  We denote by $\fks(F_0)_\rs = \fks_n(F_0)_\rs$ the set of regular semi-simple elements in $\fks(F_0)$. 

For $i \in \{0,1\}$, let
\[
   \fku_i := \Lie U_i.
\]
An element of $\fku_i(F_0)$ is called \emph{regular semi-simple} if its orbit under $H_i$ is closed and of maximal dimension.  We denote by $\fku_i(F_0)_\rs$ the set of regular semi-simple elements in $\fku_i(F_0)$. 

As in \S\ref{homog case}, the choice of a basis for $W_i^\flat$, extended by $u_i$ to a basis for $W_i$, determines an embedding $\fku_i \inj \Res_{F/F_0} \M_n$, which we again call a \emph{special embedding}.  An element $y\in\fks(F_0)_\rs$ \emph{matches} an element $x\in \fku_i(F_0)_\rs$
if these two elements are conjugate under $\GL_{n-1}(F)$ when considered as elements of 
$\M_n(F)$.  The matching relation is independent of the special embedding and induces a bijection \cite[\S5]{JR}
\[
   [ \fku_0(F_0)_\rs]\amalg [\fku_1(F_0)_\rs] \simeq [\fks(F_0)_\rs] ,
\]
where on the left-hand side are the respective sets of orbits under $H_0(F_0)$ and $H_1(F_0)$, and on the right-hand side the set of orbits under $\GL_{n-1}(F_0)$.  As before, we get a disjoint union decomposition
\[
   \fks(F_0)_\rs = \fks_{\rs,0} \amalg \fks_{\rs,1},
\]
where $\fks_{\rs,i}$ denotes the set of elements in $\fks(F_0)_\rs$ that match with elements in $\fku_i(F_0)_\rs$.

\subsection{Linear algebra characterizations}\label{ss:LAchar} 
We now recall the linear algebra characterizations of regular semi-simple elements and of their matching, in the inhomogeneous group setting and the Lie algebra setting. First we introduce the \emph{discriminant} $\Delta$, which is a morphism of varieties
\begin{equation}\label{Delta}
	\begin{gathered}
   \Delta\colon
	\xymatrix@R=0ex{
	   \Res_{F/F_0} \M_n \ar[r]  &  \Res_{F/F_0} \BA\\
		x \ar@{|->}[r]  &  \det(\tensor*[^t] e {} x^{i+j}e)_{0\leq i, j\leq n-1}
	}  
	\end{gathered}
\end{equation}
over $F_0$. We remark that beginning in \S\ref{reduction to LA} we will actually work with a rescaled version of $\Delta$, cf.~\eqref{Delta rescaled}.

An element $\gamma\in S(F_0)$ is  regular semi-simple  if and only if, considering $\gamma$ as an element of $ \GL_n(F)$, the sets of vectors $\{\gamma^ie\}_{i=0}^{n-1}$ and $\{\tensor*[^t] e {}\gamma^i\}_{i=0}^{n-1}$ are both linearly independent \cite[Th.~6.1]{RS}.\footnote{In \cite{RS}, the Lie algebra version is considered, but it is easy to deduce the group version from this.}  Equivalently, $\Delta(\gamma) \neq 0$. 
Hence  the regular semi-simple locus inside $S$ is the complement of the locus $\Delta=0$, i.e.~the Zariski-open subscheme 
\[
   S_\rs := S_{n,\rs} := \bigl\{\, \gamma\in S \bigm| \Delta(\gamma)\neq 0 \,\bigr\} .
\]
Thus we may write $S(F_0)_\rs$ and $S_\rs(F_0)$ interchangeably.

To $x\in \M_n(F)$ we associate the following numerical invariants: the $n$ coefficients of the characteristic polynomial 
 $\charac_x(T)\in F[T]$, and the $n-1$ elements $\tensor*[^t] e {}x^ie\in F$ for $ i=1,\dotsc,n-1$.  Then two  elements of $S(F_0)_\rs$ are conjugate under $\GL_{n-1}(F_0)$ if and only if they have the same numerical invariants when considered as elements of $\M_n(F)$, cf.~\cite{Z12}. 
 
For $i \in \{0, 1\}$ and $U = U_i$, an element $g\in U(F_0)$ is regular semi-simple if and only if, when considered as an element of $\GL_n(F)$ via any special embedding,  $g$ satisfies the conditions above, i.e.~the sets of vectors $\{g^ie\}_{i=0}^{n-1}$ and $\{\tensor*[^t] e {}g^i\}_{i=0}^{n-1}$ are both linearly independent. Equivalently, $\Delta(g) \neq 0$, so that the regular semi-simple set is the Zariski open complement to the locus $\Delta(g) = 0$, and we write $U(F_0)_\rs$ and $U_{\rs}(F_0)$ interchangeably. Two elements $g\in U(F_0)_\rs$ and $\gamma\in S(F_0)_\rs$ are matched if and only if their numerical invariants, when both are considered as elements of $\M_n(F)$, coincide. 
 
The theory in the Lie algebra setting is entirely analogous.  An element $y\in \fks(F_0)$ is  regular semi-simple  if and only if, considering $y$  as an element of $\M_n(F)$, the sets of vectors $\{y^ie\}_{i=0}^{n-1}$ and $\{\tensor*[^t] e {} y^i\}_{i=0}^{n-1}$ are both linearly independent \cite[Th.~6.1]{RS}.   Equivalently, $\Delta(y) \neq 0$, i.e.~the regular semi-simple set is the Zariski open complement to the locus $\Delta(y) = 0$, and we write $\fks(F_0)_\rs$ and $\fks_\rs(F_0)$ interchangeably. Furthermore,  two  elements of $\fks(F_0)_\rs$ are conjugate under  $\GL_{n-1}(F_0)$ if and only if they have the same numerical invariants when considered as elements of $\M_n(F)$, cf.~\cite{RS}. 
 
For $i \in \{0,1\}$ and $\fku = \fku_i$, an element $x\in\fku(F_0)$ is regular semi-simple if and only if, when considered as an element of $\M_n(F)$ via any special embedding, $x$ satisfies the conditions above, i.e.~the sets of vectors $\{ x^ie\}_{i=0}^{n-1}$ and $\{\tensor*[^t] e {} x^i\}_{i=0}^{n-1}$ are both linearly independent. Equivalently, $\Delta(x)\neq 0$, and we again write $\fku(F_0)_\rs$ and $\fku_\rs(F_0)$ interchangeably.  Two elements  $x\in \fku(F_0)_\rs$ and $y\in \fks(F_0)_{\rs}$ are matched if and only if their numerical invariants, when both are considered as elements of $\M_n(F)$, coincide \cite{JR,Z12}. 

From now on in the paper, we make the blanket assumption that
\begin{equation}\label{norm 1 assumption}
	\text{\emph{the special vectors $u_0 \in W_0$ and $u_1 \in W_1$ have norm $1$.}}
\end{equation}
Under this assumption, we have the following simple formula to distinguish between $W_0$ and $W_1$ in terms of the discriminant.

\begin{lemma}\label{lem delta=det} 
For $i \in \{0,1\}$ and any $x\in \fku_{i}(F_0)_\rs$, 
\[
   \eta\bigl(\Delta(x)\bigr) = (-1)^i. 
\]
\end{lemma}

\begin{proof}
Let $u=u_i$ and $W = W_i$.
Let $h$ denote the hermitian form on $W$. By \eqref{norm 1 assumption},
\[
   \tensor*[^t]{e}{} x^{i+j} e=h(u,x^{i+j}u).
\]
Since $x$ is in the Lie algebra $\fku_i$ we have $h(xv,w)=-h(v,xw)$ for all $v,w\in W$. Hence 
\[
   h(u,x^{i+j}u) = (-1)^i h(x^{i}u,x^ju).
\]
Hence
\[
   \Delta(x) 
	   = \det\bigl((-1)^i h(x^{i}u,x^ju)_{0\leq i, j\leq n-1}\bigr)
		= (-1)^{n(n-1)/2}\det h(x ^{i}u,x^ju)_{0\leq i, j\leq n-1}.
\]
Since $x$ is regular semi-simple, the vectors $u,xu,\dotsc,x^{n-1}u$ form a basis of $W$. Hence the lemma follows from 
\eqref{det def}.
\end{proof}

\subsection{Invariants in the Lie algebra setting}\label{ss: Lie invars}
Consider the $2n-1$ maps of varieties over $F_0$
\[
   \Res_{F/F_0} \M_n \to \Res_{F/F_0} \BA
\]
defined on points by sending $x$ to the quantities
\begin{equation}\label{inv gen}
   \tr \wedge^i x,
	\quad
	1\leq i\leq n,
	\quad\text{and}\quad
	\tensor*[^t] e {} x^{j}e,
	\quad
	1\leq j\leq n-1.
\end{equation}
When restricted to the subscheme $\fks = \fks_n \subset \Res_{F/F_0} \M_n$, each of these maps factors through either $\BA \subset \Res_{F/F_0} \BA$ or $\fks_1 \subset \Res_{F/F_0} \BA$ according as $i$, resp.~$j$, is even or odd.  Here we have allowed the case $n = 1$ in the definition \eqref{fks def} of $\fks_n$, i.e.~$\fks_1$ denotes the scheme of points $y$ in $\Res_{F/F_0}\BA$ such that $\ov y = - y$.  Upon choosing special embeddings $\fku_0, \fku_1 \inj \Res_{F/F_0} \M_n$, the same statement is true of each of the maps in \eqref{inv gen} when restricted to $\fku_0$ and $\fku_1$.  Thus we obtain a map from each of $\fks$, $\fku_0$, and $\fku_1$ into the common $(2n-1)$-fold product of the corresponding $\BA$'s and $\fks_1$'s; and in the case of $\fku_0$ and $\fku_1$, this map is independent of the choice of special embedding.  These maps into the product of $\BA$'s and $\fks_1$'s are invariant for the respective actions of $H'$, $H_0$, and $H_1$ on \fks, $\fku_0$, and $\fku_1$, and it is shown in \cite[Lem.~3.1]{Z14} that they identify the target with the categorical quotients $\fks/H'$, $\fku_0/H_0$, and $\fku_1/H_1$.  In other words, the ring of global invariants on each of $\fks$, $\fku_0$, and $\fku_1$ is a polynomial ring over $F_0$ generated by the $2n-1$ functions \eqref{inv gen}.

There is another set of polynomial generators, also given in loc.~cit., which will be a little more convenient for us to work with.  Write a point $x$ in $\Res_{F/F_0} \M_n$ in the form
\begin{equation}\label{x block decomp}
   x = 
	\begin{bmatrix}
      A  &  \mathbf{b}\\
		\mathbf{c}  &  d
   \end{bmatrix} ,
\end{equation}
where the block decomposition is with respect to the special vector $e$.  Then the functions
\begin{equation}\label{inv sec}
   \tr \wedge^i A,
	\quad 
	1\leq i\leq n-1, 
	\quad
	\mathbf{c}A^j\mathbf{b}, 
	\quad
	0\leq j\leq n-2,
	\quad\text{and}\quad
	d
\end{equation}
realize the categorical quotients $\fks/H'$, $\fku_0/H_0$, and $\fku_1/H_1$ as above (again after choosing any special embeddings $\fku_0,\fku_1 \inj \Res_{F/F_0} \M_n$ for the latter two).  Explicitly, for $\fkb$ any of these three quotients, the invariants \eqref{inv sec} induce an isomorphism of schemes over $F_0$
\begin{equation}\label{cat quot isom}
\begin{gathered}
   \xymatrix@R=-3ex{
	   \fkb \ar[r]^-\sim  &  {}\overbrace{\fks_1 \times \BA \times \dotsb}^{\substack{n-1\\ \text{alternating}\\ \text{factors}}} \times
		                      \overbrace{\BA \times \fks_1 \times \dotsb}^{\substack{n-1\\ \text{alternating}\\ \text{factors}}} \times
									 \: \fks_1\\
		x \ar@{|->}[r]  &  (\tr A, \tr \wedge^2 A,\dotsc, \mathbf{c}\mathbf{b}, \mathbf{c} A \mathbf{b},\dotsc, d).
	}
\end{gathered}
\end{equation}

\section{The moduli space}\label{moduli space}

In this section we introduce the moduli space of $p$-divisible groups $\CN_n$ over $\Spf O_{\breve F}$.  It is an analog of spaces appearing in \cite{KR-U1,RTZ,V2,VW,Z12} in the unramified setting, but in the ramified setting a subtler definition is required to obtain a formally smooth space.  
In accordance with our convention in the Introduction, throughout this section we take $F_0 = \BQ_p$.
Let $n \geq 1$ be an integer.

\subsection{Unitary $p$-divisible groups}\label{unitary p-div groups}

Let $S$ be a scheme over $\Spf O_{\breve F}$.  A \emph{unitary $p$-divisible group of signature $(1,n-1)$} in the setting of this paper is a triple
\[
   (X,\iota_X,\lambda_X)
\]
consisting of a $p$-divisible group $X$ over $S$, a homomorphism
\[
   \iota_X \colon O_F \to \End_S (X),
\]
and a polarization
\[
   \lambda_X \colon X \to X^\vee,
\]
subject to the following constraints:
\begin{altitemize}
\item 
(\emph{Kottwitz condition})
for the action of $O_F$ on $\Lie X$ induced by $\iota_X$, there is an equality of polynomials
\[
   \charac\bigl(\iota_X(\pi) \mid \Lie X\bigr) = (T - \pi)(T + \pi)^{n-1} \in \CO_S[T];
\]
\item 
(\emph{wedge condition}) $\bigwedge^2_{\CO_S} \bigl(\iota_X(\pi) + \pi \mid \Lie X\bigr) = 0$;
\item 
(\emph{spin condition}) if $n > 1$, then for every point $s \in S$, the operator $\iota_X(\pi) \mid \Lie X_s$ is nonzero;
\item 
the Rosati involution on $\End_S^\circ(X)$ attached to $\lambda_X$ induces the nontrivial Galois automorphism on $O_F$; and
\item 
if $n$ is even, then $\Ker \lambda_X = X[\iota_X(\pi)]$; and if $n$ is odd, then $\Ker\lambda_X \subset X[\iota_X(\pi)]$ of height $n-1$.
\end{altitemize}

Let us make a few remarks on the definition.  First note that our formulation of the Kottwitz condition implies that
\[
   \charac\bigr(\iota_X(a) \mid \Lie X\bigl) = (T - a)(T + \ov a)^{n-1} \in \CO_S[T]
   \quad\text{for all}\quad a \in O_F,
\]
and that $\rank_{\CO_S} (\Lie X) = n$.  Since $X$ is isogenous to its dual, it follows that $X$ has height $2n$.
We also note that the Rosati condition on $\lambda_X$ is equivalent to requiring that $\lambda_X$ is $O_F$-linear, where $O_F$ acts on the dual $X^\vee$ via the rule
\begin{equation}\label{dual action}
   \iota_{X^\vee}(a) = \iota_X(\ov a)^\vee.
\end{equation}

The wedge condition is due to Pappas \cite{P}.  There is another part to the wedge condition in loc.~cit., which for signature $(1,n-1)$ is
\[
   \sideset{}{^n_{\CO_S}}\bigwedge \bigl(\iota_X(\pi) - \pi \mid \Lie X\bigr) = 0.
\]
Since $\Lie X$ has rank $n$, this condition holds automatically by the Kottwitz condition.  Similarly, the wedge condition as we have formulated it above is implied by the Kottwitz condition when $n \leq 2$.

The spin condition is based on the spin condition introduced in \cite{PR}; see Remark \ref{relation to spin cond} below for further discussion.  Note that the spin condition is a condition only on the underlying point set of $S$.  Since 
$S$ lies 
over $\Spf O_{\breve F}$, we have $\pi \cdot \kappa(s) = 0$ for every point $s \in S$.  Thus in the presence of the wedge condition, the spin condition is equivalent to the condition that $\iota_X(\pi) \mid \Lie X_s$ has rank $1$ for every $s \in S$.

\subsection{Serre tensor construction}\label{serre const}
Before continuing, we
pause to briefly review the Serre tensor construction, which will also 
play a central role in \S\ref{serre isom}.

Quite generally, let $S$ be any base scheme, $A$ a commutative ring, $M$ a finite projective $A$-module, and $X$ a contravariant functor on the category of $S$-schemes, valued in $A$-modules.  For $T$ a scheme over $S$, define
\[
   (M \otimes_A X)(T) := M \otimes_A X(T).
\]
When $X$ is a scheme, so is $M \otimes_A X$, and moreover many properties of $X$ are inherited by $M \otimes_A X$.  See \cite[\S7]{C}.  It follows easily from loc.\ cit.\ that $M \otimes_A X$ is a $p$-divisible group when $X$ is, which will be the case of interest to us.  In this case, one furthermore has canonical isomorphisms
\[
   (M \otimes_A X)^\vee \cong M^\vee \otimes_A X^\vee,
\]
where $M^\vee := \Hom_A(M,A)$ is the dual $A$-module; and
\[
   \Lie (M \otimes_A X) \cong M \otimes_A \Lie X.
\]

\subsection{Framing objects}\label{framing obs}
In this subsection we consider unitary $p$-divisible groups of signature $(1,n-1)$ over $\Spec \ov k$.  The first point of business is that, up to isogeny, there is 
only one
whose underlying $p$-divisible group is supersingular, in the following sense.

\begin{proposition}\label{k-bar isog classes}
Let $(X,\iota_X,\lambda_X)$ and $(Y,\iota_Y,\lambda_Y)$ be supersingular unitary $p$-divisible groups of signature $(1,n-1)$ over $\Spec \ov k$.  Then there exists an $O_F$-linear quasi-isogeny
\[
   \rho \colon X \to Y
\]
such that $\rho^* (\lambda_Y)$ is an $F_0^\times$-multiple of $\lambda_X$ in $\Hom_{O_F}^\circ(X,X^\vee)$.
\end{proposition}

\begin{proof}
Let $M$ denote the covariant Dieudonn\'e module of $X$, endowed with its Frobenius operator $\uF$ and Verschiebung $\uV$.  The polarization on $X$ translates to an alternating form \aform on $M$ satisfying
\[
   \langle \uF x, y \rangle = \langle x, \uV y \rangle^\sigma
	\quad\text{for all}\quad
	x,y \in M,
\]
where $\sigma$ denotes the Frobenius operator on $W(\ov k) = O_{\breve F_0}$.  The $O_F$-action on $X$ translates to an $O_F$-action on $M$ commuting with $\uF$ and $\uV$ and satisfying
\[
   \langle a x,y \rangle = \langle x, \ov a y\rangle
\]
for all $x,y$.

Let $N := M \otimes_{O_{\breve F_0}} \breve F_0$ denote the rational Dieudonn\'e module of $X$.  Then \aform extends to a non-degenerate alternating form on $N$.  We must show that the rational Dieudonn\'e module of $Y$ is isomorphic to $N$ as a polarized isocrystal (with the polarization taken up to scalar in $F_0^\times$) with $F$-action.

Let $\zeta \in O_{\breve F_0}^\times$ satisfy $\zeta^2 \varpi = -p$.  Since $N$ is supersingular, all slopes of the $\sigma$-linear operator
\begin{equation}\label{tau}
   \tau := \zeta \pi \uV^{-1} \colon N \to N
\end{equation}
are $0$.  Hence
\[
   C := N^{\tau = 1}
\]
is an $F_0$-subspace of $N$ such that
\[
   C \otimes_{F_0} \breve F_0 \isoarrow N;
\]
and in this way $\id_C \otimes \sigma$ identifies with $\tau$.  Furthermore $C$ is $F$-stable, the restriction of \aform to $C$ takes values in $F_0$, and the form
\[
   h(x,y) := \langle \pi x, y \rangle + \langle x,y \rangle \pi, \quad x,y \in C,
\]
makes $C$ into a non-degenerate $F/F_0$-hermitian space of dimension $n$, cf.~\cite[pp.~1170--1]{RTW}.\footnote{Note that the quantity $\eta$ in loc.~cit.~should be a square root of $-\epsilon^{-1}$, rather than  a square root of $\epsilon^{-1}$.}

Clearly, to classify $N$ up to isomorphism as a polarized isocrystal with $F$-action is to classify $C$ up to similarity as a hermitian space.  When $n$ is odd, all $n$-dimensional non-degenerate $F/F_0$-hermitian spaces are similar, which proves the lemma.  When $n$ is even, the two isomorphism types of $n$-dimensional non-degenerate hermitian spaces remain non-similar.  By Dieudonn\'e theory, the Lie algebra of $X$ identifies with $M / \uV M = M / \pi\tau^{-1} M$, and the spin condition is that $\dim_{\ov k} (\pi M + \pi \tau^{-1} M)/\pi \tau^{-1} M = 1$.
The condition $\Ker \lambda_X = X[\iota_X(\pi)]$ translates to $M^\vee = \pi^{-1} M$ inside $N$, where the dual lattice $M^\vee$ is the set of $x \in N$ such that $\langle x, M \rangle \subset O_{\breve F_0}$, or equivalently such that $h(x,M) \subset O_{\breve F}$.
The lemma is now a consequence of the general result in Lemma~\ref{C parity lem} below.
\end{proof}

We will need to prepare a little before coming to Lemma~\ref{C parity lem}.  Suppose that $n$ is even, and let $N$ be ``the'' $n$-dimensional non-degenerate $\breve F/\breve F_0$-hermitian space.  Let $M$ be a $\pi$-modular $O_{\breve F}$-lattice in $N$, which is to say that the dual lattice of $M$ with respect to the hermitian form is $\pi^{-1}M$. Let $U := \RU(N)$ denote the (quasi-split) unitary group of $N$ over $\Spec \breve F_0$.  Let $K$ denote the stabilizer in $U(\breve F_0)$ of $M$.  Then $K$ is a special maximal parahoric subgroup of $U(\breve F_0)$, and all special maximal parahoric subgroups of $U(\breve F_0)$ are conjugate to $K$; see e.g.~case (b) in \cite[\S4.a]{PR-TLG}.  We also need the \emph{Kottwitz homomorphism} on $U(\breve F_0)$, which is a homomorphism
\begin{equation}\label{Kottwitz}
   \kappa\colon U(\breve F_0) \to \{\pm 1\}
\end{equation}
admitting the following simple description.  For any $g \in U(\breve F_0)$, the determinant $\det_F(g)$ is a norm $1$ element in $\breve F$, and $\kappa(g)$ takes the value $\pm 1$ according as $\det_F(g) \equiv \pm 1 \bmod \pi$; this follows from e.g.~\cite[\S\S3.b.1, 4.a]{PR-TLG}, or see \cite[\S1.2.3(b)]{PR} for the closely related case of quasi-split $GU_n$.

\begin{lemma}\label{Kott lem}
For $n$ even and any $g \in U(\breve F_0)$, $\kappa(g)$ equals $1$ or $-1$ according as the $O_{\breve F}$-length of $(M + gM)/M$ is even or odd.
\end{lemma}

\begin{proof}
We use the Cartan decomposition.  We can choose a split $\breve F$-basis $e_1,\dotsc,e_n$ for $N$ (meaning that $e_i$ and $e_j$ pair to $\delta_{i,n+1-j}$ under the hermitian form) such that
\[
   M = O_{\breve F} e_1 + \dotsb + O_{\breve F} e_{n/2} + O_{\breve F} \pi e_{n/2 + 1} + \dotsb + O_{\breve F} \pi e_n.
\]
Let $T$ be the maximal torus in $U$ whose $\breve F_0$-points are
\[
   T(\breve F_0) = \bigl\{\, \diag\bigl(a_1,\dotsc,a_{n/2}, \ov a_{n/2}^{-1},\dotsc, \ov a_1^{-1}\bigr) \bigm| a_1,\dotsc,a_{n/2} \in \breve F^\times \,\bigr\}.
\]
Since $K$ is a special maximal parahoric subgroup, by the Cartan decomposition $g = k_1 t k_2$ for some $k_1,k_2 \in K$ and $t \in T(\breve F_0)$.  Then $(M + gM)/M = (M + k_1tM)/M$, which, multiplying by $k_1^{-1}$, is isomorphic to $(M + tM)/M$.  Since $k_1$ and $k_2$ of course have trivial Kottwitz invariant, we have therefore reduced the lemma to the case $g = t$, where it is obvious.  (It may be helpful to note that an element of the form $a/\ov a$, $a \in \breve F^\times$, is congruent to $(-1)^{\ord_\pi a} \bmod \pi$.)
\end{proof}

\begin{lemma}\label{C parity lem}
Let $n$ be even, let $C$ be a non-degenerate $F/F_0$-hermitian space of dimension $n$, let $N := C \otimes_{F_0} \breve F_0$, and let $\tau := \id_C \otimes \sigma$.  Let $M$ be an $O_{\breve F}$-lattice in $N$ which is $\pi$-modular with respect to the induced $\breve F/ \breve F_0$-hermitian form.  Then the $O_{\breve F}$-length of $(M + \tau^{-1}M)/\tau^{-1}M$ is even or odd according as $C$ is a split or non-split hermitian space.
\end{lemma}

\begin{proof}
The length in question is the same as the length of the module $(M + \tau M)/M$, which we will work with instead.  Let $U$ denote the unitary group of $C$ over $\Spec F_0$.  (In terms of the notation in the previous lemma, in this way we view the unitary group of $N$ as defined over $F_0$.)

If $C$ is split, then it contains a $\pi$-modular $O_F$-lattice $\Lambda$.  Let $L := O_{\breve F} \cdot \Lambda \subset N$.  Then $M = gL$ for some $g\in U(\breve F_0)$.  Hence $\tau M = \sigma(g)L$, and
\[
   (M + \tau M)/M = \bigl(M + \sigma(g)g^{-1}M\bigr)\big/M.
\]
Since $\sigma(g)$ and $g$ have the same Kottwitz invariant, we conclude from Lemma~\ref{Kott lem} that the $O_{\breve F}$-length of the displayed module is even.

If $C$ is non-split, then it can be expressed as an orthogonal direct sum of a non-split $2$-dimensional space $C_1$ and a split $(n-2)$-dimensional space $C_2$.  There exists a basis $e_1,e_2$ of $C_1$ such that the associated matrix of the hermitian form is of the form
\[
   \begin{bmatrix}
		1\\
		  &  -b
	\end{bmatrix}
\]
for some $b \in O_{F_0}^\times \smallsetminus \RN O_F^\times$.  Let $\beta \in O_{\breve F_0}$ be a square root of $b$.  Then the vectors
\[
   f_1 := e_1 + \beta^{-1}e_2,\quad f_2:= \frac 1 2(e_1 - \beta^{-1}e_2)
\]
form a split basis of $N_1 := C_1 \otimes_{F_0} \breve F_0$.  Hence $L_1 := O_{\breve F} f_1 + O_{\breve F} \pi f_2$ is a $\pi$-modular lattice in $N_1$.  Since $C_2$ is split, $N_2 := C_2 \otimes_{F_0} \breve F_0$ contains a $\pi$-modular lattice $L_2$ which is stable under $\id_{C_2} \otimes \sigma$.  Then the lattice $L := L_1 \oplus L_2$ in $N$ is $\pi$-modular, and by inspection, $L$ and $\tau L$ differ by the transformation $g_0 \in U(\breve F_0)$ which interchanges $f_1$ and $f_2$ and is the identity on $N_2$.  Then $\det(g_0) = -1$, and hence $\kappa(g_0) = -1$.  Now we argue as in the case that $C$ is split.  Writing $M = gL$ for an appropriate $g \in U(\breve F_0)$, we have
\[
   (M + \tau M)/M 
	   = \bigl(gL + \sigma(g)\tau L\bigr) / gL 
		\cong \bigl(L + g^{-1}\sigma(g)g_0L\bigr) \big/ L.
\]
By Lemma~\ref{Kott lem}, the length of the module on the right is odd.
\end{proof}

In the rest of this subsection we are going to fix particular framing objects $(\BX_n,\iota_{\BX_n},\lambda_{\BX_n})$, $n \geq 1$, over $\Spec \ov k$ for use in the rest of the paper.  When $n = 1$, we define
\[
   (\BX_1, \iota_{\BX_1}) := (\BE, \iota_\BE),
\]
where as in the Introduction, \BE is the unique (up to isomorphism) connected $p$-divisible group of dimension $1$ and height $2$ over $\Spec \ov k$, and $\iota_\BE$ is an embedding
\[
   \iota_\BE\colon O_F \inj O_D = \End_{O_{F_0}}(\BE),
\]
which makes \BE into a formal $\pi$-divisible module of relative height $1$.
(Recall that in this section $F_0 = \BQ_p$.)  The Dieudonn\'e module \BM of \BE can be identified with $W(\ov k)^2 = O_{\breve F_0}^2$ endowed with the Frobenius operator given in matrix form by
\[
   \begin{bmatrix}
	   0  &  \varpi\\
	   1  &  0
   \end{bmatrix}
   \sigma,
\]
where $\sigma$ denotes the usual Frobenius homomorphism on the Witt vectors.
To give a polarization on \BE is to give an alternating bilinear pairing on \BM with associated matrix of the form
\[
   \begin{bmatrix}
	   0  &  \delta\\
	   -\delta  &  0
   \end{bmatrix}
\]
for $\delta \in O_{\breve F_0}$ satisfying $\sigma(\delta) = -\delta$.  We define the (principal) polarization $\lambda_{\BE}$ by fixing any such $\delta \in O_{\breve F_0}^\times$ once and for all.  Note that any other principal polarization of \BE differs from $\lambda_\BE$ by an $O_{F_0}^\times$-multiple.  We define
\[
   \lambda_{\BX_1} := -\lambda_\BE.
\]
The $F_0$-algebra $D = \End_{O_{F_0}}^\circ(\BE)$ is the quaternion division algebra over $F_0$, and $O_D = \End_{O_{F_0}}(\BE)$ is its maximal order.  
The Rosati involution attached to $\lambda_\BE$ (and to $\lambda_{\BX_1}$) is the main involution on $D$, and therefore it induces the nontrivial Galois automorphism on $O_F$.

When $n = 2$, we define
\[
   \BX_2 := O_F \otimes_{O_{F_0}} \BE
\]
via the Serre tensor construction, with $\iota_{\BX_2}$ given by the tautological $O_F$-action on the left tensor factor.  Then canonically
\[
   \Lie(O_F \otimes_{O_{F_0}} \BE) \cong O_F \otimes_{O_{F_0}} \Lie \BE
\]
as $(O_F \otimes_{O_{F_0}} \ov k)$-modules.  It is clear from this that $(\BX_2, \iota_{\BX_2})$ satisfies the Kottwitz and spin conditions.  To define the polarization $\lambda_{\BX_2}$, first note that canonically
\[
   \BX_2^\vee \cong \ov{O_F^\vee} \otimes_{O_{F_0}} \BE^\vee
\]
as $O_F$-modules (where $O_F$ acts on $\BX_2^\vee$ as prescribed by \eqref{dual action}); here $\ov{O_F^\vee}$ is the $O_{F_0}$-linear dual of $O_F$, made into an $O_F$-module by the rule
\[
   (x\cdot f)(y) = f(\ov x y) \quad\text{for}\quad x,y \in O_F, \quad f \in O_F^\vee.
\]
Define the (injective but not surjective) $O_F$-linear map
\begin{equation}\label{varphi}
   \varphi\colon O_F  \to  \ov{O_F^\vee}, \quad x \mapsto \bigl[y \shortmapsto \tfrac 1 2 \tr_{F/F_0}(\ov x y)\bigr].
\end{equation}
Then we define $\lambda_{\BX_2}$ to be the map
\[
    O_F \otimes_{O_{F_0}} \BE 
        \xra{\varphi \otimes \lambda_\BE}
		\ov{O_F^\vee} \otimes_{O_{F_0}} \BE^\vee \cong (O_F \otimes_{O_{F_0}} X)^\vee.
\]
Note that $\lambda_{\BX_2}$ is anti-symmetric because $\varphi$ is symmetric and $\lambda_\BE$ is anti-symmetric,\footnote{There is a mistake in \cite[(6.2)]{KR-U1}: the $\delta^{-1}$ in loc.~cit.~should be eliminated to obtain an anti-symmetric homomorphism into the dual, rather than a symmetric one. } and one readily verifies that $\Ker \lambda_{\BX_2} = \BX_2[\iota_{\BX_2}(\pi)]$.

The triple $(\BX_2,\iota_{\BX_2},\lambda_{\BX_2})$ can be expressed in more concrete terms after choosing an $O_{F_0}$-basis for $O_F$.  Indeed the choice of basis $1$, $\pi$ induces an isomorphism of $O_{F_0}$-modules
\begin{equation}\label{BX_2 concrete}
   O_F \otimes_{O_{F_0}} \BE \simeq  \BE \times \BE.
\end{equation}
The action of $\pi$ on the left-hand side of this isomorphism translates to the action of the matrix
\[
   \begin{bmatrix}
	  0  &  \varpi\\
	  1  &  0
   \end{bmatrix}
\]
on the right-hand side.  Using the dual basis to identify $\ov{O_F^\vee} \otimes_{O_{F_0}} \BE^\vee \simeq (\BE^\vee)^2$, the polarization $\lambda_{\BX_2}$ is given by the matrix
\begin{equation}\label{explicit lambda_BX_2}
   \begin{bmatrix}
	  \lambda_\BE  &  0\\
	  0  &  - \varpi\lambda_\BE
   \end{bmatrix}.
\end{equation}

\begin{remark}\label{KR framing object}
A framing object in our setting in the case $n = 2$ is also described in \cite[\S5~d)]{KR-alt}, but contrary to the claim made there, it does not give rise to a formally smooth moduli space.  Indeed this object manifestly does not satisfy the spin condition; or one can check directly that the hermitian space corresponding to it (defined as in \S\ref{aut of framing obs} below) is split.  The object in loc.~cit.~should be replaced with $\BX_2$ as we have defined it.
\end{remark}

Now we define framing objects for $n > 2$. As in the Introduction, we denote by $\ov \BE$ the same $O_{F_0}$-module as $\BE$, but with $O_F$-action $\iota_{\ov\BE}$ equal to $\iota_\BE$ precomposed by the nontrivial Galois automorphism.  Then $\lambda_\BE$ is still $O_F$-linear with respect to $\iota_{\ov\BE}$, and we denote it by $\lambda_{\ov\BE}$.  If $n$ is even, then we define
\begin{align*}
   \BX_n &:= \BX_2 \times \ov\BE^{n-2},\\
   \iota_{\BX_n} &:= \iota_{\BX_2} \times \iota_{\ov\BE}^{n-2},
\end{align*}
and, in matrix form,
\[
   \lambda_{\BX_n} := \lambda_{\BX_2} \times 
      \diag \biggl(
	     \underbrace{\begin{bmatrix} 
	           0  &  \lambda_{\ov\BE} \iota_{\ov\BE}(\pi)\\
	           -\lambda_{\ov\BE} \iota_{\ov\BE}(\pi)  &  0
			\end{bmatrix}
			,\dotsc,
			\begin{bmatrix} 
	           0  &  \lambda_{\ov\BE} \iota_{\ov\BE}(\pi)\\
	           -\lambda_{\ov\BE} \iota_{\ov\BE}(\pi)  &  0
			\end{bmatrix}}_{(n-2)/2 \text{ times}}\biggr).
\]
Then indeed $\Ker \lambda_{\BX_n} = \BX_n[\iota_{\BX_n}(\pi)]$. 
If $n$ is odd, then we define
\begin{equation}\label{odd framing object}
   (\BX_n,\iota_{\BX_n}, \lambda_{\BX_n}) :=
      (\BX_{n-1} \times \ov\BE, \iota_{\BX_{n-1}} \times \iota_{\ov\BE},\lambda_{\BX_{n-1}} \times \lambda_{\ov\BE}).
\end{equation}
Then indeed $\Ker \lambda_{\BX_n} \subset \BX_n[\iota_{\BX_n}(\pi)]$ of height $n-1$.  For either parity of $n$, if $n > 1$, then it is clear that $\iota_{\BX_n}(\pi)$ acts on $\Lie \BX_n$ with rank $1$.  Hence $(\BX_n,\iota_{\BX_n}, \lambda_{\BX_n})$ is a unitary $p$-divisible group of signature $(1,n-1)$ for all $n$.

\subsection{Automorphisms of framing objects}\label{aut of framing obs}

For $n \geq 1$, let $g \mapsto g^\dag$ denote the Rosati involution on $\End_{O_F}^\circ(\BX_n)$ induced by $\lambda_{\BX_n}$. Define
\begin{equation}\label{Aut(BX_n)}
   U(\BX_n) := U(\BX_n,\iota_{\BX_n},\lambda_{\BX_n}) := 
      \bigl\{\, g \in \End_{O_F}^\circ(\BX_n) \bigm| gg^\dag = \id_{\BX_n} \,\bigr\}.
\end{equation}
Thus $U(\BX_n)$ is the group of $O_F$-linear self-quasi-isogenies of $\BX_n$ which preserve $\lambda_{\BX_n}$ on the nose.

Next define the space of \emph{special quasi-homomorphisms}
\begin{equation}\label{BV_n}
   \BV_n := \Hom_{O_F}^\circ \bigl(\ov\BE, \BX_n\bigr);
\end{equation}
cf.~e.g.~\cite[Def.~3.1]{KR-U1}.  Then $\BV_n$ is an $n$-dimensional $F$-vector space.  It carries a natural non-degenerate $F/F_0$-hermitian form $h$: for $x,y \in \BV_n$, the composite
\[
   \ov \BE \xra{y} \BX_n \xra{\lambda_{\BX_n}} \BX_n^\vee \xra{x^\vee} \ov\BE^\vee \xra{\lambda_{\ov\BE}^{-1}} \ov\BE
\]
lies in $\End_{O_F}^\circ(\ov\BE)$, and hence identifies with an element $h(x,y) \in F$ via the isomorphism 
\[
   \iota_{\ov\BE} \colon F \isoarrow \End_{O_F}^\circ(\ov\BE).
\]

\begin{lemma}
The hermitian space $(\BV_n,h)$ is non-split for all $n$.
\end{lemma}

\begin{proof}
When $n = 1$, the space $\BV_1 = \Hom_{O_F}^\circ(\ov\BE, \BE)$ is the $-1$-eigenspace in $D$ for the conjugation action by $\iota_\BE(\pi)$, endowed with the hermitian norm $(x, x) = - \RN x$, which is indeed non-split. When $n = 2$, using the $O_{F_0}$-linear isomorphism $\BX_2 \simeq \BE \times \BE$ in \eqref{BX_2 concrete}, $\BV_2$ identifies with
\[
   \biggl\{
	\begin{bmatrix}
		a\\
		b
	\end{bmatrix}
	   \in \RM_{2 \times 1}(D)
	\biggm|
	\begin{bmatrix}
		  & \varpi\\
	   1
	\end{bmatrix}
	\begin{bmatrix}
		a\\
		b
	\end{bmatrix}
	=
	\begin{bmatrix}
		a\\
		b
	\end{bmatrix}
	\iota_{\ov\BE}(\pi)
	\biggr\}
	=
	\biggl\{
	\begin{bmatrix}
		b\iota_{\ov\BE}(\pi)\\
		b
	\end{bmatrix}
	\biggm|
	b \in D
	\biggr\}.
\]
For $b\in D$, one computes from the explicit form \eqref{explicit lambda_BX_2} of the polarization that $\bigl[\begin{smallmatrix} b\iota_{\ov\BE}(\pi) \\ b \end{smallmatrix}\bigr]$ pairs with itself under $h$ to $-2\varpi \RN b$.  Thus $\BV_2$ has no nonzero isotropic vectors, which characterizes it as the non-split hermitian space of dimension $2$.
	
To complete the proof for higher $n$, note that the definition of $\BV_n$ as a hermitian space makes sense for any polarized formal $O_F$-module in place of $\BX_n$.  Doing this for the pair $(\ov\BE,\lambda_{\ov\BE})$, we obtain a $1$-dimensional space $V_1$ which is obviously split; and doing this for
\[
   \biggl( \ov\BE^2, 
	   \begin{bmatrix}
         0  &  \lambda_{\ov\BE} \iota_{\ov\BE}(\pi)\\
         -\lambda_{\ov\BE} \iota_{\ov\BE}(\pi)  &  0
		\end{bmatrix}
	\biggr),
\]
we obtain a $2$-dimensional space $V_2$ which is obviously split.  When $n$ is even, we conclude that
\[
   \BV_n \cong \BV_2 \oplus V_2^{(n-2)/2}
\]
is non-split, because it is the orthogonal direct sum of a non-split space and an even-dimensional split space.  When $n$ is odd, we conclude that
\[
   \BV_n \cong \BV_{n-1} \oplus V_1
\]
is non-split, because it is the orthogonal direct sum of an even-dimensional non-split space with a split space.
\end{proof}

The group $U(\BX_n)$ acts naturally from the left on $\BV_n$, and in this way identifies with $\U(h)$.  Thus in terms of the notation \eqref{U_i def}, we may, and will, choose an isomorphism
\begin{equation}\label{Aut(BX_n)=U(W_1)}
	U(\BX_n) \simeq U_1(F_0)
\end{equation}
for all $n$.

\subsection{The moduli space}\label{CN_n def}

We now define the moduli space $\CN_n$.  For $S$ a scheme over $\Spf O_{\breve F}$, let
\[
   \ov S := \Spec \CO_S/\pi\CO_S.
\]
The $S$-points on $\CN_n$ are isomorphism classes of quadruples
\[
   (X,\iota_X,\lambda_X,\rho_X),
\]
where $(X,\iota_X,\lambda_X)$ is a unitary $p$-divisible group of signature $(1,n-1)$, and where
\[
   \rho_X \colon X \times_S \ov S \to \BX_n \times_{\Spec \ov k} \ov S
\]
is an $O_F$-linear quasi-isogeny of height $0$ such that $\rho_X^*(\lambda_{\BX_n}\times_{\Spec \ov k} \ov S)$ is locally an $O_{F_0}^\times$-multiple of $\lambda_X \times_S \ov S$ in $\Hom_{O_F}^\circ(X_{\ov S}, X_{\ov S}^\vee)$, i.e., locally on $S$, 
\begin{equation}\label{multiplierconst}
   \rho_X^*(\lambda_{\BX_n}\times_{\Spec \ov k} \ov S) = c(\lambda_X)\cdot (\lambda_X \times_S \ov S), \quad c(\lambda_X)\in O_{F_0}^\times .
\end{equation}  Here an isomorphism between quadruples $(X,\iota_X,\lambda_X,\rho_X) \isoarrow (Y,\iota_Y,\lambda_Y,\rho_Y)$ is an $O_F$-linear isomorphism of $p$-divisible groups $\alpha\colon X \isoarrow Y$ over $S$ such that $\rho_Y \circ (\alpha \times_S \ov S) = \rho_X$ and such that $\alpha^* (\lambda_Y)$ is locally an $O_{F_0}^\times$-multiple of $\lambda_X$.

Each $g$ in the group $U(\BX_n)$ \eqref{Aut(BX_n)} is a quasi-isogeny of height $0$, and therefore $U(\BX_n)$ acts naturally on $\CN_n$ on the left via the rule $g\cdot (X,\iota_X,\lambda_X,\rho_X) = (X,\iota_X,\lambda_X, g\rho_X)$.

\begin{remark}\label{alt formulation}
We have formulated the moduli problem defining $\CN_n$ in a way that conforms with other instances of RZ spaces in the literature, by taking the polarization $\lambda_X$ up to scalar in $O_{F_0}^\times$.  But the definition can be reformulated without reference to scalar factors: consider the moduli problem of quadruples $(X,\iota_X,\lambda_X,\rho_X)$ as above, except where $\rho_X^*(\lambda_{\BX_n} \times_{\Spec \ov k} \ov S)$ is required to equal $ \lambda_X \times_S \ov S$ on the nose, and where isomorphisms $\alpha$ as above satisfy $\alpha^* (\lambda_Y) = \lambda_X$ on the nose.  If $(X,\iota_X,\lambda_X,\rho_X)$ determines a point on this variant moduli problem, then $(X,\iota_X,\lambda_X,\rho_X)$ also determines a point on $\CN_n$, and it is easy to see that this defines an isomorphism from the variant moduli problem to $\CN_n$. In some situations this variant of the moduli problem is a little more convenient to work with, cf.\ Remark \ref{alt simpler} and the proof of Proposition \ref{Serre isom}.
\end{remark}

\begin{example}[$n = 1$]\label{CN_1}
Let us make the definition of $\CN_n$ explicit in the case $n = 1$.  The framing object $\BE = \BX_1$ is a formal $\pi$-divisible $O_F$-module of height $1$ via $\iota_\BE$.  When $\ov S$ is reduced, any framing map $\rho$ of height $0$ into $\BE \times_{\Spec \ov k} \ov S$ as above must be an isomorphism, and it follows that $\CN_1$ is the universal deformation space of $\BE$ over $\Spf O_{\breve F}$, which is just $\Spf O_{\breve F}$ itself.  We write \CE for the universal $p$-divisible group over $\CN_1$, endowed with its $O_F$-action $\iota_\CE$, principal polarization $\lambda_\CE$, and framing $\rho_\CE\colon \CE_{\ov k} \isoarrow \BE$.  Of course, the triple $(\CE,\iota_\CE,\rho_\CE)$ is nothing but the canonical lift of $(\BE,\iota_\BE)$ in the sense of Lubin--Tate theory, as in the Introduction.
\end{example}

\subsection{Formal smoothness and essential properness}\label{formal smoothness section}
In this subsection we prove the following basic result on the geometry of $\CN_n$.

\begin{proposition}\label{formally sm ess proper}
The formal scheme $\CN_n$ is formally locally of finite type and formally smooth over $\Spf O_{\breve F}$ of relative formal dimension $n-1$ at every point.  If $n$ is even, then $\CN_n$ is essentially proper over $\Spf O_{\breve F}$.
\end{proposition}

Here \emph{essentially proper} means that every irreducible component of $(\CN_n)_\red$ is proper over $\Spec \ov k$.  The proof of the proposition is via the \emph{local model}, whose definition we now recall in the situation at hand.  Let
\[
   m := \lfloor n/2 \rfloor.
\]
Let $e_1,\dotsc,e_n$ denote the standard basis in $F^n$.  Let $\phi$ be the (split) $F/F_0$-hermitian form on $F^n$ specified by
\[
   \phi(ae_i,be_j) = \ov a b \delta_{i,n+1-j} \quad\text{(Kronecker delta)}.
\]
Let \aform 
be the 
alternating 
$F_0$-bilinear form $F^n \times F^n \to F_0$ defined by
\[
   \langle x,y \rangle := \frac 1 2 \tr_{F/F_0} \bigl(\pi^{-1}\phi(x,y)\bigr).
\]
Then
\begin{equation}\label{angle eq}
   \langle \pi x , y \rangle = - \langle x, \pi y \rangle
	\quad\text{for all}\quad
	x,y \in F^n.
\end{equation}
For $i = bn + c$ with $0 \leq c < n$, define the $O_F$-lattice
\[
   \Lambda_i := \sum_{j = 1}^c \pi^{-b-1}O_F e_j + \sum_{j = c+1}^n \pi^{-b} O_F e_j \subset F^n.
\]
For each $i$, the form \aform induces a perfect $O_{F_0}$-bilinear pairing
\[
   \Lambda_i \times \Lambda_{-i} \to O_{F_0}.
\]
In this way, for fixed nonempty $I \subset \{0,\dotsc,m\}$, the set
\[
   \Lambda_I := \{\, \Lambda_i \mid i \in \pm I + n\BZ \, \}
\]
forms a polarized chain of $O_F$-lattices over $O_{F_0}$ in the sense of \cite[Def.\  3.14]{RZ}.
The following definition of the naive local model is an alternative formulation of \cite[Def.\ 3.27]{RZ} in our situation, in the case $I = \{m\}$.

\begin{definition}\label{LM def}
The \emph{naive local model $M^\naive$} is the scheme over $\Spec O_F$ representing the functor that sends each $O_F$-scheme $S$ to the set of all families
\[
   (\CF_i \subset \Lambda_i \otimes_{O_{F_0}} \CO_S)_{i \in \pm m + n\BZ}
\]
such that
\begin{altenumerate}
\renewcommand{\theenumi}{\arabic{enumi}}
\item for all $i$, $\CF_i$ is an $O_F \otimes_{O_{F_0}} \CO_S$-submodule of $\Lambda_i \otimes_{O_{F_0}} \CO_S$ which is Zariski-locally on $S$ an $\CO_S$-direct summand of rank $n$;
\item for all $i < j$, the natural arrow $\Lambda_i \otimes_{O_{F_0}} \CO_S \to \Lambda_j \otimes_{O_{F_0}} \CO_S$ carries $\CF_i$ into $\CF_j$;
\item\label{LM periodic cond} for all $i$, the isomorphism $\Lambda_i \otimes_{O_{F_0}} \CO_S \xra[\undertilde]{\pi \otimes 1} \Lambda_{i-n} \otimes_{O_{F_0}} \CO_S$ identifies $\CF_i \isoarrow \CF_{i-n}$;
\item\label{LM perp cond} for all $i$, the perfect $\CO_S$-bilinear pairing
\[
   (\Lambda_i \otimes_{O_{F_0}} \CO_S) \times (\Lambda_{-i} \otimes_{O_{F_0}} \CO_S)
   \xra{\aform \otimes \CO_S} \CO_S
\]
identifies $\CF_i^\perp$ with $\CF_{-i}$ inside $\Lambda_{-i} \otimes_{O_{F_0}} \CO_S$; and
\item (\emph{Kottwitz condition}) for all $i$, the section $\pi \otimes 1 \in O_F \otimes_{O_{F_0}} \CO_S$ acts on $\CF_i$ as an $\CO_S$-linear endomorphism with characteristic polynomial
\[
   \charac(\pi \otimes 1 \mid \CF_i) = (T-\pi)^{n-1}(T+\pi) \in \CO_S[T].
\]
\setcounter{filler}{\value{enumi}}
\end{altenumerate}

The \emph{local model $M^\loc$} is the subscheme of $M^\naive$ defined by the additional conditions
\begin{altenumerate}
\renewcommand{\theenumi}{\arabic{enumi}}
\setcounter{enumi}{\value{filler}}
\item\label{LM wedge cond} (\emph{wedge condition}) for all $i$,
\[
   \sideset{}{_{\CO_S}^2}\bigwedge ( \pi \otimes 1 - 1\otimes \pi \mid \CF_i ) = 0;
\]
and
\item\label{LM spin cond} if $n > 1$, then $\pi \otimes 1 \mid \CF_i \otimes_{\CO_S} \kappa(s)$ is nonvanishing for all $s \in S$ and all $i$.
\end{altenumerate}
\end{definition}

Plainly $M^\naive$ is representable by a closed subscheme of a product of Grassmannians.  Furthermore the inclusion $M^\loc \subset M^\naive$ is an isomorphism on generic fibers, which both identify naturally with $\BP_F^{n-1}$ \cite[\S1.5.3]{PR}.

\begin{proposition}\label{LM smooth}
$M^\loc$ is smooth over $\Spec O_F$ of relative dimension $n-1$ at every point, and a closed subscheme of $M^\naive$ when $n$ is even.
\end{proposition}

\begin{proof}
This is essentially a matter of collecting results in the literature.  Let $M^\wedge$ denote the closed subscheme of $M^\naive$ defined by the wedge condition.  Then $M^\loc$ is an open subscheme of $M^\wedge$.  As is explained in \cite[\S3.3]{PR}, the geometric special fiber $M_{\ov k}^\naive$ of the naive local model embeds into an affine flag variety for $\GU_n$, where it and $M_{\ov k}^\wedge$ decompose (topologically) into unions of Schubert cells.  In the present setting, the Schubert cells $C_r$ that occur in $M_{\ov k}^\naive$ are described in \S2.4 of loc.~cit.  They are indexed by the rank $r$ of $\pi \otimes 1$ acting on $\CF_m$ at each point.  On the level of topological spaces, it follows from the definitions that $M_{\ov k}^\wedge = C_0 \cup C_1$ and $M_{\ov k}^\loc = C_1$.  Now, Arzdorf \cite[Prop.~3.2]{A} and \cite[\S5.3]{PR} show respectively for odd and even $n$ that $M^\wedge$ contains an open subscheme isomorphic to $\BA_{O_F}^{n-1}$, and it is immediate from these references that this open subscheme is contained in $M^\loc$.\footnote{Strictly speaking, these references work with signature 
$(n-1,1)$, whereas 
we are working with the
 opposite signature $(1,n-1)$.  But these situations are isomorphic:  explicitly, the isomorphism $F^n \isoarrow F^n$ given by applying the nontrivial Galois automorphism on each factor induces an isomorphism of lattice chains $\Lambda_I \isoarrow \Lambda_I$, which in turn induces an isomorphism between the corresponding local models.  Also, when $n$ is odd, we note that the scheme denoted $M^\loc$ in \cite{A} does not coincide with $M^\loc$ as we have defined it.\label{signature flip}}
Since $M^\loc$ has generic fiber $\BP_F^{n-1}$, and since $C_1$ is an orbit for a group action, it follows that $M^\loc$ is everywhere smooth of relative dimension $n-1$.

Furthermore, when $n$ is even, one readily verifies that the perfect pairing
\[
   \Lambda_m \times \Lambda_m \xra[\undertilde]{\id \times \pi} \Lambda_m \times \Lambda_{-m} \xra\aform O_{F_0}
\]
is split symmetric.  Hence by conditions \eqref{LM periodic cond} and \eqref{LM perp cond} above, $M^\naive$ embeds into the orthogonal Grassmannian $\OGr$ of totally isotropic $n$-planes in $2n$-space; cf.~\cite[Rem.~2.32]{PRS}.  The scheme $\OGr$ has $2$ connected components, and it is easy to see that these separate $C_0$, which consists of a single point in the special fiber, from the rest of $M^\wedge$.  Hence $M^\loc = M^\wedge \smallsetminus C_0$ is closed in $M^\wedge$, and hence in $M^\naive$.
\end{proof}

\begin{remark}\label{relation to spin cond}
When $n$ is even, and in the presence of the other conditions in the definition of $M^\loc$, condition \eqref{LM spin cond} is equivalent to the \emph{spin condition} formulated in \cite[\S7.2]{PR}.  For a general signature and parahoric level structure, the spin condition in loc.~cit.~is not a purely topological condition, but in this special case it is; see again Rem.~2.32 and the paragraph following it in \cite{PRS}.

When $n$ is odd, and again in the presence of the other conditions, condition \eqref{LM spin cond} \emph{implies} the spin condition in \cite{PR}. Indeed, the spin condition is a closed condition on $M^\naive$ which is satisfied on the generic fiber, and hence on the flat closure of the generic fiber, and we have just seen that $M^\loc$ is smooth, and hence flat, with the same generic fiber.  However, \eqref{LM spin cond} and the spin condition are not equivalent, since imposing the spin condition on $M^\wedge$ does not eliminate the point $C_0$.

The above relationships are the origin of the term ``spin condition'' in the definition of unitary $p$-divisible group in \S\ref{unitary p-div groups}. But, for the reason just explained, note that this is a slight abuse of language when $n$ is odd.
\end{remark}

\begin{remark}\label{better def}
More is true when $n$ is odd.  Indeed, in this case let $N$ denote the scheme-theoretic closure of the generic fiber of $M^\naive$ in $M^\naive$.  Then $N$ contains $M^\loc$ as an open subscheme, and is itself smooth over $\Spec O_F$ by Richarz \cite[Prop.~4.16]{A}.  The main result of \cite{S} establishes a moduli description for $N$, by introducing a condition which is weaker than \eqref{LM spin cond} above but stronger than the spin condition in \cite{PR}.  However, this condition is much more complicated to state than \eqref{LM spin cond}, and for the purposes of this paper it suffices to work just with \eqref{LM spin cond} instead.
\end{remark}

The link between Propositions \ref{formally sm ess proper} and \ref{LM smooth} is via the \emph{local model diagram} \cite{RZ}.  Let $\CN_n^\naive$ denote the moduli functor over $\Spf O_{\breve F}$ defined in the same way as $\CN_n$, except without the wedge and spin conditions.  By \cite[Th.~3.25]{RZ}, $\CN_n^\naive$ is representable by a formal scheme which is formally locally of finite type over $\Spf O_{\breve F}$.  We will see in the course of proving Proposition \ref{formally sm ess proper} below that the inclusion
\[
   \CN_n \subset \CN_n^\naive
\]
is a locally closed embedding, and a closed embedding when $n$ is even.

Let $\wt \CN_n^\naive$ denote the functor that associates to each scheme $S$ over $\Spf O_{\breve F}$ the set of isomorphism classes of quintuples
\[
   (X,\iota_X,\lambda_X,\rho_X,\gamma),
\]
where $(X,\iota_X,\lambda_X,\rho_X) \in \CN_n^\naive(S)$, and $\gamma$ is an isomorphism of polarized chains of $O_F \otimes_{O_{F_0}} \CO_S$-modules
\[
   [{}\dotsb \xra{\iota_X(\pi)_*} M(X) \xra{\iota_X(\pi)_*} M(X) \xra{\iota_X(\pi)_*} \dotsb{}] \xra[\undertilde]\gamma \Lambda_{\{m\}} \otimes_{O_{F_0}} \CO_S
\]
when $n$ is even, and
\[
   [{}\dotsb \to M(X) \xra{(\lambda_X)_*} M(X^\vee) \xra{\mu_*} \dotsb {}] \xra[\undertilde]\gamma \Lambda_{\{m\}} \otimes_{O_{F_0}} \CO_S
\]
when $n$ is odd, in the terminology of \cite{RZ}.  Let us explain the notation.  We have denoted by $M$ the functor that assigns to a $p$-divisible group the Lie algebra of its universal vector extension.  Our requirements for $\iota_X$ and $\lambda_X$ imply that there is a unique (necessarily $O_F$-linear) isogeny $\mu\colon X^\vee \to X$ such that the composite
\[
   X \xra{\lambda_X} X^\vee \xra\mu X
\]
is $\iota_X(\pi)$; thus $\mu$ is an isomorphism or of height $1$ according as $n$ is even or odd.  Upon applying $M$, this diagram extends periodically to the chains appearing above (and it explains the source of the chain's polarization when $n$ is even).

The functor $\wt\CN_n^\naive$ is representable by a formal scheme over $\Spf O_{\breve F}$, and it sits in a diagram
\[
   \xymatrix@C-4ex{
      & \wt\CN_n^\naive \ar[dl]_-{\varphi} \ar[dr]\\
	  \CN_n^\naive & & M_{\Spf O_{\breve F}}^\naive.
   }
\]
Here $\varphi$ is the natural map that forgets $\gamma$; it is a torsor under the automorphism scheme of $\Lambda_{\{m\}}$ as a polarized $O_F$-lattice chain over $O_{F_0}$, which is smooth.  The arrow on the right sends an $S$-point $(X,\iota_X,\lambda_X,\rho_X,\gamma)$ to the family $(\CF_i \subset \Lambda_i \otimes_{O_{F_0}} \CO_S)_{i \in \pm m + n\BZ}$, where for each $i$, $\CF_i$ is the image under $\gamma$ in $\Lambda_i \otimes_{O_{F_0}} \CO_S$ of the $\Fil^1$ term in the covariant Hodge filtration for the $p$-divisible group.  With this in hand, we now arrive at the proof of Proposition \ref{formally sm ess proper}.

\begin{proof}[Proof of Proposition \ref{formally sm ess proper}]
The key point is that by \cite[Prop.\ 3.33]{RZ}, after passing to an \'etale cover $\CU \to \CN_n^\naive$, the map $\varphi$ admits a section $s$ such that the composite $\CU \to M_{\Spf O_{\breve F}}^\naive$ in the diagram
\[
   \xy
	   (-17,11)*+{\CU}="U";
	   (7,15)*+{\wt\CN_n^\naive}="N~";
		(-7,0)*+{\CN_n^\naive}="N";
		(21,0)*+{M_{\Spf O_{\breve F}}^\naive}="M";
		{\ar "U";"N~" ^-{s}};
		{\ar "U";"N"};
		{\ar "N~";"N" _-\varphi};
		{\ar "N~";"M"};
	\endxy
\]
is formally \'etale.  We claim that the respective inverse images of $\CN_n$ and $M_{\Spf O_{\breve F}}^\loc$ in \CU coincide.  In fact, we will show that the respective wedge conditions on the one hand, and the spin condition and condition \eqref{LM spin cond} on the other hand, separately pull back to equivalent conditions on \CU.

First consider the wedge conditions.  Let $(\CF_i \subset \Lambda_i \otimes_{O_{F_0}} \CO_{M^\naive})_i$ denote the universal object over $M^\naive$.  For fixed $i$, consider the tautological exact sequence
\begin{equation}\label{exact seq}
   0 \to \CF_i \to \Lambda_i \otimes_{O_{F_0}} \CO_{M^\naive} 
	         \to (\Lambda_i \otimes_{O_{F_0}} \CO_{M^\naive}) / \CF_i \to 0.
\end{equation}
It is easy to see that on the explicit affine charts on $M^\naive$ calculated in \cite[\S5.1--5.3]{PR},
\[
   \sideset{}{_{\CO_{M^\naive}}^2}\bigwedge ( \pi \otimes 1 - 1\otimes \pi \mid \CF_i ) = 0
	\iff
	\sideset{}{_{\CO_{M^\naive}}^2}\bigwedge \bigl( \pi \otimes 1 + 1\otimes \pi \mid (\Lambda_i \otimes_{O_{F_0}} \CO_{M^\naive}) / \CF_i \bigr) = 0.
\]
Since these charts meet every Schubert cell in the special fiber (after embedding the geometric special fiber into an affine flag variety, as discussed in the proof of Proposition \ref{LM smooth}), the equivalence in the display holds on all of $M^\naive$.%
\footnote{Strictly speaking, when $n$ is odd, \cite{PR} computes an affine chart for a different maximal parahoric level structure than the one we are using.  But the calculations in \cite{PR} are easily adapted to our case.  See also Arzdorf \cite{A}.}
Now consider the $\CO_{M^\naive}$-linear dual of \eqref{exact seq},
\[
   0 \to \bigl((\Lambda_i \otimes_{O_{F_0}} \CO_{M^\naive}) / \CF_i\bigr)^\vee 
      \to (\Lambda_i \otimes_{O_{F_0}} \CO_{M^\naive})^\vee
	  \to \CF_i^\vee
	  \to 0. 
\]
The isomorphism $\Lambda_{-i} \otimes_{O_{F_0}} \CO_{M^\naive} \isoarrow (\Lambda_i \otimes_{O_{F_0}} \CO_{M^\naive})^\vee$ induced by \aform identifies $\CF_{-i}$ with $((\Lambda_i \otimes_{O_{F_0}} \CO_{M^\naive}) / \CF_i )^\vee$ by condition \eqref{LM perp cond}, and it identifies the operator $\pi \otimes 1$ on $\Lambda_{-i} \otimes_{O_{F_0}} \CO_{M^\naive}$ with $-\pi \otimes 1$ on $(\Lambda_i \otimes_{O_{F_0}} \CO_{M^\naive})^\vee$ by \eqref{angle eq}.  It follows from these observations and from condition \eqref{LM periodic cond} that the wedge condition on $M^\naive$ holds for all $i$ as soon as it holds for a single $i$, and that it and the wedge condition on $\CN_n^\naive$ pull back to equivalent conditions on \CU.

Now consider condition \eqref{LM spin cond} on $M^\naive$.  This condition is automatically satisfied in the generic fiber, so let $S$ be the spectrum of a field $\kappa$ which is an extension of $k$. Let $(\CF_i \subset \Lambda_i \otimes_{O_{F_0}}\kappa)_i$ be an $S$-point on $M^\naive$.  Since $\Ker(\pi \otimes 1 \mid \Lambda_i \otimes_{O_{F_0}}\kappa) = \Im(\pi \otimes 1 \mid \Lambda_i \otimes_{O_{F_0}}\kappa)$, we have
\[
   (\pi \otimes 1 \mid \CF_i ) = 0 
   \;\iff\;
   \CF_i = (\pi \otimes 1) \cdot (\Lambda_i \otimes_{O_{F_0}} \kappa)
   \;\iff\;
   \bigl(\pi \otimes 1 \mid (\Lambda_i \otimes_{O_{F_0}} \kappa)/ \CF_i \bigr) = 0.
\]
Using conditions \eqref{LM periodic cond} and \eqref{LM perp cond}, it follows as above that \eqref{LM spin cond} holds for all $i$ as soon as it holds for a single $i$, and that it and the spin condition on $\CN_n^\naive$ pull back to equivalent conditions on \CU.

Thus we have shown that $\CN_n$ and $M_{\Spf O_{\breve F}}^\loc$ have common inverse image in \CU.  We conclude that $\CN_n \subset \CN_n^\naive$ is a locally closed embedding in general and closed embedding when $n$ is even, since the same is true of $M^\loc \subset M^\naive$.  Thus $\CN_n$ inherits the property of being formally locally of finite type from $\CN_n^\naive$.  By Proposition~\ref{LM smooth}, it also follows that $\CN_n$ is formally smooth over $\Spf O_{\breve F}$ of relative formal dimension $n-1$.  Finally, by Prop.~2.32 and the proof of Th.~3.25 in \cite{RZ}, $\CN_n^\naive$ is essentially proper over $\Spf O_{\breve F}$.  Thus the same is true of $\CN_n$ when $n$ is even.
\end{proof}

\begin{remark}
In analogy with Remark~\ref{better def}, when $n$ is odd, one can replace the spin condition in the definition of $\CN_n$ with a weaker condition, based on the one introduced in \cite{S}, to obtain an essentially proper, formally smooth formal scheme containing $\CN_n$ as an open formal subscheme; see \cite[\S7]{RSZ}.  Although this formal scheme is in some sense a ``better'' object, its definition is more complicated to state, and for our purposes it suffices to work with $\CN_n$ as we have defined it.
\end{remark}

\section{Intersection numbers}\label{intersection numbers}

In this section we define the intersection numbers that appear in our AT conjectures.  Let $n\geq 3$ be an odd integer.

\subsection{The morphisms $\delta_\CN$ and $\Delta_\CN$}
We begin by introducing some closed embeddings of moduli spaces.
As in the Introduction, over $O_{\breve F}$ we have the canonical lift $\CE$ of $\BE$ equipped with its action $\iota_\CE \colon O_F \to \End(\CE)$, its principal polarization $\lambda_\CE$, and its framing isomorphism $\rho_\CE\colon \CE_{\ov k} \isoarrow \BE$; see also Example \ref{CN_1}.  
We further have the quadruple $(\ov\CE, \iota_{\ov\CE}, \lambda_{\ov\CE}, \rho_{\ov\CE})$, where $\ov \CE$ is the same underlying $p$-divisible group, where the $O_F$-action $\iota_{\ov\CE}$ is obtained by precomposing $\iota_\CE$ with the nontrivial Galois automorphism on $O_F$, and where $\lambda_{\ov\CE} = \lambda_\CE$ and $\rho_{\ov\CE} = \rho_\CE$.  

Using $(\ov\CE, \iota_{\ov\CE}, \lambda_{\ov\CE}, \rho_{\ov\CE})$, we define a morphism of formal moduli schemes
\begin{equation}\label{delta}
	\delta_\CN\colon \CN_{n-1}^\naive\to \CN_n^\naive 
\end{equation}
as follows. Let $S$ be a scheme over $\Spf O_{\breve F}$, and let $(Y, \iota_Y, \lambda_Y, \rho_Y)\in \CN_{n-1}^\naive(S)$, cf.~\S\S\ref{CN_n def}, \ref{formal smoothness section}. Then locally on $S$ there exists $c_Y\in O_{F_0}^\times$ such that $\rho_Y^* (\lambda_{\BX_{n-1}}\times_{\Spec \ov k} \ov S)=c_Y\cdot (\lambda_Y \times_S \ov S)$, cf.~\eqref{multiplierconst}. Recall that for odd $n$ we have the framing object $\BX_n = \BX_{n-1} \times_{\Spec \ov k} \ov\BE$ and its polarization $\lambda_{\BX_n} = \lambda_{\BX_{n-1}} \times \lambda_{\ov\BE}$.  We define 
\[
   \delta_\CN\colon (Y, \iota_Y, \lambda_Y, \rho_Y) 
   \mapsto 
   \bigl(Y \times \ov\CE, \iota_Y \times \iota_{\ov\CE},\lambda_Y \times c_Y^{-1} \lambda_{\ov\CE},\rho_Y \times \rho_{\ov\CE} \bigr) .
\]
This is a well-defined morphism of functors. Indeed, $(\rho_Y\times \rho_{\ov\CE})^* (\lambda_{\BX_{n-1}}\times  \lambda_{\ov\BE}) = c_Y(\lambda_Y \times c_Y^{-1}\lambda_{\ov\BE})$. Furthermore, if 
$(Y, \iota_{Y}, \lambda_{Y}, \rho_{Y})$ 
is isomorphic to $(Y', \iota_{Y'}, \lambda_{Y'}, \rho_{Y'})$, then there exists an $O_F$-linear isomorphism $\alpha\colon Y'\isoarrow Y$ compatible with $\rho_{Y'}$ and $\rho_Y$,  and such that, locally on the base,  $\gamma \lambda_{Y'}= \alpha^*(\lambda_Y)$ for some $\gamma\in O_{F_0}^\times$. But then $c_{Y'} = \gamma c_Y$, from which one verifies that $\alpha\times  \id_{\ov\CE} \colon Y'\times \ov\CE \isoarrow Y\times \ov\CE$ defines an isomorphism from $\delta_\CN (Y', \iota_{Y'}, \lambda_{Y'}, \rho_{Y'})$ to $\delta_\CN (Y, \iota_{Y}, \lambda_{Y}, \rho_{Y})$.

\begin{remark}\label{alt simpler}
If one uses the variant formulation of the moduli problem defining $\CN_n$ given in Remark \ref{alt formulation}, where $c_Y$ is required to equal $1$, then it is essentially obvious that $\delta_\CN$ is well-defined.
\end{remark}

\begin{lemma}\label{delta lem}
The map $\delta_\CN$ induces a closed embedding of formal schemes 
\[
   \delta_\CN\colon \CN_{n-1} \to \CN_n.
\]
\end{lemma}

\begin{proof}
To see that $\delta_\CN$ carries $\CN_{n-1}$ into $\CN_n$, let $(Y,\iota,\lambda,\rho)$ be a point on $\CN_{n-1}$. Then $(\iota \times \iota_{\ov\CE})(\pi) \mid \Lie (Y \times \ov\CE)$ is pointwise nonvanishing because $\iota(\pi) \mid \Lie Y$ is, and $\delta_\CN(Y,\iota,\lambda,\rho)$ satisfies the wedge condition because $(Y,\iota,\lambda,\rho)$ does and because $(\iota_{\ov\CE}(\pi) + \pi \mid \Lie\ov\CE) = 0$.  That $\delta_\CN$ is a closed embedding follows from the facts that it is obviously a monomorphism, that $\CN_{n-1}$ is essentially proper over $\Spf O_{\breve F}$ by Proposition \ref{formally sm ess proper}, and that $\CN_n$ is separated over $\Spf O_{\breve F}$ by \cite[Lem.\ 2.3.23]{F}.
\end{proof}

Using $\delta_\CN$, we obtain a closed embedding of formal schemes (again using separatedness of $\CN_n$), 
\begin{equation*}
   \Delta_\CN\colon \CN_{n-1} \xra{(\id_{\CN_{n-1}},\delta_\CN)} \CN_{n-1} \times_{\Spf O_{\breve F}} \CN_n.
\end{equation*}

Let us conclude this subsection by explaining the equivariance properties of the maps $\delta_\CN$ and $\Delta_\CN$.  Since $n$ is odd, the non-split hermitian space $\BV_n = \Hom_{O_F}^\circ(\ov\BE, \BX_n )$ in \eqref{BV_n} has a canonical special vector $u$ of norm $1$, namely
\begin{equation}\label{u def}
   u := (0,\id_{\ov\BE}) \in \Hom_{O_F}^\circ\bigl(\ov\BE, \BX_n \bigr)
      = \Hom_{O_F}^\circ\bigl(\ov\BE, \BX_{n-1} \times \ov\BE\bigr).
\end{equation}
The stabilizer of $u$ in $U(\BX_n) \cong \U(\BV_n)$ identifies with $U(\BX_{n-1})$. In this way we obtain an identification of
\begin{equation}\label{group identifications}
   H_1(F_0) \subset U_1(F_0)
   \quad\text{with}\quad
   U(\BX_{n-1}) \subset U(\BX_n),
\end{equation}
as in \eqref{Aut(BX_n)=U(W_1)}.
%
Via these identifications, $H_1(F_0)$ acts on $\CN_{n-1}$, $U_1(F_0)$ acts on $\CN_n$, and the product $G_{W_1}(F_0) = H_1(F_0) \times U_1(F_0)$ acts on $\CN_{n-1} \times_{\Spf O_{\breve F}} \CN_n$.  The maps $\delta_\CN$ and $\Delta_\CN$ are then equivariant for the action of $H_1(F_0)$, i.e.
\begin{equation}\label{equivdelta}
\begin{aligned}
   \delta_\CN \bigl(h \cdot(Y, \iota_Y, \lambda_Y, \rho_Y) \bigr) &= h\cdot\delta_\CN\bigl((Y, \iota_Y, \lambda_Y, \rho_Y) \bigr) ,\\
   \Delta_\CN \bigl( h \cdot(Y, \iota_Y, \lambda_Y, \rho_Y) \bigr) &= h\cdot\Delta_\CN\bigl((Y, \iota_Y, \lambda_Y, \rho_Y) \bigr),
\end{aligned}
\end{equation}
where  $H_1(F_0)$ acts via its diagonal embedding into $G_{W_1}(F_0)$ on the right-hand side of the second equation.

\subsection{Homogeneous setting}
We now use $\Delta_\CN$ to define intersection numbers. We start with the homogeneous case.  Define the closed formal subscheme of $\CN_{n-1}\times_{\Spf O_{\breve F}}\CN_n$,
\begin{equation}\label{Delta cycle}
   \Delta := \Delta_\CN(\CN_{n-1}).
\end{equation}
Of course this is not to be confused with the discriminant \eqref{Delta}; throughout the paper context should suffice to  make the meaning of $\Delta$ clear.  For $g\in G_{W_1}(F_0)$, we define the translate
\begin{equation}\label{Deltaghomog}
   \Delta_g := g \Delta.
\end{equation}
We then define the intersection number of $\Delta$ and $\Delta_g$ by the Euler-Poincar\'e characteristic of the derived tensor product (cf.~\cite{KR-U1}),
\begin{equation}\label{Int(g) homog}
   \Int(g) := \langle\Delta, \Delta_g\rangle := \chi\bigl(\CO_\Delta\otimes^\BL\CO_{\Delta_g}\bigr) . 
\end{equation}

Note that, since $\CN_{n-1}\times_{\Spf O_{\breve F}}\CN_n$ is a regular formal scheme by Proposition~\ref{formally sm ess proper}, this derived tensor product is represented by a finite complex of locally free coherent modules. Furthermore, the intersection $\Delta\cap\Delta_g$ can be identified with 
$$
\Delta\cap\Delta_g=\Delta_\CN^{-1}(\Delta_g) ,
$$
where $\Delta$ and $\Delta_g$ are  closed formal subschemes of $\CN_{n-1}\times_{\Spf O_{\breve F}}\CN_n$. It follows that $\Delta\cap\Delta_g$ may be identified with a closed formal subscheme of $\CN_{n-1}$. Since $n-1$ is an even integer, the formal scheme $\CN_{n-1}$ is essentially proper over $\Spf O_{\breve F}$ by Proposition \ref{formally sm ess proper}. It follows that $\Int(g)$ is finite, provided that the underlying reduced scheme of the intersection  $\Delta\cap\Delta_g$ is of finite type over $\Spec O_{\breve F}$, and that the ideal of definition of this formal scheme is nilpotent. Indeed, under these conditions, $\Delta\cap\Delta_g$ will be a scheme $X$ proper over $\Spec O_{\breve F}$ with support in the special fiber, and the cohomology groups of any perfect complex of $\CO_X$-modules are finite length $O_{\breve F}$-modules, and there are only finitely many of them. For situations where we can make sure this happens, see Remark \ref{g rs => Int(g) proper} below.

\begin{remark}\label{equivint}
We note that $\Int(g)$ only depends on the 
double coset of $g$ modulo $H_1(F_0)$ (with respect to the diagonal embedding $H_1(F_0) \subset G_{W_1}(F_0)$). 
Indeed, the equivariance property \eqref{equivdelta} implies that 
\begin{equation*}
   h \Delta = \Delta
	\quad\text{for all}\quad
	 h\in H_1(F_0). 
\end{equation*} 
Hence
\[
   \Delta_{gh_2} = g h_2 \Delta = g \Delta = \Delta_g,
\]
and 
\[
   \langle\Delta, \Delta_{h_1 g h_2}\rangle
	   = \langle\Delta, \Delta_{h_1 g}\rangle
		= \langle\Delta, h_1\Delta_g\rangle
		= \langle h_1^{-1}\Delta, \Delta_g\rangle
		= \langle\Delta, \Delta_g\rangle.
\]
\end{remark}

\subsection{Inhomogeneous setting}

For the inhomogeneous case, we recycle notation by introducing, for $g\in U_1(F_0)$, the closed formal subscheme of $\CN_{n-1}\times_{\Spf O_{\breve F}}\CN_n$,
\begin{equation}\label{Deltaginhomog}
\Delta_g := (1\times g) \Delta,
\end{equation}
and setting 
\begin{equation}\label{Int(g) inhomog}
   \Int(g) := \langle \Delta, \Delta_g\rangle = \chi(\CO_\Delta\otimes^\BL\CO_{\Delta_g}) .
\end{equation}
The same remarks as in the homogeneous case above apply to this definition. 

\begin{remark}\label{indepoforb}
It follows from Remark~\ref{equivint} that in the inhomogeneous setting, $\Int(g)$ only depends on the orbit of $g$ under the conjugation action of $H_1(F_0)$. 
\end{remark}

\begin{remark}\label{g rs => Int(g) proper}
Identify $U_1(F_0)$ with $U(\BX_n)$, and $G_{W_1}(F_0)$ with $U(\BX_{n-1}) \times U(\BX_n)$, according to \eqref{group identifications}.  We claim that for any $g \in U_{1,\rs}(F_0)$ in the inhomogeneous case, or any $g \in G_{W_1, \rs}(F_0)$ in the homogeneous case, the quantity $\Int(g)$ is finite.  More precisely, we claim that in either case, the intersection of $\Delta$ and $\Delta_g$ inside $\CN_{n-1}\times_{O_{\breve F}} \CN_n$  is a \emph{scheme} over $\Spf O_{\breve F}$ (i.e., any ideal of definition of this formal scheme is nilpotent) which is proper over $\Spec O_{\breve F}$.
 
Indeed, for simplicity, let us consider the inhomogeneous case. Then the proof of \cite[Lem.~2.8]{Z12} goes through, although we are lacking at the moment a suitable reference for the global facts used in loc.~cit. The argument is based on the relation with KR divisors \cite{KR-U1}. Namely, let $\CZ(u)$ be the special cycle defined by the special vector $u=u_1$. Then there is the identification $\CZ(u)=\delta_\CN(\CN_{n-1})$. Via the second projection there is an identification $\Delta\cap\Delta_g\subset  \CN_n^g$.  Therefore we obtain an inclusion
\[
 \Delta\cap\Delta_g \subset \CZ(u)\cap\CZ(gu)\cap\dotsb\cap\CZ(g^{n-1}u) .
\] 
But since $g$ is regular semi-simple, the vectors $u, gu, \dotsc, g^{n-1}u$ 
in $\BV_n$ 
are linearly independent, cf.~\S\ref{ss:LAchar}. In other words, the fundamental matrix of the special divisors $\CZ(u),\dotsc,\CZ(g^{n-1}u)$ is non-singular. Now we approximate the vectors $u, gu,\dotsc, g^{n-1}u$ by ``global vectors'' $v_1, \dotsc, v_n$, and we imitate in the present ramified case the global argument of \cite[Lem.~2.8]{Z12}, which shows that there is a chain of inclusions of schemes
\[
   \CZ(u)\cap\CZ(xu)\cap\dotsb\cap\CZ(x^{n-1}u)\subset \CZ(v_1)\cap\CZ(v_2)\cap\dotsb\cap\CZ(v_{n})\subset {\rm Sh}^{\rm ss} .
\]
Here $ {\rm Sh}^{\rm ss}$ denotes the supersingular locus of the integral model of the global Shimura variety, and we are implicitly using nonarchimedean uniformization to make these identifications. (In the unramified situation the $\CZ(v_i)$ are relative divisors on ${\rm Sh}$, whereas in our ramified situation we can only assert that the underlying point set has codimension one.) 

On the other hand, in the case $n=3$, we will show directly in Theorem \ref{main prop} below that $\Delta\cap\Delta_g$ is an artinian scheme whenever $g \in U_{1,\rs}(F_0)$.  The proof of this does not use a global argument.
\end{remark}

\subsection{Lie algebra setting}

It would be interesting to transpose the above to the Lie algebra case. To this end,  first note that  the closed formal subscheme \eqref{Deltaginhomog} can be identified with the closed sublocus of points $((Y, \iota_Y, \lambda_Y, \rho_Y), (X, \iota_X, \lambda_X, \rho_X)) \in \CN_{n-1}\times_{\Spf O_{\breve F}}\CN_n$ where
\begin{equation}\label{Deltag}
   \text{\emph{the quasi-endomorphism $g\colon \BX_n\rightarrow \BX_n$ lifts to a homomorphism $Y\times\ov \CE\rightarrow X$.}}
\end{equation}
Here we recall from \eqref{odd framing object} that $\BX_n = \BX_{n-1} \times \ov\BE$ by definition, and by ``lifts'' we mean with respect to the framings $\rho_Y$ of $Y$, resp.~$\rho_{\ov\CE}$ of $\ov\CE$, resp.~$\rho_X$ of $X$. 

Note that condition \eqref{Deltag} still makes sense when $g$ is replaced by any quasi-endomorphism $\BX_n$. Hence we may replace the formal subspace $\Delta_g = (1\times g) \cdot \Delta_\CN(\CN_{n-1})$ in the inhomogeneous version above by the analogous subspace $\Delta_x$, for any $x$ in the Lie algebra $\fku_1$ of $U_1$. Unfortunately,  we do not know how to give a reasonable definition of  $\Int(x) = \langle \Delta, \Delta_x\rangle$ in general, cf.~our remarks in the Introduction. 
Here by ``reasonable'' we mean that at least, as in Remark \ref{equivint}, $\Int (x)$ only depends on the conjugation orbit of $x$ under $H_1(F_0)$, and that if $\Delta\cap\Delta_x$ is an artinian \emph{scheme}, then $\Int(x)$ coincides with the length of $\Delta\cap\Delta_x$. One problem, pointed out to us by A.~Mihatsch, is that it may happen that the formal dimension of $\Delta_x$ is smaller than $n$. 

\begin{remark}\label{intersection conjugation invar}
For any quasi-endomorphism $x$ of $\BX_n$ and any $h_1,h_2 \in H_1(F_0)$, it is elementary to verify that
\[
   (h_2 \times h_1)\Delta_x = \Delta_{h_1 x h_2^{-1}}.
\]
Taking $h_1 = h_2 =: h$, and taking $H_1(F_0)$ to act diagonally on $\CN_{n-1} \times_{\Spf O_{\breve F}} \CN_n$ as before, we conclude that the formal subspaces $\Delta \cap \Delta_x$ and $\Delta \cap \Delta_{hxh^{-1}} = \Delta \cap h\Delta_x = h(\Delta \cap \Delta_x)$ are isomorphic.
\end{remark}

\begin{remark}\label{Intfiniteforliealg}
Similar to Remark \ref{g rs => Int(g) proper}, 
the intersection $\Delta\cap\Delta_x$ is a scheme proper over $\Spec O_{\breve F}$, provided that $x\in \fku_{1, \rs}(F_0)$. Indeed, the same proof works, once the following two observations are  taken into account. First,
\begin{equation*}
\Delta\cap \Delta_x\subset \Delta\cap \Delta_{x^i},\quad i\geq 1. 
\end{equation*}
Indeed, abusing notation in the obvious way, the left-hand side is the locus of points $(Y, X)$ in $\CN_{n-1}\times_{\Spf O_{\breve F}}\CN_n$ where $Y\times\ov\CE\simeq X$ and where $x$ lifts to an endomorphism of $X$. It is clear that then also $x^i$ lifts as an endomorphism of $X$, for any $i\geq 1$. Second,  for any $y\in \End_{O_F}(\BX_n)$, 
\begin{equation*}
\Delta_y\subset \CZ(yu) .
\end{equation*}
Indeed, the restriction of $y\colon Y\times\ov\CE\to X$ to the second factor coincides with the homomorphism $yu\colon \ov\CE\to X$, and hence the locus where $y$ extends to a homomorphism is contained in the locus $\CZ(yu)$ where the homomorphism $yu$ extends. We conclude that
\[
   \Delta\cap\Delta_x \subset \CZ(u)\cap\CZ(xu)\cap\dotsb\cap\CZ(x^{n-1}u) ,
\] 
and the proof proceeds from here as in the group case.
\end{remark}

We also conjecture the following partial converse to Remarks \ref{g rs => Int(g) proper} and \ref{Intfiniteforliealg}.  

\begin{conjecture}\label{conj Int(g) proper => g rs}
Let $g\in G_{W_1}(F_0)$ be semi-simple (i.e., its orbit under $H_1 \times H_1$ is Zariski-closed). If the intersection of $\Delta$ and $\Delta_g$ inside $\CN_{n-1}\times_{O_{\breve E}} \CN_n$  is a nonempty scheme proper over $\Spec O_{\breve F}$, then  $g$ is regular semi-simple.  The inhomogeneous version for $g \in U_1(F_0)$ and the Lie algebra version for $x\in \fku_1(F_0)$ instead of $g$ also hold true. 
\end{conjecture}

\section{Conjectures and main results}\label{conj and results}

In this section, we formulate the general conjecture that is addressed in this paper.  Following the case distinctions we have made earlier, we will formulate three variants: a homogeneous version, an inhomogeneous version, and a Lie algebra version. We continue to denote by $F/F_0$ a ramified quadratic extension (and assume $p\neq 2$). Throughout this section, $n\geq 3$ is an \emph{odd} integer. 

\subsection{Homogeneous setting}  

For $\gamma\in G'(F_0)_\rs$, for a function $f'\in C_c^\infty(G')$, and for a complex parameter $s\in \BC$, we introduce the weighted orbital integral
\begin{equation}\label{inhwithpara}
   \Orb(\gamma,f', s) := \int_{H_{1, 2}'(F_0)}f'(h_1^{-1}\gamma h_2)\lvert \det h_1 \rvert^s\eta(h_2)\,dh_1\, dh_2 .
\end{equation}
Here $\eta=\eta_{F/F_0}$ is the quadratic character corresponding to $F/F_0$, and we are using a product Haar measure on $H_{1, 2}'(F_0) = H'_1(F_0)\times H'_2(F_0)$. Also, for
\[
   h_2=(h_2', h_2'')\in H_2'(F_0) = \GL_{n-1}(F_0)\times\GL_n(F_0),
\]
we write  $\eta(h_2)$ for $\eta(\det h_2')$. Note that $\det h_1$ is an element of $F$, and we are taking the normalized absolute value on $F$ in \eqref{inhwithpara}.  We also introduce
\begin{equation*}
   \Orb(\gamma,f') := \Orb(\gamma,f', 0) 
   \quad\text{and}\quad 
   \del(\gamma,f') := \frac d{ds} \Big|_{s=0} \Orb(\gamma, f',s) . 
\end{equation*}
The integral in  \eqref{inhwithpara} is absolutely convergent, and 
\begin{equation*}
\Orb(h_1^{-1}\gamma h_2,f')= \eta(h_2)\Orb(\gamma,f') \quad\text{for}\quad (h_1, h_2)\in H_{1, 2}'(F_0)=H_1'(F_0)\times H_2'(F_0) .
\end{equation*}

Now let $i \in \{0,1\}$, and set $W := W_i$ and $H := H_i$.  For $g\in G_W(F_0)_\rs$ and for a function $f \in C_c^\infty(G_W)$, we  introduce the orbital integral
\begin{equation*}
\Orb(g,f) := \int_{H(F_0)\times H(F_0)} f(h_1^{-1} g h_2)\, dh_1\, dh_2 .
\end{equation*}
Here on $H(F_0)\times H(F_0)$ we take a product  measure of identical Haar measures on  $H(F_0)$. 

Dual to the matching of regular semi-simple elements discussed in \S\ref{setup} is the transfer relation of functions on $G'(F_0)$, $G_{W_0}(F_0)$, and $G_{W_1}(F_0)$. To define this, recall  that a \emph{transfer factor} is a function $\Omega\colon G'(F_0)_\rs\to \BC^\times$ such that
\[
   \Omega(h_1^{-1}\gamma h_2)=\eta(h_2)\Omega(\gamma) \quad\text{for all}\quad (h_1, h_2)\in H_{1, 2}'(F_0)=H'_1(F_0)\times H_2'(F_0).
\] 
Transfer factors always exist, and we will specify our particular choice below.

\begin{definition}\label{homog transfer}
A function $f'\in C_c^\infty(G')$ and a pair of functions $(f_0,f_1) \in C_c^\infty(G_{W_0}) \times C_c^\infty(G_{W_1})$ are \emph{transfers} of each other, or are \emph{associated} (for the fixed choices of Haar measures and a fixed choice of transfer factor and a fixed choice of special vectors $u_i$ in $W_i$), if for each $i \in \{0,1\}$ and each $g\in G_{W_i}(F_0)_\rs$,
\[
 \Orb(g,f_i)=\Omega(\gamma) \Orb(\gamma,f')
\]
whenever $\gamma\in G'(F_0)_\rs $ matches $g$.
\end{definition}

In the specific case at hand, we will take the transfer factor $\Omega=\omega$ to be given by the following formula. Let $\wt \eta$ be an extension of $\eta$ from $F_0^\times$ to $F^\times$ (not necessarily of order $2$). Then we take\footnote{Note that in this paper our choice of various transfer factors is slightly different from that in \cite{Z14}.} 
\begin{equation}\label{sign}
   \omega (\gamma) := \wt\eta\bigl(  \det(\wt\gamma)^{-(n-1)/2}   \det(\wt{\gamma}^ie)_{i=0,\dotsc, n-1}\bigr) ,
\end{equation}
where for $\gamma=(\gamma_1, \gamma_2)\in G'(F_0)_\rs$ we set $\wt \gamma=s(\gamma)=(\gamma_1^{-1}\gamma_2)\bigl(\ov{\gamma_1^{-1}\gamma_2}\bigr)^{-1}\in S_n(F_0)$, and where we recall the column vector $e=(0,\dotsc,0,1)\in F^n$. We point out that 
\[
   \det(\wt\gamma)^{-(n-1)/2}\det(\wt\gamma^i e)_{i=0,1,\dotsc,n-1}
\]
is an eigenvector under the Galois involution, i.e., it is in either $F_0$ or $\pi F_0$.

Let us fix the rest of the choices that go into the statement of our AT conjecture. We normalize the Haar measure in \eqref{inhwithpara} by assigning the subgroup $H'_1(O_{F_0})=\GL_{n-1}(O_{F})$ measure $1$, and by taking the product Haar measure on $H_2'(F_0) = \GL_{n-1}(F_0) \times \GL_n(F_0)$ such that  $\GL_{n-1}(O_{F_0})$ and $\GL_n(O_{F_0})$ have measure $1$.  On the unitary side, recall from \eqref{norm 1 assumption} that we assume that the special vectors $u_0 \in W_0$ and $u_1 \in W_1$ have norm $1$.   Since $n$ is odd, it follows that the perp spaces $W^\flat_0$ and $W^\flat_1$ are the respective split and non-split non-degenerate hermitian spaces of dimension $n-1$.  Since $W_0^\flat$ is furthermore even dimensional, it contains a \emph{$\pi$-modular lattice} \cite[Prop.~8.1(b)]{J62}, that is, an $O_F$-lattice $\Lambda_0^\flat$ such that $(\Lambda_0^\flat)^\vee = \pi^{-1}\Lambda_0^\flat$, where $(\Lambda_0^\flat)^\vee \subset W_0^\flat$ denotes the set of elements that pair with $\Lambda_0^\flat$ to values in $O_F$ under the hermitian form.  We fix such a $\Lambda_0^\flat$ and denote by $K^\flat_0$ its stabilizer in $H_0(F_0)$. We normalize the Haar measure on $H_0(F_0)$ such that $K^\flat_0$ gets measure $1$. The normalization of the Haar measure on $H_1(F_0)$ will not be important for us. We set
\begin{equation}\label{Lambda_0}
   \Lambda_0 := \Lambda^\flat_0\oplus O_F u_0,
\end{equation}
which is a \emph{nearly $\pi$-modular lattice} in $W_0$, i.e.~$\Lambda_0 \subset \Lambda_0^\vee \subset \pi^{-1}\Lambda_0$ with $\pi^{-1}\Lambda_0 / \Lambda_0^\vee$ of length $1$.  We denote by $K_0$ the stabilizer of $\Lambda_0$ in $U_0(F_0)$. 

\begin{remark}\label{rem syp red}
The maximal compact open subgroup $K_0^\flat \subset H_0(F_0)$ is a special maximal parahoric subgroup in the sense of Bruhat--Tits theory.  The maximal compact open subgroup $K_0 \subset U_0(F_0)$ is the full stabilizer in $U_0(F_0)$ of a special vertex in the (extended) building, and contains the associated maximal parahoric subgroup with index $2$.  See \cite[\S4.a]{PR-TLG}.  These subgroups have symplectic reduction in the sense that the quotients $(\Lambda_0^\flat)^\vee/\Lambda_0^\flat$ and $\Lambda_0^\vee/\Lambda_0$ have dimension $n-1$ over the residue field $k$, and the hermitian form on the ambient space induces a non-degenerate alternating bilinear pairing on both.
\end{remark}

We now state the homogeneous version of the arithmetic transfer conjecture in the case at hand. 
For $g \in G_{W_1}(F_0)$, recall the intersection number $\Int(g)$ from \eqref{Int(g) homog}.

\begin{conjecture}[Homogeneous AT conjecture]\label{conj homog}\hfill
\begin{altenumerate}
\renewcommand{\theenumi}{\alph{enumi}}
\item\label{conj homog a} There exists a function $f'\in  C^\infty_c(G')$ which transfers to $(\mathbf{1}_{K^\flat_0\times K_0}, 0) \in C_c^\infty(G_{W_0}) \times C_c^\infty (G_{W_1})$, and which satisfies the following identity for any $\gamma\in G'(F_0)_\rs$ matched with an element $g\in G_{W_1}(F_0)_\rs$:
$$
\omega(\gamma)\del(\gamma,f')=-\Int(g)\cdot \log q .
$$ 

\item\label{conj homog b} For any $f'\in C^\infty_c(G')$ which transfers to $(\mathbf{1}_{K^\flat_0\times K_0}, 0) \in C_c^\infty(G_{W_0}) \times C_c^\infty (G_{W_1})$, there exists a function $f'_\corr\in C^\infty_c(G')$ such that for any $\gamma\in G'(F_0)_\rs$ matched with an element $g\in  G_{W_1}(F_0)_\rs$,
$$
\omega(\gamma)\del(\gamma,f')=-\Int(g) \cdot \log q+ \omega(\gamma)\Orb(\gamma,f'_\corr) .
$$
\end{altenumerate}
\end{conjecture}

\begin{remarks}
\begin{altenumerate}
\renewcommand{\theenumi}{\roman{enumi}}
\item\label{DOrb invar} 
Both sides of these identities depend only on the respective orbits of $\gamma$ and $g$: for the right-hand side this follows from   Remark~\ref{equivint}; and for the left-hand side, this follows from $\Orb(\gamma, f')=0$ for any $\gamma\in G'(F_0)_\rs$ matched with an element $g\in  G_{W_1}(F_0)_\rs$, which holds because $f'$ transfers to $0$ on $G_{W_1}(F_0)$. 
\item 
Part \eqref{conj homog a} of Conjecture \ref{conj homog} follows from part \eqref{conj homog b}; see Proposition \ref{cor b to a} below. 
The converse \eqref{conj homog a}$\implies$\eqref{conj homog b} would follow from a conjectural density principle (Conjecture \ref{conj density}); see Lemma \ref{lem a to b} below.
\item\label{it pro-u}
The function $f'$ in Conjecture \ref{conj homog}\eqref{conj homog a} cannot be taken to lie in the Iwahori Hecke algebra of $G'(F_0)=\GL_{n-1}(F)\times\GL_n(F)$ due to the presence of the ramified quadratic character $\eta$. It is tempting to guess that such an $f'$ could be taken in the pro-unipotent Hecke algebra of $G'(F_0)$, i.e., to be bi-invariant under the pro-unipotent radical of an Iwahori subgroup.  But even in the case $n=3$ we do not know how to prove this. 
\end{altenumerate}
\label{rem pro-u}
\end{remarks}

\subsection{Inhomogeneous setting}\label{inhom setting section}
Now we give the inhomogeneous version of the conjecture.  Recall the scheme $S = S_n = \{ g \in \Res_{F/F_0}\GL_n \mid g \ov g = 1_n\}$.  For  $\gamma\in S(F_0)_\rs$,  a function $f'\in C_c^\infty(S)$, and a complex parameter $s\in \BC$, we introduce the weighted orbital integral
\begin{align}\label{eqn def inhom}
   \Orb(\gamma,f', s) := \int_{H'(F_0)}f'(h^{-1}\gamma h)\lvert \det h \rvert^s \eta(h) \, dh ,
\end{align}
as well as the special values
\begin{equation*}
   \Orb(\gamma,f') := \Orb(\gamma,f', 0)
   \quad\text{and}\quad
   \del(\gamma,f') : = \frac d{ds} \Big|_{s=0} \Orb(\gamma, f',s) . 
\end{equation*}
Here we write $\eta(h)$ for $\eta(\det h)$, as in the Introduction.  Note that $\det h\in F_0$, and in \eqref{eqn def inhom} we are taking its absolute value for the normalized absolute value on $F_0$.  As in the homogeneous setting, the integral defining $\Orb(\gamma,f', s)$ is absolutely convergent, and
\[
   \Orb(h^{-1}\gamma h,f') = \eta(h) \Orb(\gamma,f')
	\quad\text{for all}\quad
	h \in H'(F_0) = \GL_{n-1}(F_0).
\]
For $i \in \{0,1\}$, set $U := U_i$ and $H := H_i$.  For $g\in U(F_0)_\rs$ and a function $f \in C_c^\infty(U)$, we similarly introduce the orbital integral
\begin{equation*}
   \Orb(g,f) := \int_{H(F_0)}f(h^{-1} g h)\, dh .
\end{equation*}

A \emph{transfer factor} on $S(F_0)_\rs$ is a function $\Omega\colon S(F_0)_\rs\to \BC^\times$ such that $\Omega(h^{-1}\gamma h) = \eta(h)\Omega(\gamma)$ for all $\gamma \in S(F_0)_\rs$ and $h \in H'(F_0)$.  Quite analogously to the homogeneous setting, in the case at hand we take the transfer factor $\Omega = \omega$ given by
\begin{equation}\label{sign Sn}
   \omega(\gamma) := \wt\eta\bigl( \det(\gamma)^{-(n-1)/2} \det({\gamma}^i e)_{i=0,\dotsc, n-1}\bigr).
\end{equation}
By definition, this is compatible with the transfer factor \eqref{sign} on $G'(F_0)$ in the sense that, for $\gamma\in G'(F_0)$,
\begin{equation}\label{sign Sn=}
   \omega(\gamma) = \omega\bigl(s(\gamma)\bigr).
\end{equation}
We take the other normalizations (the length of the special vectors, the Haar measures on $H'(F_0)$ and $H(F_0)$) and the notation $K_0$ as in the homogeneous setting.  The transfer relation for functions is the following.

\begin{definition}
A function $f' \in C_c^\infty(S)$ and a pair of functions $(f_0,f_1) \in C_c^\infty(U_0) \times C_c^\infty(U_1)$ are \emph{transfers} of each other, or are \emph{associated}, if for each $i \in \{0,1\}$ and each $g\in U_i(F_0)_\rs$,
\[
   \Orb(g,f_i)=\omega(\gamma) \Orb(\gamma,f')
\]
whenever $\gamma\in S(F_0)_\rs $ matches $g$.
\end{definition}

The inhomogeneous version  of the conjecture takes the following form. 

\begin{conjecture}[Inhomogeneous AT conjecture]\label{inhomconj}\hfill
\begin{altenumerate}
\renewcommand{\theenumi}{\alph{enumi}}
\item\label{inhomconj a}
There exists a function $f'\in C^\infty_c(S)$ which transfers to $(\mathbf{1}_{K_0}, 0) \in C_c^\infty(U_0) \times C_c^\infty(U_1)$, and which satisfies the following identity for any $\gamma\in S(F_0)_\rs$ matched with an element $g\in U_1(F_0)_\rs$:
$$
2\omega(\gamma)\del(\gamma,f')=-\Int(g)\cdot\log q .
$$ 
\item\label{inhomconj b}
For any $f'\in C^\infty_c(S)$ which transfers to $(\mathbf{1}_{K_0}, 0) \in C_c^\infty(U_0) \times C_c^\infty(U_1)$, there exists a function  $f'_\corr\in C^\infty_c(S)$ such that for any $\gamma\in S(F_0)_\rs$ matched with an element $g\in U_1(F_0)_\rs$,
$$
2\omega(\gamma)\del(\gamma,f') = -\Int(g)\cdot\log q + \omega(\gamma)\Orb(\gamma,f'_\corr).
$$ 
\end{altenumerate}
\end{conjecture}

Note that there is a factor of $2$ in Conjecture \ref{inhomconj} which is not present in Conjecture \ref{conj homog}; this is due to the fact that the restriction to $F_0$ of the normalized absolute value of $F$  is the square of the normalized absolute value of $F_0$, cf.~Lemma \ref{lem G' to S} and its proof below.

In the rest of this subsection, we show that Conjectures \ref{conj homog} and \ref{inhomconj} are equivalent.

\begin{lemma}\label{lem G' to S}
Let $f'\in  C^\infty_c(G')$, and define the function $\wt f'$ on $S(F_0)$ by, for $g\in \GL_n(F)$,
\begin{equation}\label{eqn def wt f}
   \wt f'(g\ov g^{-1}) := \int_{\GL_{n-1}(F)\times \GL_n(F_0)}f'(h_1,h_1 g h_2)\, dh_1\,dh_2.
\end{equation}

\begin{altenumerate}
\item\label{conj equiv lem i}
$\wt f'\in C^\infty_c(S)$, and every element in $C^\infty_c(S)$ arises in this way.  
\item \label{conj equiv lem ii}
For all $\gamma\in G'(F_0)_\rs$,
\begin{equation}\label{eqn G'toS 0}
   \Orb(\gamma,f') = \Orb\bigl(s(\gamma),\wt f'\bigr).
\end{equation}
\item \label{conj equiv lem iii} If $f'$ transfers to $(f_0,0)$ for some $f_0\in  C^\infty_c(G_{W_0})$, then there exists a function $\phi'_\corr\in C^\infty_c(G')$ such that for any $\gamma$ matching an element in $G_{W_1}(F_0)_{\rs}$,
\begin{equation}\label{eqn G'toS 1}
   \del(\gamma,f') = 2 \del\bigl(s(\gamma),\wt f'\bigr)+   \Orb(\gamma,\phi'_{\corr}).
\end{equation}
\end{altenumerate}
\end{lemma}

\begin{proof}
Part \eqref{conj equiv lem i} is clear; we show parts \eqref{conj equiv lem ii} and \eqref{conj equiv lem iii}.
Note that for any $(\gamma_1,\gamma_2) \in G'(F_0) = \GL_{n-1}(F) \times \GL_n(F)$,
\[
   \Orb\bigl((\gamma_1,\gamma_2),f', s\bigr)= |\gamma_1|^{-s}\Orb\bigl((1,\gamma_1^{-1}\gamma_2),f', s\bigr).
\] 
Therefore the left-hand side of \eqref{eqn G'toS 0} is invariant if we replace $(\gamma_1,\gamma_2)$ by $(1,\gamma_1^{-1}\gamma_2)$, and the same is true of the left-hand side of \eqref{eqn G'toS 1} when $f'$ and $(\gamma_1,\gamma_2)$  satisfy the assumptions of \eqref{conj equiv lem iii}, since then $\Orb((\gamma_1,\gamma_2),f') = 0$.
Hence it suffices to consider elements of the form $(1,\gamma)\in G'(F_0)_\rs$.
By definition,
\[
   \Orb\bigl((1,\gamma),f', s\bigr) 
	   = \int_{H_{1, 2}'(F_0)} f'(h_1^{-1}h_2',h_1^{-1}\gamma h_2'') \lvert\det h_1\rvert^s \eta(h_2') \,dh_1\,dh_2'\,dh_2'',
\]
where $h_1\in H_1'(F_0) = \GL_{n-1}(F)$ and $(h_2',h_2'')\in H_2'(F_0)=\GL_{n-1}(F_0)\times\GL_n(F_0)$. Replacing $h_1$ by $h_2'h_1$, we have 
\[
   \Orb\bigl((1,\gamma),f', s\bigr)
	   = \int_{H_{1, 2}'(F_0)} f'\bigl(h_1^{-1},h_1^{-1}(h_2')^{-1}\gamma h_2''\bigr) \lvert\det(h_2'h_1)\rvert^s \eta(h_2')\,dh_1 \,dh_2'\, dh_2'' .
\]
This is equal to the sum of 
\begin{equation}\label{eqn term 1}
  \int_{H_{1, 2}'(F_0)} f'\bigl(h_1^{-1},h_1^{-1}(h_2')^{-1}\gamma h_2''\bigr) \lvert \det h_2' \rvert^s\eta(h_2')\,dh_1\,dh_2'\, dh_2''
\end{equation}
and
\begin{equation}\label{eqn term 2}
  \int_{H_{1, 2}'(F_0)}f'\bigl(h_1^{-1},h_1^{-1}(h_2')^{-1}\gamma h_2''\bigr) (\lvert\det h_1\rvert^s-1)\lvert \det h_2' \rvert^s\eta(h_2')\,dh_1\,dh_2'\, dh_2'' .
\end{equation}

Comparing  with the definition \eqref{eqn def inhom}, we see that  the term \eqref{eqn term 1} is equal to 
\[
   \int_{\GL_{n-1}(F_0)}   \wt f'\bigl((h_2')^{-1}\gamma\ov\gamma^{-1} h_2'\bigr)\lvert\det h_2'\rvert_{F_0}^{2s} \eta(h_2') \,dh_2'
	   = \Orb\bigl(\gamma\ov\gamma^{-1},\wt f', 2s\bigr).
\]
The term \eqref{eqn term 2} has an obvious zero at $s=0$ for all regular semi-simple elements $(1,\gamma)$. This proves part (\ref{conj equiv lem ii}). 
To prove part (\ref{conj equiv lem iii}), we note that, since $\lvert\det h_1\rvert^s-1$ has a zero at $s=0$, the first derivative of  the term  \eqref{eqn term 2} evaluated at $s=0$ is equal to 
\begin{equation}\label{first deriv}
  \int_{H_{1, 2}'(F_0)}f'\bigl(h_1^{-1},h_1^{-1}(h_2')^{-1}\gamma h_2''\bigr)\eta(h_2')  \log \lvert\det h_1\rvert\,dh_1\,dh_2'\, dh_2''.
\end{equation}
Now we define a function on $G'(F_0)$,
\begin{equation}
   \phi'_\corr(\gamma_1,\gamma_2) :=
	   -f' (\gamma_1,\gamma_2)\log  \lvert\det \gamma_1^{-1}\rvert, \quad (\gamma_1,\gamma_2)\in G'(F_0).
\end{equation}
This clearly belongs to $C^\infty_c(G')$, and the term  \eqref{first deriv} is equal to 
$-\Orb((1,\gamma),\phi'_\corr )$. This completes the proof.
\end{proof}

\begin{lemma}\label{hom=inhom}
Conjectures \ref{conj homog}(\ref{conj homog a}) and \ref{inhomconj}(\ref{inhomconj a}) are equivalent. Similarly for Conjectures \ref{conj homog}(\ref{conj homog b}) and \ref{inhomconj}(\ref{inhomconj b}).
\end{lemma}
\begin{proof}
We have
\begin{equation*}
   \Int(h_1gh_2) = \Int(g)
	\quad\text{for all}\quad
	g\in G_{W_1}(F_0), \quad h_1,h_2\in H_1(F_0).
\end{equation*}
Therefore it suffices to consider elements of the form $(1,g)$ for $g\in U_1(F_0)$. 
By \eqref{sign Sn=}, we have $\omega(\gamma)=\omega(s(\gamma))$ for $\gamma\in G'(F_0)$. 
The equivalence between Conjectures \ref{conj homog}\eqref{conj homog a} and \ref{inhomconj}\eqref{inhomconj a} now follows from Lemma \ref{lem G' to S}. Indeed, if $f'\in C^\infty_c(G')$ satisfies the conclusion of Conjecture \ref{conj homog}\eqref{conj homog a}, then parts \eqref{conj equiv lem ii} and \eqref{conj equiv lem iii} of Lemma \ref{lem G' to S} imply that the function $\wt f'\in C^\infty_c(S)$ satisfies the conclusion of Conjecture \ref{inhomconj}\eqref{inhomconj b}. This implies Conjecture \ref{inhomconj}\eqref{inhomconj a} by Remark \ref{remark b to a} below. Conversely, if $f''\in C^\infty_c(S)$ satisfies the conclusion of Conjecture \ref{inhomconj}\eqref{inhomconj a}, then we may choose $f'\in C^\infty_c(G')$ such that $\wt f'=f''$, and $f'$ will then satisfy the conclusion of Conjecture \ref{conj homog}\eqref{conj homog b}. This implies Conjecture \ref{conj homog}\eqref{conj homog a} again by Remark \ref{remark b to a}. 

One may similarly show the equivalence between Conjectures \ref{conj homog}\eqref{conj homog b} and \ref{inhomconj}\eqref{inhomconj b}, by taking the $f'_{\corr}$ in   Conjecture \ref{inhomconj}\eqref{inhomconj b} to be $\wt{f'_{\corr}}- \wt{\phi'_{\corr}}$ for the $f'_{\corr}$ in Conjecture \ref{conj homog}\eqref{conj homog b} and $\phi'_{\corr}$ in Lemma \ref{lem G' to S}\eqref{conj equiv lem iii}.
\end{proof}

\subsection{Lie algebra setting}\label{lie conj sec}

We would like to formulate a Lie algebra version of Conjecture \ref{inhomconj}. At least the analytic side of the conjecture makes sense. 
Recall the Lie algebra version of the symmetric space $S=S_n$,
\[
   \fks = \bigl\{\,y\in \Res_{F/F_0}\M_n \bigm| y+\ov y=0 \,\bigr\}.
\]
For a regular semi-simple element $y\in \fks(F_0)_\rs$, a function $\phi'\in C^\infty_c(\frak s)$, and a complex parameter $s$, we define the weighted orbital integral
\[
  \Orb(y,\phi',s) := \int_{H'(F_0)}\phi'(h^{-1} yh)\lvert \det h \rvert^s\eta(h) \,dh ,
\]
as well as the special values
\[
   \Orb(y,\phi') := \Orb(y,\phi', 0)
   \quad\text{and}\quad 
   \del(y,\phi') : = \frac d{ds} \Big|_{s=0} \Orb(y, \phi',s) . 
\]
As in \eqref{eqn def inhom}, here we are using the normalized absolute value on $F_0$.  As before, the integral defining $\Orb(y,\phi',s)$ is absolutely convergent, and
\[
   \Orb(h^{-1}yh,\phi') = \eta(h)\Orb(y,\phi')
	\quad\text{for all}\quad
	h \in H'(F_0) = \GL_{n-1}(F_0).
\]
For $i \in \{0,1\}$, set $\fku := \fku_i$ and $H := H_i$.  For $x\in \fku(F_0)_\rs$ and a function $\phi\in C_c^\infty(\fku)$, we similarly introduce the orbital integral
\[
   \Orb(x,\phi) := \int_{H(F_0)}\phi(h^{-1} x h)\, dh .
\]

A \emph{transfer factor} on $\fks(F_0)_\rs$ is a function $\Omega\colon \fks(F_0)_\rs \to \BC^\times$ such that $\Omega(h^{-1} y h) = \eta(h)\Omega(y)$ for all $y \in \fks(F_0)_\rs$ and $h \in H'(F_0)$.  We take the transfer factor $\Omega = \omega$ on $\fks(F_0)_\rs$ given by
\begin{equation}\label{omega lie}
   \omega(y) := \wt\eta\bigl(\det(y^ie)_{i=0,1,\dotsc,n-1}\bigr);
\end{equation}
cf.~\cite[(3.7)]{Z14}. Again we normalize the Haar measure on the symmetric space side such that $H'(O_{F_0})=\GL_{n-1}(O_{F_0})$ has measure $1$.  On the unitary side, we denote by $\fkk_0$ the stabilizer in $\fku_0(F_0)$ of the nearly $\pi$-modular lattice $\Lambda_0 \subset W_0$ defined in \eqref{Lambda_0}. Again we normalize the Haar measure on $H_0(F_0)$ such that the stabilizer $K^\flat_0$ of $\Lambda^\flat_0$ has measure $1$.  The transfer relation for functions extends readily to the Lie algebra setting as follows.

\begin{definition}
A function $\phi' \in C_c^\infty(\fks)$ and a pair of functions $(\phi_0,\phi_1) \in C_c^\infty(\fku_0) \times C_c^\infty(\fku_1)$ are \emph{transfers} of each other, or are \emph{associated}, if for each $i \in \{0,1\}$ and each $x\in \fku_i(F_0)_\rs$,
\[
   \Orb(x,\phi_i)=\omega(y) \Orb(y,\phi')
\]
whenever $y\in \fks(F_0)_\rs $ matches $x$.
\end{definition}

Now we state the Lie algebra version of the conjecture.

\begin{conjecture}[Lie algebra AT conjecture]\label{lieconj}\hfill
\begin{altenumerate}
\renewcommand{\theenumi}{\alph{enumi}}
\item\label{lieconj a} 
There exists a function $\phi'\in C^\infty_c(\frak s)$ which transfers to $(\mathbf{1}_{\fkk_0}, 0) \in C_c^\infty(\fku_0) \times C_c^\infty(\fku_1)$, and which satisfies the following identity for any $y\in \fks(F_0)_\rs$ matched with an element $x\in \fku_1(F_0)_\rs$ for which the  intersection $\Delta\cap\Delta_x$ is an artinian scheme:
\[
 2  \omega(y)\del(y,\phi') = -\lInt (x)\cdot\log q .
\]
Here we set
\[
   \lInt(x) := \length(\Delta\cap\Delta_x) .
\]
\item\label{lieconj b}
For any $\phi'\in C^\infty_c(\frak s)$  which transfers to $(\mathbf{1}_{\fkk_0}, 0) \in C_c^\infty(\fku_0) \times C_c^\infty(\fku_1)$, there exists a function $\phi'_\corr\in C^\infty_c(\fks)$ such that for any 
$y\in \fks(F_0)_\rs$ matched with an element  $x\in \fku_1(F_0)_\rs$ for which the  intersection $\Delta\cap\Delta_x$ is an artinian scheme,
\[
   2 \omega(y)\del(y,\phi') = -\lInt(x) \cdot\log q + \omega(y)\Orb(y,\phi'_\corr) .
\]
\end{altenumerate}
\end{conjecture}

\begin{remark}\label{lie conj artinian remark}
It is not currently clear to us how to formulate a conjecture without the hypothesis that $\Delta\cap \Delta_x$ is  artinian; see our comments right before Remark \ref{intersection conjugation invar}. 
\end{remark}

\subsection{Relation between parts (a) and (b) of the conjectures}\label{sec relate a b}
In this subsection we explain how parts (a) and (b) in each of Conjectures \ref{conj homog},  \ref{inhomconj}, and \ref{lieconj} are related.

The group $H'(F_0) = \GL_{n-1}(F_0)$ acts on $S(F_0)$ and hence on the space $C^\infty_c(S)$.  For a function $f'\in C^\infty_c(S)$ and an element $h\in H'(F_0)$, we denote by $\tensor[^h]{f}{^\prime} \in C_c^\infty(S)$ the function defined by 
\[
   \tensor*[^h]{f}{^\prime}\colon \gamma \mapsto f'(h^{-1} \gamma h).
\]
We also denote by $\tensor*[^{\eta(h)h-1}]{f}{^\prime} \in C_c^\infty(S)$ the function
\[
   \tensor*[^{\eta(h)h-1}]{f}{^\prime}\colon \gamma\mapsto \eta(h)f'(h^{-1}\gamma h)-f'(\gamma).
\]
We use analogous notation in the Lie algebra setting, with $\fks$ in place of $S$.

\begin{lemma}
\begin{altenumerate}
\item\label{twisted del S}
For any $f' \in C_c^\infty(S)$, $\gamma \in S(F_0)_\rs$, and $h \in H'(F_0)$,
\[
   \Orb\bigl(\gamma,\tensor*[^{\eta(h)h-1}]{f}{^\prime}\bigr) = 0
	\quad\text{and}\quad
   \del\bigl(\gamma,\tensor*[^{\eta(h)h-1}]{f}{^\prime}\bigr) = \log\lvert \det h \rvert \Orb(\gamma, f').
\]
\item\label{twisted del s}
For any $\phi'\in C^\infty_c(\fks)$, $\gamma \in S(F_0)_\rs$, and $h \in H'(F_0)$,
\[
   \Orb\bigl(\gamma,\tensor[^{\eta(h)h-1}]{\phi}{^\prime}\bigr) = 0
	\quad\text{and}\quad
   \del\bigl(\gamma,\tensor[^{\eta(h)h-1}]{\phi}{^\prime}\bigr) = \log\lvert \det h \rvert \Orb(\gamma, \phi').
\]
\end{altenumerate}
\label{twisted del}
\end{lemma}

\begin{proof}
We prove \eqref{twisted del S}; the proof of \eqref{twisted del s} is analogous. The integral $\Orb(\gamma, f',s)$ transforms under the action of $H'(F_0)$ on $f'$ by the character $\eta_s\circ\det$, where
\[
   \eta_s(a) := \eta(a)|a|^s,
	\quad
	a \in F_0^\times,
\] 
i.e.
\begin{equation*}
   \Orb\bigl(\gamma, \tensor*[^h]{f}{^\prime}, s\bigr) = \eta_s(\det h) \Orb(\gamma, f', s).
\end{equation*}
It follows that
\begin{equation*}
   \Orb\bigl(\gamma, \tensor*[^{\eta(h)h-1}]{f}{^\prime}, s\bigr) = (\lvert \det h \rvert^s-1)\Orb(\gamma, f', s).
\end{equation*} 
The proves the first desired identity, and the second follows by differentiating.
\end{proof}

\begin{lemma}
\begin{altenumerate}
\item\label{exists wt f'} 
For any $f'\in C^\infty_c(S)$, there exists $ f'^\sharp\in C^\infty_c(S)$ such that 
\[
   \Orb(\gamma,f') = \del\bigl(\gamma,f'^\sharp\bigr) \quad  \text{for all}\quad \gamma\in S(F_0)_\rs .
\]
Furthermore, we may choose $f'^\sharp$ such that it transfers to $(0, 0) \in C_c^\infty(U_0) \times C_c^\infty(U_1) $. 
\item\label{exists wt phi'}
For any $\phi'\in C^\infty_c(\fks)$, there exists $ \phi'^\sharp\in C^\infty_c(\fks)$ such that 
\[
   \Orb(y,\phi') = \del\bigl(y, \phi'^\sharp\bigr) \quad\text{for all}\quad y\in \fks(F_0)_\rs .
\]
Furthermore, we may choose $\phi'^\sharp$ such that it transfers to $(0, 0) \in C_c^\infty(\fku_0) \times C_c^\infty(\fku_1)$.
\end{altenumerate}
\label{lem btoa}
\end{lemma} 

\begin{proof} 
Again we just prove \eqref{exists wt f'}.  Choose any $h \in H'(F_0)$ with $\det h$ a non-unit, and set
\[
  f'^\sharp := \frac{\tensor*[^{\eta(h)h-1}]{f}{^\prime}}{\log\lvert \det h \rvert}.
\]
Then $f'^\sharp$ has all desired properties by Lemma \ref{twisted del}.
\end{proof}

\begin{proposition}\label{cor b to a}
Part (b) implies part (a) in each of Conjectures \ref{conj homog}, \ref{inhomconj}, and \ref{lieconj}.
\end{proposition} 

\begin{proof}
By the proof of the ST conjecture in \cite[Th.~2.6]{Z14} (resp.~its Lie algebra analog in \S4.5 of loc.~cit.), there exists a function in $C_c^\infty(S)$ transferring to $(\mathbf{1}_{K_0},0)$ (resp.\ in $C_c^\infty(\fks)$ transferring to $(\mathbf{1}_{\fkk_0},0)$).
Of course these functions needn't satisfy the conclusion of part (a) in Conjectures \ref{inhomconj} and \ref{lieconj}, respectively, but assuming part (b) in these conjectures, we may use Lemma \ref{lem btoa} to modify them into functions that do.  The case of  Conjecture \ref{conj homog} is handled in a similar way by using Lemma \ref{lem G' to S} to relate functions on $G'(F_0)$ to functions on $S(F_0)$, again using Lemma \ref{lem btoa}\eqref{exists wt f'} to control the needed modifications.
\end{proof}

\begin{remark}\label{remark b to a}
The proof shows that the existence of any one $f'\in C_c^\infty(S)$ satisfying the conclusion of part \eqref{inhomconj b} in Conjecture \ref{inhomconj} already implies part \eqref{inhomconj a}. The same remark applies to the other two conjectures. 
\end{remark}

For the converse direction from (a) to (b), we need the following. 

\begin{conjecture}[Density principle]
\label{conj density}
The orbital integrals $\Orb(\gamma,\,\cdot\,)$ for all regular semi-simple $\gamma$ span a weakly dense subspace in the space of $(H'(F_0),\eta)$-invariant distributions on $S(F_0)$. The same holds for $\fks(F_0)$. 
\end{conjecture}

\begin{remarks}
\begin{altenumerate}
\renewcommand{\theenumi}{\roman{enumi}}
\item An equivalent statement to the above density principle for $S(F_0)$ is as follows: if a function $f'\in  C^\infty_c(S)$ has vanishing orbital integrals $\Orb(\gamma,f')=0$ for all regular 
semi-simple $\gamma\in S(F_0)$, then it lies in the subspace of $C^\infty_c(S)$ spanned by functions of the form $\tensor*[^{\eta(h)h-1}]{f}{^{\prime\prime}}$ for $f'' \in  C^\infty_c(S)$  and $h\in H'(F_0)$.  (Of course the orbital integrals of $\tensor[^{\eta(h)h-1}]{f}{^{\prime\prime}}$ vanish by Lemma \ref{twisted del}.)
\item It is easy to see that the density principles for $S(F_0)$ and $\fks(F_0)$ are equivalent. 
\item The density principle holds for $\fks(F_0)$ when $n=3$ by \cite[Th.~1.1]{Z12b}; cf.~Theorem \ref{densprince} below. It is still open for $n\geq 4$.
\end{altenumerate}
\label{density remarks}
\end{remarks} 

\begin{lemma}\label{lem a to b}
Assume that Conjecture \ref{conj density} holds. Then part (a) implies part (b) in each of Conjectures \ref{conj homog},  \ref{inhomconj}, and \ref{lieconj}.
\end{lemma}

\begin{proof}
We show the implication \eqref{inhomconj a}$\implies$\eqref{inhomconj b} in Conjecture \ref{inhomconj}. The analogous implication for Conjecture \ref{lieconj} can be proved in a similar way, and that for Conjecture \ref{conj homog} follows by Lemma \ref{hom=inhom}. Suppose that $f_0'\in C_c^\infty(S)$ satisfies the conclusion of \eqref{inhomconj a}, and let $f'\in C_c^\infty(S)$ be any function transferring  to $(\mathbf{1}_{ K_0}, 0)$. Then the function $f'-f'_0$ has vanishing orbital integrals at all $\gamma\in S(F_0)_\rs$.  By Conjecture \ref{conj density}, we may assume that  $f'-f_0'$ is a sum of functions of the form $\tensor*[^{\eta(h)h-1}]{f}{^{\prime\prime}}$ for $f''\in  C^\infty_c(S)$  and $h\in H'(F_0)$. By Lemma \ref{twisted del}\eqref{twisted del S} (applied to $f''$), it follows that the function  $\del(\gamma, f'-f'_0) $ is of the form $\Orb(\gamma, f_\corr')$ for some $f_\corr'\in C_c^\infty(S)$, and hence \eqref{inhomconj b} holds. This completes the proof.
\end{proof}

\subsection{Uniqueness of the function in part (a) of the conjectures}
In this subsection we explain the extent to which the function $f'$ in Conjecture \ref{inhomconj}\eqref{inhomconj a} is unique, assuming the density principle (Conjecture \ref{conj density}).  An analogous statement holds for Conjecture \ref{conj homog}\eqref{conj homog a}.  The statement does not carry over to Conjecture \ref{lieconj}\eqref{lieconj a} as we have formulated it, owing to the additional  hypothesis (artinian intersection) that we impose, cf.~Remark \ref{lie conj artinian remark}.

Let $\epsilon := \eta(\det W_0^\flat) \in \{\pm 1\}.$ Then $-\epsilon=\eta(\det W_1^\flat)$. Here we assume as usual that the norm of the special vector $u_i\in W_i$ is $1$. Otherwise we may modify the definition of $\epsilon$ accordingly. 

We consider the action of the transpose on $S(F_0)$ and on $C_c^\infty(S)$. For $f'\in C_c^\infty(S)$, we define the function $f'^t$ by $f'^t(\gamma) := f'(\tensor*[^t]{\gamma}{})$ for $\gamma\in S(F_0)$. Clearly, each $f'$ can be written in a unique way as 
\[
   f' = f'_+ + f'_-
\]
where $f'^t_\pm=\pm \epsilon f'_\pm$; explicitly,
\[
   f'_+ = \frac 1 2(f' + \epsilon f'^t)
	\quad\text{and}\quad
	f'_- = \frac 1 2(f' - \epsilon f'^t).
\]
Now recall the decomposition $S(F_0)_\rs = S_{\rs,0} \amalg S_{\rs,1}$ from \eqref{Srs decomp}.

\begin{lemma}\label{f+ f- lem}
For any $f' \in C_c^\infty(S)$,
\[
   \Orb(\gamma, f'_+) = 0
	\quad\text{for all}\quad
	\gamma \in S_{\rs,1}
	\quad\text{and}\quad
	\Orb(\gamma, f'_-) = 0
	\quad\text{for all}\quad
	\gamma \in S_{\rs,0}.
\]
\end{lemma}

\begin{proof}
This is a consequence of how $\Orb$ transforms with respect to the transpose operation. For $\gamma\in S(F_0)_{\rs}$, since $\tensor*[^t]{\gamma}{}$ and $\gamma$ have the same invariants, there exists a unique element $h_\gamma\in H'(F_0)$ such that
\begin{equation}\label{eqn t gamma}
   \tensor*[^t]{\gamma}{} = h_\gamma^{-1}\gamma h_\gamma.
\end{equation}
Moreover, the element $h_\gamma$ is symmetric, i.e.~$h_\gamma = \tensor*[^t]{h}{_\gamma}$; and if $\gamma \in S_{\rs,i}$, then $h_\gamma$ defines a hermitian space isometric to $W^\flat_i$, i.e.~$\eta(h_\gamma)=\eta(\det W_i^\flat)$ (cf.~the proof of \cite[Lem.~2.3]{Z12}).  Writing $\eta_s(h) := \eta_s(\det h)$ for $h \in H'(F_0)$, it follows from \eqref{eqn t gamma} and suitable substitutions that
\begin{equation*}
\begin{aligned}
  \Orb(\gamma,  f'^t, s)
     &= \int_{H'(F_0)} f'(\tensor*[^t]{h}{} \tensor*[^t]{\gamma}{} \tensor*[^t]{h}{^{-1}})\eta_s(h)\,dh\\
     &= \int_{H'(F_0)} f'(h h_{\gamma}^{-1}\gamma h_\gamma h^{-1})\eta_s(h)\,dh\\
     &= \int_{H'(F_0)} f'(h^{-1} \gamma  h)\eta_s(h^{-1} h_\gamma )\,dh\\
     &= \eta_s(h_\gamma)  \notag \Orb(\gamma, \,  f', -s).
\end{aligned}
\end{equation*}
In particular,
\[
   \Orb(\gamma,\, f'^t)=\eta(h_\gamma)\Orb(\gamma, f').
\]
The lemma follows from this.
\end{proof}

\begin{lemma}\label{lem uniqueness}
Assume that Conjecture \ref{conj density} holds. Let $f' \in C_c^\infty(S)$ be a function satisfying the conclusion of Conjecture \ref{inhomconj}(\ref{inhomconj a}).  Then $f'$ is unique up to adding a linear combination of functions of the form
\begin{altenumerate}
\item\label{i} $\tensor[^{\eta(h_1)h_1-1}]{(}{}\tensor[^{\eta(h_2)h_2-1}]{\theta}{})$;
\item\label{ii} $\tensor*[^{\eta(h)h-1}]{(\Theta + \epsilon \Theta^t)}{}$; and
\item\label{iii} $\sum_{i=1}^m \tensor*[^{\eta(h_i)h_i-1}]{{f_i'}}{}$ such that $\sum_{i=1}^m \log\lvert \det h_i \rvert  f_i'=0$,
\end{altenumerate}
where $h,h_i\in H'(F_0)$ and $\theta,\Theta,f_i'\in C_c^\infty(S)$. Conversely, adding any  function  of the form \eqref{i}, \eqref{ii}, or \eqref{iii} to $f'$ gives a new function satisfying the conclusion of Conjecture \ref{inhomconj}(\ref{inhomconj a}).
\end{lemma}

\begin{proof} 
First consider the following two properties of a function $f'' \in C_c^\infty(S)$:
\begin{altenumerate}
\renewcommand{\theenumi}{\arabic{enumi}}
\item\label{f' Orb}
$\Orb(\gamma,f'')=0$ for all $\gamma\in S(F_0)_\rs$; and
\item \label{f' del}
$\del(\gamma,f'')=0$ for all $\gamma\in S_{\rs,1}$.
\end{altenumerate}
If $f''$ is of type \eqref{i}, \eqref{ii}, or \eqref{iii} above, then $f''$ satisfies \eqref{f' Orb} and \eqref{f' del} by Lemma \ref{twisted del}\eqref{twisted del S}, with an assist from Lemma \ref{f+ f- lem} for property \eqref{f' del} when $f''$ is of type \eqref{ii}.  For such $f''$, the function $f' + f''$ therefore satisfies the conclusion of Conjecture \ref{inhomconj}\eqref{inhomconj a} whenever $f'$ does.
	
Now suppose that both $f_1'$ and $f_2'$ satisfy the conclusion of Conjecture \ref{inhomconj}\eqref{inhomconj a}, and set
\[
   f'':=f_1'-f_2'. 
\]
Then $f''$ satisfies \eqref{f' Orb} and \eqref{f' del}. By \eqref{f' Orb}, it follows from Conjecture \ref{conj density} that
\[
   f'' = \sum_{i=1}^m \tensor*[^{\eta(h_i)h_i-1}]{{f_i''}}{}
	\quad\text{for some}\quad
	f_i'' \in C_c^\infty(S),
	\quad
	h_i\in H'(F_0).
\]
If $\log\lvert \det h_i \rvert = 0$ for all $i$, then $f''$ is of type \eqref{iii} above, and we're done. If not, then say $\log\lvert \det h_1 \rvert \neq 0$.  
By Lemma \ref{twisted del}\eqref{twisted del S},
\[
   \del(\gamma, f'')=\Orb(\gamma,  \Theta)
	\quad\text{for}\quad
	\Theta := \sum_{i=1}^m \log\lvert \det h_i \rvert f_i''.
\]
By \eqref{f' del}, we have $\Orb(\gamma,  \Theta)=0$  for all $\gamma\in S_{\rs,1}$. This and Lemma \ref{f+ f- lem} imply that $\Orb(\gamma,\Theta_-)=0$ for all $\gamma\in S(F_0)_\rs$. So by another application of Conjecture \ref{conj density}, we may write
\[
   \Theta_-=\sum _{j=1}^k \tensor[^{\eta(h'_j)h'_j-1}]{\theta}{_j}
	\quad\text{for some}\quad
	\theta_j \in C_c^\infty(S),
	\quad
	h_j' \in H'(F_0).
\]
It follows that $\tensor*[^{\eta(h_1)h_1-1}]{\Theta}{} = \tensor*[^{\eta(h_1)h_1-1}]{(\Theta_++\Theta_-)}{}$ is a sum of functions type \eqref{i} and \eqref{ii}.  Since
\[
   f'' - \frac{\tensor*[^{\eta(h_1)h_1-1}]{\Theta}{}}{\log \lvert \det h_1 \rvert}
		= \sum_{i=2}^m \biggl( \tensor*[^{\eta(h_i)h_i-1}]{{f_i''}}{} - \frac{\log \lvert \det h_i \rvert}{\log \lvert \det h_1 \rvert} \tensor*[^{\eta(h_1)h_1-1}]{{f_i''}}{} \biggr)
\]
is a function of type \eqref{iii}, we conclude that $f''$ is of the asserted form.
\end{proof}

\subsection{Statement of the main results}
Now we state our main results.
In the remainder of the paper, we will be concerned with the case $n = 3$,
with (apart from \S\ref{inputs section}) only occasional remarks about the case of general $n$. 
In this case, the following result gives a combined version of Remark \ref{g rs => Int(g) proper}, Remark \ref{Intfiniteforliealg}, and Conjecture \ref{conj Int(g) proper => g rs}. Recall from \eqref{Delta cycle} the notation $\Delta=\Delta_\CN(\CN_2)$, and from \eqref{Deltaghomog} (or from \eqref{Deltag} in the Lie algebra case) the notation $\Delta_g$.  Note that, by the identification \eqref{Aut(BX_n)=U(W_1)}, we may view $U_1(F_0)$ and $\fku_1(F_0)$ as subsets of $\End_{O_F}^\circ(\BX_n)$. Also, recall from the Introduction that, since the statements below involve the geometry of formal schemes, we are taking $F_0=\BQ_p$ throughout. 

\begin{theorem}\label{main prop}
Let $n=3$.  Then for any $x \in \fku_1(F_0)$, the intersection $\Delta\cap\Delta_x$ is non-empty if and only if $x \in \End_{O_F}(\BX_3)$, and the following three properties are equivalent. 
\begin{altenumerate}
\item $x\in \End_{O_F}(\BX_3)\cap\fku_1(F_0)_\rs$. 
\item The intersection of $\Delta$ and $\Delta_x$ is  a non-empty scheme proper over $\Spec O_{\breve F}$.
\item  The intersection of $\Delta$ and $\Delta_x$ is  non-empty artinian.
\end{altenumerate}
 The analog where $x\in \fku_1(F_0)$ is replaced by $g\in U_1(F_0)$ also holds true. 
\end{theorem}

We complete the proof of Theorem \ref{main prop} in \S\ref{rel to int numbers}, modulo some explicit calculations which we carry out in \S\ref{sec expl Lie}.

We also prove the AT conjectures formulated above in the case $n=3$, namely the following theorems.

\begin{theorem}[Group version]\label{thm group}
 Let  $n=3$. Then for any $g\in U_1(F_0)_\rs$, the intersection of $\Delta$ and $\Delta_g$ is an artinian scheme with two points unless it is empty, and there are no higher Tor terms in the expression \eqref{Int(g) inhomog} for $\Int(g)$.  Furthermore, statements (\ref{inhomconj a}) and (\ref{inhomconj b}) of Conjecture \ref{inhomconj} hold true. 
\end{theorem}

By Lemma \ref{hom=inhom}, this theorem also implies  the  homogeneous group version, i.e.~where $g\in G_{W_1}(F_0)_\rs$, cf.~Conjecture \ref{conj homog}. 

\begin{theorem}[Lie algebra version]
\label{thm lie}
 Let $n=3$. Then for any $x\in \fku_{1}(F_0)_{\rs}$, the intersection of $\Delta$ and $\Delta_x$ is an artinian scheme with two points, unless it is empty.  Furthermore, statements (\ref{lieconj a}) and (\ref{lieconj b}) of Conjecture \ref{lieconj}
 hold true.
\end{theorem}

The proofs of Theorems \ref{thm group} and \ref{thm lie} will occupy essentially the entire rest of the paper.  We prove these theorems by a combination of methods from geometry and from local harmonic analysis, which in rough outline goes as follows.
In \S\ref{serre isom} and \S\ref{sums of qcan divisors section}, we relate the moduli space $\CN_2$ to the formal deformation space of formal $O_{F_0}$-modules, and the special divisors on $\CN_2$  (the analog of KR-divisors in the present setting of a ramified quadratic extension)  to quasi-canonical divisors on this formal deformation space. In \S\ref{reduction to LA}, using the Cayley transform, we reduce the computation of $\Int(g)$ to the calculation of $\lInt(x)$ in the Lie algebra, and even in the \emph{reduced subset} of the Lie algebra. The calculation of $\lInt(x)$ for a reduced element $x$ in the Lie algebra  is then carried out in \S\ref{sec expl Lie}. This uses the calculation of intersection numbers of special cycles, and is based on the Gross-Keating formulas for the intersection numbers of quasi-canonical divisors. 
At this point the inputs to the proof from geometry are in place.
The local harmonic analysis set-up used in our proof is explained in \S\ref{inputs section}. 
Using this, we show
in Proposition \ref{lem1} that Theorem \ref{thm group} follows from Theorem \ref{thm lie}.  We show  in Proposition \ref{lem2} that, in turn, Theorem \ref{thm lie} follows from Theorem \ref{mainforred}, which is its analog for the reduced sets in the Lie algebra setting (cf.~\S\ref{subsec inv u1}, \S\ref{subsred}, and \S\ref{subsec inv u0}).  We then obtain Theorem \ref{mainforred} from a comparison of the result of the geometry side with the germ expansion of the orbital integral side. This part of the proof, carried out in \S\ref{germspecial} and \S\ref{section:comparison}, is explained in \S\ref{sec:apllgerm}.   The general germ expansion is established in Part \ref{germ expansion part}, and has its own introduction in 
\S\ref{sec statementofgerm}.

\part{Geometric side}
In this part of the paper we address the geometric aspects of Theorems \ref{thm group} and \ref{thm lie}, including the geometric side of the identities to be proved in these theorems.  Along the way we also prove Theorem \ref{main prop}.  Throughout this part we take $n = 3$.

\section{The Serre map}\label{serre isom}

The main aim of this section is to describe the moduli space $\CN_2$ in terms of the Serre tensor construction (see \S\ref{serre const}), and to use it to analyze special cycles.  

\subsection{The Serre map} 

The first basic fact we need about $\CN_2$ is the following.

\begin{lemma}\label{two points}
$\CN_2\bigl(\ov k\bigr)$ consists of two points.
\end{lemma}

\begin{proof}
Let $N$ denote the rational covariant Dieudonn\'e module of the framing object $\BX_2$, and recycle the notation \aform, $h$, $\tau$, and $C$ from the proof of Proposition~\ref{k-bar isog classes}.  Since $\BX_2$ satisfies the spin condition, the hermitian space $C$ is non-split by Lemma~\ref{C parity lem}.  By Dieudonn\'e theory, $\CN_2(\ov k)$ identifies with the set of $O_{\breve F}$-lattices $L$ in $N$ such that
\begin{equation}\label{Dieu latt}
   \varpi L \subset^2 \pi\tau L \subset^2 L,
\end{equation}
where the superscripts indicate the $\ov k$-dimension of the corresponding quotients, and such that $L^\vee = \pi^{-1}L$, where $L^\vee$ denotes the dual lattice in $N$ with respect to \aform, or equivalently with respect to $h$; cf.~\cite[Eg.~4.14]{RV}, or \cite[Prop.~2.2]{RTW} for the variant where the polarization in the moduli problem is principal.  Note that, since $C$ is non-split, the $p$-divisible group corresponding to any such $L$ automatically satisfies the spin condition by Lemma~\ref{C parity lem}.  Given such $L$, we have
\begin{equation}\label{Lambdas}
   L \subset^1 L + \tau L \subset^1 \pi^{-1}L.
\end{equation}
By an obvious variant of \cite[Prop.~2.17]{RZ} (with $\breve F$ in place of $W_\BQ$, $\pi$ in place of $p$, etc.), the lattice $L + \tau L$ in the $2$-dimensional $\breve F$-vector space $N$ is $\tau$-stable.  This lattice is also self-dual, since the dual lattice
\[
   (L + \tau L)^\vee = L^\vee \cap (\tau L)^\vee = L^\vee \cap \tau (L^\vee) = \pi^{-1} (L \cap \tau L)
\]
contains $L + \tau L$ by \eqref{Dieu latt}, and both 
$L + \tau L$ and $(L + \tau L)^\vee$
are contained in $\pi^{-1}L$ with codimension $1$ by \eqref{Lambdas} (and the dual of the diagram \eqref{Lambdas}).  Now, since $C$ is non-split and $2$-dimensional, 
$C$ contains a \emph{unique} self-dual $O_F$-lattice $\Lambda$.  Hence $L + \tau L = O_{\breve F} \cdot \Lambda$ inside $N$.
The hermitian form $h$ induces a \emph{symmetric} form on $V := \Lambda \otimes_{O_{ F}} \ov k$, and the image of $L$ in $V$ is an isotropic line.  Since $V$ is $2$-dimensional, there are exactly two isotropic lines in $V$, and these correspond to the two possibilities for $L$.
\end{proof}

\begin{remark}\label{tau interchange}
We record for later use that, in the notation of the proof, $\tau$ interchanges the two lattices in $N$ corresponding to the two points in $\CN_2(\ov k)$.
\end{remark}

Since $\CN_2$ is formally locally of finite type over $\Spf O_{\breve F}$, we conclude that it consists of two connected components
\[
   \CN_{2,+}
	\quad\text{and}\quad
	\CN_{2,-},
\]
each reduced to a point topologically.  These components are distinguished as the respective loci in $\CN_2$ where the framing map $\rho$ is and is not pointwise an isomorphism.  They are interchanged under the group action of $U(\BX_2) \simeq U_1(F_0)$ by elements of nontrivial Kottwitz invariant, cf.~\eqref{Kottwitz}.  Thus to understand the structure of $\CN_2$, it remains to understand either one of these components.  We will do this in the rest of this subsection via the \emph{Serre map}.

Forgetting the $O_F$-action, consider $\BE$ as a connected $p$-divisible $O_{F_0}$-module of dimension $1$ and height $2$ over $\Spec \ov k$. Let $\CM_{O_{\breve F_0}}$ denote its formal deformation space over $\Spf O_{\breve F_0}$.  By Lubin--Tate theory, $\CM_{O_{\breve F_0}} \simeq \Spf O_{\breve F_0}[[t]]$.  
This formal scheme
also admits a moduli description as (a connected component of)
an RZ-space \cite[Prop.~3.79]{RZ}: for each scheme $S$ over $\Spf O_{\breve F_0}$, $\CM_{O_{\breve F_0}}(S)$ is the set of isomorphism classes of pairs $(Y_0,\beta)$, where $Y_0$ is a $p$-divisible $O_{F_0}$-module over $S$ and $\smash{\beta\colon Y_{0,\ov S} \to \BE_{\ov S}}$ is an $O_{F_0}$-linear quasi-isogeny of height $0$.  An isomorphism between such pairs is an isomorphism between $p$-divisible $O_{F_0}$-modules over $S$ which is compatible with the framings to $\BE_{\ov S}$ in the obvious sense.  Set
\[
   \CM := \CM_{O_{\breve F_0}} \times_{\Spf O_{\breve F_0}} \Spf O_{\breve F}.
\]

\begin{proposition}\label{Serre isom}
The Serre construction $Y_0 \mapsto O_F \otimes_{O_{F_0}} Y_0$ induces an isomorphism of formal schemes over $\Spf O_{\breve F}$,
\[
   \CM\isoarrow \CN_{2,+}.
\]
\end{proposition}

\begin{proof}
Given $(Y_0,\beta\colon Y_{0,\ov S} \to \BE_{\ov S}) \in \CM(S)$, we first have to explain how to define the rest of the quadruple $(O_F \otimes_{O_{F_0}} Y_0,\iota,\lambda,\rho) \in \CN_{2,+}(S)$.  For simplicity, we use the version of the moduli problem for $\CN_2$ described in Remark~\ref{alt formulation}.
Of course,
the notation signifies that for $\iota$ we take
the tautological $O_F$-action on $O_F \otimes_{O_{F_0}} Y_0$.  
There is a canonical isomorphism
\[
   \Lie(O_F \otimes_{O_{F_0}} Y_0) \cong O_F \otimes_{O_{F_0}} \Lie Y_0
\]
as $O_F \otimes_{O_{F_0}} O_S$-modules, from which it easily follows that $(O_F \otimes_{O_{F_0}} Y_0, \iota)$ satisfies the Kottwitz and spin conditions.  We define the framing map $\rho$ to be the quasi-isogeny
\[
   \rho\colon O_F \otimes_{O_{F_0}} Y_{0,\ov S} \xra{\id_{O_F} \otimes \beta} O_F \otimes_{O_{F_0}} \BE_{\ov S} = \BX_{2, \ov S},
\]
which is an isomorphism at each point of $S$ since the same is true of $\beta$.

It remains to define the polarization $\lambda$.  Since $\CM$ is a formal scheme over $\Spf O_{\breve F}$ with a single $\ov k$-point, it suffices to assume that $S$ is the spectrum of an Artin local ring with residue field $\ov k$.
Let $\lambda_0$ be any principal polarization of $Y_0$; this exists, for example, because $\BE$ is isomorphic to the $p$-divisible group of an elliptic curve, and hence so is $Y_0$ by the Serre--Tate theorem.
Over $\ov S = \Spec \ov k$, the principal polarizations $\lambda_{0,\ov k}$ and $\beta^*(\lambda_\BE)$ of $Y_{0,\ov k}$ differ by an $O_{F_0}^\times$-multiple by our remarks in \S\ref{framing obs}.  Rescaling $\lambda_0$ as needed, we may assume that $\lambda_{0,\ov k} = \beta^*(\lambda_\BE)$.
Then we define $\lambda$ to be the polarization
\[
   \lambda\colon O_F \otimes_{O_{F_0}} Y_0 
      \xra{\varphi \otimes \lambda_0} 
	  \ov{O_F^\vee} \otimes_{O_{F_0}} Y_0^\vee \cong (O_F \otimes_{O_{F_0}} Y_0)^\vee,
\]
where $\varphi\colon O_F \to \ov{O_F^\vee}$ is the symmetric $O_F$-linear map defined in \eqref{varphi}.  Just as when we defined $\BX_2$, one readily verifies that $\Ker \lambda = (O_F \otimes_{O_{F_0}} Y_0)[\iota(\pi)]$.  Furthermore 
\[
   \lambda_{\ov k} = \varphi \otimes \lambda_{0,\ov k} = \varphi \otimes \beta^*(\lambda_\BE) = \rho^*(\varphi \otimes \lambda_\BE) = \rho^*(\lambda_{\BX_2}).
\]
Hence $(O_F \otimes_{O_{F_0}} Y_0, \iota, \lambda, \rho)$ gives a point in $\CN_{2,+}(S)$, and its isomorphism class is clearly well-defined in terms of the isomorphism class of $(Y_0,\beta)$.  This defines the map $\CM \to \CN_{2,+}$.

Now we show that this map is an isomorphism.  Since $\CM$ and $\CN_{2,+}$ both have a single $\ov k$-point and are formally smooth over $\Spf O_{\breve F}$ of relative formal dimension $1$ (using Proposition \ref{formally sm ess proper} in the case of $\CN_{2,+}$), it suffices to show that the induced map on tangent spaces is nonzero.  For this we might as well prove the a priori stronger fact that $\CM(S) \to \CN_{2,+}(S)$ is an injection for any $\Spf O_{\breve F}$-scheme $S$.  Let $(Y_0,\beta)$ and $(Y'_0,\beta')$ be $S$-points on $\CM$, and let $(O_F \otimes_{O_{F_0}}Y_0,\iota,\lambda,\rho)$ and $(O_F \otimes_{O_{F_0}}Y'_0,\iota',\lambda',\rho')$ be the corresponding quadruples as defined above.  With respect to the $O_{F_0}$-linear decompositions
\[
   O_F \otimes_{O_{F_0}}Y_0 = 1 \otimes Y_0 + \pi \otimes Y_0,
	\quad
	O_F \otimes_{O_{F_0}}Y'_0 = 1 \otimes Y'_0 + \pi \otimes Y'_0,
	\quad\text{and}\quad
	\BX_2 =  1 \otimes \BE + \pi \otimes \BE,
\]
the framings $\rho$ and $\rho'$ take the form
\[
   \rho = \diag(\beta,\beta)
	\quad\text{and}\quad
	\rho' = \diag(\beta',\beta').
\]
Thus by ridigity $(\rho')^{-1}\circ \rho$ lifts to an isomorphism $O_F \otimes_{O_{F_0}}Y_0 \isoarrow O_F \otimes_{O_{F_0}}Y'_0$ if and only if $(\beta')^{-1}\circ \beta$ lifts to an isomorphism $Y_0 \isoarrow Y'_0$, which completes the proof.
\end{proof}

\begin{remark}
Lemma~\ref{two points} and Proposition~\ref{Serre isom} make precise and supply details for the Claim for the formal scheme denoted $\CM_1$ in \cite[Eg.~4.14]{RV}.  Note however that loc.~cit.~uses the framing object described in \cite[\S5 d)]{KR-alt}; this framing object should be replaced with our $\BX_2$, as discussed in Remark \ref{KR framing object}.\footnote{Note also that the claim concerning $O_F^\vee$ at the end of \cite[Eg.~4.14]{RV} is obviously incorrect.}
\end{remark}

\subsection{Special divisors}\label{special divisors}
Recall the embedding $\iota_\BE\colon O_F \hookrightarrow \End(\BE) = O_D$ and the corresponding canonical lift $\CE$ of $\BE$ over $O_{\breve F}$, with its action $\iota_\CE \colon O_F\to \End(\CE)$ and principal polarization $\lambda_\CE$. Let $\CY_0$ denote the universal $p$-divisible group over $\CM$, and let $c\in\End_{O_{F_0}}(\BE) = \Hom_{O_{F_0}}(\BE,\ov\BE)$. Associated to $c$ is the closed sublocus $\CT_{F}(c)$ of $\CM$ where $c$ lifts to a homomorphism $\CY_0 \to \ov\CE$.  Note that the divisors $\CT_{F}(c)$ are different from the divisors considered by Gross--Keating, i.e.~the locus where $c$ deforms to an endomorphism of $\CY_0$, or in other words a divisor of the form $\Spf W[[t]]/J$ in \cite[top of p.~147]{Rapo}.

Now let \CY denote the universal $p$-divisible group over $\CN_2$, and let $b\in\Hom_{O_F}(\BX_2, \ov\BE)$. Associated to $b$ is the closed sublocus $\CZ(b)$ of $\CN_{2,+}$ where the $O_F$-linear homomorphism $b\colon \BX_2 \to \ov\BE$ lifts to 
a
homomorphism $\CY \to \ov\CE$ (the analog in our present ramified setting of a \emph{KR divisor} in the unramified setting \cite{KR-U1}\footnote{Note however that in loc.~cit.~deformations of homomorphisms $\ov\BE \to \BX_2$ are considered.}). By identifying, via Proposition~\ref{Serre isom}, the restriction of \CY to $\CN_{2,+}$ with $O_F \otimes_{O_{F_0}} \CY_0$, adjunction  implies the following lemma. 

\begin{lemma}\label{compwserre}
If $c$ corresponds to $b$ under the adjunction isomorphism
\[
   \Hom_{O_{F_0}}\bigl(\BE, \ov\BE\bigr) \cong \Hom_{O_F}\bigl(O_F \otimes_{O_{F_0}} \BE, \ov\BE\bigr) = \Hom_{O_F}\bigl(\BX_2, \ov\BE\bigr) ,
\]
then the Serre isomorphism $\CM \cong \CN_{2,+}$ identifies
\begin{flalign*}
	\phantom{\qed} & &
   \CT_{F}(c) \cong \CZ(b).
	& & \qed
\end{flalign*}
\end{lemma}

From now on, we often drop the field $F$ from the notation $\CT_F(c)$. 

\begin{proposition}
If $b$ and $c$ are nonzero, then both $\CZ(b)$ and $\CT(c)$ are relative divisors.
\end{proposition}

\begin{proof}
Of course it suffices to prove this for $\CT(c)$. Consider $\CS := \CM \times_{\Spf O_{\breve F}} \CM$, with its universal object $\CY_0\times\CY'_0$. Recall  from \cite[Prop.~5.1]{Wewers} that the locus inside $\CS$ where $c$ lifts to a homomorphism $\CY_0\to\CY'_0$ is a relative divisor $\CZ$ in $\CS$. Another divisor $\CD$ inside $\CS$ is given by the locus where $\CY'_0= \ov\CE$. Now $\CD\simeq \CM$ is an irreducible divisor that is not contained in $\CZ$ (otherwise $c$ would lift to a homomorphism $\CY_0\to \ov\CE$ over all of $\CM$, which is absurd). Hence $\CT(c)=\CZ\cap \CD$ is a divisor on $\CM$. It is a relative divisor because $c$ does not lift to a homomorphism $\CY_0\to \ov\CE$ over the whole special fiber $\overline{\CM}=\CM\times_{\Spf O_{\breve F}}\Spec \ov k$. 
\end{proof}

\begin{remark}\label{qcanon} There are also quasi-canonical variants of this construction. Let $j\geq 1$.  Let
\[
   O_j: = O_{F, j} := O_{F_0} + \pi^j O_{F}
\]
be the order of conductor $j$ in $O_F$.  Let $W_j$ be the ring of integers of the ring class field extension of $\breve F$ corresponding to $O_j$, and let $\CW_j := \Spec W_j$.  Let $\CE_j$ be the quasi-canonical lifting of level $j$ over $\CW_j$ \cite[Def. 3.1]{Wewers}.  In particular, $W_0 = O_{\breve F}$ and $\CE_0 = \CE$.  Put
\[
   \CM_j := \CM\times_{\Spec O_{\breve F}} \CW_j
	\quad\text{and}\quad
	\CN_{2,+, j} := \CN_{2,+}\times_{\Spec O_{\breve F}} \CW_j.
\]
Let $\ov\CE_j$ denote the same object as $\CE_j$, but where the $O_j$-action is precomposed by the nontrivial Galois automorphism.  Inside $\CM_j$, we have the locus $\CT_{F,j}(c)$ where the endomorphism $c\in O_D$ lifts to a homomorphism $\CY_0 \to \ov\CE_j$; and inside  $\CN_{2,+, j}$, we have the locus $\CZ_j(b)$ where the $O_F$-linear homomorphism $b\colon \BX_2 \to \ov\BE$ lifts to an $O_j$-linear homomorphism $\CY \to \ov\CE_j$.
 If $c$ corresponds to $b$ under the adjunction isomorphism
$
\Hom_{O_{F_0}}(\BE, \ov\BE) \cong \Hom_{O_F}(O_F \otimes_{O_{F_0}} \BE, \ov\BE)=\Hom_{O_F}(\BX_2, \ov\BE) ,
$
then
\[
   \CT_{F,j}(c) \cong \CZ_j(b)
\]
under the Serre isomorphism.  
\end{remark}

\section{Special divisors as sums of quasi-canonical divisors}\label{sums of qcan divisors section}

In this section we express the special divisor $\CT(c)$ defined in \S\ref{special divisors} as a sum of quasi-canonical divisors, where $c \in O_D$ is nonzero.  We identify $F$ with its image in $D$ via $\iota_\BE$, but we also consider the conjugate embedding $\iota := \tensor[^{\ov c}]{\iota}{_\BE} \colon F \inj D$ and its image $\tensor*[^{\ov c}]{F}{}$.  Recall that here the bar denotes the main involution on $D$, and the conjugate embedding is defined in \eqref{conjugate emb}. Corresponding to $\iota$, we have the quasi-canonical divisor $\CW_{\tensor*[^{\ov c}]{F}{}, j}$ on \CM and the quasi-canonical lift $\CE_{\iota,j}$ of level $j$ over $\CW_{\tensor*[^{\ov c}]{F}{}, j}$, as well as the canonical lift $\CE_\iota$ over $\Spf O_{\breve F}$.  

\begin{proposition}\label{sums of qcan divisors prop}
There is an equality of divisors on \CM,
\[
   \CT(c)=\sum_{0\leq j\leq v_D(c)}\CW_{\tensor*[^{\ov c}]{F}{}, j} . 
\]
\end{proposition}

\begin{proof} 
Note that in the definition of $\CT(c)$, it is harmless to replace $\ov\BE$ with $\BE$ and $\ov\CE$ with $\CE$, since in both cases the underlying $O_{F_0}$-modules are the same.  Set $\gamma := v_D(c)$, and write $c = \pi^\gamma c_0$ with $c_0 \in O_D^\times$. Let $j\leq \gamma$.  The element $\pi^{\gamma-j} c_0 \in O_D$ conjugates $\iota$ into $\iota_\BE$, and therefore lifts to a homomorphism $\CE_\iota \to \CE$.  Over the locus $\CW_{\tensor*[^{\ov c}]{F}{}, j}$, the endomorphism $\iota(\pi^j)$ of $\BE$ lifts to a homomorphism $\psi_j\colon \CE_{\iota, j} \to \CE_\iota$.  Since $\iota = \tensor[^{c_0^{-1}\!\!\!}]{\iota}{_\BE}$,  we thus obtain a diagram
\[
   \xymatrix@C+2ex{
	   \CE_{\iota, j} \ar[r]^-{\psi_j} \ar@{-}[d]  &  \CE_\iota \ar[r] \ar@{-}[d]  &  \CE \ar@{-}[d]\\
		\BE \ar[r]^-{c_0^{-1} \pi^j c_0}  &  \BE \ar[r]^-{\pi^{\gamma-j}c_0}  &  \BE,
	}
\]
where the vertical lines indicate reduction to $\ov k$, and where the bottom row evidently composes to $c$.
This shows that the divisor $\CW_{\tensor*[^{\ov c}]{F}{}, j}$ is a component of $\CT(c)$. We therefore obtain an inequality of divisors on $\CM$, 
\[
   \sum_{j=0}^\gamma \CW_{\tensor*[^{\ov c}]{F}{}, j} \leq \CT(c).
\]
The equality will follow by comparing the intersections of both sides with the special fiber $\ov\CM$. For the left-hand side, we note that $(\CW_{\tensor*[^{\ov c}]{F}{}, j}\cdot \ov\CM) = [W_j: O_{\breve F}]=q^j$. For the right-hand side, we apply the following lemma.
\end{proof}
\begin{lemma}\label{multspfib}
 The intersection multiplicity of the cycle $\CT(c)$ with the special fiber $\ov\CM$ is
\[
   \bigl(\CT(c)\cdot \ov\CM\bigr) = \sum\nolimits_{0\leq j\leq v_D(c)}q^j.
\]
\end{lemma}

\begin{proof} 
We use the Kummer congruence, cf.~\cite[Th.~4.1]{Rapo}. Identify $\ov\CM$ with $\Spf \ov k [[t]]$, and the product of $\ov\CM$ with itself with $\Spf \ov k [[t, t']]$. Let $(\CY_0, \CY_0')$ be the universal $p$-divisible group over $\Spf \ov k [[t, t']]$.  Let $I$ be the ideal in $\ov k [[t, t']]$ describing the closed sublocus where $c\colon \BE\to \BE$ lifts to a homomorphism $\wt c: \CY_0 \to \CY_0'$. By loc.~cit.,\ the uniformizers $t$ and $t'$ may be chosen such that $I$ is generated by the element
\[
   g := \bigl(t-(t')^{q^\gamma}\bigr) \bigl(t^q-(t')^{q^{\gamma-1}}\bigr) \dotsb \bigl(t^{q^\gamma}-t'\bigr), 
\]
where as in the previous proof $\gamma = v_D(c)$.  
On the other hand, the locus where $\CY' = \CE$ is defined by $t'=0$. Hence the intersection multiplicity in question is equal to the length of the Artin ring
\[
   \ov k [[t, t']]/(t', g) = \ov k [[t]]/(t^{1+q+\dotsb+q^\gamma}) .
\]
The claim follows.
\end{proof}

\begin{remarks}
\begin{altenumerate}
\item
Note that our convention that $q = p$ when working with formal schemes is in force in Lemma \ref{multspfib}.  Strictly speaking, we need it to appeal to the Kummer congruence.
\item
The analog of Lemma \ref{multspfib} in the case when  $F/F_0$ is unramified is \cite[Prop.~8.2]{KR-U1}.\footnote{Note that the quantity $v$ in loc.~cit.\ should be replaced by half of its value.} The proof of this analog in loc.~cit., due to Th.~Zink, uses displays and is difficult.  The proof of Lemma \ref{multspfib} given here transposes in the obvious way to the unramified case, which gives a drastic simplification of loc.~cit. 

Conversely, one can reduce Lemma \ref{multspfib} to the unramified case in \cite{KR-U1} as follows.  Let $F'$ denote the unramified quadratic extension of $F_0$.  We first point out that the Serre isomorphism in Proposition \ref{Serre isom} also holds in the unramified setting (with $\CM_{O_{\breve F_0}}$ isomorphic to the entire space $\CN_{F'/F_0,2}$), as does the compatibility of special divisors in Proposition \ref{compwserre}. Now let $\CE'/\Spf O_{\breve F_0}$ denote the canonical lifting of $\BE$ for $F'/F_0$ (relative to any embedding of $F'$ in $D$).  Of course
\[
   \CE \times_{\Spf O_{\breve F}} \Spec \ov k \simeq \CE' \times_{\Spf O_{\breve F_0}} \Spec \ov k
\] 
as formal $O_{F_0}$-modules over $\Spec \ov k$.  Hence, identifying the special fibers of $\CM$ and $\smash{\CM_{O_{\breve F_0}}}$,
we get an identification of the corresponding divisors 
\[
   \CT_F(c) \times_{\Spf O_{\breve F}} \Spec \ov k = \CT_{F'}(c) \times_{\Spf O_{\breve F_0}} \Spec \ov k. 
\]
Therefore Lemma \ref{multspfib} follows from Proposition~8.2 of \cite{KR-U1}. 
\end{altenumerate}
\end{remarks}

\section{Reduction to the Lie algebra}\label{reduction to LA}

In this section we lay the framework to reduce the geometric calculations in Theorem~\ref{thm group} (the group setting) to those in Theorem~\ref{thm lie} (the Lie algebra setting).  Modulo these calculations, which will be carried out in the next section, we also prove Theorem \ref{main prop}.

\subsection{Coordinates on $U_1$ and $\fku_1$}\label{U_1 and fku_1 coords}
We begin by presenting the unitary group $U_1(F_0) \simeq U(\BX_3)$ and its Lie algebra $\fku_1(F_0)$ in terms of explicit coordinates.  First recall the embedding
\[
   \iota_{\ov\BE} \colon O_F\incl \End_{O_{F_0}}\bigl(\ov\BE\bigr) = O_D.
\]
Except where stated to the contrary, from now on we will tacitly regard $O_F$ as a subring of $O_D$ via $\iota_{\ov\BE}$, and likewise for $F \subset D$, and drop $\iota_{\ov\BE}$ from the notation.  Write 
\begin{equation}\label{Deigen}
   D = D_+ \oplus D_-
\end{equation}
for the decomposition of $D$ into its respective $+1$ and $-1$ eigenspaces under the conjugation action of $\pi$.  Then $F = D_+$.  For any $\alpha \in D$, we denote by $\alpha_+$ and $\alpha_-$ its respective components with respect to this decomposition.  Note that $\RN \alpha = \RN \alpha_+ + \RN\alpha_-$.
We also write $D^{\tr = 0}$ for the set of traceless elements in $D$, and we analogously define $O_D^{\tr = 0}$, $F^{\tr = 0} = F_0\pi$, and $O_F^{\tr = 0} = O_{F_0}\pi$.

Recall from \S\ref{framing obs} that we have defined the framing object
\begin{equation}\label{BX_3 prod decomp}
   \BX_3 = \BX_2\times\ov\BE=(O_F\otimes_{O_{F_0}}\BE)\times \ov\BE.
\end{equation}
Identifying $O_F\otimes_{O_{F_0}}\BE \simeq \BE \times \BE$ as in \eqref{BX_2 concrete}, we obtain
\begin{equation}\label{End(BX_3) = M_3(O_D)}
   \End_{O_{F_0}}(\BX_3) \simeq \M_3(O_D).
\end{equation}
In terms of this identification, the $O_F$-action $\iota_{\BX_3}$ sends
\[
   \pi \mapsto
   \begin{bmatrix}
	   0  &  \varpi  &  0\\
	   1  &  0  &  0\\
	   0  &  0  & \pi
   \end{bmatrix}.
\]
Hence we identify
\[
   \End_{O_F}(\BX_3) = 
   \left\{\left.
      \begin{bmatrix}
         \alpha  &  \beta\varpi  &  b\pi \\
         \beta  &  \alpha  & b\\
		 c  &  \pi c  &  d
      \end{bmatrix}\,\right|\, \alpha,\beta,b,c\in O_D,\ d\in O_F\right\}.
\]
We emphasize that here and below the symbol $\pi$ always means $\iota_{\ov\BE}(\pi)$, in accordance with our convention.

With regard to the decomposition $\BX_3 \simeq \BE \times \BE \times \ov\BE$, the polarization $\lambda_{\BX_3}$ is given by
\[
   \lambda_{\BX_3} = \diag (\lambda_\BE,  -\varpi\lambda_\BE,  \lambda_{\ov\BE});
\]
see \eqref{explicit lambda_BX_2}.  Since the Rosati involution on $D$ induced by $\lambda_\BE = \lambda_{\ov\BE}$ is the main involution $a \mapsto \ov a$, the Rosati involution $x \mapsto x^\dag = \lambda_{\BX_3}^{-1} \circ x^\vee \circ \lambda_{\BX_3}$ on $\End_{O_F}^\circ(\BX_3)$ is given by the formula
\[
   \begin{bmatrix}
      \alpha  &  \beta\varpi  &  b\pi \\
      \beta  &  \alpha  & b\\
      c  &  \pi c  &  d
   \end{bmatrix}^\dag
   =
   \begin{bmatrix}
	  \ov \alpha  &  - \ov \beta\varpi  &  \ov c \\
      -\ov \beta  &  \ov \alpha  &  \ov c \pi^{-1} \\
	  -\pi \ov b  &  -\ov b\varpi  &  \ov  d
   \end{bmatrix}.
\]

Attached to $\BX_3$ is the identification of unitary groups from \eqref{Aut(BX_n)=U(W_1)},
\begin{equation}
\begin{aligned}\label{U_1 coords}
   U_1(F_0) \simeq U(\BX_3) &= \bigl\{\, g\in\End^\circ_{O_F}(\BX_3) \bigm| gg^\dag = 1 \,\bigr\}\\
      &\qquad\qquad\qquad\qquad\subset
      \left\{\left.
         \begin{bmatrix}
            \alpha  &  \beta\varpi  &  b\pi \\
            \beta  &  \alpha  & b\\
   		 c  &  \pi c  &  d
         \end{bmatrix}\,\right|\, \alpha,\beta,b,c\in D,\ d\in F\right\}.
\end{aligned}
\end{equation}
Thus we get an identification of Lie algebras,
\begin{align}
   \fku_1(F_0) 
	   &\simeq \bigl\{\, x\in\End^\circ_{O_F}(\BX_3) \bigm| x + x^\dag = 0 \,\bigr\} \notag\\
      &= \left\{\left.
         \begin{bmatrix}
            \alpha  &  \beta\varpi  &  b\pi \\
            \beta  &  \alpha  & b\\
		    \pi \ov b  &  \ov b \varpi  &  d
         \end{bmatrix}\,\right|\, \alpha \in D^{\tr = 0},\ \beta \in F_0,\ b \in D,\ d\in F^{\tr = 0}\right\}. \label{fku_1 coords}
\end{align}
Intersecting with $\End_{O_F}(\BX_3)$ gives a natural compact open subgroup in each,
\[
   K_1 := U_1(F_0)\cap \End_{O_F}(\BX_3)
      = \bigl\{\, g \in \End_{O_F}(\BX_3) \bigm| gg^\dag = 1 \,\bigr\}
\]
and 
\begin{align}
   \fkk_1 := {}& \fku_1(F_0) \cap \End_{O_F}(\BX_3) \notag\\
      = {}& \bigl\{\, x \in \End_{O_F}(\BX_3) \bigm| x + x^\dag = 0 \,\bigr\} \notag\\
	  = {}& \left\{\left.
         \begin{bmatrix}
            \alpha  &  \beta\varpi  &  b\pi \\
            \beta  &  \alpha  & b\\
		    \pi \ov b  &  \ov b \varpi  &  d
         \end{bmatrix}\,\right|\, \alpha \in O_D^{\tr = 0},\ \beta \in O_{F_0},\ b \in O_D,\ d\in O_F^{\tr = 0}\right\}. \label{fkk_1 coords}
\end{align}
Note that $K_1$ is the stabilizer in $U_1(F_0)$ of the standard basepoint in $\CN_3$, i.e.~the point
\[
   (\BX_3, \iota_{\BX_3}, \lambda_{\BX_3}, \id_{\BX_3}) \in \CN_3\bigl(\ov k\bigr).
\]
Furthermore we set
\[
   K_{1,\rs} := K_1 \cap U_{1,\rs}(F_0)
	\quad\text{and}\quad
	\fkk_{1,\rs} := \fkk_1 \cap \fku_{1,\rs}(F_0).
\]

\subsection{Invariants on $\fku_1$ and reduced elements}\label{subsec inv u1}
We now make explicit the invariants on $\fku_1$ discussed in \S\ref{ss: Lie invars} in terms of the coordinates just introduced.  Let
\[
   x =
   \begin{bmatrix}
		\alpha  &  \beta \varpi  &  b \pi\\
		\beta  &  \alpha  &  b\\
		\pi \ov b  &  \ov b \varpi  &  d
	\end{bmatrix}
	\in \fku_1(F_0)
\]
be expressed in the form \eqref{fku_1 coords}.  Write
\begin{equation*}
   A' :=
	\begin{bmatrix}
		\alpha  & \beta \varpi \\ 
	   \beta  &  \alpha
	\end{bmatrix},
	\quad
   \mathbf b' := \begin{bmatrix}b\pi\\b\end{bmatrix}, 
	\quad
   \mathbf c' := \begin{bmatrix}\pi \ov b  &  \ov b \varpi \end{bmatrix},
\end{equation*}
so that
\[
   x =
	\begin{bmatrix}
		A'  &  \mathbf{b}'\\
		\mathbf{c}'  &  d
	\end{bmatrix}.
\]
Note that this block decomposition for $x$ is \emph{not the same} as the earlier one in \eqref{x block decomp}, since here we allow matrix entries in $D$.  With respect to the identifications \eqref{BX_3 prod decomp} and \eqref{End(BX_3) = M_3(O_D)}, we have
\begin{equation}\label{block identifications}
   A' \in \End_{O_F}^\circ(\BX_2),
	\quad
	\mathbf{b}' \in \Hom_{O_F}^\circ\bigl(\ov\BE, \BX_2\bigr) = \BV_2,
	\quad\text{and}\quad
	\mathbf{c}' \in \Hom_{O_F}^\circ\bigl(\BX_2, \ov\BE\bigr).
\end{equation}
Using these identifications, one sees that the quantities
\begin{equation}\label{fku_1 invariants}
   \lambda(x) := \operatorname{det}_F(A' \mid \BV_2),\quad
   u(x) := \varpi^{-1}\mathbf{c}' \mathbf{b}',\quad
   w(x) := \varpi^{-1}\mathbf{c}' A' \mathbf{b}',\quad
	\operatorname{tr}_F(A' \mid \BV_2),\quad
	d
\end{equation}
are the five polynomial generators of the invariant ring listed in \eqref{inv sec}, except for the factor $\varpi^{-1}$ in front of the second and third invariants, which we have inserted to give a more convenient normalization; see e.g.~Lemma \ref{intimage} below.

To make the first and fourth invariants in \eqref{fku_1 invariants} explicit, note that the map
\[
   \BV_2 = 
	\biggl\{
	\begin{bmatrix}
		b\pi\\
		b
	\end{bmatrix} 
	\biggm| b \in D \biggr\}
	\to
	D,
	\quad
	\begin{bmatrix}
		b\pi\\
		b
	\end{bmatrix}
	\mapsto b,
\]
is an $F$-linear isomorphism, where $F$ acts naturally on the right on source and target. In this way $A'$ acting on $\BV_2$ identifies with the $F$-linear map
\[
   b \mapsto \alpha b + b \beta \pi
\]
on $D$.  Hence 
\[
   \operatorname{tr}_F(A' \mid \BV_2) = 2\beta \pi
	\quad\text{and}\quad
	\lambda(x) = \operatorname{det}_F(A' \mid \BV_2) = \RN \alpha + \beta^2 \varpi.
\]

\begin{definition}\label{fku_1 reduced}
An element $x \in \fku_1(F_0)$ written as above is called \emph{reduced} if its invariants $\tr_F(A' \mid \BV_2)$ and $d$ are $0$, that is, if $\beta = d = 0$.  We denote by $\fku_{1,\red}(F_0)$ the subspace of reduced elements in $\fku_1(F_0)$.
\end{definition}

Of course, the invariants $\tr_F(A' \mid \BV_2)$ and $d$ as written here arise from regular functions on $\fku_1$, and their vanishing therefore defines $\fku_{1,\red}$ as a closed subscheme of $\fku_1$; hence the notation.  We also set
\[
   \fku_{1,\red,\rs} := \fku_{1,\red} \cap \fku_{1,\rs},
	\quad
	\fkk_{1,\red} := \fkk_1 \cap \fku_{1,\red}(F_0),
	\quad\text{and}\quad
	\fkk_{1,\red,\rs} :=	\fkk_{1,\red} \cap \fkk_{1,\rs}.
\]

There is a natural map $\fku_1(F_0) \to \fku_{1,\red}(F_0)$, which we denote by $x\mapsto x_\red$, defined by
\[
   \begin{bmatrix} 
		\alpha  &  \beta\varpi  &  b\pi\\
		\beta  &  \alpha  &  b\\
		\pi\ov b  &  \ov b \varpi  &  d
	\end{bmatrix}
	\mapsto 
	\begin{bmatrix}
		\alpha  &  0  &  b\pi\\
		0  &  \alpha  &  b\\
		\pi \ov b  &  \ov b \varpi  &  0
	\end{bmatrix}.
\] 
Taking this map together with the last two invariants in \eqref{fku_1 invariants} gives a product decomposition
\begin{equation}\label{u1 prod decomp}
	\begin{gathered}
   \xymatrix@R=0ex{
   \fku_1(F_0) \ar[r]^-\sim  &  \fku_{1, \red}(F_0) \times \fks_1(F_0) \times \fks_1(F_0)\\
	x \ar@{|->}[r]  &  (x_{\red},  2\beta \pi, d) . 
	}
	\end{gathered}
\end{equation}

In \S\ref{rel to int numbers} we are going to explain how to reduce the calculation of intersection numbers not just to the Lie algebra setting, but to \emph{reduced} elements in $\fku_1(F_0)$.  The first basic fact in this direction is the following.

\begin{lemma}\label{rs for reduced u1}
An element $x\in \fku_1(F_0)$ is regular semi-simple if and only if  $x_\red$ is.
\end{lemma}

\begin{proof}
By the linear algebra characterization of regular semi-simple elements in \S\ref{ss:LAchar}, relative to the canonical special vector $u \in \BV_3$ in \eqref{u def}, it suffices to show that the three vectors $u, xu, x^2u$ are linearly independent over $F$ if and only if the vectors  $u, x_\red u, x_\red^2u$ are, and analogously for $\tensor*[^t]{u}{}, \tensor*[^t]{u}{} x, \tensor*[^t]{u}{} x^2$ and $\tensor*[^t]{u}{}, \tensor*[^t]{u}{} x_\red, \tensor*[^t]{u}{} x_\red^2$. For clarity we denote the $F$-action on $\BV_3$ as a right action.  Expressing $x$ in terms of the coordinates \eqref{fku_1 coords}, the first of these equivalences follows from the easily verified relations
\[
   xu = x_\red u + ud
\]
and
\[
   x^2u = xx_\red u + xud = x_\red^2 u + x_\red u \beta \pi + x_\red u d + u d^2 = x_\red^2 u + x_\red u(\beta\pi + d) + ud^2.
\]
The second, ``transposed''  equivalence is proved in a similar way.
\end{proof}

We conclude this subsection by giving a simple characterization of regular semi-simplicity for reduced elements. Let
\begin{equation}\label{coordofred}
   x=
   \begin{bmatrix}
	  \alpha  &  0  &  b\pi \\
      0  &  \alpha  &  b\\
	  \pi\ov b  &  \ov b\varpi  &  0
   \end{bmatrix}
   \in \frak u_{1, \red}(F_0),
	\quad
	\alpha \in D^{\tr = 0},
	\quad
	b \in D.
\end{equation}
The first three invariants in \eqref{fku_1 invariants} take the values on $x$,
\begin{equation}\label{u w lambda}
\begin{aligned}
   \lambda(x) &= \RN \alpha=\RN \alpha_++\RN \alpha_-=\RN \alpha'_++\RN \alpha'_-,\\
   u(x) &= \varpi^{-1}(\pi\ov b b\pi+ \ov b\varpi b)= 2\RN b,\\
   w(x) &= \varpi^{-1}(\pi\ov b \alpha b\pi+ \ov b\varpi \alpha b)=\RN b\cdot (\pi^{-1}\alpha'\pi+\alpha')=2\RN b \cdot \alpha'_+.
\end{aligned}
\end{equation}
Here in the expressions involving $\alpha'$, we have assumed that $b \neq 0$ and set
\[
   \alpha' := b^{-1}\alpha b;
\]
recall that the subscripts $+$ and $-$ denote the components of an element with respect to the decomposition \eqref{Deigen}.  Of course, if $b = 0$, then $u(x) = w(x) = 0$.

Now recall from  \S\ref{ss:LAchar} that $x$ is regular semi-simple if and only if $\Delta(x)\neq 0$, where 
\begin{equation}\label{Delta rescaled}
	   \Delta(x) := -\varpi^{-2} \det(\tensor*[^t] e {} x^{i+j}e)_{0\leq i, j\leq 2}.
\end{equation}  
Note that here we have rescaled the discriminant defined in \eqref{Delta}, which will give us a more convenient normalization later on.  From now on we will always understand $\Delta$ in the sense of \eqref{Delta rescaled}.

\begin{lemma}\label{rs crit for reduced}
A reduced element $x\in\fku_{1,\red}(F_0)$ in the form \eqref{coordofred} is  regular semi-simple if and only if $b\neq 0$ and $\alpha'_-\neq 0$. 
\end{lemma}
\begin{proof}
We calculate the discriminant as 
\begin{equation}\label{Deltaonu1}
\begin{aligned}
	\Delta(x) &= -\varpi^{-2}\det
	     \begin{bmatrix}
			1  &  0  &  \varpi u\\
			0  &  \varpi u  &  \varpi w\\
			\varpi u  &  \varpi w  &  -\lambda \varpi u+\varpi^2u^2
		 \end{bmatrix}\\
     &= \lambda u^2+w^2\\
	  &= 4(\RN b)^2 \RN \alpha'_-.
\end{aligned}
\end{equation}
\end{proof}

\subsection{The Cayley transform}\label{cayley u_1}
Our main tool in passing from the group setting to the Lie algebra setting will be the Cayley transform $x\mapsto (1+x)(1-x)^{-1}$ from $\frak u_1(F_0)$ to $U_1(F_0)$. 
More precisely, let 
\begin{equation}\label{u1circ}
   \fku_1^\circ(F_0) := \bigl\{\, x\in \fku_1(F_0) \bigm| 1-x \text{ is invertible} \,\bigr\}.
\end{equation}
Then the Cayley transform is defined on $\fku_1^\circ(F_0)$. In fact, we will need the variant
\[
   \fkc_{\xi}\colon \fku_1^\circ(F_0) \to  U_1({F_0})
\]
defined by
\begin{equation}\label{cayley fmla}
	x \mapsto \xi \frac{1+x}{1-x},
\end{equation}
where $\xi \in U_1(F_0)$ is a fixed element of the form $\diag(\pm 1_2, \pm 1)$, expressed in the presentation \eqref{U_1 coords} of $U_1(F_0)$.  We remark that in \S\ref{group setting} we will give a slightly more general definition of the Cayley transform for $\fku_1$, and also define it for $\fks$ and $\fku_0$.

\begin{lemma}[Cayley transform for $\fkk_1$]\label{lem cayley U1}
There is an inclusion $\fkk_1\subset \fku_1^\circ(F_0)$, and hence the restriction of the Cayley map $\fkc_\xi$ to $\fkk_1$ is well-defined, where $\xi = \diag(\pm 1_2,\pm1)$. Furthermore, this restriction factors through $K_1$, 
\[
   \fkc_\xi \colon \fkk_1 \to K_1,
\]
 and the images $\fkc_\xi(\fkk_1)$, as $\xi$  varies over these four elements, cover $K_1$.
\end{lemma}

\begin{proof} 
Let
\[
   x=
   \begin{bmatrix}
	  \alpha  &  \beta\varpi  &  b\pi \\
      \beta  &  \alpha  &  b\\
	  \pi \ov b  &  \ov b \varpi  &  d
   \end{bmatrix}
   \in \fkk_1,
\] 
expressed in the form \eqref{fkk_1 coords}.  Then $x \bmod\pi$ is an upper triangular block matrix of the form 
\begin{equation}\label{eqn mod pi}
   \begin{bmatrix}
	  \alpha  &  0  &  0 \\
      \beta  &  \alpha  &  b\\
	  0  &  0  &  d
   \end{bmatrix}
   \mod \pi.
\end{equation}
Since $\alpha\in O_D^{\tr = 0}$ and $d\in O_F^{\tr = 0}$, we have $1 - \alpha$, $1 - d \in O_D^\times$. It follows that $1-x$ is invertible and that its inverse has entries in $O_D$. Hence $x \in \fku_1^\circ(F_0)$ and $\fkc_\xi(x) \in \End_{O_F}(\BX_3) \cap U_1(F_0) = K_1$.

To show that the images cover $K_1$, it suffices to show that for every $g \in K_1$, there exists $\xi = \diag(\pm 1_2,\pm1)$ such that $\fkc_{\xi}^{-1}(g)$ is well-defined and has integral entries.  Here
\begin{equation}\label{cayley inverse}
   \fkc_{\xi}^{-1}(g) = -\frac{1-\xi^{-1}g}{1+\xi^{-1}g}.
\end{equation}
From the equation $gg^\dagger=1$ it follows that $g\bmod \pi$ is also of the form \eqref{eqn mod pi}.  Since $1 + \alpha$ and $1-\alpha$ sum to $2 \in O_D^\times$, at least one of them is in $O_D^\times$ too, and likewise for $1 \pm d$. Thus the desired $\xi$ exists.
\end{proof}

\begin{remark}
Even though $\fkc_\xi(\fkk_1)\subset K_1$, there are  elements  $x\in \fku_1^\circ(F_0) \smallsetminus \fkk_1$ with $\fkc_\xi(x)\in K_1$. 
\end{remark}

In the case of $\fkc := \fkc_{\id_{\BX_3}}$, we also have the following.

\begin{lemma}\label{endom subalg =}
Let $x \in \fkk_1$.  Then there is an equality of subalgebras of $\End_{O_F}(\BX_3)$,
\[
   O_F[x] = O_F[\fkc(x)].
\]
\end{lemma}

\begin{proof}
Let $y := 1 - x$.  Then $y$ is an automorphism of $\BX_3$ by Lemma~\ref{lem cayley U1}.  It follows from the Cayley--Hamilton theorem that $y^{-1}$ is expressible as a polynomial with coefficients in $O_F$ in $y$, and hence in $x$.  Hence  $\fkc(x)=(1+x)(1-x)^{-1}$ is a polynomial in $x$.  Conversely, the same argument, using the inverse formula \eqref{cayley inverse}, shows that $x$ is a polynomial in $\fkc(x)$.
\end{proof}

\begin{lemma}\label{cayley rs iff rs}
Let $x \in \fkk_1$ and $\xi = \diag(\pm 1_2, \pm 1)$.  Then $x$ is regular semi-simple if and only if $\fkc_\xi(x)$ is.
\end{lemma}

\begin{proof}
Use the linear algebra characterization of regular semi-simple elements, as in the proof of Lemma~\ref{rs for reduced u1}, twice: first to deduce the lemma in the case $\xi = \id_{\BX_3}$ from Lemma~\ref{endom subalg =}, and then to see that $\fkc(x)$ is regular semi-simple if and only if $\fkc_\xi(x) = \xi \fkc(x)$ is (which is a simple exercise).
\end{proof}

\subsection{Relation to intersection numbers}\label{rel to int numbers}

We now apply the material in the previous subsections to intersection numbers.  In this subsection $\Delta = \Delta_\CN(\CN_2) \subset \CN_2 \times_{\Spf O_{\breve F}} \CN_3$ (not to be confused with the discriminant!).  We begin with a basic lemma on the geometry of intersections.  For any quasi-endomorphism $x \in \End_{O_F}^\circ(\BX_3)$, recall the subspace $\Delta_x$ of $\CN_2 \times_{\Spf O_{\breve F}} \CN_3$ defined by the condition \eqref{Deltag}.

\begin{lemma}\label{intersection nonempty conds}
For $x \in \End_{O_F}^\circ(\BX_3)$, the following are equivalent.  
\begin{altenumerate}
\item $\Delta \cap \Delta_x$ is nonempty.
\item $(\Delta \cap \Delta_x)_\red$ consists of two points.
\item $x \in \End_{O_F}(\BX_3)$.
\end{altenumerate}
\end{lemma}

\begin{proof}
By Lemma \ref{two points}, $\CN_2(\ov k)$ consists of two points, the standard basepoint
\[
   z_+ := (\BX_2,\iota_{\BX_2},\lambda_{\BX_2},\id_{\BX_2})
\]
and another point $z_-$.  Thus $\Delta_\red$ consists of the two points $(z_+, \delta_\CN(z_+))$ and $(z_-, \delta_\CN(z_-))$ in $\CN_2 \times \CN_3$.  According to the definitions, $\delta_\CN(z_+) = (\BX_3,\iota_{\BX_3},\lambda_{\BX_3},\id_{\BX_3})$.  Thus what we have to show is that $x$ either does or does not give an (honest) endomorphism of the framed $p$-divisible groups corresponding to $\delta_\CN(z_+)$ and $\delta_\CN(z_-)$ simultaneously.  For this we translate the problem to Dieudonn\'e modules.  Let $N_1$, $N_2$, and $N_3$ denote the respective rational covariant Dieudonn\'e modules of $\ov\BE$, $\BX_2$, and $\BX_3$.  We must show that the lattices in $N_3$ corresponding to $\delta_\CN(z_+)$ and $\delta_\CN(z_-)$ either are or are not simultaneously carried into themselves under the endomorphism of $N_3$ corresponding to $x$.  For $i = 1,2,3$, let $\tau_i$ be the $\tau$-operator on $N_i$ defined in \eqref{tau}.  Then
\[
   N_3 = N_2 \oplus N_1 \quad\text{and}\quad \tau_3 = \tau_2 \oplus \tau_1.
\]
Since $N_1$ is $1$-dimensional over $\breve F$, the Dieudonn\'e lattice in $N_1$ corresponding to $\ov\BE$ is $\tau_1$-stable.  By Remark~\ref{tau interchange}, $\tau_2$ interchanges the lattices in $N_2$ corresponding to $z_+$ and $z_-$.  Therefore $\tau_3$ interchanges the lattices in $N_3$ corresponding to $\delta_\CN(z_+)$ and $\delta_\CN(z_-)$.  Since the endomorphism of $N_3$ induced by $x$ commutes with $\tau_3$, the lemma follows.
\end{proof}

Thus nonzero intersection numbers can only occur for $g \in K_1$ in the group setting, and for $x \in \fkk_1$ in the Lie algebra setting.

\begin{lemma}\label{intersections =}
For $x\in \fkk_1$ and $\xi = \diag(\pm 1_2,\pm1)$, there are equalities of closed formal subschemes of $\CN_2 \times_{\Spf O_{\breve F}} \CN_3$,
\[
   \Delta\cap\Delta_x = \Delta \cap \Delta_{x_\red} = \Delta \cap \Delta_{\fkc_\xi(x)}.
\]
\end{lemma}

\begin{proof}
Let $(Y,\iota,\lambda,\rho)$ be a point on $\CN_2$.  It has to be shown that the endomorphism $x$ of $\BX_3$ lifts to an endomorphism of $Y \times \ov\CE$ (via the framing $\rho \times \rho_{\ov\CE}$) if and only if the endomorphism $x_\red$ lifts if and only if the endomorphism $\fkc_\xi(x)$ lifts.  Writing $x$ in terms of the usual coordinates \eqref{fkk_1 coords}, the first two of these conditions are equivalent because the endomorphism
\[
   \iota_{\BX_2}(\beta \pi) =
	\begin{bmatrix}
		0  &  \beta\varpi\\
		\beta  &  0
	\end{bmatrix}
	\quad
	(\beta \in O_{F_0})
\]
of $\BX_2$ and the endomorphism $d \in O_F$ of $\ov\BE$ automatically lift.
Furthermore, since $\xi$ obviously lifts, the equivalence of the first and third conditions is an immediate consequence of Lemma~\ref{endom subalg =}.
\end{proof}

Combined with Lemma \ref{intersection nonempty conds}, the following proves Theorem \ref{main prop}. 

\begin{proposition}\label{char nondeg}
The following three properties of $x\in \frak u_1(F_0)$ are equivalent.
\begin{altenumerate}
\item\label{integral rs} $x \in \fkk_{1,\rs}$.
\item\label{intersection proper} $\Delta \cap \Delta_x$ is a nonempty scheme, proper over $\Spec O_{\breve F}$.
\item\label{intersection artin} $\Delta \cap \Delta_x$ is artinian with two points.
\end{altenumerate}
The analog where $x\in \fku_1(F_0)$ is replaced by $g\in U_1(F_0)$ and $\fkk_{1,\rs}$ is replaced by $K_{1,\rs}$ is also true. Furthermore, in the group case, under these conditions, there are no higher Tor terms in the expression \eqref{Int(g) inhomog} defining $\Int(g)$.
\end{proposition}

\begin{proof}
Lemma \ref{intersection nonempty conds} immediately gives the equivalence of \eqref{intersection proper} and \eqref{intersection artin} in both the Lie algebra and group cases; and in the proof of the rest of the proposition, it also allows us to assume that $x \in \fkk_1$ in the Lie algebra case (resp.~$g \in K_1$ in the group case) and that the intersection in question is nonempty.  By Lemmas~\ref{rs for reduced u1}, \ref{lem cayley U1}, \ref{cayley rs iff rs}, and~\ref{intersections =}, the equivalence of \eqref{integral rs} and \eqref{intersection proper} in both the Lie algebra and group cases follows from their equivalence in the case that $x$ is a reduced element in the Lie algebra.  When $x \in \fkk_{1,\red}$ is regular semi-simple, we will show in \S\ref{calc of l-Int} that $\Delta \cap \Delta_x$ is an artinian scheme by explicitly computing its (finite) length.  Thus \eqref{integral rs} implies \eqref{intersection proper}.

To complete the proof of the equivalence of the three properties, we will show that if $x \in \fkk_{1,\red}$ is not regular semi-simple, then the intersection $\Delta \cap \Delta_x$ is not a scheme.  Write $x$ in the form \eqref{coordofred}, with $\alpha \in O_D^{\tr = 0}$ and $b \in O_D$.  Then $x$ is not regular semi-simple (if and) only if the elements $\alpha b$ and $b$ are linearly dependent over $F$ (which as before we take to act on $D$ on the right).

First suppose that $b \neq 0$.  Then $b^{-1}\alpha b \in F$ inside $D$.  Or in other words, $\alpha$ is in the image of the conjugate embedding $\iota := \tensor[^b]{\iota}{_\BE}$, in the notation of \eqref{conjugate emb}.
Since $F/F_0$ is ramified, 
$\iota$ makes $\BE$ into a formal $O_F$-module of height $1$, and we denote by $\CE_\iota$ the corresponding canonical lift of $\BE$ over $\Spf O_{\breve F}$.
Via the Serre construction, $O_F \otimes_{O_{F_0}} \CE_\iota$ gives a $\Spf O_{\breve F}$-point on $\CN_2$.  Since $\alpha$ lifts to an endomorphism of $\CE_\iota$, $\diag(\alpha,\alpha)$ lifts to an endomorphism of $O_F \otimes_{O_{F_0}} \CE_\iota$.  Since $b \in O_D$ conjugates $\iota_{\BE}$ into $\iota$, it lifts to a homomorphism $\CE \to \CE_\iota$.  Hence $\bigl[\begin{smallmatrix} b\pi\\ b \end{smallmatrix}\bigr]$ lifts to a homomorphism $\ov\CE \to O_F \otimes_{O_{F_0}} \CE_\iota$.  Similarly, since $\ov b = b^{-1} \RN b$ conjugates $\iota$ into $\iota_\BE$, $\begin{bmatrix} \pi\ov b  &  \ov b \varpi\end{bmatrix}$ lifts to a homomorphism $O_F \otimes_{O_{F_0}} \CE_\iota \to \ov\CE$.  This shows that $x$ lifts to an endomorphism of $(O_F \otimes_{O_{F_0}} \CE_\iota) \times \ov\CE$.  Thus we've constructed a $\Spf O_{\breve F}$-point on $\Delta \cap \Delta_x$, which shows that $\Delta \cap \Delta_x$ cannot be a scheme.

The case that $b = 0$ is even simpler: let $\CE_0$ be any formal $O_{F_0}$-module over $\Spf O_{\breve F}$ which lifts $\BE$ and for which the endomorphism $\alpha$ lifts.  Then as before $x$ lifts to an endomorphism of $(O_F \otimes_{O_{F_0}} \CE_0) \times \ov\CE$, so that we again obtain a $\Spf O_{\breve F}$-point on $\Delta \cap \Delta_x$.

Now let us prove the final assertion, i.e.~that for $g \in U_{1,\rs}(F_0)$ we have
\begin{equation}\label{nohigher}
   \Int(g) = \length(\Delta\cap \Delta_g).
\end{equation}
We follow \cite[Lem.~4.1]{Terstiege-HZ} and \cite[Prop.~11.6]{KR-U2}. Let $R$ be the local ring at a point  $x \in \Delta \cap \Delta_g$ of $\CN_2 \times \CN_3$. Then $\Delta$ is defined in $x$  by the ideal generated by a regular sequence $f_1, f_2$ of $R$. Hence the Koszul complex $K(f_1, f_2)$ is a free resolution of the $R$-module $\CO_{\Delta, x}$, and the complex $K(f_1, f_2)\otimes_R\CO_{\Delta_g, x}$ represents $(\CO_\Delta\otimes^\BL\CO_{\Delta_g})_x$. But 
\[
   K(f_1, f_2)\otimes_R\CO_{\Delta_g} = K\bigl(\ov f_1, \ov f_2\bigr) ,
\]
where $\ov f_i$ denotes the image of $f_i$ in $\CO_{\Delta_g}$, and where on the right-hand side appears the Koszul complex as $\CO_{\Delta_g, x}$-module. Since $\Delta$ and $\Delta_g$ intersect properly, $\ov f_1, \ov f_2$ forms a regular sequence in $\CO_{\Delta_g, x}$ which generates the ideal of $\Delta\cap{\Delta_g}$, we see that $ K(\ov f_1, \ov f_2)$ is a free resolution of $\CO_{\Delta\cap{\Delta_g}, x}$. Hence $(\CO_\Delta\otimes^\BL\CO_{\Delta_g})_x$ is represented by $\CO_{\Delta\cap{\Delta_g}, x}$. The asserted equality \eqref{nohigher} follows. 
\end{proof}

\begin{corollary}\label{inttoliealg}
For $x\in \fkk_{1,\rs}$ and $\xi = \diag(\pm 1_2,\pm1)$,
\[
  \lInt(x) = \lInt(x_\red) = \length\bigl(\Delta\cap\Delta_{\fkc_\xi(x)}\bigr) = \Int\bigr(\fkc_\xi(x)\bigl) .
\]
\end{corollary}

\begin{proof}
Lemma~\ref{intersections =} gives the first two equalities, and the vanishing of higher Tor terms asserted in Proposition~\ref{char nondeg} gives the last one.
\end{proof}

\section{Explicit calculations for the Lie algebra}\label{sec expl Lie}

By the results of the previous section, the calculation of intersection numbers in the situations of interest to us reduces to the calculation of $\lInt(x)$ for a reduced, regular semi-simple element $x \in \fkk_1$.  
In this section we effect this calculation.

\subsection{Keating invariants}
To begin, we briefly recall the theorem of Keating \cite[Th.~2.1]{V1} in the case that $F/F_0$ is ramified. Fix any $F_0$-embedding of $F$ into $D$, and let $\psi\in O_D$.  Let $\dist_{j}(\psi)$ be the ``distance'' of $\psi$ to the order $O_j$ of conductor $j$ in $F$, i.e.
\[
   \dist_{j}(\psi) := \max \bigl\{\, v_D(x+\psi)\bigm| x\in O_j\, \bigr\}.
\] 
Equivalently, $\dist_j(\psi)$ is the positive integer $\ell$ such that
\[
   \psi \in \bigl(O_j+\pi^{\ell}O_D\bigr) \smallsetminus \bigl( O_j+\pi^{1+\ell}O_D\bigr).
\]
(Recall that we use the uniformizer $\pi$ of $F$ as the uniformizer of $O_D$.) We may also describe the distance as the minimum 
\begin{equation}\label{distance fmla}
   \dist_j(\psi) = \min\bigl\{ \ell(\psi_-), \ell_j(\psi_+)\bigr\} ,
\end{equation}
where $\psi_+$ and $\psi_-$ are the components of $\psi$ with respect to the decomposition \eqref{Deigen}, and
\[
   \ell(\psi_-) = v_D(\psi_-)
   \quad\text{and}\quad
   \ell_j(\psi_+) = 
      \begin{cases}
         v_D\bigl(\Im(\psi_+)\bigr),  &  v_D\bigl(\Im(\psi_+)\bigr) < 2j;\\
         +\infty,  &  v_D\bigl(\Im(\psi_+)\bigr) \geq 2j .
      \end{cases}
\]
Here $\Im(\psi_+) = \bigl(\psi_+-\ov \psi_+\bigr)/2 \in F^{\tr=0}$ is the imaginary part; note that $v_D(\Im(\psi_+))$ is always odd. We adopt the usual convention that $v_D(0)=+\infty$.

\begin{proposition}[Keating]\label{keating}
Assume that $F/F_0$ is ramified.
For $j\geq 0$, let $\ell={\rm dist}_j(\psi)$ be defined as above. Then the length $n_j(\psi)$ of the locus inside the quasi-canonical divisor $\CW_{F, j}$ where $\psi$ lifts to an endomorphism of the corresponding quasi-canonical lifting is given by
\[
   n_{j}(\psi) = 
   \begin{dcases}
	  2\sum_{i=0}^{\ell/2} q^i - q^{\ell/2},  &  \ell\leq 2j \text{ is even;}\\
      2\sum_{i=0}^{(\ell-1)/2} q^i = 2\frac{q^{(\ell+1)/2} - 1}{q-1},  &  \ell\leq 2j \text{ is odd;}\\
      2\sum_{i=0}^{j-1} q^i + (\ell-2j+1)q^j,  &  \ell > 2j .
   \end{dcases}
\]
\end{proposition}

We refer to the third alternative in the statement as the \emph{stable range for $j$ relative to  $\ell$}, and the first two alternatives as the \emph{unstable range for $j$ relative to $\ell$}.

\subsection{Calculation of $\lInt(x)$ for $x\in\fkk_{1,\red}$}\label{calc of l-Int}
Now let us return to our convention that $F$ is embedded in $D$ via $\iota_{\ov\BE}$, as in \S\ref{reduction to LA}.  Let $x \in \fkk_{1,\red}$ be regular semi-simple.
By definition,
\begin{equation}\label{lInt n=3}
   \lInt(x)={\rm length}\bigl(\text{\emph{locus in $\CN_2$ where $x$ lifts to an endomorphism of $\CY \times \ov\CE$}} \bigr),
\end{equation}
where $\CY$ denotes the universal $p$-divisible group over $\CN_2$. We are going to obtain an explicit expression for this length by pulling the calculation back to \CM via the Serre map (as in Proposition~\ref{Serre isom}) and using Keating's theorem.  Of course, to do so we have to account for the fact that $\CN_2$ has two connected components, only one of which is identified with \CM under the Serre map.  As in the proof of Lemma~\ref{intersection nonempty conds}, let $z_\pm$ denote the two points in $\CN_2(\ov k)$, with $z_+$ the standard basepoint.  Write $\lInt_\pm(x)$ for the length of the locus occurring in \eqref{lInt n=3} supported at $z_\pm$, so that $\lInt(x) = \lInt_+(x) + \lInt_-(x)$.  If $g \in H_1(F_0) = U(\BX_2)$ interchanges $z_+$ and $z_-$, then via the inclusion $H_1(F_0) \subset U_1(F_0)$,
\begin{equation}\label{lInt 2 terms}
   \lInt(x) = \lInt_+(x) + \lInt_+(gxg^{-1}).
\end{equation}

We first consider the term $\lInt_+(x)$ in \eqref{lInt 2 terms}.  Write
\[
   x = 
   \begin{bmatrix}
      \alpha  &  0  &  b\pi\\
	  0  &  \alpha  &  b\\
	  \pi\ov b  &  \varpi \ov b  &  0
   \end{bmatrix},
   \quad \alpha\in O_D^{\tr = 0}, \quad b\in O_D,
\]
in the coordinates \eqref{fkk_1 coords}.  
Recall from \S\ref{U_1 and fku_1 coords} that the matrix entries are with respect to the $O_{F_0}$-linear decomposition of the framing object
\[
   \BX_3 
	   = \BX_2 \times \ov\BE 
		= (O_F \otimes_{O_{F_0}} \BE) \times \ov\BE 
		= (1 \otimes \BE + \pi \otimes \BE) \times \ov\BE 
		\simeq \BE \times \BE \times \ov\BE.
\]
By Lemma~\ref{rs crit for reduced}, since $x$ is regular semi-simple, we have $b \neq 0$ and $\alpha'_- \neq 0$, where $\alpha' = b^{-1}\alpha b$ and the minus denotes the component of $\alpha'$ with respect to the decomposition \eqref{Deigen} of $D$.  Now, inside \CM is the special divisor $\CT(\ov b)$ where $\ov b$ lifts to a homomorphism $\CY_0 \to \ov\CE$, where we recall that $\CY_0$ denotes the universal formal $O_{F_0}$-module over \CM, cf.~\S\ref{special divisors}.  Since $\lambda_\BE$ lifts to the principal polarizations $\lambda_{\CY_0}$ of $\CY_0$ and $\lambda_{\ov\CE}$ of $\ov\CE$, and since the Rosati involution on $D$ is the main involution, this is the same as the locus where $b$ lifts to a homomorphism $\ov\CE \to \CY_0$.  Over the connected component $\CN_{2,+} \subset \CN_2$, the Serre map identifies $\CY$ with $O_F \otimes_{O_{F_0}} \CY_0 = 1 \otimes \CY_0 + \pi \otimes \CY_0$.  Since $\pi \in O_D$ of course lifts to an endomorphism of $\ov\CE$, we conclude that $\CT(\ov b)$ identifies with the locus in $\CN_{2,+}$ where $\bigl[\begin{smallmatrix} b\pi \\ b \end{smallmatrix}\bigr]$ lifts to a homomorphism $\ov\CE \to \CY$ and $\begin{bmatrix} \pi \ov b  &  \varpi \ov b \end{bmatrix}$ lifts to a homomorphism $\CY \to \ov\CE$; and we further conclude that the locus in $\CT(\ov b)$ where $\alpha$ lifts to an endomorphism of $\CY_0$ identifies with the locus in $\CN_{2,+}$ where $x$ lifts to an endomorphism of $\CY \times \ov\CE$.  By Proposition~\ref{sums of qcan divisors prop}, we can write the divisor $\CT(\ov b)$ as a sum of quasi-canonical divisors
\[
  \CT\bigl(\ov b\bigr) = \sum_{0\leq j\leq v_D(b)} \CW_{\tensor*[^b]{F}{}, j},
\]
where we recall that $\tensor*[^b]{F}{}$ denotes the image of $F$ in $D$ under the conjugate embedding $\tensor[^b]{\iota}{_\BE}$.  We obtain from Proposition~\ref{keating}
\[
   \lInt_+(x) = \sum_{0\leq j\leq v_D(b)} n_j (\alpha') ,
\]
where the length $n_j(\alpha')$ depends, via Keating's formula, on the distance of $\alpha'$ to the original order in $F$ (not to the conjugate order!). 

Now consider the term $\lInt_+(gxg^{-1})$ in \eqref{lInt 2 terms}. In terms of the coordinates \eqref{U_1 coords} for $U_1(F_0)$, let us explicitly take
\[
   g :=
	\begin{bmatrix}
		0  &  \pi  &  0\\
	   \pi^{-1}  &  0  &  0\\
		0  &  0  &  1
	\end{bmatrix},
\]
which indeed has nontrivial Kottwitz invariant. Then
\[
   gxg^{-1} =
	\begin{bmatrix}
		\pi \alpha \pi^{-1}  &  0  &  \pi b\\
		0  &  \pi^{-1} \alpha \pi  &  \pi^{-1} b \pi\\
		\varpi \ov b \pi^{-1}  &  \pi \ov b \pi  &  0
	\end{bmatrix}
	=
	\begin{bmatrix}
		\tensor*[^\pi]{\alpha}{}  &  0  &  \tensor*[^\pi]{b}{} \pi\\
		0  &  \tensor*[^\pi]{\alpha}{}  &  \tensor*[^\pi]{b}{}\\
		\pi \ov{\tensor*[^\pi]{b}{}}  &  \varpi \ov{\tensor*[^\pi]{b}{}}  &  0
	\end{bmatrix},
\]
where the superscript $\pi$ denotes conjugation of the base by $\pi$ (which is the same as conjugation by $\pi^{-1}$, since $\pi$ is traceless).  Thus to compute $\lInt_+(gxg^{-1})$, we can run through exactly the same analysis as above, with $\tensor*[^\pi]{b}{}$ in place of $b$ and $\tensor*[^\pi]{\alpha}{}$ in place of $\alpha$.  Since the elements $\tensor*[^\pi]{b}{^{-1}} \tensor*[^\pi]{\alpha}{} \tensor*[^\pi]{b}{} = \tensor*[^\pi]{(\alpha')}{}$ and $\alpha'$ have the same distance to $O_j$, we conclude that $\Int_+(gxg^{-1}) = \lInt_+(x)$. Hence
\[
   \lInt(x) = 2 \lInt_+(x).
\]

We now explicitly calculate this length via Keating's formula.  Since $\tr (\alpha) = 0$, we also have $\tr (\alpha') = \tr (\alpha'_\pm) = 0$. Therefore the formula \eqref{distance fmla} for the distance gives
\[
   \dist_j(\alpha') = 
   \begin{cases}
	   \min \bigl\{v_D(\alpha'_-), v_D(\alpha'_+)\bigr\},  &  v_D(\alpha'_+) < 2j;\\
       v_D(\alpha'_-),  &  v_D(\alpha'_+) \geq 2j .
   \end{cases}
\]
To lighten notation, set
\begin{equation*}
   \ell_- := v_D(\alpha'_-),\quad \ell_+ := v_D(\alpha'_+),\quad\text{and}\quad m := v_D(b) . 
\end{equation*}
Note that $\ell_+ = v_D(\alpha'_+)$ is odd, since $\alpha'_+ \in F$ is purely imaginary. Also note that these quantities depend only on the invariants $u(x)$, $w(x)$, and $\Delta(x)$ (which are given explicitly in \eqref{u w lambda} and \eqref{Deltaonu1}), via
\begin{equation}\label{remember valu}
   v_D(u) = 2m,\quad v_D(w) = 2m+\ell_+,\quad\text{and}\quad v_D(\Delta)=4m+2\ell_- .
\end{equation}
We will make use of the following ancillary calculation at several points below: for any $r \geq 0$, since
\[
   \sum_{j=0}^r jq^j
      = \frac{rq^{r+1}}{q-1} - \frac{q^{r+1}-q}{(q-1)^2}
	  = \frac{rq^{r+2}-(r+1)q^{r+1}+q}{(q-1)^2},
\]
we have
{\allowdisplaybreaks
\begin{equation}\label{ancillary}
\begin{aligned}
	&\sum_{j=0}^{r}\biggl( 2\frac{q^{j}-1}{q-1} + (\ell_- - 2j + 1)q^j\biggr)\\
	   &\qquad = 2\frac{\frac{q^{r+1}-1}{q-1}-(r+1)}{q-1} + (\ell_- + 1)\frac{q^{r+1} - 1}{q - 1}
		                - 2 \frac{rq^{r+2}-(r+1)q^{r+1}+q}{(q-1)^2}\\
		&\qquad = q^{r+1}\frac{2 + (\ell_- + 1)(q-1) - 2(rq - r - 1)}{(q-1)^2} - \frac{\ell_- + 1 + 2(r+1)}{q-1} - \frac{2q + 2}{(q-1)^2}\\
		&\qquad = q^{r+1} \frac{2(q+1) + (\ell_- -2r - 1)(q-1)}{(q-1)^2}  - \frac{\ell_- + 2r + 1}{q-1} - \frac{4q}{(q-1)^2}.
\end{aligned}
\end{equation}}
Finally we set
\[
   t := q^{-1}.
\]

For the main calculation we now distinguish cases.

\smallskip
\paragraph{\textbf{Case I}: $\ell_-\leq\ell_+$.}
 
Then for any $j$, 
\[
   \dist_{j}(\alpha') = \ell_-.
\]
We divide this case further into three subcases.
\begin{altenumerate}
\renewcommand{\theenumi}{\arabic{enumi}}
\item\label{ell_- > 2m}
$\ell_->2m$. Then all $j$ with $0\leq j\leq m$ are in the stable range for $\dist_{j}(\alpha')$. Hence
\[
   \lInt_+(x) = \sum_{j=0}^{m}\biggl( 2\frac{q^{j}-1}{q-1} + (\ell_- - 2j + 1)q^j\biggr).
\]
Taking $r = m$ in \eqref{ancillary}, replacing $q$ with $t^{-1}$, and multiplying by $2$, we obtain 
\[
   \lInt(x) = 2t^{-m} \frac{2(1+t)+(\ell_--2m-1)(1-t)}{(1-t)^2} 
	            - \frac{2(\ell_-+2m+1)t}{1-t} - \frac{8t}{(1-t)^2} .
\]
\item\label{ell_- =< 2m and odd}  $\ell_-\leq 2m$  and $\ell_-$  odd.  In this case, there are some $j$ in the stable range and some in the unstable range. We get 
\begin{equation}\label{Int(x) Case I(2)}
   \lInt_+(x) = \sum_{j=0}^{(\ell_- - 1)/2} \biggl( 2\frac{q^{j}-1}{q-1} + (\ell_- - 2j + 1)q^j\biggr)+ \sum_{j = (\ell_-+1)/2}^{m} 2\frac{q^{(\ell_- + 1)/2} - 1}{q-1}.
\end{equation}
Multiplying by $2$ and using \eqref{ancillary} with $r = (\ell_- - 1)/2$, we get
\[
\begin{aligned}
	\lInt(x) 
	   &= 2q^{(\ell_- + 1)/2}\frac{2(q+1)}{(q-1)^2} - 2\frac{2\ell_-}{q-1} - 2\frac{4q}{(q-1)^2} + 2\biggl(m-\frac{\ell_-+1}{2}+1\biggr) \frac{2(q^{(\ell_-+1)/2}-1)}{q-1}\\
	   &= 2t^{-(\ell_--1)/2}\frac{\bigl(2m-\ell_-+3)-(2m-\ell_--1)t}{(1-t)^2} 
             - \frac{2(\ell_-+2m+1)t}{1-t} - \frac{8t}{(1-t)^2} .
\end{aligned}
\]
\item  $\ell_-\leq 2m$ and $\ell_-$  even. Again, there are stable and unstable $j$. Similarly to the previous subcase, we get
{\allowdisplaybreaks
\[
\begin{aligned}
   \lInt(x) 
	   &= 2\sum_{j=0}^{\ell_-/2-1} \biggl( 2\frac{q^{j}-1}{q-1} + (\ell_-+1-2j)q^j\biggr)
             + 2\sum_{j=\ell_-/2}^{m} \biggl(2\frac{q^{\ell_-/2+1}-1}{q-1} - q^{\ell_-/2}\biggr)\\
		&= 2q^{\ell_-/2} \frac{2(q+1) + q-1}{(q-1)^2} 
		        - 2 \frac{2\ell_- - 1}{q-1} - 2\frac{4q}{(q-1)^2} 
				  + 2\biggl(m - \frac{\ell_-}2 + 1 \biggr) \frac{q^{\ell_-/2}(q+1)-2}{q-1}\\
		&= 2t^{-\ell_-/2}\frac{(m-\ell_-/2+1)(1-t^2)+t(t+3)}{(1-t)^2} 
		     - \frac{2(\ell_-+2m+1)t}{1-t} - \frac{8t}{(1-t)^2}. 
\end{aligned}
\]
}
\end{altenumerate}

\smallskip

\paragraph{\textbf{Case II}: $\ell_->\ell_+$.}  In this case we have for the distance 
\[
   \dist_j(\alpha') =
   \begin{cases}
      \ell_+,  &  \ell_+<2j;\\
      \ell_-,  &  \ell_+\geq 2j.
   \end{cases}
\]
We consider the following two subcases.
\begin{altenumerate}
\renewcommand{\theenumi}{\arabic{enumi}}
\item $\ell_+\geq 2m.$   Then for all $j$ with $0\leq j\leq m$, we have $\dist_j(\alpha') = \ell_- > 2m$. Hence all $j$ lie in the stable range, and we get 
the same answer as in Case I\eqref{ell_- > 2m},
\[
   \lInt(x) = 2t^{-m} \frac{2(1+t)+(\ell_--2m-1)(1-t)}{(1-t)^2} 
	            - \frac{2(\ell_-+2m+1)t}{1-t} - \frac{8t}{(1-t)^2} .
\]

\item $\ell_+<2m$.  In this subcase, the relevant $j$ can be in the stable range as well as in the unstable range. Note that $\ell_+$ is odd. Hence we get, similarly to Case I\eqref{ell_- =< 2m and odd},
{\allowdisplaybreaks
\[
\begin{aligned}
	\lInt(x) 
	   &= 2\sum_{j=0}^{(\ell_+-1)/2}\biggl( 2\frac{q^{j}-1}{q-1} + (\ell_-+1-2j)q^j\biggr)
            + 2\sum_{j=(\ell_++1)/2}^{m} 2\frac{q^{(\ell_++1)/2}-1}{q-1}\\
		&= 2q^{(\ell_+ + 1)/2}\frac{2(q+1) + (\ell_- - \ell_+)(q-1)}{(q-1)^2}\\
		&\phantom{=}\qquad - 2\frac{\ell_- + \ell_+}{q-1} - 2\frac{4q}{(q-1)^2} + 2\biggl(m-\frac{\ell_+ + 1}{2}+1\biggr) \frac{2(q^{(\ell_+ + 1)/2}-1)}{q-1}\\
		&= 2t^{-(\ell_+-1)/2}\frac{(\ell_- - 2\ell_++2m+3)(1-t)+4t}{(1-t)^2} 
                  - \frac{2(\ell_- + 2m+1)t}{1-t} - \frac{8t}{(1-t)^2}. 
\end{aligned}
\]}
\end{altenumerate}

\part{Analytic side}\label{analtyic side part}
In this part of the paper we turn to the analytic side of the identities to be proved in Theorems \ref{thm group} and \ref{thm lie}, and, modulo the material in Part \ref{germ expansion part} on germ expansions of orbital integrals, we complete the proofs of these theorems.

\section{Inputs from local harmonic analysis}\label{inputs section}

In this section we formulate some basic facts about harmonic analysis on the spaces in play in the Lie algebra and group settings. Except where noted to the contrary, we allow $n$ to be arbitrary.

\subsection{Lie algebra setting}\label{lie setting}
Let 
\begin{equation}\label{mapfks}
   \pi_\fks\colon \fks\to \frak b
\end{equation}
be the categorical quotient of $\fks$ by $H'$ as discussed in \S\ref{ss: Lie invars}, say by taking either set of invariants \eqref{inv gen} or \eqref{inv sec}.  Thus $\frak b$ is an affine space over $F_0$ of dimension $2n-1$ (given explicitly in \eqref{cat quot isom} in the case of the invariants \eqref{inv sec}).  Let $\fkb_\rs$ be the image of $\fks_\rs$ in $\fkb$. Then  
\[
   \fkb_\rs = \bigl\{\, x\in \fkb \bigm| \Delta(x)\neq 0 \,\bigr\},
\]
where $\Delta$ denotes the discriminant \eqref{Delta rescaled}. Since this is a global function on $\fks$ which is $H'$-invariant, it descends to a global function on $\fkb$.

For $\phi' \in C_c^\infty(\fks)$, we note that the function $y\mapsto \omega(y)\Orb(y,\phi')$ descends to a function $\varphi$ on $\fkb_\rs(F_0)$. Let $C^\infty_\rc(\fkb_\rs)$ denote the space of locally constant functions on $\fkb_\rs(F_0)$ whose support has compact closure in $\fkb(F_0)$ (functions with \emph{relatively compact support}). By \cite[Lem.~3.12]{Z14} the function $\varphi$ lies in $C^\infty_\rc(\fkb_\rs)$. By a slight variant of \cite[Prop.~3.8]{Z14}, we have the following.
\begin{theorem}\label{locoforblie}
Let $\varphi$ be a function in $C_\rc^\infty(\fkb_{\rs})$.  The following properties are equivalent.
\begin{altenumerate}
\item\label{orb int fcn}
There exists a function $\phi'\in C^\infty_c(\fks)$ such that
\[
   \varphi\bigl(\pi_\fks(y)\bigr) = \omega(y)\Orb(y,\phi')
	\quad\text{for all}\quad 
	y \in \fks_\rs(F_0) .
\]
\item\label{loc orb int fcn}
For every $x_0\in \frak b(F_0)$, there exists an open neighborhood $V_{x_0}$ of $x_0$ and a function $\phi_{x_0}'\in C^\infty_c(\fks)$  such that
\begin{flalign*}
   \varphi\bigl(\pi_\fks(y)\bigr) = \omega(y)\Orb(y,\phi_{x_0}')
	\quad\text{for all}\quad
	y \in \pi_\fks^{-1}\bigl(V_{x_0} \cap \fkb_\rs(F_0)\bigr) = \pi_\fks^{-1}(V_{x_0}) \cap \fks_\rs(F_0).
\end{flalign*}
\qed
\end{altenumerate}
\end{theorem}

A function $\varphi\in C_\rc^\infty(\fkb_{\rs})$ satisfying property \eqref{orb int fcn} is called an \emph{orbital integral function}; a function $\varphi$ on $\fkb_\rs(F_0)$ satisfying property \eqref{loc orb int fcn}  is called a \emph{local orbital integral function} \cite[Def.~3.7]{Z14}. More precisely, if $\varphi$ satisfies  property \eqref{loc orb int fcn} locally around $x_0$, then $\varphi$ will be called an orbital integral function \emph{locally around $x_0$}.

%
%

We note that the map $\pi_\fks$ in \eqref{mapfks} induces a surjection on $F_0$-rational points, and we have a decomposition 
\begin{equation}\label{declie}
   \fkb_\rs(F_0) = \fkb_{\rs, 0}\amalg\fkb_{\rs, 1}
\end{equation}
into a disjoint union of two open (for the $p$-adic topology) subsets. Here $\fkb_{\rs,i}$ is the image under $\pi_\fks$ of the set $\fks_{\rs,i}$ of elements in $\fks_\rs(F_0)$ which match with elements in $\fku_{i,\rs}(F_0)$.

This decomposition can also be explained by a similar picture on the unitary  side.  Taking the same invariants as used in realizing the quotient map \eqref{mapfks}, by \S\ref{ss: Lie invars} we obtain for $i \in \{0,1\}$ a categorical quotient map
\[
   \pi_{\fku_i}\colon \fku_i\to \fkb.
\]
Then $\fkb_{\rs,i}$ is the image under $\pi_{\fku_i}$ of $\fku_{i,\rs}(F_0)$.  To be quite clear,
the maps $\pi_{\fku_0}$ and $\pi_{\fku_1}$ do \emph{not} induce surjections on $F_0$-rational points, cf.~the remark after Lem.~3.1 in \cite{Z14}.

\begin{proposition}
\label{eta b01}
Let $n=3$. For $i\in\{0,1\}$, an element $x \in \fkb_\rs(F_0)$ lies in $\fkb_{\rs, i}$ if and only if $\eta(-\Delta(x))=(-1)^{i}$.
\end{proposition}
\begin{proof}This follows from Lemma \ref{lem delta=det}, noting that we rescaled $\Delta(x)$ by the factor $-\varpi^{-2}$ in \eqref{Delta rescaled}.  (One can also use \eqref{Deltaonu1} for $i=1$, taking note that $\eta(-\RN \alpha'_-)=-1$, and the explicit coordinates given in \eqref{coord u0} below for $i=0$.)
\end{proof}

Of course, Lemma \ref{lem delta=det} works perfectly well to distinguish between the two summands in \eqref{declie} for any $n$; we have stated the proposition for $n = 3$ only because we are now working with the rescaled version of $\Delta$.

We conclude the subsection by noting the following.

\begin{lemma}\label{del rc lie} 
Let $\phi'\in C^\infty_c(\frak s)$ be a function with transfer $(\phi, 0) \in C_c^\infty(\fku_0) \times C_c^\infty(\fku_1)$.
Then the function
\[
   y \mapsto
	\begin{cases}
		\omega(y) \del(y,\phi'),  &  y\in \fks_{\rs,1};\\ 
      0,  &  y \in \fks_{\rs,0}
   \end{cases}
\]
descends to a function on $\fkb_\rs(F_0)$ which lies in $C_\rc^\infty(\fkb_\rs)$.
\end{lemma}

\begin{proof}
By Remark \ref{rem pro-u}\eqref{DOrb invar} (or rather, its Lie algebra analog), the function  descends to a function $\varphi$ on $\fkb_\rs(F_0)$. The map $\pi_\fks\colon\fks(F_0)\to \fkb(F_0)$ is continuous and hence sends the support of $\phi'$ (a compact set) to a compact set in $\fkb(F_0)$.  The support of $\varphi$ lies in the image of the support of $\phi'$ under $\pi_\fks$ and is therefore relatively compact. The local constancy of $\varphi$ follows from the same argument as in \cite[Lem.~3.12]{Z14} (which is about the case $y\mapsto \omega(y)\Orb(y,\phi')$).
\end{proof}

\subsection{Group setting}\label{group setting}
We now translate the contents of \S\ref{lie setting} to the group setting.  Let $B$ denote the categorical quotient of $S$ by $H'$, and $B_{U_i}$ the categorical quotient of $U_i$ by $H_i$ for $i = 0,1$. These are affine varieties with rings of global functions given by the ring of invariants. We denote by
\[
   \pi_S\colon S\to B 
	\quad\text{and}\quad
	\pi_{U_i}\colon U_i\to B_{U_i}
\]
the corresponding quotient morphisms, and by $B_\rs$, resp.~$B_{U_i, \rs}$, the images of $S_\rs$, resp.~$U_{i, \rs}$, under these morphisms. All of these are open subschemes defined by the non-vanishing of the discriminant function (which by equivariance drops to the categorical quotients).  On the level of $F_0$-rational points, we write $B_\rs(F_0)$ and $B(F_0)_\rs$, resp.~$B_{U_i,\rs}(F_0)$ and $B_{U_i}(F_0)_\rs$, interchangeably.

We are going to see that, in analogy with the Lie algebra case, the quotients $B$, $B_{U_0}$, and $B_{U_1}$ can all be naturally identified.  To facilitate the precise statement and proof of this result, we introduce the Cayley transform on each of our Lie algebra spaces (cf.~\S\ref{cayley u_1} for the case of $\fku_1$ when $n=3$).  Let\footnote{Here by $\det$ we mean the usual $n \times n$ determinant $\Res_{F/F_0} \M_n \to \Res_{F/F_0}\BA$, relative to any choice of basis for the hermitian space $W_i$ in the case of $\fku_i$. This definition of $\fku_1^\circ$ is consistent with the notation \eqref{u1circ} for $\fku_1^\circ(F_0)$, but note that the matrix representations used in \S\ref{reduction to LA} involve entries in $D$, and therefore care must be taken to correctly interpret their determinant. An analogous remark applies in the definition of $U_{1,\xi}$ below.}
\begin{equation}\label{s^circ u_i^circ}
   \fks^\circ := \bigl\{\, y\in \fks \bigm| \det(1-y)\neq 0 \,\bigr\}
	\quad\text{and}\quad
	\fku_i^\circ := \bigl\{\, x\in \fku_i \bigm| \det(1-x)\neq 0 \,\bigr\} ,\quad  i=0 \text{ or } 1.
\end{equation}
Then under the respective quotient maps $\pi_\fks$, $\pi_{\fku_0}$, and $\pi_{\fku_1}$, these open subschemes descend to a common open subscheme $\fkb^\circ$ of $\fkb$.
Next recall from the Introduction the norm $1$ subgroup $F^1=\{\xi\in F\mid \RN\xi=1\}$, and define
\[
   S^1 := \bigl\{\, \diag( \xi_1 \cdot 1_{n-1}, \xi_2) \bigm| \xi_1,\xi_2 \in F^1  \,\bigr\} \subset S(F_0).
\]
Then $S^1$ canonically identifies with a subgroup in both $U_0(F_0)$ and $U_1(F_0)$ (upon choosing any special embeddings of $U_0$ and $U_1$ as in \S\ref{homog case}).  

For $\xi \in S^1$, we define the Cayley transform for $\fku_0$ and $\fku_1$ via the same formula as in \eqref{cayley fmla},
\begin{equation}\label{cayley u_i}
	\begin{gathered}
   \fkc_\xi\colon 
	\xymatrix@R=0ex{
	   \fku_i^\circ \ar[r]  &  U_i,\\
	   x \ar@{|->}[r]  &  \xi \dfrac{1+x}{1-x}
	}
	\quad
	i = 0 \text{ or } 1.
	\end{gathered}
\end{equation}
Then $\fkc_\xi$ is $H_i$-equivariant, and, abusing notation, we continue to denote by $\fkc_\xi$ the induced map on the quotients
\[
   \fkc_\xi\colon \fkb^\circ \to B_{U_i},
	\quad
	i = 0 \text{ or } 1.
\]

\begin{remark}
This definition of $\fkc_\xi$ on $\fku_1^\circ$ generalizes the definition in \S\ref{cayley u_1} when $n = 3$ and $\xi = \diag(\pm 1_2, \pm 1)$.  But note that for more general $\xi$ there is a small subtlety between the matrix notation we are currently using and the coordinates in \S\ref{reduction to LA}: the element $\diag(\xi_1,\xi_1,\xi_2) \in S^1$ is expressed as
\[
   \begin{bmatrix}
		\alpha  &  \beta \varpi  &  0\\
		\beta  &  \alpha  &  0\\
		0  &  0  &  \xi_2
	\end{bmatrix}
	\in U_1(F_0)
\]
in the coordinates \eqref{U_1 coords}, where $\xi_1 = \alpha + \beta\pi$ with $\alpha,\beta \in F_0$.
\end{remark}

To define the Cayley transform on $\fks^\circ$, note that the formula in \eqref{cayley u_i} only gives a map into $S$ when $\xi$ is of the form $\xi_1 \cdot 1_n$.  To define $\fkc_\xi$ for an arbitrary $\xi = \diag(\xi_1\cdot 1_{n-1}, \xi_2) \in S^1$, first choose $\nu_1,\nu_2 \in F^\times$ such that $\ov \nu_1/\nu_1 = \xi_1$ and $\ov \nu_2/\nu_2 = \xi_2$, and set
\[
   \nu := \diag(\nu_1 \cdot 1_{n-1}, \nu_2).
\]
Then we define
\begin{equation}\label{cayley s}
	\begin{gathered}
   \fkc_\xi\colon 
	\xymatrix@R=0ex{
	   \fks^\circ \ar[r]  &  S\\
	   y \ar@{|->}[r]  &  \ov \nu \cdot \dfrac{1+y}{1-y} \cdot \nu^{-1}.
	}
	\end{gathered}
\end{equation}
As before, $\fkc_\xi$ is $H'$-equivariant.  Note that this definition depends on the choice of $\nu_1$ and $\nu_2$, but the induced map on the quotients (which we again abusively denote by $\fkc_\xi$)
\[
   \fkc_\xi\colon \fkb^\circ \to B
\]
depends only on $\xi$.  When $\xi$ is a scalar matrix, by convention we always take $\nu_1 = \nu_2$; then the Cayley transform into $S$ is given by the usual formula $y \mapsto \xi (1+y)(1-y)^{-1}$.

\begin{lemma}[{\cite[Lem.~3.4]{Z14}}]\label{cayley cover lem}
\begin{altenumerate}
For $\xi\in S^1$, define the open subschemes $S_\xi^\circ \subset S$ and $U_{i, \xi}^\circ \subset U_i$ by
\[
   S_\xi^\circ :=\bigl\{\, \gamma\in S \bigm| \det(\xi + \gamma)\neq 0  \,\bigr\}
	\quad\text{and}\quad 
	U_{i, \xi}^\circ := \bigl\{\, g\in U_i \bigm| \det (\xi + g)\neq 0  \,\bigr\},
	\quad 
	i = 0 \text{ or } 1. 
\]
\item\label{cayley cov lem S part}
The Cayley transform $\fkc_{\xi}$ induces  an $H'$-equivariant isomorphism  
\[
   \fks^\circ \isoarrow S_\xi^\circ . 
\]
Furthermore, let $\xi_1,\xi_2,\dotsc,\xi_{n+1}$ be $n+1$ distinct elements of $F^1$. Then,  as $j$ varies,  the open subschemes $S_{\xi_j \cdot 1_n}^\circ$ cover $S$.
\item
Analogous statement for $\fku_0$ and $U_0$ (and $H_0$-equivariance) in place of $\fks$ and $S$.
\item
Analogous statement for $\fku_1$ and $U_1$ (and $H_1$-equivariance) in place of $\fks$ and $S$.\qed
\end{altenumerate}
\end{lemma} 

The following generalizes Lemma \ref{cayley rs iff rs}.

\begin{lemma}
For any $\xi \in S^1$, an element $y$ in any of $\fks^\circ(F_0)$, $\fku_0^\circ(F_0)$, or $\fku_1^\circ(F_0)$ is regular semi-simple if and only if $\fkc_\xi(y)$ is.
\end{lemma}

\begin{proof}
We show that the sets of vectors $\{y^ie\}_{i=0}^{n-1}$ and $\{\tensor*[^t] e {} y^i\}_{i=0}^{n-1}$ are linearly independent if and only if $\{\fkc_\xi(y)^ie\}_{i=0}^{n-1}$ and $\{\tensor*[^t] e {} \fkc_\xi(y)^i\}_{i=0}^{n-1}$ are; see \S\ref{ss:LAchar}.  Set $\fkc := \fkc_{1_n}$.  First note that the same argument as in the proof of Lemma \ref{endom subalg =} shows that there is an equality of $F$-algebras $F[y] = F[\fkc(y)]$, which proves the desired equivalence when $\xi = 1_n$.

For an arbitrary $\xi = \diag(\xi_1 \cdot 1_{n-1}, \xi_2)$, we next claim that $\{\fkc_\xi(y)^ie\}_{i=0}^{n-1}$ is linearly independent if and only if $\{\fkc(y)^ie\}_{i=0}^{n-1}$ is.  To show this we make a little calculation which will also be useful later. Let $\zeta := \diag(1_{n-1}, \xi_2/\xi_1)$. By an easy induction argument, for any $A \in \M_n(F)$, there is an equality of matrices
\[
   \begin{bmatrix}
		e  &  \zeta A e  &  \dots  &  (\zeta A)^{n-1} e
	\end{bmatrix}
	=
	\begin{bmatrix}
		e  &  A e  &  \dots  &  A^{n-1} e
	\end{bmatrix}
	\begin{bmatrix}
		1 &  &  
		        \hspace{-3.5ex}\raisebox{0ex}[0ex][0ex]{\raisebox{-1.7ex}{\huge $*$}}\\
		  &  \ddots\\
		  &  &  1
	\end{bmatrix}
\]
for some upper triangular unipotent matrix on the right.  Hence for the Cayley transform on $\fku_0^\circ$ and $\fku_1^\circ$,
\begin{equation}\label{cayley det calc u_i}
	\begin{aligned}
   \det
   \begin{bmatrix}
		e  &  \fkc_\xi(y) e  &  \dots  &  \bigr(\fkc_\xi(y)\bigr)^{n-1} e
	\end{bmatrix}
	   &= \det
		   \begin{bmatrix}
				e  &  \xi \fkc(y) e  &  \dots  &  \bigr(\xi\fkc(y)\bigr)^{n-1} e
			\end{bmatrix}\\
	   &= \xi_1^{n(n-1)/2} \det
		   \begin{bmatrix}
				e  &  \zeta\fkc(y) e  &  \dots  &  \bigr(\zeta\fkc(y)\bigr)^{n-1} e
			\end{bmatrix}\\
		& = \xi_1^{n(n-1)/2} \det
		   \begin{bmatrix}
				e  &  \fkc(y) e  &  \dots  &  \fkc(y)^{n-1} e
			\end{bmatrix}.
	\end{aligned}
\end{equation}
For the Cayley transform on $\fks^\circ$, taking $\nu = \diag(\nu_1 \cdot 1_{n-1}, \nu_2)$ as in the definition of $\fkc_\xi$, we have
\[
   \nu^{-1} \cdot 
	\begin{bmatrix} 
		e  &  \fkc_\xi(y) e  &  \dots  &  \bigr(\fkc_\xi(y)\bigr)^{n-1} e
	\end{bmatrix}
	=
	\begin{bmatrix}
		\nu_2^{-1} e  &  \xi \fkc(y) \nu_2^{-1} e  &  \dots  &  \bigr(\xi\fkc(y)\bigr)^{n-1} \nu_2^{-1} e
	\end{bmatrix}.
\]
Hence, calculating as in \eqref{cayley det calc u_i},
\begin{equation}\label{cayley det calc s}
\begin{aligned}
   &\det
	\begin{bmatrix} 
		e  &  \fkc_\xi(y) e  &  \dots  &  \bigr(\fkc_\xi(y)\bigr)^{n-1} e
	\end{bmatrix}\\
	&\qquad = \det (\nu) \cdot \nu_2^{-n} \cdot \xi_1^{n(n-1)/2} \det
	   \begin{bmatrix}
			e  &  \fkc(y) e  &  \dots  &  \fkc(y)^{n-1} e
		\end{bmatrix}\\
	&\qquad = (\nu_1/\nu_2)^{n-1} \xi_1^{n(n-1)/2} \det
	   \begin{bmatrix}
			e  &  \fkc(y) e  &  \dots  &  \fkc(y)^{n-1} e
		\end{bmatrix}.
\end{aligned}
\end{equation}
The equalities \eqref{cayley det calc u_i} and \eqref{cayley det calc s} prove the claim in all cases.  A similar calculation shows that $\{\tensor*[^t] e {} \fkc_\xi(y)^i\}_{i=0}^{n-1}$ is linearly independent if and only if $\{\tensor*[^t] e {} \fkc(y)^i\}_{i=0}^{n-1}$ is, which completes the proof.
\end{proof}

For $\xi\in S^1$, let $B_\xi^\circ$ denote the image of $S_\xi^\circ$ in $B$, and let $B_{U_i, \xi}^\circ$ denote the image of $U_{i, \xi}^\circ$ in $B_{U_i}$. Then the Cayley transforms drop to isomorphisms
\[
   \fkb^\circ\isoarrow B_\xi^\circ
	\quad\text{and}\quad \fkb^\circ\isoarrow B_{U_i, \xi}^\circ, \quad i=0\text{ or } 1 . 
\]
Thus we obtain isomorphisms
\begin{equation}\label{Bgluing}
   \varphi_\xi\colon B_\xi^\circ\isoarrow B_{U_i, \xi}^\circ, \quad i=0 \text{ or } 1 .
\end{equation}

\begin{lemma}\label{B isom}
Let $i=0$ or $i=1$. 
There is a unique isomorphism $B\isoarrow B_{U_i}$ which induces for each $\xi\in S^1$ the isomorphism \eqref{Bgluing}. This isomorphism induces an identification of open subschemes $B_\rs \isoarrow B_{U_i, \rs}$. 
\end{lemma}

\begin{proof} What has to be seen is that for any $\xi, \eta\in S^1$,  the isomorphisms $\varphi_\xi$ and $\varphi_\eta$ coincide on the intersection $B_\xi^\circ \cap B_\eta^\circ$.  Now, after base extension from $F_0$ to $F$, there are standard isomorphisms of algebraic varieties
\[
   S\otimes_{F_0}F \cong U_i\otimes_{F_0}F \cong \GL_{n,F}
	\quad\text{and}\quad
	\fks \otimes_{F_0} F \cong \fku_i \otimes_{F_0} F \cong \M_{n,F}.
\]
The latter identification is compatible with the quotient maps $\pi_\fks$ and $\pi_{\fku_i}$ to $\fkb \otimes_{F_0} F$, and it also identifies $\fks^\circ \otimes_{F_0} F \cong \fku_i^\circ \otimes_{F_0} F$.  Similarly, $H'\otimes_{F_0}F \cong H_i\otimes_{F_0}F \cong \GL_{n-1,F}$.  Since formation of the categorical quotient commutes with flat base change, the first isomorphism in the display induces an isomorphism of algebraic varieties over $F$,
$$
B\otimes_{F_0}F \cong B_{U_i}\otimes_{F_0}F .
$$
Under the above identifications, the Cayley transforms $\fkc_\xi$ on $\fks^\circ \otimes_{F_0} F$ and on $\fku_i^\circ \otimes_{F_0} F$ need not coincide (indeed the Cayley transform for $\fks$ is not even well-defined in terms of $\xi$), but one readily checks that they are $\GL_{n-1}(F)$-conjugate. In other words, under these identifications, the base change to $F$ of the isomorphism $\varphi_\xi$ becomes simply the identity morphism. Now the assertion is obvious.
\end{proof}

Via the lemma, we regard $B$ as the common categorical quotient of $S$ by $H'$, of $U_0$ by $H_0$, and of $U_1$ by $H_1$.  Analogously to \eqref{declie}, we obtain a disjoint union decomposition
\begin{equation}\label{B_rs(F_0) decomp}
   B_\rs(F_0)=B_{\rs, 0}\amalg B_{\rs, 1}  ,
\end{equation}
where $B_{\rs, i}$ is the image under $\pi_{U_i}$ of $U_{i, \rs}(F_0)$. Equivalently, $B_{\rs, i}$ is the image under $\pi_S$ of the set $S_{ \rs, i}$ of elements in $S_\rs(F_0)$ which match with elements in $U_{i, \rs}(F_0)$. 
We have
\[
   S_\rs = \pi_S^{-1}(B_\rs), 
	\quad 
	S_{\rs, 0} = \pi_S^{-1}(B_{\rs, 0}), 
	\quad\text{and}\quad 
	S_{\rs, 1} = \pi_S^{-1}(B_{\rs, 1}) .
\]
Of course, setting
\[
   \fkb_\rs^\circ := \fkb_\rs \cap \fkb^\circ,
	\quad
	\fkb_{\rs,0}^\circ := \fkb^\circ(F_0) \cap \fkb_{\rs,0},
	\quad\text{and}\quad
	\fkb_{\rs,1}^\circ := \fkb^\circ(F_0) \cap \fkb_{\rs,1},
\]
the decomposition
\[
   \fkb_\rs^\circ(F_0) = \fkb_{\rs,0}^\circ \amalg \fkb_{\rs,1}^\circ
\]
is compatible with the decomposition \eqref{B_rs(F_0) decomp} under the Cayley transform $\fkc_\xi$ for any $\xi \in S^1$.

For $f' \in C_c^\infty(S)$, we note that  the function $\gamma\mapsto \omega(\gamma)\Orb(\gamma,f')$ descends to a function $\varphi$ on $B_\rs(F_0)$. Just as in the Lie algebra case, $\varphi$ is locally constant with relatively compact support on $B_\rs(F_0)$. We denote the space of such functions by  $C_\rc^\infty(B_\rs)$. 
By Prop.~3.8 and the remarks before Lem.~3.6 in \cite{Z14}, we have the following.

\begin{theorem}\label{local pro of orb fun}
Let $\varphi$ be a function in $C_\rc^\infty(B_\rs)$. The following properties are equivalent.
\begin{altenumerate}
\item	
There exists a function $f'\in C^\infty_c(S)$  such that 
\[
   \varphi\bigl(\pi_S(\gamma)\bigr) = \omega(\gamma)\Orb(\gamma,f')
	\quad\text{for all}\quad
	\gamma\in S_{\rs}(F_0) .
\] 
\item\label{loc orb int fcn B}
For every point in $B(F_0)$, there exists an open neighborhood $V$ of the point and a function $f'\in C^\infty_c(S)$ (both depending on the point) such that 
\begin{flalign*}
  \varphi\bigl(\pi_S(\gamma)\bigr) = \omega(\gamma)\Orb(\gamma,f')
  \quad\text{for all}\quad
  \gamma\in \pi_S^{-1}\bigl(V\cap  B_\rs(F_0)\bigr) = \pi_S^{-1}(V) \cap S_\rs(F_0).
\end{flalign*}
\qed
\end{altenumerate}
\end{theorem}

As in the Lie algebra setting, we call a function $\varphi \in C_{\rc}^\infty(B_\rs)$ satisfying \eqref{loc orb int fcn B} a \emph{local orbital integral function}, and if $\varphi$ satisfies \eqref{loc orb int fcn B} locally around a given point, then we call $\varphi$ an orbital integral function \emph{locally around} the point.

Similarly  to Lemma \ref{del rc lie}, we have the following.

\begin{lemma}\label{del rc group}
Let $f'\in C^\infty_c(S)$ be a function with transfer $(f, 0) \in C_c^\infty(U_0) \times C_c^\infty(U_1)$.
Then the function
\[
   \gamma\mapsto
	\begin{cases}
		\omega(\gamma) \del(\gamma,f'),  &  \gamma\in S_{\rs,1};\\ 
		0,  &  \gamma \in S_{\rs,0}
   \end{cases}
\]
descends to a function on $B_\rs(F_0)$ which lies in $C_\rc^\infty(B_\rs)$.
\end{lemma}

\begin{proof}
The same argument as in the proof Lemma \ref{del rc lie} shows that the function descends to a function with relatively compact support. The local constancy of the descended function reduces to Lemma \ref{del rc lie} by a partition of unity argument (cf.~the proof of \cite[Lem.~3.6]{Z14}).
\end{proof}

\subsection{Reduced Lie algebra setting}\label{subsred}
In this subsection, we formulate a variant of the Lie algebra version which eliminates ``trivial'' factors from $\frak s$. 

For $y$ a point on $\fks$, write $y$ in the block form
\[
   y =
	\begin{bmatrix}
		A  &  \mathbf{b} \\
      \mathbf{c}  &  d
   \end{bmatrix}
   \in \Res_{F/F_0}\M_n,
\]
as in \eqref{x block decomp}.  In analogy with Definition \ref{fku_1 reduced}, let $\fks_\red$ denote the closed subscheme of \emph{reduced} points in $\fks$, defined by 
\[
   \fks_\red := \{\, y\in \fks \mid \tr A=0 \text{ and } d=0 \,\} . 
\]
Then $\fks_\red$ is an $H'$-invariant subscheme of $\fks$.  Define 
$\fkb_\red$ to be the product of the middle $2n-3$ factors in the target in \eqref{cat quot isom},
\[
   \fkb_\red := 
	   \overbrace{\BA \times \fks_1 \times \dotsb}^{\substack{n-2\\ \text{alternating}\\ \text{factors}}} \times
		\overbrace{\BA \times \fks_1 \times \dotsb}^{\substack{n-1\\ \text{alternating}\\ \text{factors}}} \, .
\]
Then the composite map
\[
   \pi_\red\colon
	\xymatrix@R=0ex{
	    \fks_\red\, \ar@{^{(}->}[r]  &  \fks \ar[r]^{\pi'_\fks}   &  \fkb_\red\\
		   &  y \ar@{|->}[r]   &  (\tr \wedge^2 A, \dotsc, \tr \wedge^{n-1} A, \mathbf{c}\mathbf{b}, \dotsc, \mathbf{c} A^{n-2} \mathbf{b})
	}
\]
is a categorical quotient for $\fks_\red$ by $H'$.  In terms of this notation, we also realize the quotient map \eqref{mapfks} for $\fks$ by taking $\fkb = \fkb_\red \times \fks_1 \times \fks_1$ and
\[
   \pi_\fks\colon
	\xymatrix@R=0ex{
	   \fks \ar[r]  &  \fkb_\red \times \fks_1 \times \fks_1\\
		y \ar@{|->}[r]  &  \bigl(\pi'_\fks(y), \tr A, d\bigr)
	}
\]
(of course this is nothing but a reordering of the factors in \eqref{cat quot isom}).

There is a natural $H'$-equivariant map $\fks \to \fks_\red$, $y \mapsto y_\red$, sending
\[
   \begin{bmatrix}
		A  &  \mathbf{b}\\
		\mathbf{c}  &  d
	\end{bmatrix}
	\mapsto
	\begin{bmatrix}
		A - \frac{\tr A}{n-1} \cdot 1_{n-1}  &  \mathbf{b}\\
		\mathbf{c}  & 0
	\end{bmatrix}.
\]
This induces an evident $H'$-equivariant product decomposition $\fks \cong \fks_\red \times \fks_1 \times \fks_1$, where the map onto the last two factors is given by taking $\tr A$ and $d$.  We obtain a commutative diagram
\begin{equation}\label{s prod decomp}
	\begin{gathered}
   \xymatrix{ 
      \fks \ar[r]^-\sim \ar[d]_-{\pi_\fks}  &  \fks_\red\times \fks_1 \times \fks_1 \ar[d]^-{\pi_\red \times \id \times \id}\\
      \fkb \ar@{=}[r]  &  \fkb_\red\times \fks_1 \times \fks_1.
	}
	\end{gathered}
\end{equation}
We also denote by $x \mapsto x_\red$ the natural projection $\fkb \to \fkb_\red$, and we regard $\fkb_\red$ as a closed subscheme of $\fkb$ via the $0$ section.

The following extends both the statement and proof of Lemma \ref{rs for reduced u1}
to the case of $\fks$.

\begin{lemma}\label{rs for reduced s}
An element $y \in \fks(F_0)$ is regular semi-simple if and only if $y_\red$ is.
\end{lemma}

\begin{proof}
By an easy induction argument, there is an equality of matrices
\begin{equation}\label{y y_red matrix eqn}
   \begin{bmatrix}
		e  &  y_\red e  &  \dots  &  y_\red^{n-1} e
	\end{bmatrix}
	=
	\begin{bmatrix}
		e  &  y e  &  \dots  &  y^{n-1} e
	\end{bmatrix}
	\begin{bmatrix}
		1 &  &  
		        \hspace{-3.5ex}\raisebox{0ex}[0ex][0ex]{\raisebox{-1.7ex}{\huge $*$}}\\
		  &  \ddots\\
		  &  &  1
	\end{bmatrix}
\end{equation}
for some upper triangular unipotent matrix on the right.  Similarly, the matrices
\[
   \begin{bmatrix}
		\tensor[^t]{e}{} \\
		\tensor[^t]{e}{} y_\red\\
		\vdots\\
		\tensor[^t]{e}{} y_\red^{n-1}
	\end{bmatrix}
	\quad\text{and}\quad
   \begin{bmatrix}
		\tensor[^t]{e}{} \\
		\tensor[^t]{e}{} y\\
		\vdots\\
		\tensor[^t]{e}{} y^{n-1}
	\end{bmatrix}
\]
differ by left multiplication by a lower triangular unipotent matrix.  The conclusion now follows from the linear algebra characterization of regular semi-simple elements in \S\ref{ss:LAchar}.
\end{proof}

All concepts introduced in \S\ref{lie setting} in the Lie algebra context have obvious analogs in the ``reduced'' setting. The analog of Theorem \ref{locoforblie} for the reduced set is the following. The proof is essentially the same.  Set
\[
   \fks_{\red,\rs} := \fks_\red \cap \fks_\rs
	\quad\text{and}\quad
	\fkb_{\red,\rs} := \fkb_\red \cap \fkb_\rs.
\]

\begin{theorem}\label{thm stab const red}
Let $\varphi$ be a function in $C_\rc^\infty(\frak b_{{\rm red, rs}})$.  The following properties are equivalent.
\begin{altenumerate}
\item
There exists a function $\phi'\in C^\infty_c(\fks_\red)$  such that 
\[
   \varphi\bigl(\pi_\red(y)\bigr) = \omega(y)\Orb(y,\phi')
	\quad\text{for all}\quad
	 y\in \fks_{\red, \rs}(F_0).
\]
\item
For every $x_0\in \fkb_\red(F_0)$, there exists an open neighborhood $V_{x_0}$ of $x_0$ and a function $\phi_{x_0}'\in C^\infty_c(\fks_\red)$ such that 
\[
   \varphi\bigl(\pi_\red(y)\bigr) = \omega(y)\Orb(y,\phi_{x_0}')
	\quad\text{for all}\quad
	y \in \pi_\red^{-1}\bigl(V_{x_0} \cap \fkb_{\red,\rs}(F_0)\bigr) = \pi_\red^{-1}(V_{x_0}) \cap \fks_{\red,\rs}(F_0).
\]
\qed
\end{altenumerate}
\end{theorem}

\begin{remark}\label{last one}
A useful class of orbital integral functions  in $C_\rc^\infty(\frak b_{{\rm red, rs}})$ is provided by Corollary \ref{cor germ converse}. Similar results hold for $\frak b_{{\rm rs}}$ and $B_{{\rm rs}}$.
\end{remark}

\section{Reduction to the Lie algebra}

The main aim of this section is to show that Theorem \ref{thm group} (the main group theorem) follows from Theorem \ref{thm lie} (the main Lie algebra theorem).  Except where noted to the contrary, from now on we specialize to the case $n = 3$.

\subsection{Renormalized invariants on $\fks_3$}\label{subsec inv s} 
We first fix a slight renormalization of the categorical quotient map $\pi_\fks$ in \S\ref{subsred} when $n = 3$, in analogy with the renormalization of the invariants on $\fku_1$ given in \eqref{fku_1 invariants}.  As in \S\ref{subsred}, we take
\[
   \fkb = \BA \times \BA \times \fks_1 \times \fks_1 \times \fks_1
	\quad\text{and}\quad
	\fkb_\red = \BA \times \BA \times \fks_1
\]
over $F_0$.  Then we realize the quotient maps $\pi_\fks$ and $\pi_\red$ as
\[
   \begin{gathered}
	\pi_\fks\colon
	\xymatrix@R=0ex{
	   \fks \ar[r]  &  \fkb\\
		y \ar@{|->}[r]  &  \bigl(\lambda(y), u(y), w(y), \tr A, d\bigr)
	}
	\end{gathered}
	\qquad\text{and}\qquad
   \begin{gathered}
	\pi_\red\colon
	\xymatrix@R=0ex{
	   \fks_\red \ar[r]  &  \fkb_\red\\
		y \ar@{|->}[r]  &  \bigl(\lambda(y), u(y), w(y) \bigr),
	}
	\end{gathered}
\]
where we write the point $y$ in the usual block form
\[
   y =
	\begin{bmatrix}
		A  &  \mathbf{b} \\
      \mathbf{c}  &  d
   \end{bmatrix}
   \in \Res_{F/F_0}\M_3,
\]
and where
\begin{equation}\label{u,w,lambda}
   \lambda(y) := \det A,
	\quad
	u(y) := \varpi^{-1} \mathbf{c} \mathbf{b},
	\quad\text{and}\quad
	w(y) := \varpi^{-1} \mathbf{c} A \mathbf{b}.
\end{equation}

Of course, the invariants used in this definition of $\pi_\fks$, regarded as defined on $\Res_{F/F_0} \M_3$ in the obvious way, give exactly the invariants \eqref{fku_1 invariants} on $\fku_1$, relative to any special embedding of $\fku_1$ in the sense of \S\ref{Lie algebra setting}.  From now on we realize the quotient map $\pi_{\fku_1} \to \fkb$ via these invariants.

\subsection{Coordinates on $U_0$ and $\fku_0$}\label{subsec inv u0}

In parallel with \S\ref{U_1 and fku_1 coords}, we now describe the unitary group $U_0$ and its Lie algebra $\fku_0$ in terms of explicit coordinates.  Define the $F/F_0$-hermitian matrices
\[
   J_0^\flat :=
	\begin{bmatrix}
		0  &  \pi\\
	   -\pi  &  0
	\end{bmatrix}
	\in \M_2(F)
	\quad\text{and}\quad
	J_0 := 
	\begin{bmatrix}
		J_0^\flat  &  0\\
		0  &  1
	\end{bmatrix}
	\in \M_3(F).
\]
Then $J_0^\flat$ and $J_0$ determine non-degenerate split $F/F_0$-hermitian spaces of dimensions $2$ and $3$, respectively.  We take
\[
   H_0 = \U(J_0^\flat)
	\quad\text{and}\quad
	U_0 = \U(J_0).
\]
Explicitly,
\[
   U_0(F_0) = \bigl\{\, g \in \M_3(F) \bigm| gg^\dag = 1 \,\bigr\},
\]
where the adjoint $g^\dag = J_0^{-1} \tensor*[^t]{\ov g}{} J_0$ is given in coordinates by
\begin{equation}\label{U_0 dagger}
   \begin{bmatrix}
		a_1  &  a_2  &  b_1\\
		a_3  &  a_4  &  b_2\\
		c_1  &  c_2  &  d
	\end{bmatrix}^\dag
	=
	\begin{bmatrix}
		\ov a_4  &  -\ov a_2  &  -\pi^{-1} \ov c_2\\
		-\ov a_3  &  \ov a_1  &  \pi^{-1} \ov c_1\\
		-\pi\ov b_2  &  \pi\ov b_1  &  \ov d
	\end{bmatrix}.
\end{equation}
Of course, $H_0$ is described similarly in terms of $J_0^\flat$, and it embeds in $U_0$, via the rule $h \mapsto \diag(h,1)$, as the stabilizer of the special vector $e = (0,0,1)$.  We also note that under these coordinates, the tautological embedding of $\SL_2$ into $\Res_{F/F_0} \M_2$ identifies
\begin{equation}\label{SL2 = SU2}
	\SL_2 \isoarrow \SU(J_0^\flat).
\end{equation}
The Lie algebra $\fku_0$ is given in these coordinates by
\begin{align*}
   \fku_0(F_0)
	   &= \bigl\{\, x \in \M_3(F) \bigm| x + x^\dag = 0 \,\bigr\}\\
	   &= \left\{\left.
         \begin{bmatrix}
            a_1  &  a_2  &  b_1 \\
            a_3  &  - \ov a_1  &  b_2\\
		      \pi\ov b_2  &  -\pi\ov b_1  &  d
         \end{bmatrix}
			\,\right|\, 
			a_1,b_1,b_2 \in F,\ a_2,a_3 \in F_0,\ d\in F^{\tr = 0}\right\}.
\end{align*}

Recall that our formulation of the AT conjecture in \S\ref{conj and results} involved the choice of a $\pi$-modular lattice $\Lambda_0^\flat \subset W_0^\flat = F^2$.  We now take for $\Lambda_0^\flat$ the standard lattice $O_F^2 \subset F^2$, which is indeed $\pi$-modular for $J_0^\flat$.  As in \eqref{Lambda_0}, we then take $\Lambda_0 = O_F^3 \subset F^3$, and $K_0$ and $\fkk_0$ are the respective subgroups of $U_0(F_0)$ and $\fku_0(F_0)$ stabilizing $\Lambda_0$:
\[
   K_0 = U_0(F_0) \cap \M_3(O_F) = \bigl\{\, g \in \M_3(O_F) \bigm| gg^\dag = 1 \,\bigr\}
\]
and
\begin{align*}
   \fkk_0
	   &= \fku_0(F_0) \cap \M_3(O_F)\\
		&= \left\{\left.
         \begin{bmatrix}
            a_1  &  a_2  &  b_1 \\
            a_3  &  - \ov a_1  &  b_2\\
		      \ov b_2\pi  &  -\ov b_1\pi  &  d
         \end{bmatrix}
			\,\right|\, 
			a_1,b_1,b_2 \in O_F,\ a_2,a_3 \in O_{F_0},\ d\in O_F^{\tr = 0}\right\}.
\end{align*}
We also set
\[
   \fkk_{0,\rs} := \fkk_0 \cap \fku_{0,\rs}(F_0).
\]

Given a point $x$ in $\fku_0$, write $x$ in the block form
\[
   x =
	\begin{bmatrix}
		A  &  \mathbf{b}\\
		\mathbf{c}  &  d
	\end{bmatrix}.
\]
We realize the categorical quotient of $\fku_0$ by $H_0$ by taking the same invariants as in the previous subsection,
\begin{equation}\label{invaronu0}
\begin{aligned}
   \pi_{\fku_0}\colon 
	\xymatrix@R=0ex{
	   \fku_0 \ar[r]  &  \fkb = \BA \times \BA \times \fks_1 \times \fks_1 \times \fks_1\\
		x \ar@{|->}[r]  &  \bigl( \lambda(x),u(x), w(x), \tr A, d\bigr),
	}
	\end{aligned}
\end{equation}
where $\lambda$, $u$, and $w$ are as defined in \eqref{u,w,lambda}.

As in the cases of $\fku_1$ and \fks, we say that $x$ is \emph{reduced} if $\tr A = d = 0$.  We write $\fku_{0,\red}$ for the closed subscheme of reduced points in $\fku_0$.  In terms of explicit coordinates,
\begin{align}\label{coord u0}
   \fku_{0,\red}(F_0) =
	\left\{\left.
	\begin{bmatrix}
	   a_1  &  a_2  &  b_1 \\
	   a_3  &  - a_1  &  b_2\\
		\ov b_2\pi  &  -\ov b_1\pi  &  0
	\end{bmatrix}
	\,\right|\, 
	a_1,a_2,a_3 \in F_0,\ b_1,b_2 \in F \right\}.
\end{align}
As in previous cases, there is a natural map $\fku_0 \to \fku_{0,\red}$, $x \mapsto x_\red$, sending
\[
   \begin{bmatrix}
		A  &  \mathbf{b}\\
		\mathbf{c}  &  d
	\end{bmatrix}
	\mapsto
	\begin{bmatrix}
		A - \frac 1 2 (\tr A) \cdot 1_2  &  \mathbf{b}\\
		\mathbf{c}  & 0
	\end{bmatrix},
\]
and this gives rise to an $H_0$-equivariant product decomposition
\begin{equation}\label{u0 prod decomp}
	\begin{gathered}
   \xymatrix@R=0ex{
      \fku_0 \ar[r]^-\sim  &  \fku_{0,\red} \times \fks_1 \times \fks_1\\
   	x \ar@{|->}[r]  &  (x_\red, \tr A, d).
	}
	\end{gathered}
\end{equation}
Just as in Lemma \ref{rs for reduced u1} for $\fku_1$ and Lemma \ref{rs for reduced s} for $\fks$, an element $x \in \fku_0(F_0)$ is regular semi-simple if and only if $x_\red$ is.  We set
\[
   \fku_{0,\red,\rs} := \fku_{0,\red} \cap \fku_{0,\rs},
	\quad
	\fkk_{0,\red} := \fkk_0 \cap \fku_{0,\red}(F_0),
	\quad\text{and}\quad
	\fkk_{0,\red,\rs} :=	\fkk_{0,\red} \cap \fkk_{0,\rs}.
\]

\subsection{Integral Cayley transform on $\fku_0$}
In this subsection we prove an analog for $\fku_0$ of Lemma \ref{lem cayley U1}, which pertained to the Cayley transform on $\fku_1$.  Let $\xi \in S^1$.  Recall from \eqref{s^circ u_i^circ} the open subscheme $\fku_0^\circ \subset \fku_0$, which is the locus where the Cayley transform $\fkc_\xi$ is defined, and recall from Lemma \ref{cayley cover lem} its image $U_{0,\xi}^\circ \subset U_0$.  Define the sets of $F_0$-rational points
\[
   \fku_0^{\circ\circ} := \bigl\{\, x\in\fku_0^\circ(F_0) \bigm| \det(1-x) \in O_{F}^\times \,\bigr\}
	\quad\text{and}\quad
	U_{0,\xi}^{\circ\circ} := \bigl\{\, g\in U_{0,\xi}^\circ(F_0) \bigm| \det(\xi + g) \in O_{F}^\times \,\bigr\}.
\]
It is trivial to verify that $\fkc_\xi$ carries $\fku_0^{\circ\circ}$ isomorphically onto $U_{0,\xi}^{\circ\circ}$, and we then have the following.

\begin{lemma}[Cayley transform for $\fkk_0$]\label{lem cayley U}
For any $\xi \in S^1$, the restriction of the Cayley map to $\fkk_0\cap\fku_0^{\circ\circ}$ induces an isomorphism 
\[
   \fkc_{\xi}\colon
   \xymatrix@R=0ex{
      \fkk_0\cap\fku_0^{\circ\circ} \ar@{->}[r]^-\sim  &  K_0 \cap U_{0,\xi}^{\circ\circ}\\
      x \ar@{|->}[r]  &  \xi \dfrac{1+x}{1-x}.
   }
\]
Furthermore, the sets $K_0 \cap U_{0,\xi}^{\circ\circ}$, as $\xi$ varies over the four elements $\diag(\pm 1_2,\pm1)$, cover $K_0$.
\end{lemma}

\begin{proof} 
Clearly $\fkc_\xi( \fkk_0\cap\fku_0^{\circ\circ} ) \subset K_0 \cap U_{0,\xi}^{\circ\circ}$, and it is also clear from the inverse formula
\[
   \fkc_{\xi}^{-1}(g) = \frac{\xi^{-1}g-1}{\xi^{-1}g+1}
\]
that $\fkc_\xi^{-1}(K_0 \cap U_{0,\xi}^{\circ\circ}) \subset \fkk_0\cap\fku_0^{\circ\circ}$.  This proves the first assertion.  It remains to show that if $g \in K_0$, then $\det(\xi + g) \in O_F^\times$ for some $\xi = \diag(\pm 1_2,\pm 1)$.  In terms of the coordinates in the previous subsection, write $g$ in the block form
\[
   g =
	\begin{bmatrix}
		A  &  \mathbf{b}\\
		\mathbf{c}  &  d
	\end{bmatrix}.
\]
Since $g \in K_0 = U_0(F_0) \cap \GL_3(O_F)$, we may reduce the entries of $g$ mod $\pi$.
Since $g^\dag = g^{-1}$ also has integral entries, \eqref{U_0 dagger} shows that that $g_k$ is of the form
\[
   g_k =
	\begin{bmatrix}
		A_k  &  \mathbf{b}_k \\
		0  &  \pm 1
	\end{bmatrix},
\]
where the subscript $k$ everywhere denotes reduction mod $\pi$.  Since $K_0$ has symplectic reduction in the sense of Remark \ref{rem syp red}, $A_k \in \Sp_2(k) = \SL_2(k)$.  Hence $A_k + 1_2$ or $A_k - 1_2$ is invertible, since otherwise $A_k$ has eigenvalues $1$ and $-1$, contradicting $A_k \in \SL_2(k)$.  The lemma follows.
\end{proof} 

\begin{remark}\label{fku_1 circ circ}
 Note that $\fkk_0$ is not contained in $\fku_0^{\circ\circ}$, nor even in $\fku_0^\circ(F_0)$, i.e.~$\fkc_\xi$ is not defined on all of $\fkk_0$. This differs from the situation for $\fkk_1$. Indeed, defining $\fku_1^{\circ\circ} \subset \fku_1(F_0)$ in the obvious way (by\footnote{Meaning the determinant of the $F$-linear endomorphism $1-x$ acting on the hermitian space $W_1$; as before, care is required when working with the coordinates for $\fku_1$ in \S\ref{reduction to LA}.} $\det(1-x)\in O_F^\times$), the proof of Lemma \ref{lem cayley U1} shows that $\fkk_1\subset \fku_1^{\circ\circ}$. 
\end{remark}

\subsection{Integral points on $\fkb$}
We now collect some facts related to integral points on the quotient space $\fkb = \BA \times \BA \times \fks_1 \times \fks_1 \times \fks_1$. Note that this space is naturally defined over $O_{F_0}$, and
\[
   \fkb(O_{F_0}) = O_{F_0} \times O_{F_0} \times O_F^{\tr = 0} \times O_F^{\tr = 0} \times O_F^{\tr = 0}.
\]
We claim that 
\begin{equation}\label{fkb(O_F_0) union}
   \fkb(O_{F_0}) = \pi_{\fku_0} ( \fkk_0)\cup \pi_{\fku_1}  (\fkk_1).
\end{equation}
Indeed, the reverse inclusion is obvious from the explicit form of the invariants \eqref{invaronu0} and \eqref{fku_1 invariants} on $\fku_0$ and $\fku_1$, respectively.  For the forward inclusion, we give the following more precise lemma. Recall the decomposition $\fkb_\rs(F_0) = \fkb_{\rs,0} \amalg \fkb_{\rs,1}$ from \eqref{declie}, and for $i \in \{0,1\}$ set
\[
   \fkb(O_{F_0})_{\rs, i} := \fkb(O_{F_0})\cap\fkb_{\rs, i},
	\quad
	\fkb_{\red,\rs,i} := \fkb_\red(F_0) \cap \fkb_{\rs,i},
	\quad
	\fkb(O_{F_0})_{\red,\rs,i} :=  \fkb(O_{F_0}) \cap \fkb_{\red,\rs,i}.
\]

\begin{lemma}
\begin{altenumerate}
\item\label{integral b rs}
For $i\in \{ 0, 1\}$, 
\[
   \fkb(O_{F_0})_{\rs, i}=\pi_{\fku_i}(\fkk_{i, \rs}) . 
\]
\item\label{integral b non rs}
Let $x\in \fkb(O_{F_0})\smallsetminus \fkb_\rs(F_0)$. Then \smallskip
\begin{altenumerate2}
\makeatletter
\renewcommand{\p@enumii}{}
\makeatother
\item\label{nonexceptional}
$x\in \pi_{\fku_0}(\fkk_0)$.
\item\label{exceptional} 
$x\in \pi_{\fku_1}(\fkk_1)$ unless $ x_\red=( \lambda, u, w)$  with $-\lambda\in F_0^{\times, 2}$.\smallskip
\end{altenumerate2}
Furthermore, if $x \in \fkb(F_0)$ lies in the closure of $\fkb_{\rs, 1}$, then $x$ is not exceptional in the sense of \eqref{integral b non rs}\eqref{exceptional}.
\end{altenumerate}
\label{intimage}
\end{lemma}

\begin{proof} 
The statements are immediately reduced to the corresponding ones for the reduced sets.
 
To prove \eqref{integral b rs}, first let $x=(\lambda, u, w)\in\fkb(O_{F_0})_{\red,\rs, 1}$.  We will show the existence of an element in $\fkk_{1,\red,\rs}$ in terms of the explicit coordinates \eqref{coordofred} whose image is $x$.  We use the formulas \eqref{u w lambda} for the invariants. By Lemma \ref{rs crit for reduced}, $u \neq 0$ since $x \in \fkb_{\rs,1}$.  Choose any $b\in D$ such that $\RN b=u/2$. Then $b \in O_D$, since $p\neq 2$.  Let $\alpha'_+ := w/u$. We claim that $\alpha_+'$ is integral. Indeed, if instead $|w/u|>1$, then $|w^2|>|\lambda u^2|$, since $\lambda$ is integral.  Then by \eqref{Deltaonu1} and the fact that $w$ is traceless,
\[
   \eta(-\Delta) = \eta\bigl(-(\lambda u^2+w^2)\bigr) = \eta(-w^2) = 1,
\]
a contradiction to Proposition \ref{eta b01}. Finally choose $\alpha'_- \in D$ of norm $\Delta/u^2$.  Then the same argument shows that $\alpha_-'$ is integral, and \eqref{Deltaonu1} shows that this suffices to solve our problem.
A similar analysis, using the coordinates in \eqref{coord u0},  shows that $\fkb(O_{F_0})_{\red,\rs, 0}=\pi_{\fku_0}(\fkk_{0,\red, \rs})$. 

To prove \eqref{integral b non rs}, part \eqref{nonexceptional} follows from the explicit construction of elements in \S\ref{orbitsinfksminusnull} below. Part \eqref{exceptional} is straightforward, again using the explicit coordinates \eqref{coordofred} and the formulas for the invariants \eqref{u w lambda}.  
For example, the condition $-\lambda\notin F_0^{\times, 2}$ holds on all of $\pi_{\fku_1}(\fku_{1,\red}(F_0))$ by virtue of the facts $\lambda=\RN\alpha$ and $\alpha\in D^{\tr=0}$ 
(recall that $D$ is the quaternion division algebra).

Finally, suppose $x=(\lambda, u, w)\in\fkb_\red(F_0)$ lies in the closure of $\fkb_{\red,\rs, 1}$, with $\lambda\neq 0$.  Then for $x'=(\lambda', u', w')\in \fkb_{\red,\rs, 1}$ sufficiently close to $x$, $\lambda$ and $\lambda'$ will lie in the same class in $F_0^{\times}/ F_0^{\times, 2}$.  So $-\lambda \notin F_0^{\times, 2}$ by the fact just cited.
\end{proof}

\begin{remark}
Although $\fks$ is also naturally defined over $O_{F_0}$, the map $\fks(O_{F_0}) \to \fkb(O_{F_0})$ is not a surjection, in contrast to the map $\fks(F_0) \twoheadrightarrow \fkb(F_0)$ on $F_0$-rational points.
\end{remark}

For the next statement, note that the function $x \mapsto \det(1-x)$ descends from each of $\fks$, $\fku_0$, and $\fku_1$ to a common function on $\fkb$, and define
\[
   \fkb^{\circ\circ} := \bigl\{\, x\in\fkb^\circ(F_0) \bigm| \det(1-x)\in O_{F}^\times \,\bigr\}.
\]

\begin{lemma}\label{b circ circ lemma}
Let $x \in \fkb(O_{F_0})$, and suppose that $x$ lies in the closure of $\fkb_{\rs,1}$.  Then $x \in \fkb^{\circ\circ}$.
\end{lemma}

\begin{proof}
By Lemma \ref{intimage}, $x \in \pi_{\fku_1}(\fkk_1)$. The conclusion then follows from Remark \ref{fku_1 circ circ}.
\end{proof}

Now recall our general discussion of the Cayley transform from \S\ref{group setting}, and that we regard $B$ as the common categorical quotient of $S$, $U_0$, and $U_1$ via Lemma \ref{B isom}.

\begin{lemma}\label{pi(K1) notin pi(K0)}
There is an inclusion of subsets of $B(F_0)$,
\[
   \pi_{U_1}(K_1) \smallsetminus B_\rs(F_0) \subset \pi_{U_0}(K_0).
\]
\end{lemma}

\begin{proof}
Suppose that $x^\nat \in \pi_{U_1}(K_1)$ is not regular semi-simple.  By Lemma \ref{lem cayley U1}, we may choose $\xi = \diag(\pm 1_2, \pm 1)$ such that $x := \fkc_\xi^{-1}(x^\nat)$ is defined and contained in $\pi_{\fku_1}(\fkk_1) \subset \fkb(O_{F_0})$.  Then $x \in \pi_{\fku_0}(\fkk_0)$ by Lemma \ref{intimage}\eqref{integral b non rs}, and $x \in \fkb^{\circ\circ}$ by Lemma \ref{b circ circ lemma}.  Hence $x^\nat = \fkc_\xi(x) \in  \pi_{U_0}(K_0)$ by Lemma \ref{lem cayley U}.
\end{proof}

\subsection{Intersection numbers as a function on the quotient}\label{intersection numbers quot section}

In this subsection we consider intersection numbers as a function on the categorical quotient in  both the group and Lie algebra settings.  Recall the decomposition $B_\rs(F_0) = B_{\rs,0} \amalg B_{\rs,1}$ from \eqref{B_rs(F_0) decomp}.  For any odd $n$, by Remark \ref{indepoforb} the function $g\mapsto \Int(g)$ on $U_1(F_0)_\rs$ descends to a function on $B_{\rs, 1}$. We extend it by zero to $B_{\rs, 0}$, and still denote by $\Int$ the resulting function on $B_\rs(F_0)$. 
By Remark \ref{intersection conjugation invar},
the same applies to the function  $x\mapsto \lInt(x)$ on $\fku_{1, \rs}$, which we consider as a function on $\fkb_\rs(F_0)$.  Note that, by Remark \ref{g rs => Int(g) proper} and Remark \ref{Intfiniteforliealg} respectively, or by Lemma \ref{intersection nonempty conds} and Proposition \ref{char nondeg} when $n = 3$, these are finite-valued functions.

\begin{proposition}\label{Intrelcomp}
 Let $n=3$. 
The function $\Int$ belongs to $C^\infty_\rc(B_\rs)$; similarly, the function $\lInt$ belongs to    $C^\infty_\rc(\fkb_\rs)$.
\end{proposition}
\begin{proof}  To prove that the closure of the support of $\Int$ is compact, it suffices to prove that the closure inside $U_1(F_0)$  of the support of the function  $g\mapsto \Int(g)$ on $U_1(F_0)_\rs$ is compact. But by Lemma \ref{intersection nonempty conds}, this support is contained in the compact subgroup $K_1$, and hence the assertion is clear. A similar argument applies to $\lInt$. 

Now we prove that the function $\lInt$ on $\fkb_{\rs, 1}$ is locally constant. By Corollary \ref{inttoliealg}, this assertion follows from  the corresponding statement for the function $\lInt$ on $\fkb_{ \red, \rs, 1}$. But this   follows in turn from the  expressions  for this function in \S\ref{calc of l-Int} in terms of the quantities $\ell_-, \ell_+, m$ or, equivalently by \eqref{remember valu}, in terms of the functions $u, w, \Delta$ on $\fkb_\red(F_0)$, all of which are locally constant  on $\fkb_{ \red, \rs}(F_0)$. 

The fact that $\Int$ is locally constant on $B_{\rs,1}$ now follows from Lemma \ref{lem cayley U1}, Corollary \ref{inttoliealg}, and the fact that the Cayley transform $\fkc_\xi$, for any $\xi \in S^1$, is a local homeomorphism (on its domain of definition). 
\end{proof}

\begin{remark}
We conjecture that the preceding assertion regarding the function $\Int$ on $B_\rs(F_0)$ continues to  hold for arbitrary odd $n$. 
\end{remark}

\subsection{The reduction step} 
In this subsection we complete the proof that Theorem \ref{thm group} follows from  Theorem \ref{thm lie}. 
We begin by noting the following compatibility between transfer factors under the Cayley transform for $\fks$, which holds for all odd $n$.  Set $\fks_\rs^\circ := \fks_\rs \cap \fks^\circ$.

\begin{lemma}[{\cite[Lem.~3.5]{Z14}\footnote{The proof of loc.~cit.~contains some miscalculations and should be corrected accordingly. This does not affect the results in loc.~cit.}}]\label{lem transfer factor gp2lie}
Let $n$ be odd. Then for any $y\in \fks_\rs^\circ(F_0)$ and $\xi\in S^1$, 
\begin{equation*}
   \omega\bigl(\fkc_{\xi}(y)\bigr) = \wt{\eta} \bigl( 2^{n(n-1)/2} \bigr)\omega(y).
\end{equation*}
\end{lemma}

\begin{proof}
Let $\gamma := \fkc_{\xi}(y)\in S_\rs(F_0)$.  Write $\xi = \diag(\xi_1 \cdot 1_{n-1}, \xi_2)$, and choose $\nu = \diag(\nu_1 \cdot 1_{n-1},\nu_2)$ with $\ov\nu \nu^{-1} = \xi$ as in the definition \eqref{cayley s} of $\fkc_\xi$.  By the definition \eqref{sign Sn} of $\omega$ on $S_\rs(F_0)$,
\begin{alignat*}{2}
	\omega(\gamma) 
	   &= \wt\eta \bigl( \det(\gamma)^{-(n-1)/2} \det(\gamma^ie)_{i=0,\dotsc,n-1}\bigr)\\
		&= \wt\eta \Bigl( \det\bigl(\ov\nu \fkc_{1_n}(y) \nu^{-1}\bigr)^{-(n-1)/2} 
		         (\nu_1/\nu_2)^{n-1} \xi_1^{n(n-1)/2} \det(\fkc_{1_n}(y)^ie)_{i=0,\dotsc,n-1} \Bigr)  
		             & \quad\text{(by \eqref{cayley det calc s})}\\
		&= \wt\eta \bigl( (\xi_1/\xi_2)^{(n-1)/2} (\nu_1/\nu_2)^{n-1} \bigr) \omega\bigl(\fkc_{1_n}(y)\bigr).
\end{alignat*}
Since $\xi_1 \nu_1^2= \ov\nu_1 \nu_1$ and $\xi_2 \nu_2^2 = \ov\nu_2 \nu_2$ are both norms, we conclude that $\omega(\gamma) = \omega(\fkc_{1_n}(y))$. Therefore we reduce to the case $\xi=1$. Then
\[
   \gamma=(1+y)(1-y)^{-1}=-1+T,
\]
where we set $T := 2(1-y)^{-1}$. We compute
\begin{align*}
	\det(\gamma^i e)_{i=0,1,\dotsc,n-1}
	   &= \det\bigl((-1+T)^i e\bigr)_{i=0,1,\dotsc,n-1}\\
		&= \det(T^i e)_{i=0,1,\dotsc,n-1}\\
		&= 2^{n(n-1)/2} \det(1-y)^{1-n} \det\bigl((1-y)^i e\bigr)_{i = n-1,n-2,\dotsc,0}\\
		&= 2^{n(n-1)/2}\det(1-y)^{1-n}\det(y^i e)_{i=0,1,\dotsc,n-1}.
\end{align*}
Note that $1-y=1+\ov y$, since $y \in \fks(F_0)$. Hence $\det(1+y)\det(1-y)\in\RN F^\times$, and 
\[
   \wt{\eta}\bigl(\det(1+y)\det(1-y)\bigr) = 1.
\]
Recalling the definition \eqref{omega lie} of $\omega$ on $\fks_\rs(F_0)$, we conclude that 
\begin{align*}
   \omega(\gamma)
	   &= \wt{\eta} \Bigl( \det\bigl((1+y)(1-y)^{-1}\bigr)^{-\frac{n-1}{2}} 2^{n(n-1)/2} \det(1-y)^{1-n}  \det(y^i e)_{i=0,1,\dotsc,n-1}\Bigr)\\
      &= \wt{\eta}\bigl( 2^{n(n-1)/2} \bigr) \wt{\eta}\bigl(  \det(1+y)^{-\frac{n-1}{2}}\det(1-y)^{-\frac{n-1}{2}}  \bigr) \omega(y)\\
		&= \wt{\eta}\bigl( 2^{n(n-1)/2} \bigr) \omega(y).\qedhere
\end{align*}
\end{proof}

Now we return to $n = 3$.  The main part in the reduction step is given by the following lemma.  Recall from Lemma \ref{del rc lie} that if a function $\phi' \in C_c^\infty(\fks)$ has vanishing orbital integrals $\Orb(y,\phi') = 0$ for all $y \in \fks_{\rs,1}$, then the function $y \mapsto \omega(y) \del(y,\phi')$ on $\fks_{\rs,1}$ descends to a function on $\fkb_{\rs,1}$ which, when extended by zero to $\fkb_{\rs,0}$, lies in $C_\rc^\infty(\fkb_\rs)$; and recall from Lemma \ref{del rc group} that if $f' \in C_c^\infty(S)$ is such that $\Orb(\gamma,f') = 0$ for all $\gamma \in S_{\rs,1}$, then the function $\gamma\mapsto \omega(\gamma)\del(\gamma,f')$ on $S_{\rs,1}$ analogously descends to an element of $C_\rc^\infty(B_\rs)$.

\begin{lemma}\label{lem orb cayley}
Let $f' \in C_c^\infty(S)$  transfer to $(\mathbf{1}_{K_0}, 0) \in C_c^\infty(U_0) \times C_c^\infty(U_1)$, and let $\phi' \in C_c^\infty(\fks)$ transfer to $(\mathbf{1}_{\fkk_0}, 0) \in C_c^\infty(\fku_0) \times C_c^\infty(\fku_1)$. Fix 
$\xi \in S^1$.
Let $x_0\in \fkb^\circ(F_0)$  be an element with the property that if $x_0\notin \fkb(O_{F_0})$, then $\fkc_\xi(x_0) \notin \pi_{U_0}({K_0})$. Then the difference function
\[
   x\mapsto 
	\begin{cases} 
		\omega\bigl(\fkc_\xi(y)\bigr)\del\bigl(\fkc_\xi(y),f'\bigr) - \omega(y)\del(y,\phi'),  &  x=\pi_\fks(y)\in \fkb_{\rs, 1}^\circ;\\
		0,  &  x\in \fkb_{\rs, 0}
   \end{cases}
\]
is locally around $x_0$ an orbital integral function.\footnote{Note that this function is not defined on 
all of $\fkb_\rs(F_0)$,
but this raises no issue since the conclusion concerns only the local behavior of the function near the point $x_0 \in \fkb^\circ(F_0)$.}
\end{lemma}

\begin{proof}
Let $f := \mathbf{1}_{K_0}$ and $\phi := \mathbf{1}_{\fkk_0}$. We may assume that $x_0$ lies in the closure of $ \fkb_{ \rs, 1}$. 
 Choose an open and compact (hence closed) neighborhood $V_{x_0}$ of $x_0$ contained in $\fkb^\circ(F_0)$. Consider the functions
\[
   \phi'_{x_0} := \phi'\cdot \mathbf{1}_{\pi_{\fks}^{-1}(V_{x_0})}\in C_c^\infty(\fks)
	\quad\text{and}\quad
	\phi_{x_0} := \phi\cdot \mathbf{1}_{\pi_{\fku_0}^{-1}(V_{x_0})}\in C_c^\infty(\fku_0),
\]
and similarly
\[
   f'_{x_0} := f'\cdot \mathbf{1}_{\pi_{S}^{-1}(\fkc_\xi (V_{x_0}))}\in C_c^\infty(S)
	\quad\text{and}\quad
	f_{x_0} := f\cdot \mathbf{1}_{\pi_{U_0}^{-1}(\fkc_\xi(V_{x_0}))}\in C_c^\infty(U_0) .
\]
Then $\phi'_{x_0}(y) = \phi'(y)$ for all $y \in \pi_{\fks}^{-1}(V_{x_0})$, so that $\Orb(y, \phi'_{x_0}, s) = \Orb(y, \phi', s)$ for all such $y$; and similarly for $f'_{x_0}$ and $f'$.
The assertion of the lemma is therefore reduced to the same assertion for $f'_{x_0}$ and $\phi'_{x_0}$. 

We claim that, after possibly shrinking $V_{x_0}$, we have with respect to the Cayley transform $\fkc_\xi \colon \fku_0^\circ \to U_0$,
\begin{equation}\label{eqn phi=f}
   \phi_{x_0} = \fkc_\xi^* (f_{x_0})
\end{equation}
(where the right-hand side is extended by $0$ from $\fku_0^\circ(F_0)$ to $\fku_0(F_0)$).
Indeed, first suppose that $x_0\in \fkb(O_{F_0})$ is integral.  We have
\[
   \supp \phi_{x_0} = \fkk_0\cap\pi_{\fku_0}^{-1}(V_{x_0})
	\quad\text{and}\quad
   \supp f_{x_0} = K_0\cap\pi_{U_0}^{-1}\bigl(\fkc_\xi(V_{x_0})\bigr) .
\]
Since $x_0$ lies in the closure of $ \fkb_{\rs, 1}$, we have $x_0 \in \fkb^{\circ\circ}$ by Lemma \ref{b circ circ lemma}.
Shrinking $V_{x_0}$ if necessary, we may therefore assume that $V_{x_0}\subset \fkb(O_{F_0})\cap\fkb^{\circ\circ}$.  
Then by Lemma \ref{lem cayley U}, $\fkc_\xi$ carries the left-hand set in the display isomorphically onto the right-hand set, which proves \eqref{eqn phi=f}.

If $x_0$ is not integral, then by hypothesis $\fkc_\xi(x_0)\notin \pi_{U_0}(K_0)$. Hence, after   possibly shrinking $V_{x_0}$,  both functions $\phi_{x_0}$ and $ f_{x_0}$   vanish identically,  so that again  identity \eqref{eqn phi=f} is satisfied. 

By \eqref{eqn phi=f}, for all $y \in \fks_\rs^\circ(F_0)$ matched with an element $x \in \fku_0^\circ(F_0)$ , we have
\[
   \omega(y)\Orb(y,\phi'_{x_0})
      = \Orb(x,\phi_{x_0})
	  = \Orb\bigl(\fkc_\xi(x),f_{x_0}\bigr)
	  = \omega\bigl(\fkc_\xi(y)\bigr) \Orb\bigl(\fkc_\xi(y),f'_{x_0}\bigr).
\]
By Lemma \ref{lem transfer factor gp2lie}, we have $\omega(\fkc_\xi(y))=c\cdot \omega(y)$ for the constant $c := \wt{\eta}( 2^{n(n-1)/2})$. Hence the difference function $c\cdot \fkc_\xi^\ast ( f'_{x_0}) - \phi'_{x_0}$  (viewed as an element in $C^\infty_c(\fks)$) has identically vanishing orbital integrals on $\fks_{\rs,0}$.  The same trivially holds on $\fks_{\rs,1}$, since $\phi'_{x_0}$ transfers to $(\phi_{x_0},0)$ and $f'_{x_0}$ transfers to $(f_{x_0},0)$.
Now the assertion follows from  Corollary \ref{cor density} below.
\end{proof}

The proof of Corollary \ref{cor density}
is based on the following theorem, which is the $n = 3$ case of Conjecture \ref{conj density} (see also Remarks \ref{density remarks}).
Recall from \S\ref{sec relate a b} that for $\phi' \in C_c^\infty(\fks)$ and $h\in H'(F_0) = \GL_{n-1}(F_0)$, we define
\[
   \tensor*[^h]{\phi}{^\prime}(y) = \phi'(h^{-1}yh)
	\quad\text{and}\quad
	\tensor*[^{\eta(h)h-1}]{\phi}{^\prime}(y) = \eta(h)\phi'(h^{-1}y h) - \phi'(y).
\]

\begin{theorem}[Density principle]\label{densprince} 
Let $n=3$. 
Let $\phi'\in C^\infty_c(\fks)$ be such that $\Orb(y,\phi')=0$ for all $y\in \fks_\rs(F_0)$. 
Then $\phi'$ is in the kernel of the natural projection $C^\infty_c(\fks)\to C^\infty_c(\fks)_{H',\eta}$, i.e.\ $\phi'$ is a linear combination of functions of the form $\tensor*[^{\eta(h)h-1}]{\phi}{^{\prime\prime}}$ for $\phi''\in C^\infty_c(\fks)$ and $h\in H'(F_0)$.
\end{theorem}

\begin{proof}
By \cite[Th.~1.1]{Z12b}, the set of orbital integrals of regular semi-simple elements spans a weakly dense subspace of the space of all $(H'(F_0),\eta)$-invariant distributions on $\fks(F_0)$.\footnote{In \cite{Z12b} this is proved for $\fks_\red$, but it is trivial to extend the result to $\fks$.}  Therefore if a test function $\phi'\in C^\infty_c(\fks)$ is such that $\Orb(y,\phi')=0$ for all $y\in \fks_\rs(F_0)$, then $\phi'$ is annihilated by all  $(H'(F_0),\eta)$-invariant distributions on $\fks(F_0)$. This implies that $\phi'$ lies in the kernel of the natural  projection $C^\infty_c(\fks)\to C^\infty_c(\fks)_{H',\eta}$.
\end{proof}

\begin{corollary}\label{cor density}
Let $\phi'\in C^\infty_c(\fks)$ be such that $\Orb(y,\phi')=0$ for all $y\in \fks_\rs(F_0)$. 
Then there exists a function $\phi'^\flat \in C^\infty_c(\fks)$ such that
\[
   \omega(y)\del(y,\phi') = \omega(y)\Orb\bigl(y, \phi'^\flat \bigr) \quad\text{for all}\quad y\in \fks_\rs(F_0) . 
\]
In other words, $y \mapsto \omega(y)\del(y,\phi')$ is an orbital integral function on $\fks_\rs(F_0)$.
\end{corollary}

\begin{proof}

By the density principle, we may assume that $\phi'$ is of the form $\tensor*[^{\eta(h)h-1}]{\phi}{^{\prime\prime}}$ for some $\phi''\in C^\infty_c(\fks)$ and $h\in H'(F_0)$. By Lemma \ref{twisted del}\eqref{twisted del s},
\[
   \del\bigl(\gamma, \tensor*[^{\eta(h)h-1}]{\phi}{^{\prime\prime}}\bigr) 
	   = \log\lvert \det h \rvert \Orb(\gamma, \phi'').
\]
Setting $\phi'^\flat := \log\lvert \det h \rvert \cdot \phi''$ completes the proof. 
\end{proof}

\begin{remark} 
Note that this corollary is essentially  a converse to Lemma \ref{lem btoa}\eqref{exists wt phi'}. Indeed, in Corollary \ref{cor density} above, we are given $\phi'$, and are writing $\del(y,\phi')$ as an orbital integral; in Lemma \ref{lem btoa}\eqref{exists wt phi'}, we are given $\phi'$, and are writing $\Orb(y,\phi')$ as a derivative of an orbital integral. 
\end{remark}

\begin{proposition}\label{lem1}
Theorem \ref{thm lie} implies Theorem \ref{thm group}. 
\end{proposition}

\begin{proof} 
What we need to show is that Conjecture \ref{lieconj} implies Conjecture \ref{inhomconj}\eqref{inhomconj b} when $n = 3$; Proposition \ref{cor b to a}, Lemma \ref{intersection nonempty conds}, and Proposition \ref{char nondeg} then take care of the rest.  Suppose that $f'\in C^\infty_c(S)$ transfers to $(\mathbf{1}_{K_0}, 0)$, and let $\phi' \in C_c^\infty(\fks)$ be a function satisfying the conclusion of Conjecture \ref{lieconj}\eqref{lieconj a}.

As in \S\ref{intersection numbers quot section}, we consider $\Int$ as a function in $C_\rc^\infty(B_\rs)$ which vanishes identically on $B_{\rs,0}$, and $\lInt$ as a function in $C_\rs^\infty(\fkb_\rs)$ which vanishes identically on $\fkb_{\rs,0}$.
As in Lemmas \ref{del rc group} and \ref{del rc lie}, respectively, we consider the function $\gamma\mapsto \omega(\gamma)\del(\gamma,f')$ 
as an element in $C_\rc^\infty(B_\rs)$ which vanishes identically on $B_{\rs, 0}$, 
and the function $y \mapsto \omega(y) \del(y,\phi')$ 
as an element in $C_\rc^\infty(\fkb_\rs)$ which vanishes identically on $\fkb_{\rs,0}$.
Our task is to prove that the sum
\begin{equation}\label{sum fcn}
   \Int(g) \cdot \log q + 2 \omega(\gamma)\del(\gamma,f'),
\end{equation}
regarded in this way as a function on $B_\rs(F_0)$, is an orbital integral function on $B_\rs(F_0)$.  By Theorem \ref{local pro of orb fun}, it suffices to show that \eqref{sum fcn} is an orbital integral function locally around each $x_0^\nat\in B(F_0)$. 
If $x_0^\nat \in B_{\rs,1}$, then \eqref{sum fcn} is constant in a neighborhood of $x_0^\nat$ since both terms are, the first by Proposition \ref{Intrelcomp} and the second by Lemma \ref{del rc group}. If $x_0^\nat$ is outside the closure of $B_{\rs,1}$, then \eqref{sum fcn} is identically $0$ in a neighborhood of $x_0^\nat$. 
Hence in these two cases the conclusion follows from Corollary \ref{cor germ converse} below (or rather its analog for $B_{\rs}$, cf.\ Remark \ref{last one}). In the rest of the proof we assume that $x_0^\nat$ lies in the closure of $B_{\rs,1}$ but is not itself regular semi-simple.

By Lemma \ref{cayley cover lem}\eqref{cayley cov lem S part}, we may choose $\xi \in S^1$ such that the inverse Cayley transform is defined at $x_0^\nat$.  Set $x_0 := \fkc_\xi^{-1}(x_0^\nat) \in \fkb^\circ(F_0)$.  If $x_0^\nat \in \pi_{U_0}(K_0)$, then by Lemma \ref{lem cayley U} and \eqref{fkb(O_F_0) union}, we may furthermore choose $\xi = \diag(\pm 1_2,\pm1)$ such that $x_0 \in \fkb(O_{F_0})\cap\fkb^\circ(F_0)$.  Note that this in fact gives us $x_0 \in \fkb(O_{F_0}) \iff x_0^\nat \in \pi_{U_0}(K_0)$, since if $x_0 \in \fkb(O_{F_0})$, then $x_0 \in \pi_{\fku_0}(\fkk_0) \cap \fkb^{\circ\circ}$ by Lemmas \ref{intimage}\eqref{integral b non rs} and \ref{b circ circ lemma}, and hence $x_0^\nat \in \pi_{U_0}(K_0)$ by Lemma \ref{lem cayley U}.

We claim that for all $x \in \fkb_\rs^\circ(F_0)$ contained in a sufficiently small neighborhood of $x_0$,
\begin{equation}\label{eqn int=}
   \lInt(x) = \Int\bigl(\fkc_\xi(x)\bigr) .
\end{equation}
Of course this holds trivially for $x \in \fkb_{\rs,0}$.  If $x_0 \in \fkb(O_{F_0})$, then any $x \in \fkb_{\rs,1}$ which is sufficiently near $x_0$ will be contained in $\fkb(O_{F_0})_{\rs,1}$, which equals $\pi_{\fku_1}(\fkk_{1,\rs})$ by Lemma \ref{intimage}\eqref{integral b rs}.  Hence \eqref{eqn int=} holds for such $x$ by Corollary \ref{inttoliealg}.  If $x_0 \notin \fkb(O_{F_0})$, then $x_0^\nat \notin \pi_{U_0}(K_0)$ by our choice of $\xi$.  Since $x_0^\nat$ is not regular semi-simple, $x_0^\nat \notin \pi_{U_1}(K_1)$ by Lemma \ref{pi(K1) notin pi(K0)}.  Hence by Lemma \ref{intersection nonempty conds} both sides of \eqref{eqn int=} vanish for $x \in \fkb_{\rs,1}$ sufficiently near $x_0$.  This proves the claim.

We conclude that for all $x \in \fkb_\rs(F_0)$ near $x_0$,
\begin{multline*}
   \Int\bigl(\fkc_\xi(x)\bigr)\cdot\log q + 2 \omega\bigl(\fkc_\xi(x)\bigr)\del\bigl(\fkc_\xi(x),f'\bigr)\\
	= \lInt(x)\cdot\log q + 2\omega(x) \del(x,\phi') + 2\Bigl( \omega\bigl(\fkc_\xi(x)\bigr)\del\bigl(\fkc_\xi(x),f'\bigr) -  \omega(x) \del(x,\phi') \Bigr).
\end{multline*}
By the choice of $\phi'$, the first two terms on the right-hand side cancel. By Lemma \ref{lem orb cayley}, the remaining expression on the right is an orbital integral function on $\fkb_\rs(F_0)$ locally around $x_0$. It follows that \eqref{sum fcn} is an orbital integral function locally around $x_0^\nat=\fkc_\xi(x_0)$ since the Cayley transform preserves the property of ``being an orbital integral function.'' The proposition follows.
\end{proof}

\section{Reduction to the reduced set}
We continue to take $n = 3$.  In this section we enact a further reduction step: we reduce the main Lie algebra result Theorem \ref{thm lie} to an analog for the reduced sets $\fks_\red(F_0)$, $\fku_{0,\red}(F_0)$, and $\fku_{1, \red}(F_0)$.  Recall from \eqref{s prod decomp}, \eqref{u0 prod decomp}, and \eqref{u1 prod decomp} that we have product decompositions\footnote{Strictly speaking, \eqref{u1 prod decomp} only gives the product decomposition for $\fku_1$ on the level of $F_0$-rational points, but this obviously extends to an isomorphism of schemes.}
\begin{equation}\label{prod decomps}
   \fks \cong \fks_\red \times \fks_1 \times \fks_1,
	\quad
	\fku_0 \cong \fku_{0,\red} \times \fks_1 \times \fks_1,
	\quad\text{and}\quad
	\fku_1 \cong \fku_{1,\red} \times \fks_1 \times \fks_1,
\end{equation}
all of which are compatible with the categorical quotient maps to $\fkb = \fkb_\red \times \fks_1 \times \fks_1$.
The matching relation for regular semi-simple elements in \S\ref{Lie algebra setting} obviously respects reducedness on both sides.  Since each of the above reduced sets is stable under the group action on the ambient space, the formulas for orbital integrals in \S\ref{lie conj sec} make sense in the obvious way for reduced regular semi-simple elements and functions on the reduced set (and of course we keep the same normalizations).  The transfer relation for functions then extends in the obvious way to the reduced setting: a function $\phi_\red' \in C_c^\infty(\fks_\red)$ and a pair $(\phi_{0,\red},\phi_{1,\red}) \in C_c^\infty(\fku_{0,\red}) \times C_c^\infty(\fku_{1,\red})$ are \emph{transfers} of each other if for each $i \in \{0,1\}$ and each $x \in \fku_{i,\red,\rs}(F_0)$,
\[
   \Orb(x,\phi_{i,\red}) = \omega(y) \Orb(y,\phi'_\red)
\]
whenever $y \in \fks_{\red,\rs}(F_0)$ matches $x$.  Here the transfer factor $\omega$ is the obvious one, namely the restriction of \eqref{omega lie} to $\fks_{\red,\rs}(F_0)$.  We are going to reduce Theorem  \ref{thm lie} to the following statement.

\begin{theorem}\label{mainforred}
Let $n = 3$.  Then for any function $\phi'_\red\in C^\infty_c(\frak s_\red)$  which transfers to the pair $(\mathbf{1}_{\fkk_{0,\red}}, 0) \in C_c^\infty(\fku_{0,\red}) \times C_c^\infty(\fku_{1,\red})$, there exists a function $\phi'_{\red, \corr}\in C^\infty_c(\fks_\red)$ such that for any 
$y\in \fks_{\red, \rs}(F_0)$ matched with an element  $x\in \fku_{1,\red,\rs}(F_0)$,
\[
  2 \omega(y)\del(y,\phi'_\red) = - \lInt(x) \cdot\log q+ \omega(y)\Orb(y,\phi'_{\red, \corr}) .
\]
\end{theorem}

We will prove Theorem \ref{mainforred} in \S\ref{subsection pf 12.1}. To see that Theorem \ref{mainforred} implies Theorem \ref{thm lie}, let us first note the following straightforward lemma.

\begin{lemma}\label{phired}
If $\phi'_\red\in C^\infty_c(\fks_\red)$ transfers to $(\mathbf{1}_{\fkk_{0, \red}}, 0)$, then the function
\[
   \phi'_\red\otimes \mathbf{1}_{\fks_1(O_{F_0}) \times \fks_1(O_{F_0})}\in C^\infty_c(\fks) \cong C_c^\infty(\fks_\red \times \fks_1 \times \fks_1)
\]
transfers to $(\mathbf{1}_{\fkk_0}, 0) \in C_c^\infty(\fku_0) \times C_c^\infty(\fku_1)$. \qed
\end{lemma}

We also record the following fact about the transfer factor $\omega$, which holds for any odd $n$.  It is an immediate consequence of the equation \eqref{y y_red matrix eqn}.

\begin{lemma}\label{red omega compat}
Let $n$ be odd. Then for any $y \in \fks_\rs(F_0)$,
\begin{flalign*}
	\phantom{\qed} & &
   \omega(y) = \omega(y_\red).
	& & \qed
\end{flalign*}
\end{lemma}

Now we return to the case $n = 3$.

\begin{proposition}\label{lem2}
Theorem \ref{mainforred} implies Theorem \ref{thm lie}. 
\end{proposition}

\begin{proof}
The setup is parallel to the proof of Proposition \ref{lem1}: on account of Proposition \ref{cor b to a}, Lemma \ref{intersection nonempty conds}, and Proposition \ref{char nondeg}, what we need to show is that Theorem \ref{mainforred} implies Conjecture \ref{lieconj}\eqref{lieconj b} when $n = 3$.  Suppose that $\phi' \in C_c^\infty(\fks)$ transfers to $(\mathbf{1}_{\fkk_0},0)$.  Then 
$\phi_\red' := \phi'|_{\fks_\red(F_0)}$ transfers to $(\mathbf{1}_{\fkk_{0,\red}},0)$.  Let $\phi'_{\red,\corr} \in C_c^\infty(\fks_\red)$ be a function which satisfies the conclusion of Theorem \ref{mainforred} for $\phi'_\red$.  Set
\[
   \phi'' := \phi'_\red \otimes \mathbf{1}_{\fks_1^2(O_{F_0})}
	\quad\text{and}\quad
	\phi_{\corr}'' := \phi'_{\red,\corr} \otimes \mathbf{1}_{\fks_1^2(O_{F_0})}.
\]
We claim that for any $y\in \fks(F_0)_\rs$ matched with an element $x\in \fku_1(F_0)_\rs$,
\begin{equation}\label{phi'' eqn}
   2\omega(y) \del(y,\phi'') = -\lInt(x) \cdot \log q + \omega(y) \Orb(y,\phi''_\corr).
\end{equation}

Before proving the claim, let us explain how it implies the proposition.  By Lemma \ref{phired}, $\phi''$ transfers to $(\mathbf{1}_{\fkk_0},0)$.  Hence $\phi' - \phi''$ has vanishing orbital integrals at all regular semi-simple elements.  Hence the conclusion of Conjecture \ref{lieconj}\eqref{lieconj b} for $\phi'$ follows from the claim and from Corollary \ref{cor density}.

Now we prove the claim. Since the matching elements $y$ and $x$ have the same invariants, their last two components with respect to the respective product decompositions 
in \eqref{prod decomps}
are the same.  If either of these two common components does not lie in $\fks_1(O_{F_0})$, then $\phi''$ and $\phi''_\corr$ vanish identically on the $H'(F_0)$-orbit of $y$, and $x \notin \fkk_1$.  Hence, using Lemma \ref{intersection nonempty conds} for the $\lInt$ term, every term in \eqref{phi'' eqn} vanishes, and the claim holds trivially.

Now suppose that the last two components of $y$ and $x$ do lie in $\fks_1(O_{F_0})$.  Then
\[
   \del(y,\phi'') = \del(y_\red,\phi'_\red)
	\quad\text{and}\quad
	\Orb(y,\phi''_\corr) = \Orb(y_\red, \phi'_{\red,\corr}).
\]
Furthermore, in this case $x \in \fkk_1$ if and only if $x_\red \in \fkk_1$.  Hence by Lemma \ref{intersection nonempty conds} and Corollary \ref{inttoliealg},
\[
   \lInt(x) = \lInt(x_\red).
\]
Since $y_\red$ and $x_\red$ match, and since $\omega(y) = \omega(y_\red)$ by Lemma \ref{red omega compat}, we conclude that \eqref{phi'' eqn} holds because $\phi'_{\red,\corr}$ satisfies the conclusion of Theorem \ref{mainforred} for the function $\phi'_\red$.  This completes the proof.
\end{proof}

\begin{remark}\label{changeoftr}
We have formulated the notion of transfer above with respect to the particular transfer factor $\omega$, but this notion of course make sense relative to any transfer factor, as in Definition \ref{homog transfer} in the homogeneous group setting.  In particular, for $c \in \BC^\times$, note that a function $\phi_\red' \in C_c^\infty(\fks_\red)$ transfers to $(\mathbf{1}_{\fkk_{0,\red}},0)$ with respect to $\omega$ if and only if the function $c^{-1} \phi_\red'$ transfers to $(\mathbf{1}_{\fkk_{0,\red}},0)$ with respect to the transfer factor $c\omega$.  It is easy to see from this that the truth of  Theorem   \ref{mainforred}  is unaffected when we replace $\omega$ by a nonzero constant multiple. 
\end{remark}

\section{Application of  the germ expansion principle}\label{sec:apllgerm}
\label{sec 13}
In the rest of  Part \ref{analtyic side part} of the paper, we suppress the subscript in  the notation $\phi'_\red$, and simply call this function $\phi'$.  Also, from now on we systematically 
abuse notation and suppress the expression $(F_0)$ when referring to sets of $F_0$-rational points,
so that $\fkb$ means $\fkb(F_0)$, $\fku_i$ means $\fku_i(F_0)$, etc.

In Part \ref{germ expansion part} we  prove a germ expansion of orbital integrals around each element $x_0\in\fkb_{\red}\smallsetminus\fkb_{\red,\rs}$. The rough form of this germ  expansion is as in Theorem \ref{thm germ s}.  Taking the derivative at $s=0$ of both sides, we obtain  a sum decomposition as in  (\ref{eqn leibniz}),
\begin{align*}
   \omega\bigl(\sigma(x)\bigr)\del\bigl(\sigma(x),\phi'\bigr) 
	   =\omega\bigl(\sigma(x)\bigr)\del_1\bigl(\sigma(x),\phi'\bigr) + \omega\bigl(\sigma(x)\bigr)\del_2\bigl(\sigma(x),\phi'\bigr) .
\end{align*}
Here $\sigma(x)$ is an explicit element in $\fks_\red$ above $x\in\fkb_{\red,\rs}$, defined by \eqref{eqn sigma} and \eqref{eqn sigma 1} below. We have
\begin{equation}\label{germofdel1}
   \del_1\bigl(\sigma(x),\phi'\bigr) = \sum_{n\in \pi_\red^{-1}(x_0)/H'} \partial\Gamma_n(x) \Orb(n,\phi',0),
\end{equation}
where 
$$
\partial\Gamma_n(x):=\frac d{ds} \Big|_{s=0} \Gamma_n(x, s) .
$$
 The explicit forms of the germ functions $\Gamma_n(x, s)$, and the definitions of the orbital integrals $\Orb(n,\phi',s)$ associated to elements which are not regular semi-simple, are given in Part \ref{germ expansion part}. 

Now we assume that $\phi'$ transfers to $(\phi_0,\phi_1)$ where $\phi_i\in C_c^\infty(\fku_{i,\red})$. Throughout the rest of Part \ref{analtyic side part}, we will assume that $\phi_1=0$. 
We will be concerned with the function on $\fkb_{\red, \rs}$, 
\[
   x\mapsto 
	\begin{cases}
		\del(\sigma(x),\phi'),  &  x\in\fkb_{\red,\rs, 1};\\
      0,  &  x\in \fkb_{\red,\rs, 0}.
\end{cases}
\]
By Theorem \ref{thm error term}, the term $\omega(\sigma(x))\del_2(\sigma(x),\phi')$ is an orbital integral function.  We will calculate $\del_1(\sigma(x),\phi')$. To use the  formula \eqref{germofdel1}, we need to calculate $\partial\Gamma_n(x)$ and $\Orb(n,\phi',0)$ for all $n\in \pi_\red^{-1}(x_0)/H'$. Their values are summarized by the following table.
Let us explain some of the notation we use. 
\begin{enumerate}
\item
Here $F'=F_0[X]/(X^2+\lambda_0)$ is a quadratic $F_0$-algebra  for $x_0=(\lambda_0,u_0,w_0)\in \fkb_\red$.
\item $\Delta=\Delta(x)$ for $x\in \fkb_{\red,\rs}$.
\item
The $\club$ values are omitted since the corresponding values of $\Orb(n,\phi',0)$ vanish in our case.
\item The $\spade$ values are not needed in our case since we will only need those $x_0=(\lambda_0,u_0,w_0)\in \fkb_\red\smallsetminus \fkb_{\red,\rs}$ which are also in the closure of $\fkb_{\red,\rs,1}$ (cf.~Lemma \ref{intimage}).
\end{enumerate}

{\setlongtables
\renewcommand{\arraystretch}{1.25}
\begin{longtable}{|c|c|c|c|c|}
\hline
\begin{varwidth}{\linewidth}
   \centering
   Element \\
	$x_0 \in \fkb_\red \smallsetminus \fkb_{\red,\rs}$
\end{varwidth} 
   &  \begin{varwidth}{\linewidth}
         \centering
   	   Orbit\strut\\ 
	      representative\\
	      $n \in \fks_\red$ over $x_0$\strut
      \end{varwidth}
   &  \begin{varwidth}{\linewidth}
         \centering
      	Reference\\ 
      	for orbit\\ 
      	representative
      \end{varwidth}  
   &  Value of $\partial\Gamma_n(x)$  
   &  \begin{varwidth}{\linewidth}
         \centering 
	      Reference for\\ 
      	$\Orb(n,\phi',0)$ 
      \end{varwidth}
   \\
\hline
\hline
   &  $n(\mu)$, $\mu \in F_0$  
   &  \eqref{nilpins} 
	& Theorem \ref{thm germ x=0}
   &  Lemma \ref{lem nil orb}
   \\
$0$  
   &  $n_{0,+}$  
	&  \multirow{2}{*}{\eqref{nilpins pm}} 
	&  0
   &  \multirow{2}{*}{Lemma~\ref{nilpintreg}}
	\\
   &  $n_{0,-}$
	&  
	&  $\log|\Delta/\varpi|$
	&
	\\
\hline
\multirow{2}{*}{
   \begin{varwidth}{\linewidth}
	  \centering
      $(\lambda_0,0,0)$ for $\lambda_0 \neq 0$\\
      and $F' \not\simeq F, F_0\times F_0$
   \end{varwidth}}  
  &  $y_+$ 
  &  \multirow{2}{*}{\eqref{eqn case 0i}}
  &  0  
  &  \multirow{2}{*}{Lemma \ref{lem orb ss u 0}\eqref{lem orb ss part i}}
  \\
  &  $y_-$  
  & 
  &  $\eta(-\lambda_0 )\log|\Delta/\lambda_0 |$ 
  &
  \\
\hline
\multirow{2}{*}{
   \begin{varwidth}{\linewidth}
	   \centering
      $(\lambda_0,0,0)$ for $\lambda_0 \neq 0$\\ 
      and $F' \simeq F_0\times F_0$
	\end{varwidth}}  
   &  $y_0$ 
	&  \eqref{eqn case 0i y0} 
	&  \multirow{2}{*}{Corollary \ref{cor F'=F}}
	&  $\spade$
	\\
   &  $y_\pm$ 
	&  \eqref{eqn case 0i} 
	&  
	&  $\spade$\\
\hline
\multirow{4}{*}{
   \begin{varwidth}{\linewidth}
	  \centering
      $(\lambda_0,0,0)$ for $\lambda_0 \neq 0$\\ 
      and $F' \simeq F$
   \end{varwidth}}  
  &  $y_{++}$ 
  &  \multirow{2}{*}{\eqref{eqn case 0ii}}
  &  0 
  &  \multirow{4}{*}{Lemma \ref{lem orb ss u 0}\eqref{lem orb ss part ii}}
  \\
  &  $y_{+-}$ 
  & 
  &  $\club$ 
  &  
  \\
  &  $y_{--}$  
  &  \multirow{2}{*}{\eqref{eqn case 0ii y--}}
  &  $\club$ 
  &
  \\
  &  $y_{-+}$  
  &
  & $\eta(-1)\log|\Delta/\lambda_0 |$ 
  &
  \\
\hline
\multirow{2}{*}{
   \begin{varwidth}{\textwidth}
		\centering
		$(\lambda_0,u_0,w_0)$ for\\
		$u_0 \neq 0$
	\end{varwidth}}  
	&  $y_+$  
	&  \multirow{2}{*}{\eqref{eqn case 1}}
	&  0  
	&  \multirow{2}{*}{Lemma \ref{lem orb ss u 1}}
	\\
   &  $y_-$  
	&
	&  $\log\bigl|\Delta/(u_0^2\varpi) \bigr|$
	&
	\\
\hline
\end{longtable}
}

\section{The germ expansion of a function with special transfer}\label{germspecial}

Throughout this section, we fix a function $\phi'\in C_c^\infty(\fks_\red)$ with transfer $(\mathbf{1}_{\fkk_{0,\red}}, 0)$. We set 
\[
   \phi := \mathbf{1}_{\fkk_{0,\red}}\in C_c^\infty(\fku_{0,\red}) .
\]
In this section, we will  calculate $\del_1 (\sigma(x),\phi')$. Note that, by Remark \ref{rem del1}, the result is independent of the choice of the non-unique matching function $\phi'$.

\subsection{Nilpotent orbital integrals} In this section, we determine the nilpotent integrals of $\phi'$. The nilpotent orbits are listed in \S\ref{germzeropart4}.1. We start with the two regular nilpotent orbits. We denote $\zeta(s)=\zeta_{F_0}(s)=(1-q^{-s})^{-1}$.

\begin{lemma}\label{nilpintreg}
We have
\[
   \Orb(n_{0,-},\phi') = -q^{-1}\zeta(1)
	\quad\text{and}\quad
   \Orb(n_{0,+},\phi') = -\eta(-1)q^{-1}\zeta(1).
\]
\end{lemma}
\begin{proof}
Since $\phi'$ transfers to $(\phi,0)$, this follows  from Theorem \ref{thm match nil Orb} and (\ref{eqn def orb 0}). 
\end{proof}
Next we calculate the nilpotent orbital integrals  $\Orb(n(\mu),\phi')$ where $n(\mu)\in\fks_\red$ is the family of  nilpotent elements parametrized by $\mu\in F_0$, given by \eqref{nilpins}. When we consider  $\Orb(n(\mu),\phi')$ as a function on $F_0$, we denote it by $\Orb_{\phi'}$,
\[
\Orb_{\phi'}(\mu) := \Orb\bigl(n(\mu),\phi'\bigr) .
\]

\begin{lemma}
\label{lem nil orb}
$$
\Orb\bigl(n(\mu),\phi'\bigr)=q\frac{\eta(-1)}{\log q}\cdot\begin{dcases}0 ,& |\mu|\leq 1;\\ \frac{\eta(\mu)\log |\mu|}{|\mu|},&  |\mu|> 1.
\end{dcases}
$$
\end{lemma}
\begin{proof}
We first calculate the orbital integrals of $\phi$ over the family of nilpotent elements in $\fku_0$ given by \eqref{eqn n(mu) u}. 
As in \cite[\S2.1]{Z12b}, we use the Iwasawa decomposition $H_0(F_0)=KAN$ for $K$ the special parahoric subgroup (hence $K$ contains $\SL_2(O_{F_0})$ and all diagonal elements 
$\diag(a,\ov a^{-1})$ for $a\in O_F^\times$, and $\vol(K)=1$).  Write
\begin{align}\label{eqn iwasawa}
h=
k\begin{bmatrix}z & \\
  &  \ov z^{-1}\end{bmatrix}
\begin{bmatrix}1 & t\\
  &  1\end{bmatrix},\quad dh=\zeta(1) \,dk\, dz\,dt.
  \end{align}
  By \eqref{eqn def orb n}  we then have
\begin{align*}
   \Orb\bigl(n(\mu),\phi\bigr) 
      &= \zeta(1)\bigints_{F}\phi_K\left(\pi \begin{bmatrix}0&\mu \pi z\ov z& z\\ 0&0 &0 \\
0&\pi \ov z&0\end{bmatrix} \right) d z\\
      &=\zeta(1)|\pi|^{-1}_F\bigints_{F}\phi_K\left( \begin{bmatrix}0&-\mu z\ov z& z\\ 0&0 &0 \\
0&-\pi \ov z&0\end{bmatrix} \right) d z\\
      &=q\zeta(1)\cdot \vol\bigl\{\,z\in F \bigm| |z|\leq 1 \text{ and } |\mu z\ov z|\leq 1\,\bigr\}.
\end{align*}
Note now that $F/F_0$ is \emph{ramified}.  We obtain
$$
\Orb\bigl(n(\mu),\phi\bigr) = q\zeta(1)\cdot \begin{cases}1, &|\mu|\leq 1;\\
1/|\mu|,& |\mu|>1.\end{cases}
$$

Recall from the proof of \cite[Prop.~4.3]{Z12b} the functions 
\begin{alignat*}{2}
\phi_0 &:= {\bf 1}_{O_{F_0}}, &
\phi_1(x)&:=\begin{dcases}\frac{\eta(x)}{|x|},&  |x|> 1;\\ 0
,& |x|\leq 1,
\end{dcases}\\
\phi_2(x) &:= \begin{dcases}\frac{1}{|x|}, & |x|> 1;\\ 0 ,&
|x|\leq 1,
\end{dcases}\qquad &
\phi_3(x) &:= \begin{dcases}\frac{\eta(x)\log |x|}{|x|},& |x|> 1;\\
0 ,& |x|\leq 1.
\end{dcases}
\end{alignat*}
Then we have the following  table for their extended Fourier transforms, cf.~loc.~cit., 
\begin{alignat*}{2}
\wt\phi_0(v) &= \begin{dcases}0,&
|v|\leq 1;\\ \frac{\eta(-v)}{|v|},&  |v|>1.
\end{dcases} & \wt\phi_1(v) &= \begin{dcases}q^{-1},&
|v|\leq 1;\\ 0,& |v|>1.
\end{dcases}\\
\wt\phi_2(v) &= \begin{dcases}0,&
|v|\leq 1;\\
\frac{\eta(-1)}{\log q\,\zeta(1)}\frac{\eta(v)\log|v|}{|v|}-\frac{\eta(-v)}{|v|},&
 |v|>1.
\end{dcases}\qquad & \wt\phi_3(v) &= \begin{dcases}-\frac{\zeta'(1)}{\zeta(1)},&
|v|\leq 1;\\ -\frac{\zeta'(1)}{\zeta(1)}\frac{1}{|v|},& 
|v|>1.
\end{dcases}
\end{alignat*}
In terms of the four fundamental functions $\phi_i$, we may
rewrite the nilpotent orbital integral as
$$
\Orb_\phi=q\zeta(1)\cdot(\phi_0+\phi_2).
$$
Note that $-\frac{\zeta'(1)}{\zeta(1)}=\zeta(1)q^{-1}\log\,q$.
Hence, by Theorem \ref{thm match nil Orb}, we obtain for the orbital integral of $\phi'$, 
$$
\Orb_{\phi'}=q\frac{\eta(-1)}{\log q}\phi_3.
$$
This completes the proof.
\end{proof}

\subsection{Germ expansion of $\del_1(\sigma(x),\phi')$  around $x_0=0$}

We now calculate the germ expansion of $\del_1(\sigma(x),\phi')$  around $x_0=0$. In the sequel, we denote by $v$ the normalized valuation for $F_0$. 

We are using the section $\sigma$ of $\pi_{\fks_{\red}}|_{\fks_{\red, \rs}}$ in a neighborhood of $x_0=0$ introduced in  \eqref{eqn sigma}. 
\begin{theorem}\label{thm expl x=0}

For  $x=(\lambda,u,w)\in\fkb_{\red,\rs,1}$  near zero,
\[
   \del_1\bigl(\sigma(x),\phi'\bigr) = \Phi(x)-q^{-1}\zeta_{}(1)\log|\Delta(x)/\varpi|,
\]
where the function $\Phi(x)$ is as follows. Set $t=q^{-1}$ below.
\smallskip

{\bf Case I}: $|\Delta|\geq |w|^2$

\smallskip

\begin{enumerate}
\item If $v(\Delta)> 4v(u)$, then 
\begin{align*}
\Phi(x)=- t^{-v(u)} \frac{2 (1+ t) + \bigl(v(\Delta/u^4)-1\bigr)(1- t)}{(1 - t)^2}\cdot \log q.
\end{align*}
\item If $v(\Delta)\leq 4v(u)$ and $v(\Delta)$ is odd, then 
\begin{align*}
\Phi(x)=-t^{\frac{2v(u)-v(\Delta)+1}{2}} \frac{ \bigl(4v(u)-v(\Delta)+3\bigr)-\bigl(4v(u)-v(\Delta)-1\bigr) t}{(1 - t)^2}\cdot \log q.
\end{align*}
\item If $v(\Delta)\leq 4v(u)$ and $v(\Delta)$ is even, then 
\begin{align*}
   \Phi(x)=-t^{\frac{2v(u)-v(\Delta)}{2}} \frac{\bigl( 2v(u)-\frac{v(\Delta)}{2}+1\bigr) (1-t^2)+ t(3+ t)}{(1 - t)^2}\cdot \log q.
\end{align*}
\end{enumerate}

\smallskip

{\bf Case II}: $|\Delta|< |w|^2$

\smallskip

\begin{enumerate}
\item If $v(w/\pi) \geq 2v(u)$, then 
\begin{align*}
\Phi(x)=-  t^{-v(u)}  \frac{2 (1+ t) + \bigl(v(\Delta/u^4)-1\bigr)(1- t)}{(1 - t)^2}\cdot \log q.
\end{align*}

\item If $v(w/\pi)<  2v(u)$,  then  
   \begin{align*}&
  \Phi(x)= -  t^{v(u)-v(w/\pi)}\frac{4t + \bigl(v(\Delta)+4v(u)-4v(w/\pi)+1\bigr) (1 -t)}{(1-t)^2}\cdot \log q.
\end{align*}
\end{enumerate}
\end{theorem}

The proof of this theorem will occupy the entire subsection. The term $\del_1(\sigma(x),\phi')$ is a sum of two parts, one from the two regular nilpotent orbits $n_{0\pm}$ and the other from the one-dimensional nilpotent family $n(\mu)$.

We first determine the contribution of the two regular nilpotents. We use Lemma \ref{nilpintreg}. Noting that $\eta(\Delta/\varpi)=-1$ when $x\in\fkb_{\red,\rs,1}$ (cf.~Proposition \ref{eta b01}), and taking into account the values of $\partial\Gamma(n_{0\pm})$ (see table above), the total contribution of the two regular nilpotents to $\del_1(\sigma(x),\phi')$ is equal to 
\[
-\eta(\Delta/\varpi)\log|\Delta/\varpi| \Orb(n_{0,-},\phi')
   = \eta(\Delta/\varpi) q^{-1}\zeta(1)\log|\Delta/\varpi|
   = -q^{-1}\zeta(1)\log|\Delta/\varpi|.
\]

We now calculate the contribution from the one-dimensional family $n(\mu)$, which we denote by
$$
\Phi(x)=\int_{F_0}\partial\Gamma_{n(\mu)}(x)\Orb\bigl(n(\mu),\phi',0\bigr)\,d\mu.
$$
It is easier to use an equivalent formula, namely \eqref{eqn germ mu}:
\begin{align}\label{eqn mu log}
\Phi(x)=
 -\eta(-1)|u|^{-1}
\bigintssss_{F_0} \Orb\Bigl(n\bigl(u^{-2}(t+t^{-1}\Delta/\varpi+2w/\pi)\bigr),\phi'\Bigr)\eta(t)\log|t|\,\frac{dt}{|t|}.
\end{align} 
Set
\[
   w'=w/u^2,\quad \lambda'=\lambda/u^2,\quad \Delta'=\Delta/u^4=\lambda'+w'^2.
\]
Note that $\eta(-\Delta)=-1$ (cf.~Proposition \ref{eta b01}). Denote
\begin{equation*}
\eta_1(t)=\eta(t)|t|^{-1},\quad t\in F_0^\times.
\end{equation*}
By the formula for $\Orb_{\phi'}$ in Lemma \ref{lem nil orb}, and substituting $t\mapsto tu^2$, the integral \eqref{eqn mu log} is equal to
\begin{align}\label{eqn mu log0}
\Phi(x)=- q^{-1}(\log\,q)^{-1}|u|^{-1} \cdot \Xi(x) ,
 \end{align}
 where
 \begin{align}\label{eqn mu log1}
\Xi(x)= \int_{F}\log|t+\Delta'/(\varpi t)+2w'/\pi| \,\eta_1(t^2+\Delta'/\varpi+2w't/\pi)\log|t|\,dt,
\end{align}
where the integrand is subject to the condition
$$
|t+\Delta'/(\varpi t)+2w'/\pi|>1.
$$

A simple observation is that the contribution of $t$ with $|t|=|\Delta'/(\varpi t)|$ is always zero, since we may pair $t$ with $\Delta'/(\varpi t)$ to see that the sum is canceled (use $\eta(-\varpi)=1$,  $\eta(-\Delta')=-1$ by Proposition \ref{eta b01}).
Hence the integral in \eqref{eqn mu log1} is equal to
\begin{equation}\label{eqn log0}
   \Xi(x) = \int \log|t+\Delta'/(\varpi t)+2w'/\pi| \, \eta_1(t^2+\Delta'/\varpi+2w't/\pi) (2\log|t|-\log|\Delta'/\varpi|)\,dt,
\end{equation}
where the integrand is subject to the conditions
$$
|t|>|\Delta'/\varpi t|,\quad |t+\Delta'/\varpi t+2w'/\pi|>1.
$$

We distinguish two cases.

\smallskip

\paragraph{{\bf Case I}: $|\Delta'|\geq |w'|^2$}
Then the last integral \eqref{eqn log0} is equal to
\begin{align*}
   \Xi(x) &= \int \log|t+\Delta'/(\varpi t)|\,\eta_1(t^2)(2\log|t|-\log|\Delta'/\varpi|)\,dt\\
          &=\int \log|t|\,\eta_1(t^2)(2\log|t|-\log|\Delta'/\varpi|)\,dt,
\end{align*}
where the integrand is subject to the conditions
\[
   t|>|\Delta'/\varpi t|
   \quad\text{and}\quad
   |t+\Delta'/\varpi t+2w'/\pi|=|t|>1.
\]
The integral is equal to
\begin{align*}
\Xi(x)=\int_{|t| >\max\{1,|\Delta'/\varpi|^{1/2}\}  }\log|t|\,\eta_1(t^2)(2\log|t|-\log|\Delta'/\varpi|)\,dt.
\end{align*}
Making the substitution  $t\mapsto t^{-1}$, we obtain 
$$
\Xi(x)=\int _{|t| <\min\{1,|\Delta'/\varpi|^{-1/2}\}  } \log|t|(2\log|t|+\log|\Delta'/\varpi|)\,dt.
$$
Set $n:=1+\max \{0,[-v(\Delta'/\varpi)/2]\}>0$. Then the integral is equal to
\begin{align}\label{eqn log1}
\Xi(x)=\zeta(1)^{-1}\biggl(\sum_{i\geq  n}\bigl(-2i-v(\Delta'/\varpi)\bigr)(-i) q^{-i}\biggr)\cdot (\log q)^2.
\end{align}
For later use we tabulate the following elementary formulas:

\begin{align*}
\sum_{i\geq n}i t^{i-1} &= \frac{nt^{n-1}-(n-1)t^n}{(1-t)^2},
\\
\sum_{i\geq n}i(i+1) t^i &= \frac{t\bigl(n(n-1)t^{n+1}-2(n^2-1)t^n+n(n+1)t^{n-1}\bigr)}{(1-t)^{3}},
\\
\sum_{i\geq n}i^2 t^{i} &= \frac{(n-1)^2t^{n+2}-(2n^2-2n-1)t^{n+1}+n^2t^n}{(1-t)^3}.
\end{align*}

\smallskip

We now see that 
the integral \eqref{eqn log1} is given as follows. 

\smallskip

\begin{enumerate}
\item If $v(\Delta')> 0$, then $n=1$ and
\[
   \Xi(x) = \zeta(1)^{-1}\biggl(\sum_{i=1}^\infty\bigl(2i+v(\Delta'/\varpi)\bigr)it^i\biggr)\cdot (\log q)^2
      =  \frac{t \bigl(2 (1+ t) + v(\Delta'/\varpi)(1- t)\bigr)}{(1 - t)^2}\cdot (\log q)^2.
\]
\item If $v(\Delta')\leq 0$ is odd, then $-v(\Delta')=2n-3$ and
\[
   \Xi(x) = \zeta(1)^{-1}\biggl(\sum_{i=n}^\infty(2i-2n+2)it^i\biggr)\cdot (\log q)^2
      =  \frac{2 t^n (-n - 2 t + n t)}{(1 - t)^2}\cdot (\log q)^2.
\]
\item If $v(\Delta')\leq 0$ is even, then $-v(\Delta')=2n-2$ and
\[
   \Xi(x) = \zeta(1)^{-1} \biggl(\sum_{i=n}^\infty(2i-2n+1)it^i\biggr)\cdot (\log q)^2
      =  \frac{t^n (n + 3 t + t^2 - n t^2)}{(1 - t)^2}\cdot (\log q)^2.
\]
\end{enumerate}

\smallskip

\paragraph{{\bf Case II: $|\Delta'|<|w'|^2$.}} 

In this case  $-\lambda'/w'^2\in 1+\pi O_F$ (note that $-\Delta'=-w'^2-\lambda'$ is not a norm). Since $w'\in \pi F_0$, it follows that $-\lambda'/\varpi$ is a square and hence
the following equation has two roots
$$
t^2+2tw'/\pi+\Delta'/\varpi=0,\quad t=-w'/\pi \pm\sqrt{-\lambda'/\varpi}\in F_0.
$$
We  label $t_0$ as the unique root such that 
\begin{align}\label{eqn eq}
|t_0|=|w'/\pi|=|\lambda'/\varpi|^{1/2},
\end{align}
and label $t_1$ the other root. Then $|t_0|>|t_1|$ and
the integral (\ref{eqn log0}) is equal to
\[
   \Xi(x) = \int \log\bigl\lvert(t-t_0)(t-t_1)/t\bigr\rvert\, \eta_1\bigl((t-t_0)(t-t_1)\bigr)(2\log|t|-\log|\Delta'/\varpi|)\,dt,
\]
where the integrand is subject to 
$$
|t|>|\Delta'/\varpi t|, \quad \bigl|(t-t_0)(t-t_1)/t\bigr|>1.
$$ When $|t|>|\Delta'/\varpi t|$, we always have $|t|>|t_1|$. Hence the integral is reduced to
\begin{align}\label{eqn log3}
\Xi(x)=\int \log\lvert t-t_0\rvert \, \eta_1\bigl((t-t_0)t\bigr)(2\log|t|-\log|\Delta'/\varpi|)\,dt,
\end{align}
subject to 
\begin{align}\label{eqn ineq}
|t|>|\Delta'/\varpi t|,\quad |t-t_0|>1.
\end{align}
We break the integral up as a sum of three pieces according to whether $|t|$ is less than, greater than, or equal to $|w'/\pi|$.

\begin{lemma}
When $|t|<|w'/\pi|$, the contribution to \eqref{eqn log3} is zero.
\end{lemma}

\begin{proof}
Now we have $|t-t_0|=|t_0|$ and hence the contribution is zero unless $|w'/\pi|>1$, which we assume now. Then the contribution to (\ref{eqn log3}) is
\begin{align*}
\log|w'/\pi|\int_{|w'/\pi|>|t|>|\Delta'/\varpi t|} \,\eta_1(t_0 t)\,(2\log|t|-\log|\Delta'/\varpi|)\,dt=0,
\end{align*}
since $\eta$ is \emph{ramified}!
\end{proof}

\begin{lemma}
When $|t|>|w'/\pi|$,  the contribution to \eqref{eqn log3} is
 \begin{align}\label{eqn log4}
\int_{|t|<\min\{1,|w'/\pi|^{-1}\}}(2\log|t|+\log|\Delta'/\varpi|) \,\log|t|\,dt.
\end{align}
\end{lemma}
\begin{proof}

When $|t|>|w'/\pi|$, we have
\begin{align*}
\int_{|t|>\max\{1,|w'/\pi|\}}\log|t|\,(2\log|t|-\log|\Delta'/\varpi|)\, |t|^{-2}\,dt.
\end{align*}
Substituting $t\to 1/t$, this becomes
\begin{align*}
\int_{|t|^{-1}>\max\{1,|w'/\pi|\}}\log|t|\,(2\log|t|+\log|\Delta'/\varpi|) \,dt.
\end{align*}This completes the proof.
\end{proof}

\begin{lemma}
When $|t|=|w'/\pi|$,  the contribution to \eqref{eqn log3} is
 \begin{align*}
(\log|w'/\pi| )\,(\log|\Delta'/w'^2|)\, |w'/\pi|^{-1}\,q^{-1},
\end{align*}
when $|w'/\pi|>1$, and zero otherwise.
\end{lemma}
\begin{proof} 
The constraint $|t|>|\Delta'/\varpi t|$ in (\ref{eqn ineq}) is now superfluous. Under the assumption $|t|=|w'/\pi|=|t_0|$  (cf.~(\ref{eqn eq})), we divide the integral (\ref{eqn log3}) into two cases: $|t-t_0|<|t_0|$ and $|t-t_0|=|t_0|$.

First we show that when $|t-t_0|<|t_0|$, the contribution to (\ref{eqn log3})  is zero. Let $x=t-t_0$ and we see that the integral (\ref{eqn log3}) becomes
\begin{align*}
\int\log|x|\,\eta_1(xt_0)\,(2\log|t_0|-\log|\Delta'/\varpi|)\,dx,
\end{align*}
where $x$ satisfies conditions coming from (\ref{eqn ineq})
$$
|x|<|t_0|,\quad |x|>1.
$$
This integral is equal to
\begin{align*}
\eta_1(t_0)(2\log|t_0|-\log|\Delta'/\varpi|)\int_{1<|x|<|t_0|}\log|x|\,\eta_1(x)\,dx.
\end{align*}
Since $\eta$ is ramified, we have
$$
\int_{1<|x|<|t_0|}\log|x|\,\eta_1(x)\,dx=0.
$$

It remains to consider the contribution when $|t-t_0|=|t_0|$. We hence have $|t-t_0|=|t_0|=|t|=|w'/\pi|$ (cf.~(\ref{eqn eq})). The integral (\ref{eqn log3}) becomes
\begin{align*}
\log|w'/\pi|\,(2\log|w'/\pi|-\log|\Delta'/\varpi|) \,\int \eta_1\bigl((t-t_0)t\bigr)\,dt,
\end{align*}
subject to 
$$
|t-t_0|=|t|=|t_0|,\quad |t-t_0|=|t_0|>1.
$$The integral is zero unless $|t_0|=|w'/\pi|>1$, in which case we have
\[
   \int_{|t-t_0|=|t|=|w'/\pi|} \eta_1\bigl((t-t_0)t\bigr)\,dt
      = |w'/\pi|^{-2}\int_{|t-t_0|=|t|=|w'/\pi|} \eta(1-t_0/t)\,dt 
		= -q^{-1}|w'/\pi|^{-1}.
\]
This completes the proof.
\end{proof}

Set $n:=1+\max\{0,-v(w'/\pi)\}$.
From the last three lemmas, we find that 
the integral (\ref{eqn log3}) is given as follows.

\smallskip

\begin{altenumerate}
\item[(1)] $\,$ If $v(w'/\pi)\geq 0$, then $n=1$ and the integral \eqref{eqn log3} is equal to
\begin{align*}
   \Xi(x)
      &= \zeta(1)^{-1} \biggl(\sum_{i=1}^\infty \bigl(-2i-v(\Delta'/\varpi)\bigr)(-i)t^i\biggr)\cdot (\log q)^2 \\
      &= |u|^{-1} \frac{t \bigl(2 (1+ t) + v(\Delta'/\varpi)(1- t)\bigr)}{(1 - t)^2}\cdot (\log q)^2.
\end{align*}
\item[(2)] $\,$  If $v(w'/\pi)< 0$, then $n=-v(w'/\pi)+1$. The integral \eqref{eqn log4} is equal to $(\log q)^2$ times the factor
\begin{multline*}
   \zeta(1)^{-1}\sum_{i=n}^\infty(2i+v(\Delta'/\varpi))it^i\\
      = -\frac{t^n (-d n - 2 n^2 - 2 t - d t - 4 n t + 2 d n t + 4 n^2 t - 2 t^2 + 
         d t^2 + 4 n t^2 - d n t^2 - 2 n^2 t^2)}{(1 - t)^2},
\end{multline*}
where $d=v(\Delta'/\varpi)$. The contribution from the last lemma is 
\[
   (\log|w'/\pi| )(\log|\Delta'/w'^2|)|w'/\pi|^{-1}q^{-1}\\
      =- (n-1)(d+2n-2)t^{n}\cdot (\log q)^2.
\]
It follows that the sum (\ref{eqn log3}) is given by
   \begin{align*}&
\Xi(x)=     t^n\frac{4t + (d + 4 n-2) (1 -t)}{(1-t)^2}\cdot (\log q)^2 .
\end{align*}
\end{altenumerate}

Note that $d + 4 n-2=v(\Delta)+4v(u)-4v(w/\pi)+1$. Hence, by (\ref{eqn mu log0}),  (\ref{eqn mu log1}) and  {\bf Case I} (1), (2), (3), {\bf Case II} (1), (2), we complete the proof of Theorem \ref{thm expl x=0}.

\subsection{$x_0$-nilpotent integrals for $x_0\neq 0$}

In this subsection, we consider $x_0\in \fkb_\red\smallsetminus\fkb_{\red, \rs}$ with $x_0\neq 0$, and calculate the orbital integrals of the $x_0$-nilpotent elements in $\fks_\red$, i.e., of the elements in $\fks_\red$ mapping to $x_0$ under $\pi_\fks$, cf.~\S\ref{sec statementofgerm}. These elements are listed in \S\ref{sec germ nonzero}.1.
\begin{lemma}\label{lem orb ss u 0}
Let $x_0=(\lambda_0,0,0)\in\fkb_\red$, $\lambda_0\in F_0^\times$. Assume that $-\lambda_0\notin F_0^{\times, 2}$. Set $F'=F_0[X]/(X^2+\lambda_0)$. 
\begin{altenumerate}
\renewcommand{\theenumi}{\alph{enumi}}
\item\label{lem orb ss part i} In case $(0i)$ (i.e., $F'\not\simeq F$), 
$$
\Orb(y_+,\phi')=\eta(-\lambda_0)\Orb(y_-,\phi')=\frac{1}{2}\zeta(1)\cdot \begin{dcases}-2q^{-1}+q^{\frac{v(\lambda_0)}{2}}(1+q^{-1}),  & 2\mid v(\lambda_0),
\\
2q^{-1}(q^{\frac{v(\lambda_0)+1}{2}}-1), & 2\nmid v(\lambda_0),
\end{dcases}
$$
when $\lambda_0\in O_{F_0}$, and $\Orb(y_+,\phi')=\Orb(y_-,\phi')=0$ when $\lambda_0\notin O_{F_0}$.
\item\label{lem orb ss part ii} In case $(0ii)$ (i.e., $F'\simeq F$),  
$$\Orb(y_{-+},\phi')=\eta(-1)\Orb(y_{+-},\phi')=0 .$$
Furthermore,  
$$
\Orb(y_{++},\phi')=\eta(-1)\Orb(y_{--},\phi')=\frac{1}{2}\eta(-\alpha)\zeta(1)\cdot  \begin{dcases}-2q^{-1}+q^{\frac{v(\lambda_0)}{2}}(1+q^{-1}),  & 2\mid v(\lambda_0),
\\
2q^{-1}(q^{\frac{v(\lambda_0)+1}{2}}-1), & 2\nmid v(\lambda_0),
 \\
\end{dcases}$$ when $\lambda_0\in O_{F_0}$, and $\Orb(y_{++},\phi')=\Orb(y_{--},\phi')=0$ when $\lambda_0\notin O_{F_0}$.
\end{altenumerate}
\end{lemma}

\begin{proof}Note that $\phi'$ matches the pair $(\phi_0=\phi,\phi_1=0)$.
By Theorem \ref{thm match x non0}, we have case by case the following.  
\begin{altenumerate}
\item[(a) ]
In case $(0i)$, we have
\[
   \Orb(y_+,\phi') = \eta(-\lambda_0) \Orb(y_-,\phi')= \frac{1}{2}\Orb(y_0,\phi_0),
\]
where $y_0$ is any representative of the unique semi-simple orbit in $\fku_{0,\red}$ above $x_0\in\fkb_\red$.
\item[(b) ] In case $(0ii)$, since $ \Orb(y_\pm,\phi_1)=0$, we have 
\[
   \Orb(y_{++},\phi')=\eta(-1)\Orb(y_{--},\phi')=\frac{1}{4}\eta(-\alpha) \bigl(\Orb(y_+,\phi_0)+\Orb(y_-,\phi_0)\bigr)
\]
and
\[
   \Orb(y_{-+},\phi')=\eta(-1)\Orb(y_{+-},\phi')=\frac{1}{4}\eta(-\alpha) \bigl(\Orb(y_+,\phi_0)-\Orb(y_-,\phi_0)\bigr),
\]
where $y_\pm\in\fku_{0,\red}$ are representatives of the two semi-simple orbits in $\fku_{0,\red}$ above $x_0\in\fkb_\red$.
\end{altenumerate}

Therefore it suffices to show the following in both cases: \emph{let $y_0\in \fku_\red$ be any semi-simple element with  invariants $x_0$.
Then we have
\[
\Orb(y_0,\phi)=\zeta(1)\cdot \begin{dcases}-2q^{-1}+q^{\frac{v(\lambda_0)}{2}}(1+q^{-1}),  & 2\mid v(\lambda_0),
\\
2q^{-1}\bigl(q^{\frac{v(\lambda_0)+1}{2}}-1\bigr), & 2\nmid v(\lambda_0),
\end{dcases}
\]
when $\lambda_0\in O_{F_0}$, and $\Orb(y_0,\phi)=0$ when $\lambda_0\notin O_{F_0}$.} In particular, the independence on the choice of $y_0$ implies that $\Orb(y_+,\phi_0)=\Orb(y_-,\phi_0)$ in case $(0ii)$.

We may assume that $y_0$ is of the form   
$$
y_0=\begin{bmatrix}  0 & -\lambda_0/ \epsilon &0 \\   \epsilon   & 0 &  0\\
0& 0 & 0
\end{bmatrix},\quad  |\epsilon|=1.
$$
We use the Iwasawa decomposition \eqref{eqn iwasawa} as in the proof of Lemma \ref{lem nil orb}.
Note that $\phi$ is $K$-invariant. Then we have by definition
\[
   \Orb(y_0,\phi)=\zeta(1)\bigints_{z\in F,t\in F_0}\phi\left(
\begin{bmatrix} \epsilon t & z\ov z (-\lambda_0/ \epsilon - \epsilon  t^2)&0 \\   \epsilon  /z\ov z & -  \epsilon  t&  0\\
0& 0 & 0
\end{bmatrix}\right)dz\,dt,
\]
where the integrand is constrained by
$$
|t|\leq 1,\quad  | (z\ov z)^{-1}|\leq 1, \quad |z\ov z(-\lambda_0 - \epsilon^2  t^2)|\leq 1.
$$
Therefore we see that $|\lambda_0|\leq 1$ or the integral vanishes. Note that $-\lambda_0$ is not a square by assumption. The integral is then equal to
\begin{align}\label{eqn int1}
\int_{
|t^2|\leq |\lambda_0|,1\leq |z\ov z|\leq |\lambda_0^{-1}|}dz\,dt+\int_{
 |\lambda_0|< |t^2|\leq1,1\leq| z\ov z| \leq |t^{-2}|}dz\,dt.
\end{align}
The first integral is equal to
$$
q^{-\lfloor\frac{v(\lambda_0)+1}{2}\rfloor}(|\lambda_0|^{-1}-q^{-1}),
$$
and the second one is equal to
\begin{align*}
&\int_{|\lambda_0|< |t^2|\leq1}(|t|^{-2}-q^{-1})dt
\\=& \int_{ |\lambda_0|< |t^{-2}|\leq1} dt-q^{-1} \int_{ |\lambda_0|< |t^2|\leq1 }dt.
\end{align*}
We distinguish two cases.
\begin{altenumerate}
\item \emph{$v(\lambda_0)$ is even.} Then the integral (\ref{eqn int1}) is equal to
\[
   q^{-\frac{v(\lambda_0)}{2}}(q^{v(\lambda_0)}-q^{-1}) + ((q^{\frac{v(\lambda_0)-2}{2}}-q^{-1})-q^{-1}(1-q^{-\frac{v(\lambda_0)}{2}}))
      = -2q^{-1}+q^{\frac{v(\lambda_0)}{2}}(1+q^{-1}).
\]
\item \emph{$v(\lambda_0)$ is odd.} Then the integral (\ref{eqn int1}) is equal to
\[
   q^{-\frac{v(\lambda_0)+1}{2}}(q^{v(\lambda_0)}-q^{-1}) + ((q^{\frac{v(\lambda_0)-1}{2}}-q^{-1}) - q^{-1}(1-q^{-\frac{v(\lambda_0)+1}{2}}))
      = 2q^{-1}(q^{\frac{v(\lambda_0)+1}{2}}-1).
\]
\end{altenumerate}
Noting the factor $\zeta(1)$, 
this completes the proof.
\end{proof}

\begin{lemma}\label{lem orb ss u 1} Let $x_0=(\lambda_0,u_0,w_0)\in \fkb_\red\smallsetminus\fkb_{\red, \rs}$  with $u_0\neq 0$. Then 
$$
 \Orb(y_+,\phi')=\Orb(y_-,\phi')= \frac{1}{2}\zeta(1)\cdot\begin{dcases}2q^{-1}(q^{v(u_0)+1  }-1), & |\lambda_0|<|u_0|^2,
\\
2q^{-1}\bigl(q^{\frac{v(\lambda_0)+1}{2}}-1\bigr), & |\lambda_0|>|u_0|^2,
\end{dcases}
$$
when $\lambda_0,u_0\in O_{F_0},w_0\in O_F$, and $ \Orb(y_+,\phi')=\Orb(y_-,\phi')=0$ otherwise.

\end{lemma}
\begin{proof} By Theorem \ref{thm match x non0} (item ${\rm (c)}$, i.e., case $(1)$), we have
\begin{align*}
 \Orb(y_+,\phi')=\Orb(y_-,\phi')= \frac{1}{2}\Orb(y_0,\phi_0) ,
\end{align*}
where $y_0$ is any representative of the unique semi-simple orbit in $\fku_{0,\red}$ above $x_0$. It remains to show that 
$$
\Orb(y_0,\phi)=\zeta(1)\cdot\begin{dcases}2q^{-1}(q^{v(u_0)+1  }-1), & |\lambda_0|<|u_0|^2,
\\
2q^{-1}(q^{\frac{v(\lambda_0)+1}{2}}-1), & |\lambda_0|>|u_0|^2,
\end{dcases}
$$
when $\lambda_0,u_0\in O_{F_0},w_0\in O_F$, and $\Orb(y_0,\phi)=0$ otherwise.

First we assume $\lambda_0\neq 0$. We use the representative with invariants $(\lambda_0,u_0,w_0)$ (cf.~(\ref{eqn r=1 u}) and (\ref{eqn r=1 lambda=0}))
$$
y_0=\begin{bmatrix}0&-\lambda_0 &\pi\alpha b\\1  & 0&b\\ \cdots&\cdots& 0
\end{bmatrix},
$$
where  $\alpha^2=- \lambda_0/ \varpi, \alpha\in F_0^\times,$ and $-2b\ov b \alpha=u_0$. 
Again by the Iwasawa decomposition we arrive at 
$$
\Orb(y_0,\phi)=\zeta(1)\bigints_{z\in F,t\in F_0}\phi\left(
\begin{bmatrix}  t & z\ov z (-\lambda_0 -  t^2)& bz(\pi\alpha +t) \\  1   /z\ov z & -   t&  b/\ov z\\
\cdots& \cdots & 0
\end{bmatrix}\right)\,dz\,dt.
$$
It is easy to see that $\Orb(y,\phi)=0$ unless $\lambda_0,u_0,w_0\in O_F$ which we assume from now on. 
The integrand is constrained by
$$
|t|\leq 1,\quad  | (z\ov z)^{-1}|\leq 1,\quad  |z\ov z(-\lambda_0 -  t^2)|\leq 1,
$$
and 
$$|bz(\pi\alpha +t)|\leq 1,\quad |b/\ov z|\leq 1.$$
Note that  $-\lambda_0=\alpha^2\varpi$ is not a square in $F_0$.  We distinguish two subcases:
\begin{altitemize}
\item \emph{Case $|\alpha|\geq |u_0|$, $($i.e., $|\lambda_0|>|u_0|^2)$.} Then we have $|b|\leq 1$.
The conditions $|bz(\pi\alpha +t)|\leq 1$ and $|b/\ov z|\leq 1$  are redundant by $|b|\leq 1$, $ |z\ov z(\lambda_0 -  t^2)|\leq 1$, $ | (z\ov z)^{-1}|\leq 1$.
In this case we may simply apply the previous lemma, noting that $v(\lambda_0)$ is odd. 

\item \emph{Case $|\alpha|< |u_0|$, $($i.e., $|\lambda_0|<|u_0|^2).$} Then we have $|b|>1$. In this case the constraints are reduced to
$$
|bzt|\leq 1,\quad | z|\geq |b|, \quad|bz\pi\alpha|\leq 1.
$$
Let $v_F$ denote the normalized valuation on $F$. We calculate the orbital integral according to  $v_F(z)$, which varies from $v_F(b)$ to $-v_F(b\pi\alpha)$:
 \begin{align}\label{doubleint}
\Orb(y_0,\phi)&=\zeta(1)\int_{|b|\leq |z|\leq 1/|b\pi\alpha|} \biggl(\int_{|t|_F\leq 1/|bz|}dt\biggr)dz.
\end{align}
The double integral in \eqref{doubleint} is equal to
 \begin{align*}
&(1-q^{-1})\bigl(q^{v_F(b)}  q^{-v_F(b)}+q^{v_F(b)-1}  (q^{-v_F(b)+1} +q^{-v_F(b)+2})+ \dotsb +q^{-v_F(\alpha)/2-1}q^{v_F(b)+1+v_F(\alpha)}\bigr).
\end{align*}
We hence arrive at
 \begin{align*}
\Orb(y_0,\phi)&=\zeta(1)(1-q^{-1})\Bigl(1+1+q+q+\dotsb+q^{v_F(b)+\frac{v_F(\alpha)}{2}}+q^{v_F(b)+\frac{v_F(\alpha)}{2}}\Bigr)
\\&=2\zeta(1)(1-q^{-1})\frac{q^{v_F(b)+\frac{v_F(\alpha)}{2}+1}-1}{q-1}
\\&=2\zeta(1)q^{-1}(q^{v(u_0)+1}-1).
\end{align*}
Here we used that $v_F(u_0)=v_F(\alpha)+2v_F(b)$ and $v_F(u_0)=2v(u_0)$.
\end{altitemize}

The case $\lambda_0=0$ is similar to the case $|\lambda_0|<|u_0|^2$ above, and we omit the details. This completes the proof.
\end{proof}

\subsection{Germ expansion of $\del_1\bigl(\sigma(x),\phi'\bigr)$  around $x_0\neq 0$}
The following theorem together with the values in Lemma \ref{lem orb ss u 0} and Lemma \ref{lem orb ss u 1} give the germ expansion around nonzero elements $x_0\in\fkb_\red\smallsetminus\fkb_{\red, \rs}$. We use the classification of such elements in \S\ref{sec germ nonzero}.1. 
\begin{theorem}\label{thm expl x neq 0}
Let $x_0=(\lambda_0,u_0,w_0)\in\fkb_\red\smallsetminus\fkb_{\red, \rs}$ be a nonzero element. Assume that, if $u_0=w_0=0$ then $-\lambda_0\notin F_0^{\times, 2}$, and set $F'=F_0[X]/(X^2+\lambda_0)$.
Let  $x=(\lambda,u,w)\in\fkb_{\red,\rs,1}$ be in a small neighborhood of $x_0$. 
\begin{altenumerate}
\renewcommand{\theenumi}{\alph{enumi}}
\item In case $(0i)$ (i.e., $u_0=w_0=0$ and $F'\not\simeq F$), 
$$
   \del_1\bigl(\sigma(x),\phi'\bigr)=\eta(-\lambda)\Orb(y_-,\phi')\log|\Delta(x)|+C_1;
   $$
\item In case $(0ii)$ (i.e., $u_0=w_0=0$ and  $F'\simeq F$),  
$$
  \del_1\bigl(\sigma(x),\phi'\bigr)=\eta(-1)\Orb(y_{--},\phi')\log|\Delta(x)|+C_2;
 $$
\item In case $(1)$ (i.e., $u_0\neq 0$), 
$$
  \del_1\bigl(\sigma(x),\phi'\bigr)=\Orb(y_{-},\phi')\log|\Delta(x)/u(x)^2|+C_3;
 $$
\end{altenumerate}
where $C_1,C_2, C_3$ are constants (depending on $x_0$ only, independent of the choice of $\phi'$). Here  $\sigma(x)$ is the section defined by \eqref{eqn sigma} in cases $(0i)$ and $(1)$, and by \eqref{eqn sigma 1} in case $(0ii)$.

\end{theorem}

\begin{proof}We simply denote $\Delta$ for $\Delta(x)$. Note that $\eta(-\Delta)=-1$ for $x\in \fkb_{\red,\rs,1}$ (cf.~Proposition \ref{eta b01}). We apply Theorem \ref{thm germ ss s} and use the notation in its statement. Case $(0i)$ follows from the fact $|\lambda|$ is a constant when $x$ is near $x_0$. In case $(0ii)$, by Lemma \ref{lem orb ss u 0}, we have $\Orb(y_{-+},\phi')=\Orb(y_{+-},\phi')=0$. This case then follows easily by $\eta(z_1z_2)=\eta(\Delta)$ and since  $|z_1z_2/\Delta|$ is a constant when $x$ is near $x_0$.  Case $(1)$ follows from the fact that $|u(x)|$ is a constant when $x$ is near $x_0$.
\end{proof}

\section{Comparison}\label{section:comparison}

\subsection{Statement of the theorem}
We denote by $\varphi$ the function on $\fkb_{\red,\rs}$, 
\begin{equation}\label{eqn varphi4 12.1}
   \varphi(x):= 
	\begin{cases}
		2\omega(y)\del(y,\phi')+\lInt(x)\cdot\log q,  &  y\in \fks_\red,\ \pi_{\fks}(y)=x\in \fkb_{\red,\rs,1};\\
		0,  &  x\in \fkb_{\red,\rs,0}.
   \end{cases}
\end{equation}
We define $\varphi_1$ to be the analogous function where we replace $\del$ by $\del_1$, and we define  $\varphi_2$ to be the function
\[
   \varphi_2(x):= 
	\begin{cases}
		2\omega(y)\del_2(y,\phi'),  &  y\in \fks_\red,\ \pi_\fks(y)=x\in \fkb_{\red,\rs,1};\\
      0,  &  x\in \fkb_{\red,\rs,0}.
\end{cases}
\]
Hence
\begin{equation}\label{eqn varphi=1+2}
\varphi=\varphi_1+\varphi_2.
\end{equation}
To prove Theorem \ref{mainforred}, it suffices to show that both $\varphi_1$ and $\varphi_2$ are orbital integral functions.
We now show that  $\varphi_1$ is an orbital integral function. By Corollary \ref{cor germ converse}, it suffices to show the following result.

\begin{theorem}\label{locconsthm}
For every $x_0\in \fkb_\red$, there exists an open neighborhood $V_{x_0}$ of $x_0$ such that $\varphi_1|_{V_{x_0}\cap \fkb_{\red,\rs, 1}}$ is a constant function. 
\end{theorem}

The only non-trivial case is when $x_0$ lies in the closure of $ \fkb_{\red,\rs, 1}$, but not in $ \fkb_{\red,\rs, 1}$ itself. So, let us assume this. The proof of Theorem \ref{locconsthm} will occupy the rest of this section. We first treat the case $x_0=0$ and then move on to the case where $x_0\neq 0$, in the order appearing in the table in \S\ref{sec 13}. 

Before proceeding, we recall from \eqref{remember valu} that
$$
v_D(u)=2m,\quad v_D(w)=2m+\ell_+,\quad v_D(\Delta)=4m+2\ell_- ,
$$
and
\begin{equation*}
v(\lambda)=\min\{\ell_+,\ell_-\}.
\end{equation*}
The last identity follows because in the expression $\lambda={\rm N}\alpha'_+ +{\rm N}\alpha'_-$, when the valuations of both summands are identical, there can be no cancellation.

\subsection{The case $x_0=0$}
Theorem \ref{thm expl x=0} calculates $\omega(y)\del_1(y,\phi')$ for $y=\sigma(x)$ the explicit section introduced in \eqref{eqn sigma} below. Recall  that in Theorem \ref{thm expl x=0}, the valuation is taken to be the normalized valuation for $F_0$ and therefore $v_D(\,\cdot\,)=2v(\,\cdot\,)$.

We compare the formulas in  Theorem \ref{thm expl x=0} and the formulas for $\Int(x)$ in \S\ref{sec expl Lie}. We see that the case distinctions in both formulas are identical; we furthermore see case-by-case that, when $x\in \fkb_{\red,\rs,1}$ is close to $0$,
$$
 \varphi_1(x)=  \frac{4t(t-3)}{(1-t)^2}\cdot\log q
$$
is  constant. Here $t=q^{-1}$ for the rest of this section. Hence Theorem \ref{locconsthm} follows in this case.

\subsection{The case $x_0\neq 0$.}
The theorem holds trivially if $x_0$ is not integral, so we assume from now on that $x_0$ is integral. We first consider the case $x_0=(\lambda_0,0,0)$, $\lambda_0\neq 0$, and derive from \S\ref{sec expl Lie} a convenient expression for the quantity $\lInt(x)$ in a small neighborhood of $x_0$.

\begin{lemma}
\label{lem int case r=0}
Let $x_0=(\lambda_0,0,0)$, $\lambda_0\neq 0$. For $x\in \fkb_{\red,\rs,1}$ in a small neighborhood of $x_0$, 
\[
   \lInt(x) =
	\begin{dcases}v(\Delta/\lambda)\frac{t^{-v(\lambda_0)/2}(1+t)-2t}{1-t}+C_1\bigl(v(\lambda_0)\bigr),& 2\mid v(\lambda_0), \\
   v(\Delta/\lambda)  \frac{2t(t^{-(v(\lambda_0)+1)/2}-1)}{1-t}+C_2\bigl(v(\lambda_0)\bigr),& 2 \nmid v(\lambda_0),
   \end{dcases}
\]
where $C_1(n)$ and $C_2(n)$  are explicit polynomials in $n$. 
\end{lemma}

\begin{proof}
For $x\in\fkb_{\red,\rs,1}$ near $x_0$, we may assume that $|\lambda|=|\lambda_0|$ is fixed, while $v_D(u)=2m$ and $v_D(w)=2m+\ell_+$ are very large. Hence we may assume that $m$ is larger than at least one of $\ell_+$ or $\ell_-$. To make the comparison, 
we will rewrite $\lInt(x)$ from {\bf Case I} $(2),(3)$ and  {\bf Case II} $(2)$ in \S\ref{sec expl Lie}.

\begin{altitemize}
\item If $v(\lambda_0)$ is even (and so is $v(\lambda)$),  then the minimum among $\ell_+,\ell_-$ must be $\ell_-$ since $\ell_+$ is odd. Hence we may assume $\ell_-=\ell_{-0}=v(\lambda_0)$, and we are in case {\bf Case I} $(3)$. Hence  the intersection number is given by
\begin{align*}
   \lInt(x) 
	   &=2 t^{-\ell_-/2}\frac{(m-\ell_-/2+1)(1-t^2)+t(t+3)}{(1-t)^2}
                 +\frac{-2(\ell_-+2m+1)t}{1-t} +\frac{-8t}{(1-t)^2}\\
      &= 2m \frac{t^{-\ell_{-}/2}(1+t)-2t}{1-t}\\
		&\phantom{=}\qquad +2t^{-\ell_-/2}\frac{(-\ell_-/2+1)(1-t^2)+t(t+3)}{(1-t)^2}-\frac{2(\ell_-+1)t}{1-t} +\frac{-8t}{(1-t)^2}\\
		   &= v(\Delta/\lambda)\frac{t^{-v(\lambda_0)/2}(1+t)-2t}{1-t}+C_1\bigl(v(\lambda_0)\bigr) ,
\end{align*}
where $C_1(n)$ is an explicit function of $n$ defined by the last equality.

\item If $v(\lambda_0)$ is odd  (and so is $v(\lambda)$), we are in  {\bf Case I} $(2)$ or {\bf Case II} $(2)$, depending on whether $\ell_-\leq \ell_+$ or not (equivalently, depending on whether $|\Delta|\geq |w|^2$ or not). In other words,  this gives a partition of the intersection of  a neighborhood of $x_0$ with the regular semi-simple set as a disjoint union of two sets.

\smallskip

In {\bf Case I} $(2)$ ($\ell_-\leq \ell_+$, $\ell_-\leq 2m$ and $\ell_-$  odd), we have 
$$
v_D(\lambda)=2\ell_-,\quad v_D(\Delta/\lambda)=4m ,
$$
and we may assume that $\ell_-=\ell_{-0}=v(\lambda_0)$. Then we have
\begin{align*}
\lInt(x)&=2t^{-(\ell_--1)/2}\frac{(2m-\ell_-+3)-(2m-\ell_--1)t}{(1-t)^2}+\frac{-2(2m+\ell_-+1)t}{1-t}+\frac{-8t}{(1-t)^2} 
\\
&= 4m \frac{t(t^{-(\ell_-+1)/2}-1)}{1-t}+2t^{-(\ell_--1)/2}\frac{(-\ell_-+3)+(\ell_-+1)t}{(1-t)^2}-\frac{2(\ell_-+1)t}{1-t}+\frac{-8t}{(1-t)^2}
\\&=v(\Delta/\lambda)  \frac{2t(t^{-(v(\lambda_0)+1)/2}-1)}{1-t}+C_2\bigl(v(\lambda_0)\bigr),
\end{align*}
where $C_2(n)$ is an explicit function of $n$ defined by the last equality.

\smallskip

In {\bf Case II} $(2)$ ($\ell_->\ell_+$, $\ell_+<2m$),
we have 
$$
v_D(\lambda)=2\ell_+,\quad v_D(\Delta/\lambda)=4m+2\ell_--2\ell_+.
$$
Hence 
\begin{align*}
   \lInt(x)
	   &= 2t^{-(\ell_+-1)/2}\frac{(\ell_--2\ell_++2m+3)(1-t)+4t}{(1-t)^2} + \frac{-2(\ell_-+2m+1)t}{1-t} + \frac{-8t}{(1-t)^2}\\
      &= 2(2m+\ell_--\ell_+) \frac{t(t^{-(\ell_++1)/2}-1)}{1-t}\\
		&\phantom{=}\qquad + 2t^{-(\ell_+-1)/2}\frac{(-\ell_++3)(1-t)+4t}{(1-t)^2} - \frac{2(\ell_++1)t}{1-t} + \frac{-8t}{(1-t)^2}\\
		&= 2(2m+\ell_--\ell_+) \frac{t(t^{-(\ell_++1)/2}-1)}{1-t}\\
		&\phantom{=} \qquad + 2t^{-(\ell_+-1)/2}\frac{(-\ell_++3)+(\ell_++1)t}{(1-t)^2}-\frac{2(\ell_++1)t}{1-t}+\frac{-8t}{(1-t)^2}
\\&=v(\Delta/\lambda)  \frac{2t(t^{-(v(\lambda_0)+1)/2}-1)}{1-t}+C_2\bigl(v(\lambda_0)\bigr),
\end{align*}
where $C_2(n)$ is the same function of $n$ as in the last case (this is crucial).\qedhere
\end{altitemize}
\end{proof}

We return to the proof of Theorem \ref{locconsthm} for $x_0=(\lambda_0,0,0)$, $\lambda_0\neq 0$.  Note that $-\lambda_0\notin F_0^{\times, 2}$ since $x_0$ is in the image of $\fku_{1,\red}$, cf.~Lemma \ref{intimage}. Hence we can apply Theorem \ref{thm expl x neq 0}.
We have the two subcases (a) and (b) of that theorem. 
\begin{altenumerate}
\item[(a) ]
\emph{ Case $(0i)$}: $F'\not\simeq F$. 
By Theorem \ref{thm expl x neq 0} and Lemma \ref{lem orb ss u 0}, we have the values of the orbital integrals $\del_1(\sigma(x),\phi')$. 
By comparison with Lemma \ref{lem int case r=0}, we find that 
\[
   \varphi_1(x)= \Bigl(C_1+C_i\bigl(v(\lambda_0)\bigr)\Bigr)\cdot\log q,\quad x\in\fkb_{\red,\rs,1},
\]
is a constant,
where $C_1$ is the constant in Theorem \ref{thm expl x neq 0}, and $C_i(n)$ is the polynomial of $n$ for $i$ with the same parity as $v(\lambda_0)$.

\item[(b) ]
\emph{Case $(0ii)$: $F'\simeq F$} (so that $v(\lambda_0)$ is odd by $F'=F_0[\sqrt{-\lambda_0}]\simeq F_0[\sqrt{\varpi}]$). Similarly we find by Theorem \ref{thm expl x neq 0} and Lemma  \ref{lem orb ss u 0}, and by comparing with   Lemma \ref{lem int case r=0}, that 
$$
\varphi_1(x)=\Bigl(C_2+C_1\bigl(v(\lambda_0)\bigr)\Bigr)\cdot\log q,\quad x\in\fkb_{\red,\rs,1},
$$is a constant,
where $C_2$ is the constant in Theorem \ref{thm expl x neq 0},  and $C_1(n)$ is the polynomial of $n$  in Lemma \ref{lem int case r=0}.

 Hence Theorem \ref{locconsthm} is proved for $x_0=(\lambda_0, 0, 0)$. 
\end{altenumerate}

We now consider the case $x_0=(\lambda_0,u_0,w_0)$, $u_0\neq 0$. Again, we first derive a convenient expression for the quantity $\lInt(x)$ in a small neighborhood of $x_0$.

\begin{lemma}
\label{lem int case r=1}
Let $x_0=(\lambda_0,u_0,w_0)$, where $u_0\neq 0$. Then for $x\in \fkb_{\red,\rs,1}$ in a small neighborhood of $x_0$, 
\begin{align*}
   \lInt(x) =
	\begin{dcases} 2v(\Delta/u^2)  \frac{t(t^{-v(u_0)-1}-1)}{(1-t)} + C_1\bigl(v(u_0)\bigr),  &  |\lambda|<|u|^2, \\
2v(\Delta/u^2)\frac{t(t^{-(v(\lambda_{0})+1)/2}-1)}{(1-t)}   +C_2\bigl(v(\lambda_{0}),v(u_0)\bigr) ,  &  |\lambda|>|u|^2,
   \end{dcases}
\end{align*}
where $C_1,C_2$  are explicit polynomials.
\end{lemma}

\begin{proof} First we assume that $\lambda_0\neq 0$. Then for $x$ near $x_0$, we have $|\lambda|=|\lambda_0|$, $|u|=|u_0|,|w|=|w_0|$, and $\ell_-$ is very large. 
\begin{itemize}
\item
In {\bf Case II} $(1)$( $\ell_->\ell_+$, $\ell_+>2m$, i.e., $|\lambda|<|u|^2$) we have 
\begin{align*}
\lInt(x)&=2t^{-m}\frac{2(1+t)+(\ell_--2m-1)(1-t)}{(1-t)^2}+\frac{-2(\ell_-+2m+1)t}{1-t}+\frac{-8t}{(1-t)^2}
\\&= 2\ell_- \frac{t(t^{-m-1}-1)}{(1-t)} +2\left(t^{-m}\frac{2(1+t)+(-2m-1)(1-t)}{(1-t)^2}+\frac{-(2m+1)t}{1-t}+\frac{-4t}{(1-t)^2}\right)
\\&= 2\ell_- \frac{t(t^{-m_0-1}-1)}{(1-t)}+C_1(m_0).
\end{align*}

\item

In {\bf Case II} $(2)$ ($\ell_->\ell_+$, $\ell_+<2m$, i.e., $|\lambda|>|u|^2$) we have
\begin{align*}
\lInt(x)
&=2t^{-(\ell_+-1)/2}\frac{(\ell_--2\ell_++2m+3)(1-t)+4t}{(1-t)^2}+\frac{-2(\ell_-+2m+1)t}{1-t}+\frac{-8t}{(1-t)^2}
\\&=2\ell_- \left(t^{-(\ell_{+0}-1)/2}\frac{1}{(1-t)}-\frac{t}{1-t} \right)+C_2(\ell_{+0},m_0)
\\&=2\ell_-\frac{t(t^{-(\ell_{+0}+1)/2}-1)}{(1-t)}   +C_2(\ell_{+0},m_0) .
\end{align*}

\end{itemize}
Now we assume that $\lambda_0=0$. Then, since $0\neq x_0\notin\fkb_{\red, \rs}$, it follows that $x_0$ has the form $x_0=(0,u_0,0)$, $u_0\neq 0$. Then for $x$ near $x_0$ we have $|u|=|u_0|$ and $\ell_+,\ell_-$ are very large. Then the asymptotic behavior of $\lInt(x)$ follows from {\bf Case I} $(1)$ and {\bf Case II} $(1)$ in \S\ref{calc of l-Int}, 
\[
\lInt(x) = 2\ell_- \frac{t(t^{-m_0-1}-1)}{(1-t)} +C_1(m_0).\qedhere
\]
\end{proof}
Now we can finish the proof of Theorem  \ref{locconsthm} for $x_0=(\lambda_0, u_0, w_0)$ with $u_0\neq 0$. We are in case (c) of Theorem \ref{thm expl x neq 0}.  Noting that $v(\Delta/u^2)=\ell_-$,
by comparing Lemma \ref{lem int case r=1} with Theorem \ref{thm expl x neq 0} and Lemma \ref{lem orb ss u 1},  we find that
$
\varphi_1(x)
$ is a constant (explicitly depending on $x_0$) when $x\in \fkb_{\red,\rs,1}$ is near $x_0$.

\subsection{Completion of the proofs of Theorem \ref{locconsthm} and Theorem \ref{mainforred}}\label{subsection pf 12.1}
In view of the table in \S\ref{sec 13} (and the explanation), we have considered all $x_0\in \fkb_{\red}\smallsetminus \fkb_{\red,\rs}$  in the closure of $ \fkb_{\red,\rs, 1}$. This completes the proof of Theorem \ref{locconsthm}.  
This implies that the function $\varphi_1$ in \eqref{eqn varphi=1+2} is an orbital integral function. By Theorem \ref{thm error term}, $\varphi_2$ is an orbital integral function. It follows that the function $\varphi$ defined by \eqref{eqn varphi4 12.1} is an orbital integral function. This completes the proof of Theorem \ref{mainforred}.

\part{Germ expansion}\label{germ expansion part}

In this part of the paper, we assume that $n=3$ and that $F/F_0$ is any quadratic extension of non-archimedean local fields of characteristic not equal to $2$ (not necessarily ramified nor of odd residue characteristic). We write $F=F_0[\pi]$ for $\pi=\sqrt{\varpi}$, $\varpi\in F_0^\times$.

\section{Statement of the germ expansion}\label{sec statementofgerm}
Recall from \S\ref{subsec inv s} we have
\[
   \begin{gathered}
	\pi_\red\colon
	\xymatrix@R=0ex{
	   \fks_\red \ar[r]  &  \fkb_\red = \BA\times\BA\times\fks_1\\
	y \ar@{|->}[r]  &  \bigl(\lambda(y), u(y), w(y)\bigr)
	}
	\end{gathered}
\]	
where we write 
$y$ in the block form
$$
y=\begin{bmatrix}A & {\bf b} \\
 {\bf c}   & 0
\end{bmatrix},
$$and $$ \lambda(y)=\det A ,\quad u(y)=\varpi^{-1}{\bf c} {\bf b},\quad \text{and}\quad w(y)=\varpi^{-1}{\bf c} A{\bf b} .$$
We  also have 
$$
\Delta(y)=\lambda(y)u(y)^2+w(y)^2.
$$

Let $y$ be \emph{any} element in $\fks_\red$ (not necessarily semi-simple nor regular). We say that $y$ is \emph{relevant} if its stabilizer $H'_y$ is contained in $\SL_2$. For a relevant element $y$, the determinant $\det$ is well-defined on $H_y'\bs H'$, and so is $\lvert\det\rvert^s$ for $s\in\BC$.
 Let $\phi'\in C^\infty_c(\fks_\red)$,  and let $y\in \fks_\red$ be relevant. We consider the integral 
\begin{equation}
\label{eqn def O(s)}
   \Orb(y,\phi',s)= \tau(H_y')\int_{H_y'\bs H'}\phi'(h^{-1} yh)\eta(\det h)\lvert \det h \rvert^s \,dh,
\end{equation} 
where $ \tau(H_y')=\vol(H_y')$ if $H_y'$ is compact and $ \tau(H_y')=1$ otherwise. 
In all cases except the one in Lemma \ref{lem orb F' split}, the integral \eqref{eqn def O(s)} is absolutely convergent when $\Re(s)\gg 0$ and extends to a meromorphic function of $s\in\BC$. Even in the exceptional case,  Lemma \ref{lem orb F' split} defines $\Orb(y,\phi',s)$ as a meromorphic function. When the integral has no pole at $s=0$, we use the notation  
\[
\Orb(y,\phi'):=\Orb(y,\phi',0).
\]

For $x_0\in\fkb_\red$, the  elements of $\fks_\red$ in $ \pi_\fks^{-1}(x_0)$ will be called   \emph{$x_0$-nilpotent}. When $x_0=0$, we use the term  \emph{nilpotent} instead of $0$-nilpotent.  

\begin{theorem}\label{thm germ s}
Let $\phi'\in C^\infty_c(\fks_\red)$ and $x_0\in \frak b_\red$. There exist an open neighborhood $V_{x_0}$ of $x_0$, a section $\sigma\colon\fkb_{\red,\rs}\to \fks_\red$ defined on $V_{x_0}\cap \frak b_{\rm red, rs}$, and (explicit) continuous functions $\Gamma_n(x, s)$   on $V_{x_0}\cap \frak b_{\rm red, rs}$ such that, as meromorphic functions in the complex variable $s\in \BC$,
$$
\Orb\bigl(\sigma(x),\phi',s\bigr) = \sum_{n\in \pi_\fks^{-1}(x_0)/H'}\Gamma_n(x, s)\Orb(n,\phi',s),
$$
where the sum runs over an explicit set of representatives of relevant $H'$-orbits of $x_0$-nilpotent elements, and where the sum should be replaced by an integral with respect to a suitable measure when there is a continuous family of orbits $n(\mu)$, $\mu\in F_0$ (which only occurs when $x_0=0$).
\end{theorem}

\begin{proof}
When $x_0\in \fkb_{\red, \rs}$,
by (\ref{eqn orb s3}) below it is easy to see that we have 
$$
\Orb\bigl(\sigma(x),\phi',s\bigr) = \Orb\bigl(\sigma(x_0),\phi',s\bigr),
$$
when $x$ is near $x_0$, which proves Theorem \ref{thm germ s} in this case. 
When $x_0\in \fkb_\red\smallsetminus\fkb_{\red, \rs}$, Theorem \ref{thm germ s} will  follow from the explicit germ expansion given by Theorem \ref{thm germ x=0}   for the case $x_0=0$, and by Theorem \ref{thm germ ss s} for $x_0\neq 0$.
\end{proof}

We also need a converse to the theorem above, specialized to $s=0$, proved in \S \ref{sec proof}. Let $\CC_1(F_0)$ be the space defined in \S\ref{sec 19.3}.
\begin{theorem}
\label{thm germ converse}
Let $\varphi\in C_\rc^\infty(\fkb_{\red,\rs})$. Then $\varphi$ is an orbital integral function if and only if, for every $x_0 \in \fkb_{\red} $, there exists an open neighborhood $V_{x_0}$ of $x_0$, such that
\begin{align}
 \label{eqn varphi}
\varphi(x)= \omega\bigl(\sigma(x)\bigr)\sum_{n\in \pi_\red^{-1}(x_0)/H'}\Gamma_n(x, 0) \varphi_{x_0}(n)
\quad\text{for all}\quad
 x\in V_{x_0}\cap  \fkb_{\red,\rs},
\end{align}
where $\varphi_{x_0}(n)\in\BC$;
 when $x_0=0$, the sum is to be interpreted as an integral for the one-dimensional family of nilpotent orbits $n(\mu),\mu\in F_0$, and the function $\mu\mapsto \varphi_{x_0}(n(\mu))$ is required to define an element in $\CC_1(F_0)$; when $x_0=(\lambda_0,0,0)$ with $F_0[\sqrt{-\lambda_0}]\simeq F$, this is understood as a function of the form $\phi_0\log|\Delta(x)|+\phi_1$ for constants $\phi_0$ and $\phi_1$ (cf.~(\ref{eqn F' split})).
\end{theorem}

\begin{corollary}
\label{cor germ converse}
Let $\varphi\in C_\rc^\infty(\fkb_{\red,\rs})$ be such that the restriction of $\varphi$ to $V_{x_0}\cap  \fkb_{\red,\rs,0}$ is zero. Assume that, for every $x_0 \in \fkb_{\red} $, there exists an open neighborhood $V_{x_0}$ of $x_0$, such that
the restriction of $\varphi$ to $V_{x_0}\cap  \fkb_{\red,\rs,1}$ is  constant. Then $\varphi$ is an orbital integral function.
\end{corollary}

\begin{proof}
It suffices to verify that such $\varphi$ is of the form (\ref{eqn varphi}) in Theorem \ref{thm germ converse} for every $x_0\in \fkb_{\red}$. We first consider $x_0=0$. In this case the nilpotent orbits in (\ref{eqn varphi}) are as in Theorem \ref{thm germ x=0}. By Theorem \ref{thm germ x=0}  the function $\Gamma_{n_{0,-}}(x,0)|_{V_{x_0}\cap  \fkb_{\red,\rs,0}}=-\Gamma_{n_{0,-}}(x,0)|_{V_{x_0}\cap  \fkb_{\red,\rs,1}}$ is a (nonzero) constant and $\Gamma_{n_{0,+}}(x,0)|_{V_{x_0}\cap  \fkb_{\red,\rs}}$ is a (nonzero) constant. Therefore, by suitably choosing $\varphi_{x_0}(n_{0,-})$ and $\varphi_{x_0}(n_{0,+})$, we see that $\varphi$ is of the form 
$$
\Gamma_{n_{0,+}}(x,0)\varphi_{x_0}(n_{0,+})+\Gamma_{n_{0,-}}(x,0)\varphi_{x_0}(n_{0,-}).
$$
Setting $\varphi_{x_0}(n)$ to be zero for $n$ in the one-dimensional family of nilpotents, we see that $\varphi$ is of the form (\ref{eqn varphi}) for $x_0=0\in \fkb_{\red}$. 

For $x_0\neq 0$, the proof is similar using Theorem \ref{thm germ ss s}.
\end{proof}

We will be interested in the first derivative of the orbital integral $\partial \Orb(\sigma(x),\phi') $ at $s=0$ in a neighborhood of $x_0$. We have a decomposition according to the Leibniz rule
\begin{align}\label{eqn leibniz}
\del\bigl(\sigma(x),\phi'\bigr)=\del_1\bigl(\sigma(x),\phi'\bigr)+\del_2\bigl(\sigma(x),\phi'\bigr),
\end{align}
where we define the two terms as
$$\del_1\bigl(\sigma(x),\phi'\bigr) := \sum_{n\in \pi_\red^{-1}(x_0)/H'} \biggl(\frac d{ds} \Big|_{s=0} \Gamma_n(x, s) \biggr)\Orb(n,\phi',0)$$
and
$$\del_2\bigl(\sigma(x),\phi'\bigr) := \sum_{n\in \pi_\red^{-1}(x_0)/H'}\Gamma_n(x, 0)\biggl(\frac d{ds} \Big|_{s=0} \Orb(n,\phi',s)\biggr).$$
\begin{remark}\label{rem del1}
We point out that the first term $\del_1(\sigma(x),\phi')$ depends only on the quantities $\Orb(n,\phi')=\Orb(n,\phi',0)$ (and the intrinsically defined functions $\Gamma_n(x,s)$). In particular, the values of $\Orb(\sigma(x),\phi'),$ for regular semi-simple $x$ already determine $\del_1(\sigma(x),\phi')$. This observation does  not hold for $\del_2(\sigma(x),\phi')$.
\end{remark}
The explicit germ expansion also shows the following result, proved in \S \ref{sec proof} below.

\begin{theorem}
\label{thm error term}
Fix $i\in \{0,1\}$.  Let $\phi'\in C^\infty_c(\fks_\red)$ be a smooth transfer of $(\phi,0)$ where $\phi\in C^\infty_c(\fku_{i,\rm red})$.
Then the function $$\varphi(x)=\begin{cases} \omega\bigl(\sigma(x)\bigr)\del_2\bigl(\sigma(x),\phi'\bigr), & x\in \fkb_{\red,\rs, 1-i};
\\ 0,& x\in \fkb_{\red,\rs, i}
\end{cases}
$$ is an orbital integral function.
\end{theorem}

\section{Germ expansion around $x_0=0$}\label{germzeropart4}
In this section, we give the explicit germ expansion around $x_0\in\fkb_\red$ when $x_0=0$. 
\subsection{Statement of the theorem}
The nilpotent orbits in $\fks_\red$ are classified in \cite[\S2.1]{Z12b}. We  list here only the relevant orbits. 
\begin{altitemize}
\item
a continuous family with representatives 
\begin{align}\label{nilpins}
 n(\mu):=\pi \begin{bmatrix}0 & \mu & 1 \\
 0& 0 &   0\\  0 & 1 & 0
\end{bmatrix}, \quad \mu\in F_0
\end{align}
with stabilizer $N$, the upper-triangular unipotent matrices, and

\item
 two regular (i.e., with trivial stabilizer) nilpotents with representatives
\begin{align}\label{nilpins pm}
n_{0,+}=\pi
\begin{bmatrix} 0  &1  & 0\\
 0 &   0 &1\\ 0&0 &0
\end{bmatrix}, \quad n_{0,-}=^t\!\!n_{0,+}.\end{align}
\end{altitemize}
The nilpotent orbital integrals for the continuous family $n(\mu)$ and for the regular nilpotent orbits $n_{0,\pm}$ are defined by \eqref{eqn def O(s)} and are both   holomorphic
at $s=0$, cf.\ \cite[Lem.~2.1]{Z12b}. Note that we  choose a Haar measure on $N(F_0)$ by transporting the one on $F_0$.

  We define a section 
$$\sigma\colon\fkb_{\red}\to\fks_{\red}$$
by
\begin{align}\label{eqn sigma}
\sigma(x)=\pi\begin{bmatrix}0 &- \lambda/\pi^2 & 1\\
1 &  0 &0 \\ u& w/\pi& 0 \end{bmatrix},\quad  x=(\lambda,u,w)\in \fkb_\red .
\end{align}

\begin{theorem}\label{thm germ x=0}
For $x_0=0$ and $x=(\lambda,u,w)\in \fkb_{\red, \rs}$ near $x_0$, there is a germ expansion of $\Orb\bigl(\sigma(x),\phi',s\bigr)$ as
\[
\int_{F_0}\Gamma_{n(\mu)}(x, s)\Orb\bigl(n(\mu),\phi',s\bigr)\, d\mu
 +\eta(\Delta/\varpi)|\Delta/\varpi|^s    \Orb(n_{0,-},\phi',s)+\eta(-1)    \Orb(n_{0,+},\phi',s),
\]
where 
\[
   \Gamma_{n(\mu)}(x,s) =
   \begin{dcases}
      0,   &\text{ if } (u^2\mu-2w/\pi)^2-4\Delta /\varpi\notin F_0^{\times, 2};\\
      \frac{\eta(-\nu)\bigl(|\nu|^{-s}+\eta(\Delta/\varpi) |\Delta/\varpi|^{-s}\nu^{s}\bigr)}
	          {\bigl|(u^2\mu-2w/\pi)^2-4\Delta /\varpi \bigr|^{1/2}},
           &\text{ if } (u^2\mu-2w/\pi)^2-4\Delta /\varpi\in  F_0^{\times, 2} .
   \end{dcases}
\] 
Here $\nu$
denotes one of the two roots of 
$$
u^2\mu=\nu+\frac{\Delta/\varpi}{\nu}+2w/\pi 
$$
( $\Gamma_{n(\mu)}(x,s)$ is independent of the choice of $\nu$).
\end{theorem}

\begin{corollary}\label{cor germ=0}
If $x\in\fkb_{\red,\rs,1}$, then
$$\Gamma_{n(\mu)}(x, 0)=0.
$$ 
\end{corollary}

\begin{proof}
	When $x\in\fkb_{\red,\rs,1}$ we have 
$\eta(\Delta/\varpi)=-1$, cf.~Proposition \ref{eta b01}. The result follows from the formula above.
\end{proof}

The proof of this theorem will occupy the rest of this section. We will rely on \cite[\S3]{Z12b}.
 At this point we warn the reader that we will use slightly different notation from \cite{Z12b}.
Let $\fks\fkl_\red$ be the subspace $\pi^{-1}\fks_\red$ of $\fkg\fkl_3=\M_{3,F_0}$.
We fix an isomorphism as representations of $H'=\GL_{2}$,
\begin{equation*}
	\vcenter{
	\xymatrix@R=0ex{
	   \fks_\red \ar[r]^-\sim  &  \fks\fkl_\red\\
	   x \ar@{|->}[r]  &  \pi^{-1}x.
	} 
	}
\end{equation*}

\subsection{Proof of Theorem \ref{thm germ x=0}}
 The proof of Theorem \ref{thm germ x=0} is essentially the same as that of Theorem 2.7 in \cite{Z12b} except for some notational changes. We often write $h\cdot x=h^{-1}xh$, for $h\in H'$ and $x\in\fks_\red$.
 To be consistent with \cite{Z12b}, for $x\in\fkb_{\rs}$ we  rewrite (\ref{eqn def O(s)}) as
\[
\label{eqn def O(s)1}
\Orb(x,\phi',s)= \int_{H'/H_x'}\phi'(h xh^{-1})\eta(\det h)\lvert \det h \rvert^{-s} \,dh.
\]
 We use the Iwasawa decomposition $H'(F_0)=KAN$ for $K=\SL_2(O_{F_0})$,
$$h=
k\begin{bmatrix}b& \\
  &  a^{-1}b\end{bmatrix}
\begin{bmatrix}1 & t\\
  &  1\end{bmatrix},\quad dg=dk\, \frac{da\,db\,dt}{|b|}.$$
   We may write $\Orb\bigl(\sigma(x),\phi',s\bigr)$ as (cf.~\cite[(3.1)]{Z12b}\footnote{We note the following notational changes: 1)  a sign difference in $\lambda$; 2) $a,b$ in \cite{Z12b} become $u,w/\pi$; and 3) the measure $dx\,dy\,du$ becomes $da\,db\,dt$.})
\begin{align}\label{eqn orb s3}
\bigints_{a, b,t\in F_0}\phi'_{K}\left(\pi \begin{bmatrix} t &
a(-\lambda/\varpi-t^2 )& b\\
  1/a & -t &0\\b^{-1}u& (w/\pi-ut)ab^{-1}&0 \end{bmatrix}\right)\eta(a)|a^{-1}b^2|^{-s}\,\frac{da\,db\,dt}{|b|}.
\end{align}
Here we used the $K$-invariant function $$\phi'_K(x):=\int_{K}\phi'(kxk^{-1})dk.
$$
Without loss of generality, we may assume that $\phi'$, and hence $\phi'_K$, is invariant under translation by $\fks_\red(O_{F_0}):=\pi \gl_3(O_{F_0})\cap \fks$.

Now all the equations from (3.1) to (3.20) in \cite{Z12b} hold for the orbital integral $\Orb(\sigma(x),\phi',s)$ (instead of simply its value at $s=0$). By this we mean to interpret them as the integral against $\eta(a)|a^{-1}b^2|^s\frac{da\,db\,dt}{|b|}$. Indeed,  all of these equations are derived from partitioning the domain of integration into various pieces,  and the fact that the function $\phi_K$ is invariant under translation by $\fks_\red(O_{F_0})$ with compact support. It is never used in \cite[\S3]{Z12b} that $s=0$.

\begin{lemma}
The term (3.4) in \cite{Z12b} is equal to\
$$\int_{F_0}\Gamma_{n(\mu)}(x, s)\Orb\bigl(n(\mu),\phi',s\bigr)\,d\mu.$$
\end{lemma}
\begin{proof}
After a suitable substitution,  we may write \cite[(3.4)]{Z12b} as
\[ 
   \eta(-1)|u|^{-1}
\bigints \phi'_K\left(\pi\begin{bmatrix} 0 &
\frac{(t+w/\pi)^2 +u^2\lambda/\varpi}{u^2t}ab& b\\
 0 & 0 &0\\ 0& a&0 \end{bmatrix}\right)\eta(abt)|a^{-1}bt|^{-s} \,\frac{da\,db\,dt}{|t|}.
\]
This is equal to
\[ 
\eta(-1)|u|^{-1}
\bigintsss_{F_0} \Orb\biggl(n\left(\frac{(t+w/\pi)^2 +u^2\lambda/\varpi}{u^2t}\right),\phi',s\biggr)\eta(t)|t|^{-s} \, \frac{dt}{|t|}.
 \]
We rewrite this as 
\begin{equation}\label{eqn germ mu}
 \eta(-1)|u|^{-1}
\bigintsss_{F_0} \Orb\Bigl(n\bigl(u^{-2}(t+t^{-1}\Delta/\varpi+2w/\pi)\bigr),\phi',s\Bigr)\eta(t)|t|^{-s} \, \frac{dt}{|t|}.
\end{equation}
This is the form of the germ expansion we will use to do calculations later on. The rest is the same as the proof of \cite[Lem.~ 3.1]{Z12b}.
\end{proof}

Lemmas 3.2 and 3.3 of \cite{Z12b} remain unchanged. We recall them and indicate the necessary changes in the proof.

\begin{lemma}
The sum of (3.3) and (3.5) in \cite{Z12b} is zero.
The sum of (3.6) and (3.15) in \cite{Z12b}  is zero.
\end{lemma}

\begin{proof}

The proof in \cite{Z12b} can actually be simplified. In both sums, it suffices to show that 
\begin{align}
\label{eqn key id}
\int_{|y|\leq 1}|y|^{2s-1}dy+\int_{|y|> 1}|y|^{2s-1}dy=0,
\end{align}
where we understand both integrals as meromorphic functions of $s\in\BC$ obtained by analytic continuation as follows. The first integral converges when $\Re(s)>0$,
$$\int_{|y|\leq 1}|y|^{2s-1}dy=\frac{1}{1-q^{-s}}\int_{|y|=1}\frac{dy}{|y|}  ,
$$
and the second one converges when $\Re(s)<0$, 
$$\int_{|y|>1}|y|^{2s-1}dy= \frac{q^s}{1-q^{s}}\int_{|y|=1}\frac{dy}{|y|} .
$$These equalities give the claimed analytic continuation.
The sum of the two terms is then obviously zero.
\end{proof}

\begin{lemma}
The sum of the following terms in \cite{Z12b} is zero: (3.2), (3.8), (3.10), (3.16), (3.17), (3.20).
\end{lemma}

\begin{proof}The same proof as that of Lem.~3.3 in \cite{Z12b} still works,  noting that
the only ingredient is the identity (\ref{eqn key id}) above (for instance,  in loc.~cit.~one only uses (\ref{eqn key id}), when splitting the term $(II)$ into (3.12) plus the term following  it). 
\end{proof}

\begin{lemma}The sum of (3.12) and (3.19)  in \cite{Z12b} is zero.

 \end{lemma}

\begin{proof}
We work with the corresponding function $\phi'$ on $\fks\fkl_\red$.
The integral (3.12) is equal to
\begin{multline*}
 \bigints_{H'}\phi'\left(h\begin{bmatrix}0 &0& 0\\
 1 &0&0\\u& w &0
 \end{bmatrix}h^{-1}\right)\eta(\det h)\lvert \det h \rvert^{-s} \,dh\\
 = \bigints_{H'}\phi'\left(h\begin{bmatrix}0 &0& 0\\
 1 &0&0\\0& w &0
 \end{bmatrix}h^{-1}\right)\eta(\det h)\lvert \det h \rvert^{-s} \,dh.
\end{multline*}
After a substitution, the integral (3.19) is equal to 
{\allowdisplaybreaks
\begin{multline*}
 - \bigints_{H'}\phi'\left(h\begin{bmatrix} 0 &
-(w/u)^2 & 0\\
0& 0&0\\u& 0&0
 \end{bmatrix}h^{-1}\right)\eta(\det h)\lvert \det h \rvert^{-s} \,dh \\
 \begin{aligned}
  &= -\bigints_{H'}\phi'\left(h\begin{bmatrix} 0 &
0 & 0\\
(w/u)^2& 0&0\\0& u&0
 \end{bmatrix}h^{-1}\right)\eta(\det h)\lvert \det h \rvert^{-s} \,dh \\
   &= -\bigints_{H'}\phi'\left(h\begin{bmatrix} 0 &
0 & 0\\
1& 0&0\\0& w&0
 \end{bmatrix}h^{-1}\right)\eta(\det h)\lvert \det h \rvert^{-s} \,dh,
 \end{aligned}
\end{multline*}}
where the first and the last equality follow, respectively, from the substitutions
\[
   h\mapsto h 
	\begin{bmatrix} 
		0  &  -1\\
		1  &  0
 \end{bmatrix}
 \quad\text{and}\quad
 h\mapsto h
 \begin{bmatrix} 
	 w/u  &  0 \\
    0  &  u/w
 \end{bmatrix}. \qedhere
\]
\end{proof}

To complete the proof,  it suffices to treat the remaining terms: (3.7) and (3.9) in \cite{Z12b}. The term (3.7) in \cite{Z12b} yields, after a suitable substitution,
\begin{multline*}
   \bigints \phi'_K\left(\pi\begin{bmatrix} t &
-at^2 & b\\
  1/a& -t &0\\ 0 & 0&0 \end{bmatrix}\right)\eta(a)|a^{-1}b^2|^{-s} \,\frac{da\,db\,dt}{|b|} \\
   \begin{aligned}
		&= \bigints_{H'}\phi'\left(h\cdot\pi \begin{bmatrix} 0 &
0 & 1\\
1& 0&0\\0& 0&0
 \end{bmatrix}h^{-1}\right)\eta(\det h)\lvert \det h \rvert^{-s} \,dh \\
      &= \eta(-1)   \Orb(n_{0,+},\phi',s).
	\end{aligned}
\end{multline*}
The term (3.9) in \cite{Z12b} yields, after suitable substitutions,
\begin{multline*}
   \bigints \phi'_K\left(\pi\begin{bmatrix} 0 &
(-\lambda/\varpi-(w/u\pi)^2 )a& 0\\
 0 & 0&0\\b^{-1}u& (w/\pi-at)ab^{-1}&0
 \end{bmatrix}\right)  \eta(a)|a^{-1}b^2|^{-s} \,\frac{da\,db\,dt}{|b|}\\
   \begin{aligned}
		&= \bigints_{H'}\phi'\left(h\cdot \pi\begin{bmatrix} 0 &
-\lambda/\varpi-(w/u\pi)^2  & 0\\
0& 0&0\\ u& 0&0
 \end{bmatrix} h^{-1}\right)\eta(\det h)\lvert \det h \rvert^{-s} \,dh \\
      &=\eta(\Delta/\varpi)|\Delta/\varpi|^{-s}    \Orb(n_{0,-},\phi',s).
	\end{aligned}
\end{multline*}
This completes the proof of Theorem \ref{thm germ x=0}.

\section{Germ expansion around $x_0\neq 0$}\label{sec germ nonzero}

In this section we consider the germ expansion near a non-zero element $x_0\in \fkb_\red\smallsetminus \fkb_{\red, \rs}$.

\subsection{Orbits in $\fks_\red\smallsetminus \fks_{\red, \rs}$}\label{orbitsinfksminusnull}
We need to classify the  orbits in the fiber of $x_0=(\lambda_0,u_0,w_0)\in\fkb_\red\smallsetminus \fkb_{\red, \rs}$. 
Let 
$$
y_0=
\begin{bmatrix}A & {\bf b} \\
 {\bf c}   & 0
\end{bmatrix}\in\fks_\red
$$ be a semi-simple element in the fiber of $x_0$. It follows that
the dimension $r$ of the subspace spanned by ${\bf b},A{\bf b}$ is then either zero or one.

\medskip
\noindent\emph{Case $r=0$.} Then ${\bf b}=0$ and $x_0=(\lambda_0,0,0)$ where $\lambda_0\neq 0.$ We introduce the semi-simple quadratic $F_0$-algebra, 
\begin{equation*}
F'=F_0[X]/(X^2+\lambda_0) .
\end{equation*}
 
In the fiber of $x_0=(\lambda_0,0,0)$, there is one semi-simple orbit, and there  are two or four non-semi-simple orbits, depending on whether $F'$ is isomorphic to $F$ or not.

\smallskip

\noindent\emph{Subcase 0i: $F'\not\simeq F$.} In the fiber of such $x_0$, there is one  semi-simple orbit with representative
\begin{align}\label{eqn case 0i y0}
y_0=\pi\begin{bmatrix} 0&-\lambda_0/\varpi &0\\1&0  &  0
\\0&0&0
\end{bmatrix}.\end{align}
There are two non-semi-simple orbits with representatives
\begin{align}\label{eqn case 0i}
y_+=\pi\begin{bmatrix}0&-\lambda_0/\varpi &1\\ 1  & 0&0\\ 0&0& 0
\end{bmatrix},\quad y_-=\pi\begin{bmatrix}0&-\lambda_0/\varpi &0\\ 1  & 0&0\\ 1&0& 0
\end{bmatrix},
\end{align}
with trivial stabilizer. 
\smallskip

\noindent\emph{Subcase 0ii: $F'\simeq F$.} In this case there is one  semi-simple orbit with representative
$$
y_0=\pi\begin{bmatrix}\alpha&0&0\\0  & - \alpha&0\\ 0&0& 0
\end{bmatrix},\quad \alpha^2=-\lambda_0/\varpi ,
$$ and
there are four non-semi-simple with  representatives
\begin{align}\label{eqn case 0ii}
y_{++}=\pi\begin{bmatrix}\alpha&0&1\\0  & - \alpha&1\\ 0&0& 0
\end{bmatrix},\quad y_{+-}=\pi\begin{bmatrix}\alpha&0&1\\0  & - \alpha&0\\ 0&1& 0
\end{bmatrix},
\end{align}
and 
\begin{align}\label{eqn case 0ii y--}
y_{--}=^t\!\!y_{++},\quad  y_{-+}=^t\!\!y_{+-},
\end{align}
with trivial stabilizer. 

\smallskip

\noindent\emph{Case  $r=1$.}  Then ${\bf b}\neq 0$ and $A{\bf b}$ is a multiple of ${\bf b}$, i.e.,  ${\bf b}$ is an eigenvector of $A$, and $x_0=(\lambda_0, u_0, w_0)$ with $u_0\neq 0$. We may choose a semi-simple  representative as
$$
y_0=\pi\begin{bmatrix}\alpha&0&1\\0  & - \alpha&0\\ u_0&0& 0
\end{bmatrix},\quad \alpha\in F_0,\quad u_0\neq 0,\quad \alpha^2=-\lambda_0/\varpi.
$$
The stabilizer of $y_0$ is $\GL_1$ sitting inside $\GL_2$ in the upper left corner. However, the character $\eta\circ\det$ is nontrivial on the stabilizer and hence it is not a relevant orbit. The non-semi-simple  representatives are
\begin{align}\label{eqn case 1}
y_+=\pi\begin{bmatrix}\alpha&0&1\\ 1  & - \alpha&0\\ u_0&0& 0
\end{bmatrix},\quad y_-= \pi\begin{bmatrix}\alpha&1&1\\ 0  & - \alpha&0\\ u_0&0& 0
\end{bmatrix},
\end{align}
and they have trivial stabilizer.

\subsection{Orbital integrals}

We first define the orbital integrals of the $x_0$-nilpotent elements.  All of them are defined by \eqref{eqn def O(s)} with one exceptional case:
\begin{lemma} \label{lem orb F' split} 
Let $x_0=(\lambda_0, 0, 0)$ with $\lambda_0\neq 0$. 
  Assume that $F'\simeq F$, i.e., Case 0ii. Let $y=y_{\pm\pm}$ be a non-semi-simple element on $\fks_\red$ mapping to $x_0$. 
 Then for  $\phi'\in C_c^\infty(\fks_\red)$, the integral 
$$
\Orb(y,\phi,s_1,s_2)= \bigintsss_{a,b\in F_0^\times,t\in F_0} \phi'_K\left(\begin{bmatrix}1 & t\\
  &  1\end{bmatrix} \begin{bmatrix}a& \\
  &  b\end{bmatrix}\cdot y\right)\eta(ab)|a|^{s_1}|b|^{s_2}\, d^\times\! a \,d^\times\! b \,dt
$$is absolutely convergent when $\Re(s_1)\gg 0$, $\Re(s_2)\gg 0$ for $y=y_{++}$; when $\Re(s_1)\gg 0$, $\Re(s_2)\ll 0$ for $y=y_{+-}$; etc. It  
 has a meromorphic continuation to $(s_1,s_2)\in\BC^2$; its restriction to the diagonal $s_1=s_2$ is meromorphic  and is holomorphic at $(s_1,s_2)=(0,0)$.
\end{lemma}
\begin{proof}
The proof is analogous to that of \cite[Lem.~2.1]{Z12b} by the method of Tate's thesis. We omit the details.
\end{proof}
In this case we  define the orbital integral $\Orb(y,\phi',s)$ to be the value at $s_1=s_2=s$. By Lemma \ref{lem orb F' split}, this is  a meromorphic function of $s$.

\begin{theorem} \label{thm germ ss s}
Fix $x_0=(\lambda_0,u_0,w_0)\in \fkb_\red\smallsetminus\fkb_{\red, \rs}$ with $x_0\neq 0$.  For 
$x\in \fkb_{\red, \rs}$  in a small neighborhood of $x_0$, the orbital integral $\Orb(\sigma(x),\phi',s)$ is equal
\begin{altenumerate}
\renewcommand{\theenumi}{\alph{enumi}}
 \item in  case $(0i)$ to 
 $$\Orb(y_+,\phi',s)+\eta(\lambda^{-1}\Delta)|\lambda^{-1}\Delta|^{-s} \Orb(y_-,\phi',s),$$
where $\sigma(x)$ is the section defined in \eqref{eqn sigma};

\item  in  case $(0ii)$ to
\begin{multline*}
 \Orb(y_{++},\phi',s)+\eta(z_1)|z_1|^{-s}\Orb(y_{-+},\phi',s)\\
 + \eta(z_2)|z_2|^{-s}\Orb(y_{+-},\phi',s)+\eta(z_1z_2)|z_1z_2|^{-s}\Orb(y_{--},\phi',s),
\end{multline*}
 where the section $\sigma\colon \fkb_{\red, \rs}\to\fks_{\red, \rs}$ in a neighborhood of $x_0$  is defined by
\begin{align}\label{eqn sigma 1}
\sigma(x)=\pi\begin{bmatrix}\alpha &0& 1\\0 & -\alpha & 1 \\  z_1&z_2&0
\end{bmatrix} ,
\end{align}
with entries defined by the identities
$$
\lambda=-\alpha^2\varpi,\quad u=z_1+z_2, \quad w=\alpha(z_1-z_2)\pi ,\quad\Delta=\lambda u^2+w^2=-4\alpha^2\varpi z_1z_2;
$$

 \item in case $(1)$ to
 $$
  \Orb(y_+,\phi',s) +\eta(\Delta/\varpi)|u^{-2}\Delta/\varpi|^{-s}\Orb(y_-,\phi',s).$$ 
where $\sigma(x)$ is the section defined in \eqref{eqn sigma}.

\end{altenumerate}
Moreover, all the orbital integrals above are holomorphic at $s=0$ except in case $(0i)$ when $F'=F_0\times F_0$, in which case $\Orb(y_+,\phi',s)$ and $\Orb(y_-,\phi',s)$ both  have a simple pole at $s=0$.
\end{theorem}

\subsection{Proof of Theorem \ref{thm germ ss s}}

Let $x_0=(\lambda_0,u_0,w_0)\neq 0\in \fkb_\red$. We distinguish three cases, labeled by $(0i)$, $(0ii)$ and $(1)$, according to the case distinction in subsection \ref{orbitsinfksminusnull}. Fix a real number
$R$ such that the support of $\phi'$ is contained in the set
  \begin{equation}\label{eqn def R}
\bigl\{y=\pi \cdot (y_{ij})\in\fks_{\red}\bigm| |y_{ij}|\leq R\bigr\}.
\end{equation}

\noindent\emph{Case $(0i)$} (i.e., $r=0$, $\lambda_0\neq 0$ and $F'\neq F$):
Then $x_0=(\lambda_0,0,0)$ with $\lambda_0\neq 0$. 
We use (\ref{eqn orb s3}) to express the integral $\Orb(\sigma(x),\phi',s)$ for $x=(\lambda,u,v)\in\fkb_\red$ in a small neighborhood of $(\lambda_0,0,0)$. We therefore assume that $|\lambda|=|\lambda_0|\neq 0$ and hence the integrands have the following constraints,
$$
|t|\leq R,\quad  |a|\leq R/|\lambda_0|.
$$
 We split the integral over $b$ as a sum of two pieces, according as $|b|>1$ or $|b|\leq 1$. The contribution for  $|b|>1$ is equal, when $(u,w)$ is small enough,  to
\[
   \bigints_{|b|>1,a,t\in F_0}\phi'_{K}\left(\pi\begin{bmatrix} t &
a(-\lambda/\varpi-t^2 )& b\\
  1/a & -t &0\\0& 0&0 \end{bmatrix}\right)\eta(a)|a^{-1}b^2|^{-s}\,\frac{da\,db\,dt}{|b|}.
\]
This can be written as 
$$
\int_{a,b \in F,t\in F}\cdots\lvert \det h \rvert^sdh- \int_{|b|\leq 1,a \in F,t\in F}\cdots\lvert \det h \rvert^s\,dh,
$$
where both integrals converge absolutely when $\Re(s)>0$.
The first term extends to a meromorphic function 
\[
	\Orb(y_+,\phi',s)= 
   \bigints_{a,b,t\in F_0}\phi'_{K}\left(\pi\begin{bmatrix} t &
a(-\lambda/\varpi-t^2 )& b\\
  1/a & -t &0\\0& 0&0 \end{bmatrix}\right)\eta(a) |a^{-1}b^2|^{-s} \,\frac{da\,db\,dt}{|b|},
\]
while the second term is
\begin{equation}\label{eqn 1st term}
\begin{aligned}
   &\bigints_{|b|\leq 1,a,t\in F_0}\phi'_{K}\left(\pi\begin{bmatrix} t &
a(-\lambda/\varpi-t^2 )& 0\\
  1/a & -t &0\\0& 0&0 \end{bmatrix}\right)\eta(a) |a^{-1}b^2|^{-s} \,\frac{da\,db\,dt}{|b|}\\
  &\qquad =
   \biggl(\bigintssss_{|b|\leq 1}|b|^{2s}\frac{db}{|b|}\biggr)
   \bigints_{a ,t\in F_0}\phi'_{K}\left(\pi\begin{bmatrix} t &
a(-\lambda/\varpi-t^2 )& 0\\
  1/a & -t &0\\0& 0&0 \end{bmatrix}\right)\eta(a) |a|^{s}\,da\,dt\\
  &\qquad =\zeta(-2s)Q(s) , 
\end{aligned}
\end{equation} where we denote by $Q(s)$ the second integral in the next-to-last equation. From $F'\not\simeq F$ it follows that $-\lambda_0/\varpi$ (hence $-\lambda/\varpi$) is not a square. The norms $|t|, |a|$ and $|a|^{-1}$ are all bounded independent of $s$ and hence $Q(s)$ is an entire function.  Again by $F_0[\sqrt{-\lambda_0}]\not\simeq F$, the quadratic character $\eta$ is nontrivial on the stabilizer of the semi-simple representative $y_0$. It follows easily that
$
Q(0)=0 ,
$
and hence the second term \eqref{eqn 1st term} is holomorphic at $s=0$ and hence so is $\Orb(y_+,\phi',s)$.

Now we consider the contribution from the piece $|b|\leq 1$, 
\[
\bigints_{|b|\leq 1,a,t\in F_0}\phi'_{K}\left(\pi\begin{bmatrix} t &
a(-\lambda/\varpi-t^2 )& 0\\
  1/a & -t &0\\b^{-1}u& (w/\pi-ut)ab^{-1}&0\end{bmatrix}\right)\eta(a)\lvert \det h \rvert^{-s}\,\frac{da\,db\,dt}{|b|}.
\]
This can be written as 
$$
\int_{a,b \in F,t\in F}\cdots\lvert \det h \rvert^{-s}\,dh- \int_{|b|> 1,a \in F,t\in F}\cdots\lvert \det h \rvert^{-s}\,dh,
$$
where both integrals converge absolutely when $\Re(s)<0$.
The second term is equal to 
\begin{align*}
   &\bigints_{|b|> 1,a,t\in F_0}\phi'_{K}\left(\pi\begin{bmatrix} t &
a(-\lambda/\varpi-t^2 )& 0\\
  1/a & -t &0\\0& 0&0\end{bmatrix}\right)\eta(a)|a^{-1}b^2|^{-s}\,\frac{da\,db\,dt}{|b|}\\
   &\hspace{55ex} =\biggl(\bigintssss_{|b|>1}|b|^{-2s}\frac{db}{|b|}\biggr)Q(s)\\
   &\hspace{55ex} =q^{-2s}\zeta(2s)Q(s).
\end{align*}
Now note that  $
\zeta(-2s)+q^{-2s}\zeta(2s)=0.
$ We see that this last term cancels \eqref{eqn 1st term}.
The first term can be rewritten as
$$
\Orb\left( \begin{bmatrix} 0 &
-\lambda/\varpi& 0\\
  1 & 0&0\\u& w/\pi&0\end{bmatrix},\phi',s\right)=  \bigints_{H'} \phi'\left(h^{-1} \cdot \pi \begin{bmatrix} 0 &
-\lambda/\varpi& 0\\
  1 & 0&0\\u& w/\pi&0\end{bmatrix}h\right)\eta(\det h)\lvert \det h \rvert^s\,dh.
$$
Since $u$ and $w$ cannot be simultaneously zero, we may make a change of  variables 
$
h\mapsto h_0h
$, where 
$$
h_0:=\frac{1}{\det h_1}h_1 \quad \text{with}\quad h_1:=\begin{bmatrix} u &-
w/\pi\\
  w\pi /\lambda& u \end{bmatrix}.
$$
The integral becomes 
\begin{multline*}
\eta(h_0)\lvert \det h_0 \vert^{s}
 \bigints_H \phi'\left(h^{-1} \cdot \pi \begin{bmatrix} 0 &-
\lambda/\varpi& 0\\
  1 & 0&0\\1& 0&0\end{bmatrix}h\right)\eta(h)\lvert \det h \rvert^s\,dh\\
=
\eta(h_0)\lvert \det h_0 \rvert^{s}\Orb(y_-,\phi',s).
\end{multline*}
Now note that $\det h_1 = \lambda^{-1}w^2+u^2=\lambda^{-1}\Delta$, and $\det h_0 = \det(h_1)^{-1}$.
In summary we have
\begin{align}\label{eqn key 1}
\Orb(y,\phi',s)=\Orb(y_+,\phi',s)+ 
\eta(\lambda^{-1}\Delta)|\lambda^{-1}\Delta|^{-s} \Orb(y_-,\phi',s) ,
\end{align}
which proves Theorem \ref{thm germ ss s} in this case. 

\smallskip

\emph{Case $(0ii)$} (i.e., $r=0$ and $F'= F$):
In this case we may assume that $$-\lambda_0/\varpi=\alpha_0^2\neq 0.$$ We use the section $\sigma(x)$ defined by \eqref{eqn sigma 1}.
When $x=(\lambda,u,w)$ is in a small neighborhood of  $(\lambda_0,0,0)$, we may assume $|\alpha|=|\alpha_0|$. We use a variant of the Iwasawa decomposition
$$h=
k
\begin{bmatrix}1 & t\\
  &  1\end{bmatrix} \begin{bmatrix}a& \\
  &  b\end{bmatrix},\quad dg=dk\, \frac{da\,db\,dt}{|a||b|}.$$
Consider
$$
\phi'_{KN}( y)=\int_{KN}\phi'(kn\cdot x)\,dk\,dn,\quad y=\begin{bmatrix}A& {\bf b}\\ {\bf c} & 0
\end{bmatrix}, \quad A=\begin{bmatrix}\alpha & 0\\0 & -\alpha
\end{bmatrix},
$$
which may be viewed as a function of $({\bf b},{\bf c})$ on $M_{1,2}(F_0)\times M_{2,1}(F_0)$. For $\alpha\neq 0$, we see that the integral is absolutely convergent.
Then we have
  $$
  \Orb(y,\phi',s)=
  \bigints \phi'_{KN}\left(\pi\begin{bmatrix}\alpha & 0& a\\0 & -\alpha & b \\z_1a^{-1}&z_2 b^{-1}&0
\end{bmatrix}\right)\eta(ab)|ab|^{-s}\,\frac{da\,db}{|a||b|}.
  $$

Before we proceed, we consider a toy model: let $F_0^\times$ act on $F_0\times F_0$ by $a\cdot (x,y)=(a^{-1}x,ay)$ and let $\eta$ be a quadratic character (possibly trivial). 
\begin{lemma}Let
$\phi\in  C^\infty_c(F_0\times F_0)$, and define for $x\in F_0^\times$,
$$
\Phi(x,s)=\int_{F_0^\times}\phi(x/a,a)\eta(a)|a|^s\, d^\times\! a,
$$
which is absolutely convergent when $\Re(s)\gg 0$. For $x$  in a small neighborhood of $0$ (depending on $\phi$), 
$$
\Phi(x,s)=\int_{F_0^\times}\phi(0,a)\eta(a)|a|^s \,d^\times\! a+\eta(x)|x|^s \int_{F_0^\times}\phi(1/a,0)\eta(a)|a|^{s} \,d^\times\! a.
$$
The identity is understood after an analytic continuation of each term on the right-hand side as a meromorphic function of $s\in\BC$ (without pole at $s=0$, if $\eta$ is non-trivial).
\end{lemma}

\begin{proof}The proof is again analogous to that of \cite[Lem.~2.1]{Z12b}. We omit the details.
\end{proof}

Using this lemma, we obtain that  $ \Orb(\sigma(x),\phi',s)$ is, in a neighborhood of $x_0$, the value of the following sum when $s_1=s_2=s$,
\begin{equation}\label{eqn key 2}
\begin{aligned}
&\Orb(y_{++},\phi',s_1,s_2)+\eta(z_1)|z_1|^{s}\Orb(y_{-+},\phi',s_1,s_2)\\
&\hspace{10ex}+\eta(z_2)|z_2|^{s}\Orb(y_{+-},\phi',s_1,s_2)+\eta(z_1z_2)|z_1z_2|^{s}\Orb(y_{--},\phi',s_1,s_2).
\end{aligned}
\end{equation}
 Theorem \ref{thm germ ss s} in case $(0ii)$ now follows easily from this equality.

\smallskip

\noindent\emph{Case $(1)$} (i.e., $r=1$):
In this case  $u_0\neq 0$.
We now use the expression (\ref{eqn orb s3}). We first observe that the integrand is constrained by
$$
|t|\leq R,\quad |u|/R\leq |b|\leq R ,
$$ 
where $R$ is as in (\ref{eqn def R}).
We only consider the case $\lambda_0=0$ (hence $w_0=0$). Otherwise, the proof of the previous cases  still applies. 

We break the integral over $a$ up into two pieces: $|a|$ is large or small.  Choose a constant $C$. Note that we will consider $\lambda$, $w$ close to zero, and $|u|=|u_0|\neq 0$.
When $|a|\leq C$,  so that $|a\lambda|\leq 1$ and $|aw|R\leq 1$, we have
\begin{equation}\label{eqn orb s3 1}
\bigints_{|a|\leq C, b ,t\in F_0}\phi'_{K}\left(\pi\begin{bmatrix} t &-at^2 
& b\\
  1/a & -t &0\\b^{-1}u& -utab^{-1}&0 \end{bmatrix}\right)\eta(a)|a^{-1}b^2|^{-s}\,\frac{da\,db\,dt}{|b|}.
\end{equation}
Now consider $|a|>C$. Substitute $t\to t+w/u$, and note that we may assume $|w/u|<1$:
\begin{align*}
\bigints_{|a|>C, b,t\in F_0}\phi'_{K}\left(\pi\begin{bmatrix} t&
a(-\lambda/\varpi-(t-w/\pi u)^2)& b\\
 1/a & -t&0\\b^{-1}u& -utab^{-1}&0 \end{bmatrix}\right)\eta(a)|a^{-1}b^2|^{-s}\,\frac{da\,db\,dt}{|b|}.
\end{align*}
The condition $\lvert-utab^{-1}\rvert < R$ implies that
\[
   |at|<\frac{R^2}{|u|}
	\quad\text{and}\quad
	\bigl|at^2\bigr| = \biggl|\frac{(at)^2}{a}\biggr| < \frac{R^4}{|u|^2C}.
\]
Furthermore, for $C$ sufficiently large, we have $|at^2|<\frac{R^2}{|u|}<1$. Hence the last integral becomes
\begin{equation}\label{eqn orb s3 3}
\bigints_{|a|>C, b,t\in F_0}\phi'_{K}\left(\pi\begin{bmatrix} 0 &
a(-\lambda/\varpi-(w/\pi u)^2)& b\\
 0 & 0&0\\b^{-1}u& -utab^{-1}&0 \end{bmatrix}\right)\eta(a)|a^{-1}b^2|^{-s}\,\frac{da\,db\,dt}{|b|}.
\end{equation}
A similar argument as in case $(0i)$ shows that the sum of (\ref{eqn orb s3 1}) and (\ref{eqn orb s3 3}) can be written as
\begin{equation}\label{almostdone}
\Orb\left(\pi\begin{bmatrix}0&0&1\\ 1  & 0&0\\ u&0& 0
\end{bmatrix},\phi',s\right)+\Orb\left(\pi\begin{bmatrix}0&-\lambda/\varpi -(w/\pi u)^2&1\\ 0  & 0&0\\ u&0& 0
\end{bmatrix},\phi',s\right).
\end{equation}
Note that $\Delta=\lambda u^2+w^2$. It follows that the second summand in \eqref{almostdone} is equal to
$$
\eta(\Delta/\varpi) |\Delta/\varpi u^2|^{-s}\Orb\left(\pi\begin{bmatrix}0&1&1\\ 0  & 0&0\\ u&0& 0
\end{bmatrix},\phi',s\right) .
$$
In summary we have proved that in the case $r=1$, when $x=(\lambda,u,w)$ is close to $(\lambda_0,u_0,w_0)$, the integral
 $\Orb(\sigma(x),\phi',s)$ is equal to
 \begin{align}\label{eqn key 3}
 \Orb(y_+,\phi',s) +\eta(\Delta/\varpi)|u^{-2}\Delta/\varpi|^{-s}\Orb(y_-,\phi',s) .
\end{align}
The equations \eqref{eqn key 1}, \eqref{eqn key 2}, and \eqref{eqn key 3} together complete the proof of Theorem \ref{thm germ ss s} in this case.

\subsection{The exceptional case}
Theorem \ref{thm germ ss s}  gives the germ expansion for $\Orb(\sigma(x),\phi',s)$ at $s=0$, except in case $(0i)$ when $F'\simeq F_0\oplus F_0$, in which case both  $\Orb(y_+,\phi',s)$ and $\Orb(y_-,\phi',s)$ have a pole at $s=0$.

\begin{corollary}\label{cor F'=F}
Fix $x_0=(\lambda_0,0,0)\in \fkb_\red\smallsetminus\fkb_{\red, \rs}$ with $\lambda_0\neq 0$.  Assume that $F'\simeq F_0\times F_0$. 
\begin{altenumerate}
\renewcommand{\theenumi}{\alph{enumi}}
\item The sum 
$\Orb(y_+,\phi',s)+\Orb(y_-,\phi',s)$ is holomorphic at $s=0$. Denote by  $\Orb(y_\pm,\phi')$  its value at $s=0$.
\item
For 
$x\in \fkb_{\red, \rs}$  in a small neighborhood of $x_0$, the orbital integral $\Orb\bigl(\sigma(x),\phi',0\bigr)$ (for $\sigma(x)$ defined by (\ref{eqn sigma})) is equal to
\begin{align}\label{eqn F' split}
\Orb(y_\pm,\phi')-\frac{\log| \lambda^{-1}\Delta  |}{2\log q} \Orb(y_0,\phi'),
\end{align} and 
\[
   \Orb(y_0,\phi')
	   =\int_{\PGL_2(F_0)}\phi'(h^{-1}y_0h)\eta(h)\,dh,\quad y_0= \pi \begin{bmatrix} 0 &
-\lambda_0/\varpi& 0\\
  1 &0  &0\\0&0 &0 \end{bmatrix},
\]
where the measure on $\PGL_2(F_0)$ is the quotient measure on $\GL_2(F_0)$ divided by that of $F_0^\times$ with $\vol(O_{F_0}^\times)=1$.
\end{altenumerate}
\end{corollary}

\begin{proof}
We first claim that \emph{the meromorphic functions $\Orb(y_+,\phi',s)$, resp.~$\Orb(y_-,\phi',s)$, have a simple pole at $s=0$ with residue}
$$
\mp\Orb(y_0,\phi')\frac{1}{2\log q}.
$$
We now prove the corollary assuming the claim. The claim immediately implies  part (a). To show (b),
we note that the assumption $F'\simeq F_0\times F_0$ is equivalent to $-\lambda_0\in F_0^{\times, 2}$. In this case, for  all $x\in \fkb_{\red,\rs}$ near $x_0$, we always have $x\in\fkb_{\red,\rs,0}$ (cf.~Lemma \ref{intimage}), and hence $\eta(-\Delta(x))=1$ (cf.~Proposition  \ref{eta b01}). Hence $\eta(\Delta(x)/\lambda(x))=\eta(-\Delta(x))/\eta(-\lambda_0)=1$.
 By Theorem \ref{thm germ ss s}, we have  for regular semi-simple $x$ near $x_0$,
\[
	\Orb\bigl(\sigma(x),\phi',s\bigr)
	   =\bigl(\Orb(y_+,\phi',s)+\Orb(y_-,\phi',s)\bigr)+(|\lambda^{-1}\Delta|^{-s}-1) \Orb(y_-,\phi',s),
\]
where both terms in the right-hand side are holomorphic at $s=0$ by the claim. Hence
\[
	\Orb\bigl(\sigma(x),\phi',0\bigr)
      =\bigl(\Orb(y_+,\phi',s)+\Orb(y_-,\phi',s)\bigr)\big|_{s=0}+(|\lambda^{-1}\Delta|^{-s}-1) \Orb(y_-,\phi',s)\big|_{s=0}.
\]
The first term is now  $\Orb(y_\pm,\phi')$ by (a). The second term is given by 
$-\log|\lambda^{-1}\Delta|$ times the residue of $\Orb(y_-,\phi',s)$ at $s=0$. By the claim we complete the proof of (b).

We now prove the claim. We only treat $\Orb(y_+,\phi',s)$, since the other case is similar. By (\ref{eqn orb s3}),  $\Orb(y_+,\phi',s)$ is equal to
\begin{align*} 
\bigints_{a, b,t\in F_0}\phi'_{K}\left(\pi \begin{bmatrix} t &
a(-\lambda_0/\varpi-t^2 )& b\\
  1/a & -t &0\\0&0 &0 \end{bmatrix}\right)\eta(a)|a^{-1}b^2|^{-s}\,\frac{da\,db\,dt}{|b|}.
\end{align*}
We may view this as an integral of the form
\begin{align}\label{eqn Phi(s)}
\Orb(y_+,\phi',s)=\bigintssss_{b\in F_0}\Phi_s(b)|b|^{-2s}\,\frac{db}{|b|},
\end{align}
where $\Phi_s(b)$ extends to an entire function in $s\in\BC$. We may find the value
\begin{align*} 
   \Phi_0(0) 
	   &= \bigints_{a,t\in F_0}\phi'_{K}\left(\pi \begin{bmatrix} t &
a(-\lambda_0/\varpi-t^2 )& b\\
  1/a & -t &0\\0&0 &0 \end{bmatrix}\right)\eta(a)\,da\,dt\\
      &= (1-q^{-1})^{-1}\bigints_{h\in \PGL_2(F_0)}\phi'\left(h^{-1} \pi \begin{bmatrix} 0 &
-\lambda_0/\varpi& 0\\
  1 &0  &0\\0&0 &0 \end{bmatrix}h\right)\eta(\det h)\,dh\\
      &=(1-q^{-1})^{-1}\Orb(y_0,\phi'),
\end{align*}
where the extra factor is due to the different choice of measures.

By Tate's thesis, the integral \eqref{eqn Phi(s)} has a simple pole at $s=0$ with residue given by $\Phi_0(0)$ times the residue of 
$$
\bigintssss_{|b|\leq 1}|b|^{-2s}\,\frac{db}{|b|}=\sum_{i\geq 0} q^{2is} (1-q^{-1})=\frac{1-q^{-1}}{1-q^{2s}}.
$$
This last term has residue $-(1-q^{-1})/2\log q$. This shows that the function $\Orb(y_+,\phi',s)$ has a simple pole at $s=0$ with residue 
$$
-\Orb(y_0,\phi')\frac{1}{2\log q}.
$$
This completes the proof of the claim.
\end{proof}

\section{Germ expansion for $\fku_\red$, and matching}

In this section we present the germ expansion for the orbital integrals on $\fku_\red$ for $\fku=\fku(W)$, where $W$ is either $W_0$ or $W_1$, i.e., $\fku=\fku_0$ or $\fku=\fku_1$. 

We consider the invariants (cf.~\S \ref{subsec inv u1} and \S \ref{subsec inv u0})
\[
	\pi_\fku\colon \fku_\red \to  \fkb_\red = \BA\times\BA\times\fks_1
\]
 given by the  formulas \eqref{u,w,lambda} (in the case of $\fks_\red$, but the cases of $\fku_{0, \red}$ and $\fku_{1, \red}$ are the same, comp.~\eqref{fku_1 invariants} for $\fku_{1, \red}$ and \eqref{invaronu0} for $\fku_{0, \red}$), 
 $$
\begin{bmatrix}A & {\bf b} \\
 {\bf c}   & d
\end{bmatrix}\mapsto (\lambda,u,w).
$$

\subsection{Germ expansion around $x_0=0$}
The germ expansion around $x_0=0$ for $\fku_\red$ is stated in \cite[Th.~2.8]{Z12b}. Since we will not use it directly,  let us not repeat it here. 
We only recall the classification of the nilpotent orbits and their orbital integrals, which are used in our  calculations in Part \ref{analtyic side part}.

The nilpotent orbits for $\fku_{0,\red}$ are classified in \cite[\S2.1]{Z12b}. For our purposes we only need $\{0\}$ and the continuous family
\begin{align}\label{eqn n(mu) u}
n(\beta):=\pi\begin{bmatrix}0 & \beta\pi & 1 \\ 0 & 0 &   0\\
0 &\pi  & 0
\end{bmatrix}\in\fku_{0,\red},\quad\beta\in F_0 .
\end{align}
 The stabilizer of $n(\beta)$ is the standard unipotent subgroup $N$ sitting inside $\SL_2=\SU(J^\flat_0)$ (cf.~\eqref{SL2 = SU2}). We define the corresponding nilpotent orbital integrals by
\begin{align}\label{eqn def orb n}
\Orb(n(\beta),\phi)=\int_{H/N}\phi\bigl(h ^{-1} n(\beta) h\bigr)\, d\ov h,
\end{align}
and 
\begin{align}\label{eqn def orb 0}
\Orb(0,\phi)=e_{F/F_0}q^{-1}\zeta_{F_0}(1)\phi(0),
\end{align}
where $e_{F/F_0}$ is the ramification index of $F/F_0$. It is easy to see that both expressions  converge absolutely. \emph{It is important to note here that the measure on $H=U(W^\flat)$ is chosen such that $\vol (K)=1$ for the special parahoric subgroup $K$ (the hyperspecial one when $F/F_0$ is unramified)}.

We also  define the orbital integral for any $x\in \fku_{0,\red}$ with compact stabilizer or any $x\in \fku_{1,\red}$ by 
\[
\Orb(x,\phi)= \int_{H}\phi(h^{-1} xh) \,dh ,
\]
whenever the integral is absolutely convergent (this will always be true in the cases of interest to us).

\subsection{Germ expansion around $x_0\neq 0$}

Let $x_0\in\fkb_\red\smallsetminus \fkb_{\red, \rs}$. We first classify the $H$-orbits (semi-simple or not) in $\fku_\red$ in the fiber $\pi_{\fku}^{-1}(x_0)$.
Unlike the case of $\fks_\red$, two issues will affect the semi-simple orbits on $\fku$: stability, and whether $H=U(W^\flat)$ is quasi-split or non-split.

Similar to the case $\fks_\red$, we distinguish two cases according to the rank $r$ of the space spanned by ${\bf b}, A{\bf b}$.
\begin{altitemize}
\item $r=0$. Then $x_0$ is of the form $x_0=(\lambda_0,0,0)$ with $\lambda_0\in F_0^\times$. A semi-simple element in $\fku_\red$ mapping to  $x_0$ must be of the form 
\begin{equation*}
y_0=\begin{bmatrix}A &0 \\
0  & 0
\end{bmatrix} .
\end{equation*}
 We need the following lemma
concerning the stability issue. 
 \begin{lemma}\label{lem stability}
 Let $\lambda_0\in  F_0\smallsetminus\{0\}$. 
 Then the set
\[
   X_{\lambda_0} := \bigl\{\,x\in\fks\fku(W^\flat)\bigm| \det x =\lambda_0\,\bigr\}
\]
forms one  orbit under $U(W^\flat)(F_0)$, unless  $F'=F_0[X]/(X^2+\lambda_0)$ is isomorphic to $F$, in which case $X_{\lambda_0}$ 
 decomposes into  two such orbits.
\end{lemma}
\begin{proof} Obviously, $X_{\lambda_0}$ is one \emph{geometric} orbit (i.e., after passing to the algebraic closure $\ov F_0$,  all elements of $X_{\lambda_0}$ are conjugate). Let $x_0\in X_{\lambda_0}$, and let $T=T_{x_0}$ be the stabilizer subgroup of $x_0$. Then $T$ is a maximal torus in $U(W^\flat)$. By general principles, the number of orbits under  $U(W^\flat)(F_0)$ in a geometric orbit is in one-to-one correspondence with 
\begin{equation}\label{kernelgalois}
   \ker \bigl[H^1(F_0, T)\to H^1\bigl(F_0, U(W^\flat)\bigr)\bigr].
\end{equation}
Let $\underline {F}^1$ be the algebraic subtorus of $\underline {F}^\times:=\Res_{F/F_0}(\BG_m)$ defined by $\Nm_{F/F_0}=1$. Then $\underline {F}^1$  is the maximal torus quotient of $U(W^\flat)$, and \eqref{kernelgalois} is identified with 
\begin{equation}\label{kernelgaloisab}
   \ker \bigl[H^1(F_0, T)\to H^1(F_0, \underline {F}^1) \bigr] .
\end{equation}
Now, for $T$ there are the following possibilities, up to isomorphism.
\begin{altenumerate}
\renewcommand{\theenumi}{\arabic{enumi}}
\item $T=\underline {F}^\times$, mapping via $a\mapsto a/\ov a$ to $\underline {F}^1$.
\item $T=\underline {F}^1\times \underline {F}^1$, mapping via multiplication to $\underline {F}^1$.
\item $T=\underline {K}^1$, mapping via $\Nm_{K/F}$ to $\underline {F}^1$. Here $K=F'.  F$ is a bi-quadratic extension of $F_0$, and $\underline {K}^1$ is the algebraic subtorus of $\underline {K}^\times:=\Res_{K/F_0}(\BG_m)$, defined by $\Nm_{K/F'}=1$.\smallskip
\end{altenumerate}
Furthermore, case (1) corresponds to the case when $F'\simeq F_0\oplus F_0$, case (2) to the case when $F'\simeq F$, and case (3) to the remaining possibilities. 

Let $T'=\ker(T\to \underline {F}^1)$.  In case (1), $T'=\BG_m$ and \eqref{kernelgaloisab} is trivial. In case (2), $T'=\underline {F}^1$, and \eqref{kernelgaloisab} is identified with $H^1(F_0, \underline {F}^1)=F_0^\times/\Nm_{F/F_0}(F^\times)=\BZ/2$; in case (3), the map $H^1(F_0, T)\to H^1(F_0, \underline {F}^1)$ is identified with
$$ 
F'^\times/\Nm_{K/F'}(K^\times)\xra{\Nm_{F'/F_0}} F_0^\times/\Nm_{F/F_0}(F^\times) ,
$$
which is injective, and hence \eqref{kernelgaloisab} is trivial. 
The lemma is proved. 
\end{proof}

\smallskip

\emph{Subcase 0i.} 
\label{sec S 19 ss}
When $F'$ is not isomorphic to $F$, then by Lemma \ref{lem stability}, there is a unique semi-simple orbit with invariants  $x_0=(\lambda_0,0,0), \lambda_0\in F_0^\times$ and we fix a choice of representative $y_0\in\fku_\red$. A non-semi-simple orbit exists only when $W^\flat$ is split, and the
quadratic algebra $F'=F_0[X]/(X^2+\lambda_0)$ is split as $F_0\times F_0$. We exclude in the sequel the case when $F'$ is isomorphic to $F_0\times F_0$. The reason is that the closure of $\fkb_{\red,\rs,1}$ in $\fkb_\red$ does not contain such elements and therefore we do not need this case in Part \ref{analtyic side part}.

\emph{Subcase 0ii.} When $F'$ is isomorphic to $F$, by Lemma \ref{lem stability},  there are two semi-simple orbits  mapping to  $x_0=(\lambda_0,0,0)$,  and we fix the representatives $y_+, y_-\in\fku_\red$. We will label the two orbits $y_\pm$ as follows. Consider a regular semi-simple element $y$ in $\fku_\red$ near $y_\pm$. Choose a basis such that $y$ may be written in the form  
$$y=\begin{bmatrix}A &\ast\\\ast  &  0
\end{bmatrix}\text{ with  \emph{diagonal} $A=\pi\begin{bmatrix}\alpha &\\ & -\alpha 
\end{bmatrix}$, $\alpha\in F_0^\times$.}
$$
 Then we choose $y_+$,    resp.~$y_-$, such that all regular semi-simple elements near $y_+$ have the following property: $z_1:=\frac{1}{2}(u+\frac{w}{\pi\alpha})$ is a norm, resp.~a non-norm.  Here we are using  the coordinates $z_1, z_2$ from Theorem \ref{thm germ ss s}(b). An easy calculation shows that this is possible and we will choose a small open neighborhood $V_{x_0,\pm}$ of $x_0$ such that, for all $(\lambda,u,w)\in V_{x_0,\pm}\cap\fkb_{\red,\rs}$, we have $\eta(z_1)=\pm 1$.
Moreover, there are no non-semi-simple orbits in the fiber of such $x_0\in\fkb_\red$.

\item $r=1$, then ${\bf b}\neq 0$ and $A{\bf b}$ is a multiple of ${\bf b}$. It is not hard to show that there is a unique orbit (which therefore has to be semi-simple) mapping to $x_0$. For our calculation in Part \ref{analtyic side part}, we give an explicit representative when $W^\flat$ is split. If $\lambda_0\neq 0$ we choose
\begin{align}
\label{eqn r=1 u}
y_0=\begin{bmatrix}0&-\lambda_0 &b_1 \\1  & 0&b_2\\ \cdots&\cdots& 0
\end{bmatrix},
\end{align}
where  $
b_1=\pi\alpha b_2$ with $\alpha^2=- \lambda_0/ \varpi$ ($\alpha\in F_0^\times$) and $(b_1\ov b_2-b_2\ov b_1)/\pi=u_0
$. If $\lambda_0=0$, we choose any 
\begin{align}
\label{eqn r=1 lambda=0}
y_0=
\begin{bmatrix}0&0&b_1 \\0  & 0&1\\ \cdots&\cdots& 0
\end{bmatrix}, \quad \Im(b_1)\neq 0.
\end{align}In either case, the stabilizer is an anisotropic torus.  
\end{altitemize}

Having  classified the orbits in $\pi_\fku^{-1}(x_0)$, it is easy to prove the following explicit germ expansion.
\begin{theorem} \label{thm germ ss u}
Let $x_0=(\lambda_0,u_0,w_0)\neq (0,0,0)\in\fkb_\red$, and let $\phi\in  C_c^\infty(\fku_\red)$.  For $x$  in a small neighborhood of $x_0\in \pi_\fku(\fku_\red)$, let $\sigma(x)$ be any element in $\fku_\red$ mapping to $x\in \fkb_\red$.
\begin{altenumerate}
\renewcommand{\theenumi}{\alph{enumi}}
\item If $F'=F_0[X]/(X^2+\lambda_0)\neq F, F_0\times F_0$,
 then the orbital integral $\Orb\bigl(\sigma(x),\phi\bigr)$ is equal to
$
\Orb(y_0,\phi),$ where $y_0\in \fku_\red$ is any representative of the unique orbit $\pi_\fku^{-1}(x_0)$.
\item If $F'=F_0[\sqrt{-\lambda_0}]\simeq F$, then the orbital integral $\Orb(\sigma(x),\phi)$ is equal to
$$
\Orb(y_+,\phi) \,\mathbf{1}_{V_{x_0,\rs,+}}+\Orb(y_-,\phi)\, \mathbf{1}_{V_{x_0,\rs,-}}.
$$
\end{altenumerate}

\end{theorem}
\begin{proof}This is proved using the same argument as (and is easier than) the case $x_0=0$ in \cite{Z12b}.
\end{proof}

\subsection{Matching orbital integrals around $x_0=0$}\label{sec 19.3}
As in \cite[Def.\ 2.6]{Z12b}, we let
$\CC_1(F_0)$ be the space of locally constant functions
$f$ on $F_0$ such that, when
 $|x|$ is large enough, $f(x)$ is a linear combination of the following functions, 
$$ \eta(x)|x|^{-1},\quad\eta(x)\log\lvert x\rvert |x|^{-1}.$$
Let $\CC_2(F_0)$ be the space similarly defined by requiring that, when $|x|$ is large enough, $f(x)$ is a linear combination of
$$
|x|^{-1},\quad\eta(x)|x|^{-1}.$$
Let $\CC(F_0)=\CC_1(F_0)\cup \CC_2(F_0)$.
For $f\in \CC(F_0)$, we define the \emph{extended Fourier transform} by \cite[\S 4.1]{Z12b}:
\begin{equation*}
\wt{f}(v):= \bigintssss_{F_0}f(v+x)\eta(x) \,\frac{dx}{|x|},
\end{equation*}
which is understood in the sense of analytic continuation (cf.~loc.~cit.). We have
$$\widetilde{\widetilde{f\,}}\!(v)=\gamma(1,\eta)^{2}f(v),$$
where the square of the gamma factor is equal to 
\[
   \gamma(1,\eta)^{2}=
	\begin{dcases} 
		\left(\frac{L(0,\eta)}{L(1,\eta)}\right)^2=\frac{(1+q^{-1})^2}{4}
                    &  \eta \mbox{  unramified};\\
      \eta(-1)q^{-1}, &  \eta \mbox{  ramified}.
\end{dcases}
\]
 
\begin{definition} Let $(\phi_0,\phi_1)$ with $\phi_i\in C_c^\infty(\fku_{\red,i})$, and let $\phi'\in C_c^\infty(\fks_\red)$.  Then $(\phi_0,\phi_1)$ and  $\phi'$ are \emph{local transfers} around $x_0\in \fkb_\red$ if there exists  a small neighborhood $V_{x_0}$ of $x_0$ in $\fkb_\red$ such that
 $$
 \omega(y')\Orb(y',\phi')=\Orb(y,\phi_i)
 $$
 for any $y'\in \fks_\red$ with invariants $x\in V_{x_0,\rs}$, and any $y\in\fku_{\red,i}$ matching $y'$.
 \end{definition}
 We will use this definition with  the following transfer factor (we are allowed to do so by Remark \ref{changeoftr}, since this transfer factor differs from our original transfer factor by a constant multiple):
\begin{equation}\label{eqn omega n=3}
   \omega(y') = \eta\bigl(\det(\wt y^ie)_{i=0,1,...,n-1}\bigr), \quad \wt y'=y'/\pi,\quad y'\in\fks_\red .
\end{equation}
This transfer factor is  chosen such that $\omega(\sigma(x))=1$ for the section $\sigma(x)$ defined by (\ref{eqn sigma}), which we use frequently.
 
 For $\phi\in C_c^\infty(\fku_\red)$, resp.~$\phi'\in C_c^\infty(\fks_\red)$, we
define
\begin{equation*}
   \Orb_\phi(\beta) := \Orb\bigl(n(\beta),\phi\bigr), 
	\quad\text{resp.}\quad
	\Orb_{\phi'}(\mu) := \Orb\bigl(n(\mu),\phi'\bigr),
	\quad \beta, \mu\in F_0.
\end{equation*}
Then the functions $\Orb_\phi$, resp.~$\Orb_{\phi'}$ lie in $\CC(F_0)$ (cf.~loc.~cit.). Then the matching conditions around zero are essentially given by the extended Fourier transform between nilpotent orbital integrals. Indeed, set (cf.~loc.~cit.)
$$
\kappa_{F/F_0}= e_{F/F_0}L(1,\eta)^{-1}=\begin{cases} 1+q^{-1}, & \text{$F/F_0$ unramified};
\\2, &
\text{$F/F_0$ ramified}.
\end{cases}
$$

\begin{theorem}
\label{thm match nil Orb}The functions $(\phi_0,\phi_1),\phi_i\in C_c^\infty(\fku_{\red,i})$ and $\phi'\in C_c^\infty(\fks_\red)$ are local transfers around zero if and only if
$$
\Orb_{\phi_0}=2\eta(-1) |\varpi|^{-1}\kappa_{F/F_0}^{-1} \wt\Orb_{\phi'},
$$
and 
\[
\begin{aligned}
   -\Orb(0,\phi_0)&=\eta(-1)\Orb(n_{0,+},\phi')+\Orb(n_{0,-},\phi');\\
   \Orb(0,\phi_1)&=\eta(-1)\Orb(n_{0,+},\phi')-\Orb(n_{0,-},\phi').
\end{aligned}
\]
\end{theorem}
\begin{proof} The statement is easily reduced to the corresponding one for $\pi \fks_\red=\fks\fkl_\red$ and $\pi\fku_\red$ which are given by \cite[Prop.~4.4, 4.7]{Z12b}, by comparing the germ expansions on $\fku_\red$ and $\fks_\red$ (and note $\eta(\Delta(x)/\varpi)=(-1)^i$ if $x\in\fkb_{\red,\rs,i}$, cf.~Proposition \ref{eta b01}). Here we note that there is an error in the germ expansion \cite[Th.~2.8(1)(i)]{Z12b}: $\tau$ should be $\tau^{-1}$ ($\tau$ being $\varpi$ in the current notation). This leads to the correction factor $|\varpi|^{-1}$ in the  statement above. 
\end{proof}

\subsection{Matching orbital integrals around $x_0\neq 0$}

\begin{theorem}
\label{thm match x non0}
Let $x_0=(\lambda_0,u_0,w_0)\in \fkb_\red$, where $ x_0\neq 0$. The functions $(\phi_0,\phi_1)$, $\phi_i\in C_c^\infty(\fku_{\red,i})$, and $\phi'\in C_c^\infty(\fks_\red)$ are local transfers around $x_0$ if and only if
the following identities hold:

\begin{altenumerate}
\renewcommand{\theenumi}{\alph{enumi}}
 \item In case $(0i)$, and when $F'\neq F_0\times F_0$, 
\begin{align*}
 \Orb(y_0,\phi_0) &= \Orb(y_+,\phi')+\eta(-\lambda) \Orb(y_-,\phi'),\\
  \Orb(y_0,\phi_1) &= \Orb(y_+,\phi')-\eta(-\lambda) \Orb(y_-,\phi').
\end{align*}
Here $y_0\in \fku_{i, \red}$ is any representative of the unique orbit in $\pi_{\fku_{i}}^{-1}(x_0)$, and $y_\pm\in \fks_\red$ are the representatives given by \eqref{eqn case 0i} of the two non-semi-simple orbits in $\pi_{\fks}^{-1}(x_0)$.

\item  In case $(0ii)$, 
\begin{align*}
\eta(-\alpha) \Orb(y_+,\phi_0)= \Orb(y_{++},\phi')+\Orb(y_{-+},\phi')+\eta(-1)\Orb(y_{+-},\phi')+\eta(-1)\Orb(y_{--},\phi'),\\
\eta(-\alpha)  \Orb(y_-,\phi_0)= \Orb(y_{++},\phi')-\Orb(y_{-+},\phi')-\eta(-1)\Orb(y_{+-},\phi')+\eta(-1)\Orb(y_{--},\phi'),
\end{align*}
 and 
 \begin{align*}
 \eta(-\alpha)\Orb(y_+,\phi_1)= \Orb(y_{++},\phi')+\Orb(y_{-+},\phi')- \eta(-1)\Orb(y_{+-},\phi')-\eta(-1)\Orb(y_{--},\phi'),\\
 \eta(-\alpha) \Orb(y_-,\phi_1)= \Orb(y_{++},\phi')-\Orb(y_{-+},\phi')+\eta(-1)\Orb(y_{+-},\phi')-\eta(-1)\Orb(y_{--},\phi').
\end{align*}
Here $y_+,y_-\in \fku_{i, \red}$ are any representatives of the two semi-simple orbits in $\pi_{\fku_{i}}^{-1}(x_0)$ labeled in \S\ref{sec S 19 ss},  and $y_{\pm\pm}\in \fks_\red$ are the representatives given by \eqref{eqn case 0ii} and \eqref{eqn case 0ii y--} of the four non-semi-simple orbits in $\pi_{\fks}^{-1}(x_0)$.

  \item In  case $(1)$,
\begin{align*}
 \Orb(y_0,\phi_0)&= \Orb(y_+,\phi') +\Orb(y_-,\phi'), \\
  \Orb(y_0,\phi_1)&=\Orb(y_+,\phi') -\Orb(y_-,\phi').
\end{align*}
Here  $y_0\in \fku_{i, \red}$ is any representative of the unique orbit in $\pi_{\fku_{i}}^{-1}(x_0)$, and $y_\pm\in \fks_\red$ are the representatives given by \eqref{eqn case 1} of the two non-semi-simple orbits in $\pi_{\fks}^{-1}(x_0)$.
 
\end{altenumerate}
\end{theorem}

\begin{proof} This follows by comparing Theorem \ref{thm germ ss s} (specialized to $s=0$) and Theorem \ref{thm germ ss u}. Note that $x\in \fkb_{\red,\rs,0}$ if and only if $\eta(\Delta/\varpi)=\eta(-\Delta)=1$ (cf.~Proposition \ref{eta b01}). In case $(0ii)$, we note  the following facts about the section $\sigma$ defined by (\ref{eqn sigma 1}):
\begin{itemize}
\item $\Delta=-4\alpha^2\varpi z_1z_2$ and hence $\eta(\Delta)=\eta(z_1z_2)$.
\item The choice of $y_+$ is such that $\eta(z_1)=1$ for $x\in V_{x_0,\rs,+}$.
\item The transfer factor (\ref{eqn omega n=3}) is given by $\omega(\sigma(x))=\eta(-\alpha)$.\qedhere
\end{itemize}
\end{proof}

\section{Proofs of Theorems \ref{thm germ converse} and  \ref{thm error term}}
\label{sec proof}

\subsection{Proof of Theorem \ref{thm germ converse}}
To show the \emph{``only if''} part, by Theorem \ref{thm germ x=0} and \ref{thm germ ss s}, it suffices to show that the function $\Orb_{\phi'}$ lies in $\CC_1(F_0)$. This is proved in \cite[Lem.~2.3]{Z12b}.

To show the \emph{``if''} part, by \cite[Prop.~3.8]{Z14}, it suffices to show that $\varphi$ is a local orbital integral function around every $x_0\in \fkb_\red$, comp.~Theorem \ref{thm stab const red}. This amounts to showing the following two lemmas.
\begin{lemma}
For each $x_0$, and each \emph{discrete} $n_0\in \pi^{-1}(x_0)$ with nonzero germ function value $\Gamma_{n_0}(x, 0)$, there exists a function $\phi'\in C_c^\infty(\fks_\red)$ such that 
\[
   \Orb(n,\phi') =
	\begin{cases}
		1,  &  n=n_0;\\
      0,  &  \text{$n$ is not in the same orbit as $n_0$.}
\end{cases}
\]
\end{lemma}
\begin{proof}
Though not stated explicitly in \cite{Z12b}, this can be proved in the same way as \cite[Lem.~2.1, 2.3]{Z12b}.
\end{proof}

\begin{lemma}
Every function in $\CC_1(F_0)$ arises as $\Orb_{\phi'}$ for some $\phi'$, and such $\phi'$ can be chosen so that $\Orb(n, \phi')=0$ for the two regular nilpotents $n=n_{0,\pm}$.
\end{lemma}
\begin{proof}
This is proved in the same way as \cite[Lem.~2.3]{Z12b}.
\end{proof}

\subsection{Proof of Theorem \ref{thm error term}}

We need to verify the hypotheses of  Theorem \ref{thm germ converse}.

\smallskip

\noindent\emph{The case $x_0=0$.}
First assume that $i=0$. Then for the one-dimensional family $n(\mu)$,  by Corollary \ref{cor germ=0} the germ function $\Gamma_{n(\mu)}(x,0)$ vanishes identically. By Theorem \ref{thm germ x=0}, we have for $x\in\fkb_{\red,\rs}$ around $x_0=0$,
$$
\varphi(x)=\del(n_{0,+},\phi,0)\, \Gamma_{n_{0,+}}(x,0)+\del(n_{0,-},\phi,0)\, \Gamma_{n_{0,-}}(x,0).
$$
Clearly $\del(n_{0\pm},\phi,0)$ are constants. Therefore the function  satisfies the hypotheses of Theorem \ref{thm germ converse} concerning the summands for these two elements. This proves the case $i=0$.

Now assume that $i=1$. To show that the function in Theorem \ref{thm error term} satisfies the hypotheses of  Theorem \ref{thm germ converse} around $x_0=0$,  it suffices to show that the function 
\[
   \del_{\phi'}(\mu) := \frac d{ds} \Big|_{s=0} \Orb\bigl(n(\mu),\phi',s\bigr),\quad \mu\in F_0,
\]
lies in $\CC_1(F_0)$.

Since we are assuming that $\phi'$ transfers to the zero function on $\fku_{0,\red}$, by Theorem \ref{thm match nil Orb} we have $\Orb_{\phi'}=0$ identically as a function on $F_0$.  The claim follows from the next lemma. 
\begin{lemma}
If $\Orb_{\phi'}=0$, then $\del_{\phi'}\in \CC_1(F_0)$.
\end{lemma}
\begin{proof}
We have
$$
 \Orb\bigl(n(\mu),\phi',s\bigr)=  \bigints \phi'_K\left(\pi\begin{bmatrix} 0 &
\mu ab& b\\
 0 & 0 &0\\ 0& a&0 \end{bmatrix}\right)\eta(ab)|a^{-1}b|^s \,da\,db.
$$
Since $|a|$ and $|b|$ are bounded in the integration, the function $\del_{\phi'}(\mu)$  is locally constant in $\mu\in F_0$. It remains to check the asymptotic behavior when $|\mu|$ goes to $\infty$. 
So assume that $|\mu|\gg 0.$
The same idea as in the proof of Theorem \ref{thm germ ss s} shows that this integral is a sum of two terms, 
$$
 \bigints \phi'_K\left(\pi\begin{bmatrix} 0 &
\mu ab& 0\\
 0 & 0 &0\\ 0& a&0 \end{bmatrix}\right)\eta(ab)|a^{-1}b|^s \,da\,db
$$
and
$$
 \bigints \phi'_K\left(\pi\begin{bmatrix} 0 &
\mu ab& b\\
 0 & 0 &0\\ 0& 0&0 \end{bmatrix}\right)\eta(ab)|a^{-1}b|^s \,da\,db.
$$
These may be rewritten as 
$$
\eta(\mu)|\mu|^{-1-s} \bigints \phi'_K\left(\pi\begin{bmatrix} 0 &b& 0\\
 0 & 0 &0\\ 0& a&0 \end{bmatrix}\right)\eta(b)|a^{-2}b|^s \frac{\,da\,db}{|a|}
$$
and
$$
\eta(\mu)|\mu|^{-1+s} \bigints \phi'_K\left(\pi\begin{bmatrix} 0 &
a& b\\
 0 & 0 &0\\ 0& 0&0 \end{bmatrix}\right)\eta(a)|a^{-1}b^2|^s \frac{\,da\,db}{|b|},
$$
respectively. For simplicity we write the sum as
$$
\eta(\mu)|\mu|^{-1} \bigl( |\mu|^{-s} A(s)+ |\mu|^{s} B(s)\bigr),
$$
where $A(s)$ and $B(s)$ both have a simple pole at $s=0$ with opposite residues. Write the Laurent expansions as 
\[
   A(s)=\frac{A_{-1}}{s}+A_0+A_1s+\dotsb
	\quad\text{and}\quad
	B(s)=\frac{B_{-1}}{s}+B_0+B_1s+\dotsb,
\]
where $A_{-1}+B_{-1}=0$. Then the constant term of the Laurent expansion of 
$
|\mu|^{-s} A(s)+ |\mu|^{s} B(s)$ is
$$(A_0+B_0)+\log|\mu|(-A_{-1}+B_{-1}).
$$
Since $\Orb_{\phi'}=0$ by assumption, we have $A_0+B_0=0$ and $A_{-1}=B_{-1}=0$. This implies that the degree one term in the Laurent expansion of 
$
|\mu|^{-s} A(s)+ |\mu|^{s} B(s)$ is
$$
(A_1+B_1)s+\log|\mu|(-A_0+B_0)s.
$$
We conclude that when $|\mu|\gg 0$,
$$
\del_{\phi'}(\mu)=\eta(\mu)|\mu|^{-1} \bigl((A_1+B_1)+\log|\mu|(-A_0+B_0)\bigr),
$$and hence $\del_{\phi'}$
belongs to $\CC_1(F_0)$, as desired.
\end{proof}

\smallskip

\noindent\emph{The case $x_0\neq 0$.}
This follows easily from the explicit germ expansion Theorem \ref{thm germ ss s}, \ref{thm germ ss u}, and \ref{thm match x non0}, with a similar argument as in the case $x_0=0$. We omit the details.

\smallskip
\noindent With this Theorem \ref{thm error term} is proved.

\end{document}